\documentclass[a4, 12pt]{article}

\usepackage{geometry}
\usepackage{graphicx} 
\usepackage{amsmath}
\usepackage{amssymb}
\usepackage[english]{babel}
\usepackage[ut
f8]{inputenc}
\usepackage{tikz-cd}
\usepackage{tikz-network}
\usetikzlibrary{shapes,backgrounds,patterns,arrows,graphs,bending,knots}
\newcommand{\midarrow}{\tikz \draw[-triangle 90] (0,0) -- +(.1,0);}

\usepackage{amsthm}
\usepackage{mathtools}
\usepackage{enumerate}
\usepackage{mathrsfs}
\usepackage{todonotes}
\usepackage{mathbbol}
\usepackage{float}
\usepackage{pgfplots}
\usepackage[margin=1cm,font=small,labelfont=bf]{caption}

\usepackage{hyperref}

\theoremstyle{plain}
\newtheorem{theorem}{Theorem}[subsection]
\newtheorem{lemma}[theorem]{Lemma}
\newtheorem{corollary}[theorem]{Corollary}
\newtheorem{proposition}[theorem]{Proposition}
\newtheorem{conjecture}[theorem]{Conjecture}

\theoremstyle{definition}
\newtheorem{definition}[theorem]{Definition}

\newtheorem{example}[theorem]{Example}
\newtheorem{remark}[theorem]{Remark}

\usepackage[
backend=biber,
style=alphabetic,
doi=false
]{biblatex}
\title{\textbf{Geometric realisations of type \\ $\tilde{A}_n$ preprojective algebras \\
in homological mirror symmetry}}
\author{Johan Rydholm }
\date{}

\begin{document}

\maketitle
\begin{abstract}
The type $A_n$-singularity $\mathbb{C}^2/\mathbb{Z}_{n+1}$ can be resolved by hyper-Kähler manifolds $X_{\zeta}$ with underlying smooth manifolds diffeomorphic to the resolution of singularities $X_{\text{res}}$, whose hyper-Kähler structure depends on a parameter $\zeta\in H_2(X_{\text{res}};\mathbb{R})$. The structure as a complex manifold of each such hyper-Kähler manifold is equivalent to the resolution of singularities at the poles and the structure of a Milnor fibre with roots determined by $\zeta$ elsewhere; the symplectic structure is exact along the equator and is deformed by areas depending on $\zeta$ on the exceptional $(-2)$-spheres away from the equator.

We show that removing suitable divisors $D_u$ from a fixed $X_{\zeta}$ varying with $u$ in the
underlying upper hemisphere of the $S^2$-family of Kähler-structures yields a log Calabi--Yau hyper-Kähler family (in particular a family of log Calabi--Yau submanifolds), and that mirror symmetry is satisfied (partly conjectural in one direction) for this family by hyper-Kähler rotation, in particular by interchanging the structures over the equator and the pole. We furthermore show homological mirror symmetry after adding the missing divisors, which is related to attaching stops and computing singularity categories of certain Landau--Ginzburg potentials on the $A$-side and $B$-side, respectively.  

More concretely: we compute wrapped Fukaya categories and compare them with (previous and new) computations of derived categories of coherent sheaves and derived categories of singularities in algebraic geometry. We show that the relevant categories (with two exceptions) are triangulated equivalent to module categories over the additive and the multiplicative preprojective algebras of type $\tilde{A}_n$, or to deformations of these algebras depending on the parameters $\zeta$.
\end{abstract}

\tableofcontents
\section{Introduction}
The primary goal of this work is to explore symplectic invariants of the complex 2-dimensional hyper-Kähler manifolds obtained by resolving or smoothing an $A_n$-singularity, and to relate them to algebro-geometric invariants as predicted by Kontsevich homological mirror symmetry \cite{MR1403918}. This extends previous results in \cite{MR3692968}, \cite{MR4033516} and \cite[Section 9.2]{MR3502098}, where different types of related symplectic invariants in the form of Fukaya categories were computed.

Homological mirror symmetry is a collection of ideas, theories and conjectures. At its core are several formulations of a conjecture asserting, roughly, that certain geometric objects that possess both a symplectic and an algebraic structure, such as Kähler manifolds, occur in mirror pairs. The symplectic and the algebraic structures give rise to invariants known as the $A$-side and the $B$-side, respectively. Typically these invariants are in the form of triangulated categories, such as the derived category of coherent sheaves on the $B$-side or the different versions of derived Fukaya categories on the $A$-side. That two objects form a mirror pair means that they define equivalent invariants when interchanging the $A$-side and the $B$-side.

The singularities of type $ADE$ are well known from the McKay correspondence. One instance of this correspondence associates these singularities with a corresponding extended Dynkin quiver. In \cite{MR1752785}, the derived McKay correspondence was established: The derived category of coherent sheaves on the resolution of a singularity of type $ADE$ is shown to be triangulated equivalent with the derived category of modules over the preprojective algebra over the corresponding extended Dynkin quiver. The main motivation for this article was to construct the mirrors of these resolutions as Weinstein sectors, and to use the Chekanov--Eliashberg dg-algebras of the attaching links of the critical handles in the boundary at infinity of the sectors to prove the asserted equivalence (via the surgery formula \cite{MR2916289}, \cite{MR4634745}); this by realising the preprojective algebras as minimal models for such geometrically defined dg-algebras. We carry this out for type $A_n$ singularities, in what forms the main construction of the article (Section~\ref{geometry} and Section~\ref{wrappedsec}). 

In \cite{MR992334}, the $ADE$ singularities were related to a class of hyper-Kähler manifolds called ALE spaces. These hyper-Kähler manifolds, constructed from a hyper-Kähler quotient, can be taken to have the resolution of singularities as its algebraic structure at the poles (a hyper-Kähler manifold has an $S^2$-family of Kähler structures). In the present article, it is conjectured (depending on structural results of the wrapped Fukaya category in the non-exact settings that are not yet developed) that this hyper-Kähler structure is fundamental for homological mirror symmetry in this situation, and that the mirror of the Kähler manifold at the pole is obtained by hyper-Kähler rotation to the family of Kähler manifolds along the equator (cf. \cite{bruzzo1998mirrorsymmetryk3surfaces}). More precisely, such a hyper-Kähler rotation should be carried out in a log Calabi--Yau hyper-Kähler family, obtained by continuously removing a divisor from the hyper-Kähler manifold (such that the divisor can be extended to an anti-canonical divisor in a compactification). We produce such a family of (anti-canonical) divisors in the $A_n$-singularity case, and we show mirror symmetry results in both directions, modulo the analytical details of the wrapped Fukaya category in the non-exact setting: the $A$-side and $B$-side of the variety at the poles are related, respectively, to the $B$-side and the $A$-side of the varieties along the equators. (Similar results for the homological mirror symmetry of type $A_n$ resolutions are obtained in the setting of the Fukaya category for compact Lagrangians in \cite[Section 9.2]{MR3502098}.)

The non-log Calabi--Yau setting --- including the case that was the original motivation for the article, of the mirror of the resolution of $A$-singularities --- is then obtained from the log Calabi--Yau hyper-Kähler family by either adding stops or Landau--Ginzburg potentials on the mirror. This leads to a total of three examples of homological mirror symmetry. For each of these examples, we rigorously establish one of the directions, and we give a possible interpretation and sketch of the other (here, as above, certain foundational results concerning the invariants on the symplectic side in the non-exact setting are missing). Interestingly, this relates to deformations of the preprojective algebras, and to the deformed and non-deformed multiplicative versions of the preprojective algebras of type $\tilde{A}_n$.

The article contains extensive background material and preliminaries, likely known to many readers, which, while resulting in greater length, allows for a unified treatment.

Below is an outline of the results, the main arguments, and some of the connections to related previous works. 

\subsection{Results}
Section~\ref{algsec} is mainly for background and motivation: it is intended to provide the context for the results in the article. In it we briefly recall the McKay correspondence and its derived version. We recall the definition of the Ginzburg dg-algebra $\mathcal{G}(Q)$ of a quiver $Q$, and of the preprojective algebra $\Pi(Q)=H_0(\mathcal{G}(Q))$, as well as their multiplicative counterparts, and the deformed versions of these (dg-)algebras. In particular we recall the statement that $D^b(\text{Coh}(X_{\text{res}}))=D^b(\Pi(\tilde{A}_n))$, where $X_{\text{res}}$ is the resolution of singularities of $\mathbb{C}^2/\mathbb{Z}_{n+1}$, i.e. the $A_n$-resolution of singularities (or more generally, the statement holds for any finite subgroup $G\subset\text{SU}(2)$ and its corresponding extended Dynkin quiver). 

We begin by establishing the following purely algebraic result:
\begin{theorem}[Theorem~\ref{formal}]\label{formall}
Let $\tilde{A}_n$ be the (cyclically ordered) quiver with underlying graph $\tilde{A}_n$. Then the homology of $\mathcal{G}(\tilde{A}_n)$ is concentrated in degree $0$, i.e. it is equal to the (additive) preprojective algebra $\Pi(\tilde{A}_n)$ of $\tilde{A}_n$. Furthermore, $\mathcal{G}(\tilde{A}_n)$ is formal, i.e. quasi-isomorphic to its homology.
\end{theorem}

Thereafter, we briefly motivate the article by putting it in the context of homological mirror symmetry, and relate the results to the hyper-Kähler structure on the underlying manifold of the $A_n$-resolution of singularities. Given a finite group $G\subset\text{SU}(2)$, we recall the construction of the hyper-Kähler manifolds $X_{\zeta}$ from \cite{MR992334}. These are smooth manifolds diffeomorphic to the resolution of singularities $X_{\text{res}}$ of $\mathbb{C}^2/G$, each carrying the structure of a hyper-Kähler manifold depending on the parameter $\zeta\in H^2(X_{\text{res}};\mathbb{R})$ (here, in the context of \cite{MR992334}, we assume $\zeta=(\zeta,0,0)$ chosen generic). Each $X_{\zeta}$ considered with its complex structure $I_u$, $u\in S^2$, is equivalent to the resolution of singularities as an algebraic variety at the poles, i.e. when $u=(u_I,u_J,u_K)=(\pm 1,0,0)$, and to a Milnor fibre (smoothening) elsewhere. When regarded as a symplectic manifold $(X_\zeta,\omega_u)$ is exact and independent of $\zeta$ along the equator, (i.e., when $u_I=0$ with $u=(u_I,u_J,u_K)$), and non-exact elsewhere, where each exceptional $(-2)$-sphere $E_i$ has a non-zero symplectic area equal to $\zeta_i$, where $\zeta=(\zeta_1,\dots,\zeta_n)$. (See Example~\ref{HKquotient} and Figure~\ref{hkfamilypic}.)

Specialising to the $A_n$-case, that is, when $G=\mathbb{Z}_{n+1}$, for each $\zeta$ we can find $I_u$-holomorphic divisors $D_u\subset X_{\zeta}$, varying with $u\in S^2$, such that $X_{\zeta}\smallsetminus D_u$ is log Calabi--Yau (Example~\ref{loghkA}); restricting to the upper hemisphere $O\subset S^2$ we can do this continuously. This property is proven as Proposition~\ref{logCYHKprop}, stating that $X_{\zeta}\smallsetminus D_u$, $u\in O\subset S^2$, is a log Calabi--Yau hyper-Kähler family in the sense of Definition~\ref{logCYHKdef}. On the basis of this fact we hypothesise that mirrors for Kähler manifolds in this family should be given by hyper-Kähler rotation, by which we mean that a mirror for $(X_{\zeta}\smallsetminus D_u,I_u,\omega_u)$, for any $u\in S^2$, should be given by $(X_{\zeta}\smallsetminus D_{\check{u}},I_{\check{u}},\omega_{\check{u}})$, for some other $\check{u}\in S^2$. In particular, we conjecture that this mirror symmetry holds for the choices $u=(1,0,0)$ and $\check{u}=(0,u_J,u_K)$. This is proven in the article (almost, with two cases only partially proven). 

Below we outline the details, where we use the notation $X_{\text{res}}$ for any $(X_{\zeta},I_u,\omega_u)$ with $u=(1,0,0)$, and in this case $D_r=D_u$, and for $u=(0,I_J,I_K)$ we denote $(X_{\zeta},I_u,\omega_u)$ by $X_{\text{Mil}}$ and $D_m=D_u$. We start by surveying the symplectic geometry of $X_{\text{Mil}}$ and $X_{\text{Mil}}\smallsetminus D_m$, following Section~\ref{geometry}, which is devoted to these geometric constructions. 

The Milnor fibre introduced in \cite{MR239612} (which, as noted above, is diffeomorphic to the resolution of singularities obtained by repeated blow-ups) can be given the structure of a Weinstein domain \cite{MR3279026}, equivalent to a linear plumbing of $n$ copies of $T^*S^2$. We show in Proposition~\ref{geometriceq} that the Milnor fibre with a fibre removed (the Milnor fibre being the total space of a Lefschetz fibration), which is also a Weinstein domain, is equivalent as Weinstein domains to the cyclic plumbing of $n+1$ cotangent bundles of $S^2$. This is done using a covering argument and the Weinstein Kirby calculus provided by \cite{MR1668563}; see also \cite{MR4913463}.

Cyclic plumbing of cotangent bundles of spheres can be described by handle attachments on $\partial_{\infty}(\mathbb{C}\times\mathbb{C}^*)=(S^1\times S^2,\xi_{\text{std}})$, using Weinstein handlebody diagrams. We show that $(S^1\times S^2,\xi_{\text{std}})$ can be obtained by gluing together two copies $(J^1S^1,\xi_{\text{std}})$. We use this to prove the following theorem.

\begin{theorem}[Theorem~\ref{stopsect}, Lemma~\ref{chords} and Lemma~\ref{curves}]\label{introgeom}
Let $\Lambda_{\textup{stop}}\subset S^1\times S^2$ be the Legendrian given by $S^1\times\{\textup{pt}\}$. Then removing the stop $\sigma=\sigma_{\Lambda_{\textup{stop}}}$ defined by this Legendrian yields a Weinstein sector $W_{\sigma}=\mathbb{C}\times\mathbb{C}^*\smallsetminus \sigma_{\Lambda}$ with boundary at infinity $\partial_{\infty}W_{\sigma}$ equal to a convex open neighbourhood $V$ of the zero-section in $(J^1S^1,\xi_{\textup{std}})$.

Moreover, there is a cutoff Reeb vector field such that in a smaller convex open neighbourhood $U\subset V\simeq \partial_{\infty}W_{\sigma}$ of the zero-section in which the Reeb vector field satisfies $R_{\alpha}=\partial_z$, and no flow line of $R_{\alpha}$ returns to this neighbourhood after leaving it.

Finally, by making a Legendrian $\Lambda$ in $U$ smaller (in the $z$-direction), we can assume that holomorphic curves in the symplectisation of $V$ with punctures at Reeb chords of $\Lambda$ are fully contained in the symplectisation of $U$.
\end{theorem}

In particular, when computing the Chekanov--Eliashberg algebra of Legendrians in the boundary of $W_{\sigma}$, we can perform the computations of the invariant in $(J^1S^1,\xi_{\text{std}})$ (Theorem~\ref{cesector}).

Thus, removing a neighbourhood from the stop in the cyclic plumbing of cotangent bundles of spheres we obtain a Weinstein sector which we can think of as given by attaching handles along the cyclic link of standard unknots in $(J^1S^1,\xi_{\text{std}})$; see Figure~\ref{CEAn}. 

In Section 4 then, we use the result in \cite{MR4695507} which states that the inclusion induces an equivalence between the wrapped Fukaya category of a sector and the partially wrapped Fukaya catgory of the manifold with the corresponding stop. We then apply the surgery formula from \cite{MR2916289} and \cite{MR4634745} in both $(S^1\times S^2,\xi_{\text{std}})$ and $(J^1S^1,\xi_{\text{std}})$, and deduce the following from Theorem~\ref{introgeom}.
\begin{theorem}[Theorem~\ref{surgerycon}]\label{geometrythm}
The Chekanov--Eliashberg algebra $CE(\Lambda)$ of an attaching link $\Lambda\subset J^1S^1$ is quasi-isomorphic to the (non-unital) subalgebra $eCE(\iota(\Lambda)\cup\Lambda_{\textup{stop}})e\subset CE(\iota(\Lambda)\cup\Lambda_{\textup{stop}})$, where $e$ is the sum of the idempotents corresponding to the connected components of $\iota(\Lambda)$, where the latter Chekanov--Eliashberg algebra is calculated in $S^1\times S^2$. Here $\Lambda_{\textup{stop}}$ is the component corresponding to the stop in Theorem~\ref{introgeom} and $\iota:J^1S^1\rightarrow S^1\times S^2$ is the inclusion, as the complement of a neighbourhood of $\Lambda_{\textup{stop}}$.
\end{theorem}

In $S^1\times S^2$ we use the surgery formula together with Drinfeld's dg-quotient \cite{MR2028075} (see Proposition~\ref{quotgen}) to deduce the following theorem, which after surgery means that the cocores of the critical handles split-generate the category generated by attaching stops and handles, that is, that the linking disc of the stop is contained in the split-closure of the category generated by the cocores:
\begin{theorem}[Theorem~\ref{alggen}]
With notation as in Theorem~\ref{geometrythm}, the dg-category of finitely generated dg-modules over $CE(\iota(\Lambda)\cup\Lambda_{\textup{stop}})$ is derived equivalent to the dg-category of finitely generated dg-modules over $CE(\Lambda)$.
\end{theorem}

Using the generation result from \cite{MR4732675} or \cite{MR4695507} together with the surgery formula \cite{MR2916289} and \cite{MR4634745} we thus obtain:
\begin{theorem}[Theorem~\ref{cocoregen}]\label{cocoregenintro}
Let $W$ be the Weinstein manifold obtained by attaching handles to $\iota(\Lambda)\subset S^1\times S^2$. Then there is a triangulated equivalence
\[
D^{\pi}\mathcal{W}(W,\sigma)\simeq D^{\pi}(CE(\Lambda)),
\]
where $\sigma$ is the stop defined by $\Lambda_{\textup{stop}}$.
\end{theorem}

Specialising now to the realisation of $\Pi(\tilde{A}_{n})$, we have the following simple calculation: 
\begin{proposition}[Proposition~\ref{ginzcalc}]\label{ginzcalcintro}
Let $\Lambda_{\textup{cyc}}\subset J^1S^1$ be the attaching link of the critical handles, as shown in Figure~\ref{CEAn}. Then $CE(\Lambda_{\textup{cyc}})\cong\mathcal{G}(\tilde{A}_{n})$, i.e., with correct choices (of e.g. Maslov potential) the Chekanov--Eliashberg algebra of $\Lambda_{\textup{cyc}}$ is isomorphic as dg-algebras to the Ginzburg algebra of the extended $A_{n}$-quiver.
\end{proposition}

From Theorem~\ref{cocoregenintro} and Proposition~\ref{ginzcalcintro} we conclude that:
\begin{theorem}[Theorem~\ref{mainresult}]\label{final}
The partially wrapped Fukaya category of cyclic plumbing of $n+1$ cotangent bundles $T^*S^2$, or equivalently of the Milnor fibre of the $A_n$-singularity with a fibre removed $X_{\textup{Mil}}\smallsetminus D_m$, with the stop $\sigma=\sigma_{\Lambda_{\textup{stop}}}$, is derived equivalent to the category of finitely generated $\Pi(\tilde{A}_n)$-modules:
\[
D^{\pi}\mathcal{W}(X_{\textup{Mil}}\smallsetminus D_m,\sigma)\simeq D^b(\Pi(\tilde{A}_n)).
\]
\end{theorem}

We continue with some remarks. First, Theorem~\ref{formall} can be obtained from the theorem in \cite{MR3503979} by changing the orientation of one of the arrows in the $\tilde{A}_n$-quiver and providing an explicit isomorphism, but is here given an independent direct proof. Second, Theorem~\ref{final} depends on one fact which to the knowledge of the author has no complete proof existing in the literature: There are different definitions of the partially wrapped Fukaya category used in \cite{MR4634745} and \cite{MR4563001} respectively in \cite{MR4695507}, \cite{MR4106794}, which are expected to be equivalent. We need the generation result for Weinstein sectors that is proven in the second version of this category.

In Section~\ref{mirrorsec} we explain the three (families of) homological mirror symmetry statements of the article. We start (Section~\ref{otherdir}) with a sketch of the log Calabi--Yau homological mirror symmetry between the resolution of $A_n$-singularities $X_{\text{res}}$ with a divisor $D_r$ removed and with the $A_n$-Milnor fibre $X_{\text{Mil}}$ with a fibre $D_m$ removed, where the Milnor fibre is obtained by a particular hyper-Kähler rotation (See Figure~\ref{hmslogcyfigure} for a summary). Here, however, one needs to be aware that neither the resolution nor the Milnor fibre are unique when considered respectively as a symplectic manifold and as an algebraic variety (i.e. when swapping sides in the mirror conjecture). Their structures depend on an underlying choice of hyper-Kähler structure, which is explained in Example~\ref{HKquotient} and Example~\ref{loghkA}. 

In particular, the manifolds $X_{\text{res}}\smallsetminus D_r$ considered as symplectic manifolds are not Weinstein, and they depend on area parameters on the Lagrangian skeleton, where different parameters produce non-equivalent symplectic manifolds, having (supposedly) different (non-exact) wrapped Fukaya categories (once a correct definition is given). 

Likewise, the Milnor fibres $X_{\text{Mil}}$, when considered as algebraic varieties, can be described as hypersurfaces given as the vanishing locus of $P(Z)-XY$. Here $P(Z)$ is a polynomial in one variable of degree $n+1$, without repeated roots. 

Considering the resolution of singularities as a $B$-side and the Milnor fibre (with the divisor/fibre removed) as an $A$-side, we prove the following:
\begin{proposition}[Proposition~\ref{bsidecycplumb}]\label{fullfibreremoved}
The derived category of coherent sheaves on $X_{\textup{res}}\smallsetminus D_r$ is given by finitely generated modules over the non-deformed multiplicative preprojective algebras of type $\tilde{A}_n$: 
\[
D^b(\textup{Coh}(X_{\textup{res}}\smallsetminus D_r))\simeq D^b(\Lambda^1(\tilde{A}_n))
\]
\end{proposition}
Since, by \cite{MR4033516} and \cite{MR4582531} (or by direct computation of the Chekanov--Eliashberg algebra of the Legendrian link in Figure~\ref{plumbingpic} using the combinatorial model from \cite{MR3356070}, combined with Remark~\ref{formalityremark})
\[
D^{\pi}\mathcal{W}(X_{\text{Mil}}\smallsetminus D_m)\simeq D^b(\Lambda^1(\tilde{A}_n))
\]
this concludes one half of the mirror symmetry.

Swapping the roles yields (conjecturally) categories quasi-equivalent to finitely generated modules over deformed multiplicative preprojective algebras, where the deformation parameters are determined by areas on the exceptional spheres, which in turn is determined by the underlying choice of hyper-Kähler structure. Alternatively, the parameters can be deduced from the roots of the polynomials defining the Milnor fibres. In this setting we prove the following:
\begin{proposition}[Proposition~\ref{multloc}]\label{fibredremovedtwo}
Let $X_{\textup{Mil}}$ be defined as the zero-set of the polynomial $P(Z)-XY$, where $P(Z)$ is a degree $n+1$ polynomial in one variable without repeated roots. Let $D_m\subset X_{\textup{Mil}}$ be defined by $Z=0$. We have a triangulated equivalence
\[
D^b(\textup{Coh}(X_{\textup{Mil}}\smallsetminus D_m))\simeq D^b(\Lambda^q(\tilde{A}_n)),
\]
for certain parameters $q=(q_1,\dots,q_{n+1})$, $q_i\in\mathbb{C}^{\times}$ such that $q_1\cdots q_{n+1}=1$. 

If we denote the roots of $P(Z)$ by $Q_1,\dots,Q_{n+1}$, the relation between the roots and the parameters is given by $Q_i=\prod_{ j=1}^iq_j$.
\end{proposition}

The relation between the parameters $q_i$ and the underlying hyper-Kähler structure (as explained in Example~\ref{HKquotient}) is as follows (explained in Remark~\ref{moreonfamily} together with Remark~\ref{substrmk}): there are parameters $\lambda_i$ proportional to $\zeta_i$ by a common factor satisfying the relation
\[
\lambda_i=q_1\cdots q_{i-1}(q_i-1).
\]
In particular, the roots satisfy $Q_i=1+\sum_{l=1}^i\lambda_l$.

We note in the sketch of proof of Conjecture~\ref{cycconj} that a deformation of the Chekanov--Eliashberg algebra $CE^q(\iota(\Lambda_{\text{cyc}}))$ of the Legendrian link of attaching spheres for the cyclic plumbing of $T^*S^2$ in Figure~\ref{plumbingpic}, obtained by computing the Chekanov--Eliashberg algebra with parameters $t_i$ corresponding to loop space coefficients and specialising $t_i\mapsto q_i\in k$, is quasi-isomorphic to the deformed multiplicative preprojective algebra $\Lambda^q(\tilde{A}_n)$ (see also \cite{MR4033516}, combined with \cite{MR4582531} or Remark~\ref{formalityremark}). We expect that the parameters $q_i$ in this deformation are determined by the symplectic areas on the exceptional spheres. More precisely, they are expected to be in the field of formal Laurent series $\mathbb{C}((s))$ and equal to $q_i=q_i(s)=e^{s\lambda_i}$, for a (log-)Novikov parameter $s$. Thus, it is also expected that by using generation results analogous to \cite{MR4732675} or \cite{MR4695507} and surgery results analogous to \cite{MR2916289} and \cite{MR4634745}, that this observation implies the following equivalence:
\[
D^{\pi}\mathcal{W}(X_{\text{res}}\smallsetminus D_r)\stackrel{\text{conjecture}}{\simeq}D^b(\Lambda^{q(s)}(\tilde{A}_n)).
\]

In the second half of the homological mirror symmetry equivalence, on the $B$-side, deforming the complex structure is equivalent to the parameters $\lambda_i$, that is, we replace it by $s\lambda_i$. Then the induced deformations on the parameters $q_i$ can be identified with $q_i(s)$, obtained on the $A$-side, as infinitesimal variations, that is, interpreted over the ring $\mathbb{C}[[s]]/(s^2)$. Thus, by slight abuse of notation below, as infinitesimal variations there is a conjectured equivalence,  
\[
D^{\pi}\mathcal{W}(X_{\text{res}}\smallsetminus D_r)\stackrel{\text{conjecture}}{\simeq}D^b(\Lambda^{q(s)}(\tilde{A}_n))\simeq D^b(\text{Coh}(X_{\text{Mil}}\smallsetminus D_m)),
\]
which should be read transitively when $X_{\text{res}}\smallsetminus D_r$ is obtained from $X_{\text{Mil}}\smallsetminus D_m$ (and vice versa) by hyper-Kähler rotation in the same hyper-Kähler manifold (otherwise the parameters $q$ might differ, and hence the derived categories of modules may not be equivalent).

Next, if we do not remove the divisor on the $A_n$-resolution of singularities, i.e. we put back the divisor $D_r$ on the B-side, then, in terms of homological mirror symmetry, this amounts to adding a stop in the $A$-side when computing the wrapped Fukaya category. The main part of the article proved this equivalence, with $X_{\text{res}}$ considered as a $B$-side. Swapping the roles of $A$- and $B$-sides, by computations of the symplectic cohomology on the resolutions in \cite{MR2720232}, the categories should now vanish. We confirm the other side of this statement by providing a Landau--Ginzburg model, i.e. a potential (instead of a stop), on the Milnor fibre (obtained by a certain count of holomorphic curves). We show that this potential has trivial singularity category (or product of singularity categories to be precise, i.e. trivial category of $D$-branes, in the sense of \cite{MR2101296}). 
\begin{proposition}[Section~\ref{sectrivialpot}]
For the potential $w_m:X_{\textup{Mil}}\smallsetminus D_m\rightarrow\mathbb{C}$, defined by counting Maslov index $2$ holomorphic curves with boundary on the Lagrangian skeleton of $X_{\textup{res}}\smallsetminus D_r$, we have
\[
DB(w_m)=\prod_{\lambda}D_{\textup{sg}}(w_m^{-1}(\lambda))\simeq0.
\]
\end{proposition}
Thus, the following are equivalences of triangulated categories:
\[
D^b(\text{Coh}(X_{\text{res}}))\simeq D^b(\Pi(\tilde{A}_n))\simeq D^{\pi}\mathcal{W}(X_{\text{Mil}}\smallsetminus D_m,\sigma), \]
\[
D^{\pi}\mathcal{W}(X_{\text{res}})\simeq 0\simeq DB(w_m).
\]
See Section~\ref{hmsres} and Figure~\ref{hmsresfigure}.

Finally, if we instead choose to not remove the fibre on the Milnor fibre, this space is now equivalent to the linear plumbing of $T^*S^2$, and the wrapped Fukaya category is quasi-equivalent to twisted complexes over the Ginzburg algebra of type $A_n$ \cite{MR3692968} (which is not formal \cite{MR3503979}):
\[
D^{\pi}\mathcal{W}(X_{\text{Mil}})\simeq D^{\pi}(\mathcal{G}(A_n)).
\]

Considering this space as a $B$-side, we have to remember that the complex variety is no longer unique (as above with the fibres removed), but there are many affine varieties, whose underlying symplectic structure is the exact Milnor fibre. We show that their derived categories of coherent sheaves are equivalent to the derived categories of finitely generated modules over deformations of the preprojective algebra of type $\tilde{A}_n$:
\begin{proposition}[Proposition~\ref{nonlocmil}]
Let $X_{\textup{Mil}}$ be defined as the zero-set of the polynomial $P(Z)-XY$, where $P(Z)$ is a degree $n+1$ polynomial in one variable without repeated roots. We have an equivalence 
\[
D^b(\textup{Coh}(X_{\textup{Mil}}))\simeq D^b(\Pi^{\lambda}(\tilde{A}_n)),
\]
for certain parameters $\lambda=(\lambda_1,\dots,\lambda_{n+1})$ such that $\lambda_1+\cdots+\lambda_{n+1}=0$.

If we denote the roots by $\Lambda_i$, then the relation between the roots and the parameters is given by $\Lambda_i=\sum_{j=1}^{i}\lambda_j$, for $i=1,\dots,n+1$.
\end{proposition}

The relation between the parameters $\lambda_i$ and the underlying hyper-Kähler structure (explained in Example~\ref{HKquotient}) was explaine above (see also Remark~\ref{moreonfamily}).

On the other hand, as a symplectic manifold there is no unique resolution of singularities with a stop attached, but rather (as with the fibre removed above) there are many such symplecic manifolds, depending on area parameters $\lambda_i$ on the Lagrangian skeleton (in particular they are not Weinstein). Similarly to the case above, with a divisor removed, we show that the deformed Chekanov--Eliashberg algebra $CE^q(\Lambda_{\text{cyc}})$ of the Legendrian attaching link in Figure~\ref{CEAn} and Figure~\ref{lagAN}, i.e. the $\tilde{A}_n$-link in $J^1S^1$, computed with loop space coefficients $t_i$ which are then specialized to parameters $t_i\mapsto q_i\in k^{\times}$, satisfies:
\begin{proposition}[Lemma~\ref{nonloccechoice}, Lemma~\ref{substlemma}, Proposition~\ref{deformedformal}]
For any field $k$, if $q_1\cdots q_{n+1}=1$, then
\[
CE^q(\Lambda_{\textup{cyc}})\cong\mathcal{G}^{\lambda}(\tilde{A}_n),
\]
isomorphic as dg-algebras, where $\mathcal{G}^{\lambda}(\tilde{A}_n)$ is the deformed Ginzburg algebra
with parameters $\lambda_i=q_1\cdots q_{i-1}(q_i-1)$. In particular $\lambda_1+\cdots+\lambda_{n+1}=0$.

Moreover, $\mathcal{G}^{\lambda}(\tilde{A}_n)$ is formal with homology concentrated in degree $0$ and equal to
$\Pi^{\lambda}(\tilde{A}_n)$. 
\end{proposition}
Thus we conjecture these symplectic manifolds to have non-exact wrapped Fukaya categories quasi-equivalent to modules over deformed preprojective algebras of type $\tilde{A}_n$ (Conjecture~\ref{resconj}). As above, for the non-exact categories we use coefficients in the ring of formal Laurent series $\mathbb{C}((s))$ with a (log-)Novikov parameter $s$. The parameters $\lambda_i(s)$ in the deformation are determined by the area parameters $\lambda_i$ using the relations $\lambda_i(s)=q_1(s)\dots q_{i-1}(s)(q_i(s)-1)$, where $q_i(s)=e^{s\lambda_i}$. That is,
\[
D^{\pi}\mathcal{W}(X_{\text{res}}\smallsetminus D_r,\sigma)\stackrel{\text{conjecture}}{\simeq}D^b(\Pi^{\lambda(s)}(\tilde{A}_n)).
\]
In the exact same manner as above, on the $B$-side, deforming the complex structure is equivalent deforming the parameters $\lambda_i$, replacing them by $s\lambda_i$. Interpreting the deformed preprojective algebras obtained on both sides as having coefficients in $\mathbb{C}[[s]]/(s^2)$ (conjecturally) finishes one half of the the homological mirror symmetry for the Milnor fibre. By slight abuse of notation:
\[
D^{\pi}\mathcal{W}(X_{\text{res}}\smallsetminus D_r,\sigma)\stackrel{\text{conjecture}}{\simeq}D^b(\Pi^{\lambda(s)}(\tilde{A}_n))\simeq D^b(\text{Coh}(X_{\text{Mil}})),
\]
where as above with Proposition~\ref{fullfibreremoved} and Proposition~\ref{fibredremovedtwo}, the statement should be read transitively when $X_{\text{res}}$ and $X_{\text{Mil}}$ are obtained from each other by hyper-Kähler rotation.

Lastly, we find the mirror of the linear plumbing of $T^*S^2$ ($A$-side) in terms of a Landau--Ginzburg model ($B$-side) (Theorem~\ref{respot}). 
In more details: we find that the correct potential $w_r$ has a unique singular value $\lambda=0$, from which deduce that the category of $D$-branes is equivalent to the singularity category of the type $\tilde{A}_n$ preprojective algebra modulo cyclic paths in the original arrows: 
\[ 
DB(w_r)=D_{\text{sg}}(w_r^{-1}(0))\simeq D_{\text{sg}}(\Pi(\tilde{A}_n)/(\alpha_{i-1}\dots\alpha_{i+1}\alpha_i\;|\;i=1,\dots,n+1)).
\]
We show that the latter category is almost equivalent to the perfect derived category of finitely generated dg-modules over the Ginzburg algebra of type $A_n$ (Theorem~\ref{dbsgcatalg}). To be precise, there is a Novikov type parameter $T$, believed to be due to the fact that we are considering the cyclic plumbing of $n+1$ copies of $T^*S^2$ as non-exact submanifold of the linear plumbing of $n$ copies of $T^*S^2$. We summarise:
\begin{theorem}[Theorem~\ref{respot}, Lemma~\ref{dbsggeom}, Theorem~\ref{dbsgcatalg}]\label{resstopped}
We have the following equivalences of triangulated categories:
\[
DB(w_r)\simeq D_{\textup{sg}}(\Pi(\tilde{A}_n)/(\alpha^{n+1}))\simeq D^{\pi}(\mathcal{G}(A_n)\otimes_{\mathbb{C}}\mathbb{C}[T]),
\]
where $|T|=-2$. (Here we use the notation $\alpha=\alpha_1+\cdots+\alpha_{n+1}$.)
\end{theorem}
This concludes the other half of the homological mirror symmetry for $A_n$-Milnor fibres, where the apparent asymmetry due to the Novikov parameter $T$ is explained in Remark~\ref{parameterremark}. See Section~\ref{hmsmil} and Section~\ref{hmsmiltwo}, and Figure~\ref{hmsmilfigure} for a summary of these sections.

\subsection{Connections to other works}
We here mention some important related works. We do not claim that this list is exhaustive.

In \cite{MR3692968} it is shown that for the Weinstein $4$-manifold constructed by plumbing copies of $T^*S^2$ according to the $A_n$ or $D_n$ graphs (and conjecturally for the $E_6$, $E_7$ and $E_8$ graphs) the wrapped Fukaya category is quasi-equivalent to the category finitely generated modules over the corresponding Ginzburg algebra. In \cite{MR4033516} the following related result is proved: the manifold constructed by plumbing of $T^*S^2$ according to any graph has wrapped Fukaya category quasi-equivalent to finitely generated dg-modules over the derived multiplicative preprojective algebra (it is also shown to hold in the higher genus case). By \cite{MR4457411} the Ginzburg algebra is quasi-isomorphic to the derived multiplicative preprojective algebra precisely when the underlying quiver is of type $ADE$. Thus these two papers together provides a new proof of the $A_n$ and $D_n$ cases in \cite{MR3692968}, and extends the results to also include the cases $E_6$, $E_7$ and $E_8$.

Furthermore, in combination with \cite{MR4582531}, the result in \cite{MR4033516} shows that the wrapped Fukaya category of the cyclic plumbing (according to $\tilde{A}_n$) of $T^*S^2$ is quasi-equivalent to finitely generated modules over the preprojective algebra. This results can also be recovered by the results in this work by using that stop removal corresponds to localisation. 

Similar computations for the $\tilde{A}_0$-case were done in out in \cite{MR4489822}.

The homological mirror symmetry of $X_{\text{Mil}}\smallsetminus D_m$ and $X_{\text{res}}\smallsetminus D_r$ for the compact Fukaya category is computed in \cite{MR3022713}, from the SYZ viewpoint.

The same mirror symmetry result is discussed by Auroux, Abousaid, Katzarkov in \cite[Section 9.2]{MR3502098}, where they show that $X_{\text{res}}\smallsetminus D_r$ and $X_{\text{Mil}}\smallsetminus D_m$ are SYZ-mirrors, including a discussion of the related Landau--Ginzburg potentials. However, they only do one side of the Mirror symmetry and they do not construct $D^b(\text{Coh}(X_{\text{res}}))$ as a symplectic category (they only construct a full subcategory).

Mirrors of the $A_n$-Milnor fibre is also computed in terms of Landau--Ginzburg potential by Lekili and Ueda in \cite{MR4371540}. Note however that their Landau--Ginzburg potentials are in a different $G$-equivariant setting. It would be interesting to compare these different Landau--Ginzburg models for the same A-side.
\subsection{Outlook}
One would like to extend the results in this work to the cases $D_n$, $E_6$, $E_7$ and $E_8$, with their corresponding finite subgroups $G\subset\text{SU}(2)$. The problem is that there is no $G$-invariant conic $\mathbb{C}^2$, and thus the strategy used here breaks down. 

This is related to that plumbing according to the extended graphs $\tilde{D_n}$, $\tilde{E}_6$, $\tilde{E}_7$ and $\tilde{E}_8$, is known to have wrapped Fukaya category given by twisted complexes over the (non-deformed) multiplicative preprojective algebra, rather than the (additive) preprojective algebra. This is seen if one computes the Chekanov--Eliashberg algebra of the corresponding Legendrian attachment link: there will be higher order terms in the differential.
\subsection*{Acknowledgements}
\addcontentsline{toc}{subsection}{Acknowledgements}
I would to thank my PhD supervisor Georgios Dimitroglou Rizell for suggesting the topic of realisation of preprojective algebras of type $ADE$ as partially wrapped Fukaya categories. I would also like to thank him for helpful discussions, careful reading and numerous suggestions of improvements of earlier drafts. Further, I would like to thank Johan Asplund and Martin Bäcke for valuable discussions of various topics related to this work.
\section{HMS and McKay correspondence}\label{algsec}
We give a brief overview over the algebraic side, or the $B$-side, on the resolution of singularities and in particular of the derived McKay correspondence. This will serve as a motivation for the main calculation of the paper. Thereafter we give a short introduction to homological mirror symmetry via the standard example of $T^*S^1$. The structure of these computations mimics the structure of the main results in this article, concerning the invariants related to the log Calabi--Yau hyper-Kähler family which we describe in Example~\ref{loghkA}.
\subsection{Background on algebras, DGAs, and derived categories}\label{algprel}
\subsubsection{Quivers and algebras}
A \emph{quivers} is a directed graph, that is, a quadruple $Q=(Q_0,Q_1,s,t)$, where $Q_0$ is the set of vertices, $Q_1$ is the set of arrows, and where $s,t:Q_1\rightarrow Q_0$ are the maps determining the source respectively the target of each arrow. Thus a quiver defines an underlying (non-oriented) graph by simply forgetting the orientation of the arrows. 

We will mostly work over the field $k=\mathbb{C}$, although most definitions makes sense for a general commutative ring $k$, and many results also hold more generally. A representation of a quiver is an assignment $Q_0\ni i\mapsto V_i$ of a $k$-vector space to each vertex, and a linear map $F_{\alpha}:V_i\rightarrow V_j$, for each arrow $i\xrightarrow{\alpha} j$. 

Concatenation (we write this as with composition) defines a multiplication on the set of paths of a quiver (where a trivial path $e_i$ for each vertex $i$ is included). By letting $kQ$ be the vector space spanned by all paths, equipped with the multiplication induced by concatenation we get an algebra called the \emph{path algebra} of $Q$. Modules over this algebra are in one-to-one correspondence with representations of the quiver.

We denote by $e_i$ the trivial path, or idempotent, corresponding to the vertex $i$, and by $R=kQ_0$ the semi-simple ring of the trivial paths (the multiplication is defined by $e_i^2=e_i$ and $e_ie_j=0$, if $i\not=j$; also note $1=\sum_{i\in Q_0}e_i$ is the identity in $kQ$). Typically we will be interested in $R$-algebras $A=kQ/I$, where $I$ is some ideal.

Given a quiver $Q$ the following construction will play a central role:
\begin{definition}\label{defpreproj}
Let $\overline{Q}=(\overline{Q}_0,\overline{Q_1})$ denote the \emph{doubled quiver}, that is, $\overline{Q}_0=Q_0$ and $\overline{Q}_1=Q_1\cup Q_1^*$, where $Q_1^*=\{\alpha^*:j\rightarrow i\,|\,\alpha:i\rightarrow j\in Q_1\}$. Let $I$ be the following ideal in the path algebra of the doubled quiver $k\overline{Q}$,
\[
I=\sum_{i}\left(\sum_{\substack{\alpha\in Q_1 \\ s(\alpha)=i} }\alpha^*\alpha-\sum_{\substack{\beta\in Q_1 \\ t(\beta)=i} }\beta\beta^*\right),
\]
where the first sum runs over all vertices $i\in Q_0$. We define the \emph{preprojective algebra} $\Pi(Q)$ of $Q$ to be $k\overline{Q}/I$.
\end{definition}
\begin{example}
Consider the $D_4$ and $\tilde{A}_3$ cases. We pick the orientation so that $Q$, in each case, is given by the quivers below: 
\begin{figure}[H]
\centering
\begin{tikzpicture}
\node at (-1,2) {${Q}_{\tilde{A}_3}$:};
\node at (6,2) {${Q}_{D_4}$:};

\draw[fill=black] (0,0) circle (0.06cm);
\draw[fill=black] (2,-2) circle (0.06cm);
\draw[fill=black] (4,0) circle (0.06cm);
\draw[fill=black] (2,2) circle (0.06cm);

\draw[fill=black] (7,1.5) circle (0.06cm);
\draw[fill=black] (9,1.5) circle (0.06cm);
\draw[fill=black] (11,1.5) circle (0.06cm);
\draw[fill=black] (9,-0.5) circle (0.06cm);

\draw[black,stealth-] (0.2,0.2)  to (1.8,1.8) ;
\draw[black,-stealth] (0.2,-0.2)  to (1.8,-1.8) ;
\draw[black,-stealth] (2.2,-1.8)  to (3.8,-0.2) ;
\draw[black,stealth-] (2.2,1.8)  to (3.8,0.2) ;

\node at (0.8,1.3) {$\delta$};
\node at (3.2,1.3) {$\gamma$};
\node at (3.2,-1.3) {$\beta$};
\node at (0.8,-1.3) {$\alpha$};

\draw[black,-stealth] (7.2,1.5)  to (8.8,1.5) ;
\draw[black,-stealth] (9.2,1.5)  to (10.8,1.5) ;
\draw[black,-stealth] (9,1.3)  to (9,-0.3) ;

\node at (8,1.7) {$\alpha$};
\node at (10,1.7) {$\beta$};
\node at (9.2,0.5) {$\gamma$};
\end{tikzpicture}
\caption{Example of quivers whose underlying graphs are of type $\tilde{A}_3$ respectively $D_4$.}
\label{quiversexample}
\end{figure}
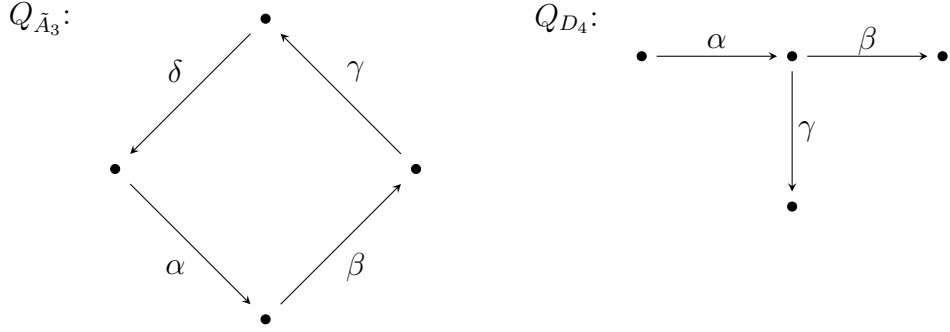
Then the preprojective algebra in each case is defined as the path algebra of the following quivers modulo the ideals $I_{\tilde{A}_3}=\langle\alpha^*\alpha-\delta\delta^*,\,\beta^*\beta-\alpha\alpha^*,\,\gamma^*\gamma-\beta\beta^*,\,\delta^*\delta-\gamma\gamma^*\rangle$, respectively $I_{D_4}=\langle\alpha^*\alpha,\,\beta^*\beta+\gamma^*\gamma-\alpha\alpha^*,\,-\beta^*\beta \rangle$:
\begin{figure}[H]
\centering
\begin{tikzpicture}
\node at (-1,2) {$\overline{Q}_{\tilde{A}_3}$:};
\node at (6,2) {$\overline{Q}_{D_4}$:};

\draw[fill=black] (0,0) circle (0.06cm);
\draw[fill=black] (2,-2) circle (0.06cm);
\draw[fill=black] (4,0) circle (0.06cm);
\draw[fill=black] (2,2) circle (0.06cm);

\draw[black,-stealth] (0.06,0.2) [out=65, in=-155] to (1.8,1.94) ;
\draw[black,stealth-] (0.2,0.06) [out=25, in=-115] to (1.94,1.8) ;

\draw[black,-stealth] (2.06,-1.8) [out=65, in=-155] to (3.8,-0.06) ;
\draw[black,stealth-] (2.2,-1.94) [out=25, in=-115] to (3.94,-0.2) ;

\draw[black,-stealth] (0.2,-0.06) [out=-25, in=115] to (1.94,-1.8) ;
\draw[black,stealth-] (0.06,-0.2) [out=-65, in=155] to (1.8,-1.94) ;

\draw[black,-stealth] (2.2,1.94) [out=-25, in=115] to (3.94,0.2) ;
\draw[black,stealth-] (2.06,1.8) [out=-65, in=155] to (3.8,0.06) ;

\node at (1.3,-0.5) {$\alpha$};
\node at (0.7,-1.6) {$\alpha^*$};

\node at (2.7,-0.5) {$\beta$};
\node at (3.3,-1.6) {$\beta^*$};

\node at (3.3,1.55) {$\gamma^*$};
\node at (2.7,0.4) {$\gamma$};

\node at (0.7,1.5) {$\delta^*$};
\node at (1.3,0.4) {$\delta$};

\draw[fill=black] (7,1.5) circle (0.06cm);
\draw[fill=black] (9,1.5) circle (0.06cm);
\draw[fill=black] (11,1.5) circle (0.06cm);
\draw[fill=black] (9,-0.5) circle (0.06cm);

\draw[black,stealth-] (7.2,1.4) [out=-20, in=-160] to (8.8,1.4) ;
\draw[black,-stealth] (7.2,1.6) [out=20, in=160] to (8.8,1.6) ;

\draw[black,stealth-] (9.2,1.4) [out=-20, in=-160] to (10.8,1.4) ;
\draw[black,-stealth] (9.2,1.6) [out=20, in=160] to (10.8,1.6) ;

\draw[black,stealth-] (8.9,1.3) [out=-110, in=110] to (8.9,-0.3) ;
\draw[black,-stealth] (9.1,1.3) [out=-70, in=70] to (9.1,-0.3) ;

\node at (8,2) {$\alpha$};
\node at (8,1) {$\alpha^*$};
\node at (10,2) {$\beta$};
\node at (10,1) {$\beta^*$};
\node at (8.5,0.5) {$\gamma^*$};
\node at (9.5,0.5) {$\gamma$};
\end{tikzpicture}
\caption{The doubled quivers corresponding to the quivers in Figure~\ref{quiversexample}, which are used to define the preprojective algebras.}
\label{quiversdoubled}
\end{figure}
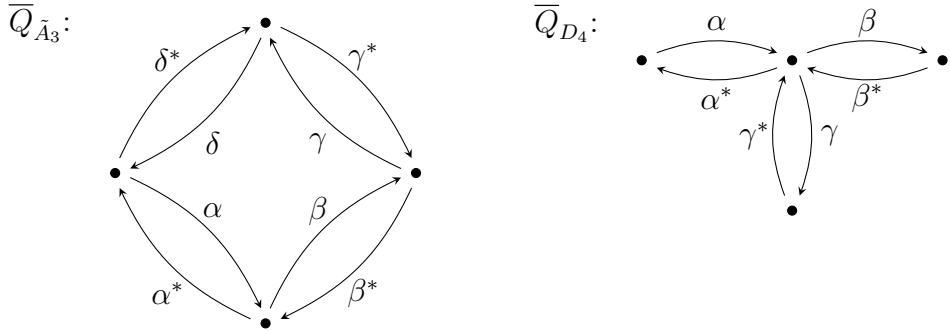
\end{example}
\begin{remark}
By $\Pi(\tilde{A}_n)$ we will mean $\Pi(Q_{\tilde{A}_n})$ where we assume that $Q_{\tilde{A}_n}$ is cyclically oriented (see Figure~\ref{Anquiv}). We will use the analogous notation with $A_n$, where we have an induced (linear) orientation as a subquiver, as well as for the deformed preprojective algebra, for the Ginzburg algebra and for the multiplicative versions, whose definitions are recalled below.
\end{remark}

Later we will also consider \emph{deformed preprojective algebras}, as introduced in \cite{MR1620538}. Given a quiver $Q$ and $\lambda_i\in k$, $i\in Q_0$, it is defined as the quotient $\Pi^{\lambda}(Q)=k\overline{Q}/I$, where
\[
I=\sum_{i}\left(\sum_{\substack{\alpha\in Q_1 \\ s(\alpha)=i} }\alpha^*\alpha-\sum_{\substack{\beta\in Q_1 \\ t(\beta)=i} }\beta\beta^*+\lambda_ie_i\right),
\]

Another related important algebra is the multiplicative preprojective algebra, first studied in \cite{MR2204754}. 
\begin{definition}\label{defmultpreproj}
Given a quiver $Q$, consider a spanning tree $T\subset Q$ (i.e. a subgraph containing all vertices, and with no cycles). Let $\overline{Q}_M=\overline{Q}\cup\{z_{\alpha}^{\pm 1}\;|\;\alpha\in Q_1\smallsetminus T_1\}$. Here the arrows $z_{\alpha}^{\pm 1}$ are loops at the vertices $s(\alpha)$.

The \emph{deformed multiplicative preprojective algebra} with parameters $q_i\in k^{\times}$, for $i\in Q_0$, is the algebra $\Lambda^q(Q)=k\overline{Q}_M/I$, where $I$ is the ideal generated by the relations
\begin{align*}
z_{\alpha}z_{\alpha}^{-1}&=z_{\alpha}^{-1}z_{\alpha}=e_{s(\alpha)}, \\
z_{\alpha}&=e_{s(\alpha)}+\alpha^*\alpha, \\
\prod_{t(\alpha)=i}(e_i+\alpha\alpha^*)&=\prod_{s(\alpha)=i}(q_ie_i+q_i\alpha^*\alpha).
\end{align*}
Here one has to choose an ordering on the vertices to define the multiplications above, however the algebra is  independent of this ordering, as well as of the spanning tree \cite{MR4033516}.
\end{definition}

\begin{example}\label{multanex}
We consider the example $Q=Q_{\tilde{A}_{n}}$, as in Figure~\ref{Anquiv}, and we take the spanning tree $T$ to be tree obtained by excluding $\alpha_{n+1}$. Then the quiver $\overline{Q}_M$ is $\overline{Q}$ with an extra pair of loops $z^{\pm 1}$ at the vertex $n+1$, satisfying $zz^{-1}=z^{-1}z=e_{n+1}$, and the other relations are $z=e_{n+1}+\alpha_{n+1}^*\alpha_{n+1}$ and $e_i+\alpha_{i-1}\alpha_{i-1}^*=q_ie_i+q_i\alpha_i^*\alpha_i$ (where the indices are read modulo $n+1$). We denote this algebra by $\Lambda^q(\tilde{A}_{n})$.
\end{example}
If $q_i=1$, for all $i$, then we get the \emph{non-deformed multiplicative preprojective algebra} $\Lambda(Q)$. In characteristic $0$ the non-deformed multiplicative preprojective algebra is isomorphic (as $R$-algebras) to the preprojective algebra if and only if $Q$ is of type $ADE$ \cite{MR4457411}. 

\subsubsection{Derived categories}\label{derivedcat}
For an algebra $A$, We will generally be interested in the category of finitely generated (right) $A$-modules, which we denote by $A$-mod, and its derived version, which we denote by $D^b(A)$. The latter can be defined as the Verdier quotient/localisation at quasi-isomorphism of the homotopy category of chain complexes of finitely generated $A$-modules which are bounded in homology. 

All objects of $D^b(A)$ can be represented by complexes of projective modules (that are bounded in homology). A complex of $A$-modules is called \emph{perfect} if it is quasi-isomorphic to a bounded complex of finitely generated projective modules. We have a triangulated inclusion $\text{Perf}(A)\subset D^b(A)$, where $\text{Perf}(A)$, called the \emph{perfect derived category} is the Verdier quotient of perfect complexes at quasi-isomorphisms.

We will also use the notation $\text{thick}\langle X_1,\dots,X_n\rangle$ to denote the smallest triangulated subcategory of $D^b(A)$ containing the (complexes of) $A$-modules $X_1,\dots,X_n$ and their summands.

\subsubsection{The triangulated category of singularities}\label{prelsingcat}
We define the \emph{triangulated category of singularities} of a Noetherian algebra $A$ by
\[
D_{\text{sg}}(A)=\frac{D^b(A)}{\text{Perf}(A)}.
\]
This measure how singular $A$ is. In particular, if $A$ is regular then $D^b(A)=\text{Perf}(A)$, and thus $D_{\text{sg}}(A)=0$. It was first studied in \cite{MR4390795} and later in \cite{MR2101296}, where it was defined also for schemes and put in the context of homological mirror symmetry, used to define Landau--Ginzburg models.

For a scheme $X$, a \emph{perfect complex} is a complex of sheaves of $\mathcal{O}_X$-modules which locally is quasi-isomorphic to a bounded complex of locally free sheaves. Denote the full triangulated subcategory of perfect complexes by $\text{Perf}(X)$. Then the \emph{singularity category} of the scheme $X$ is defined as
\[
D_{\text{sg}}(X)=\frac{D^b(\text{Coh}(X))}{\text{Perf}(X)}.
\]
As for algebras: if $X$ is smooth (and satisfies the ELF condition in the sense of \cite[Section 1.2]{MR2101296}, i.e., separated, Noetherian, of finite Krull dimension and $\text{Coh}(X)$ has enough locally free sheaves) then $D_{\text{sg}}(X)=0$ \cite[Remark 1.9]{MR2101296}. 

In \cite{MR2437083}, the construction is generalised to any triangulated category $\mathcal{D}$, as the quotient with the full subcategory of homologically finite objects: $\mathcal{D}/\mathcal{D}_{\text{hf}}$ (the definitions coincide when $X$ satisfies the ELF condition). Furthermore, in this situation it implies that a triangulated equivalence $D^b(\text{coh}(X))\simeq D^b(A)$, given e.g. by a non-commutative model (algebra) $A$, induces an equivalence $D_{\text{sg}}(X)\simeq D_{\text{sg}}(A)$.

Typically one is interested in studying $D_{\text{sg}}(X)$, where $X=w^{-1}(0)$, for a potential $:Y\rightarrow\mathbb{C}$. In the case when $Y$ is smooth and affine, $D_{\text{sg}}(X)$ is equivalent tot the usual Landau--Ginzburg model given by the category of matrix factorisation for $w$ \cite{MR2101296}.

\subsubsection{Dg-algebras and dg-categories}
Let $R$ be a commutative ring (typically we think of it as $R=kQ_0$). A \emph{dg-algebra} is a graded $R$-algebra $A$ with differential $\partial$ of degree $-1$ (i.e. decreasing the degree by $1$) satisfying the graded Leibniz rule:
\[
\partial(xy)=\partial x\cdot y+(-1)^{|x|}x\partial y,
\]
for $x,y\in A$, with $x$ homogeneous of degree $|x|$.

From now on $A=kQ/I$ will be a dg-algebra over $R=kQ_0$ whose structure as an algebra will be determined by concatenation of arrows in the quiver $Q$. We will mainly be working with dg-algebras arising as so called Chekanov--Eliashberg algebras (Section~\ref{cesection}). They are semi-free in the following sense:
\begin{definition}
A dg-algebra $A$ is \emph{semi-free} over $R=kQ_0$ if $A=kQ$ and where the set of arrows $Q_1$ has a $\mathbb{Z}$-grading which is bounded from below.
\end{definition}
The vertex set $Q_0$ will always be finite but $Q_1$ can possibly be infinitely countable.

Note that given a grading on the arrows of a quiver $Q$, we can extend this to a grading on $A=kQ$ by $\text{deg}(\beta\alpha)=\text{deg}(\beta)+\text{deg}(\alpha)$. Furthermore, we can define a differential $\partial$ on the arrows and extend this to $A$ by the graded Leibniz rule. The following example will be of interest.
\begin{definition}\label{defginzb}
Given a quiver $Q=(Q_0,Q_1)$, let $\tilde{Q}=(\tilde{Q}_0,\tilde{Q}_1)$ be the quiver given by $\tilde{Q}_0=Q_0$ and $\tilde{Q}_1=\overline{Q}_1\cup Q_0^{\text{loop}}$, where $Q_0^{\text{loop}}=\{\zeta_i:i\rightarrow i\,|\,i\in Q_0\}$. The grading is given by $|\alpha|=|\alpha^*|=0$ and $|\zeta_i|=1$. The differential is defined on generators by $\partial(\alpha)=\partial(\alpha^*)=0$ and
\[
\partial\zeta_i=\sum_{\substack{\alpha\in Q_1 \\ s(\alpha)=i} }\alpha^*\alpha-\sum_{\substack{\beta\in Q_1 \\ t(\beta)=i} }\beta\beta^*.
\]
The resulting algebra is known as the \emph{Ginzburg algebra} $\mathcal{G}(Q)$ (without potential) of $Q$.
\end{definition}
\begin{example}
The Ginzburg algebras are defined as the differential graded algebra obtained as the path algebra of the following quivers, with gradings given by $|\alpha|=|\alpha^*|=\beta|=|\beta^*|=|\gamma|=|\gamma^*|=|\delta|=|\delta^*|=0$ and $|\zeta_i|=1$, respectively $|\alpha|=|\alpha^*|=|\beta|=|\beta^*|=|\gamma|=\gamma^*|=0$ and $|\zeta_i|=1$, and with differentials $\partial \zeta_1=\alpha^*\alpha-\delta\delta^*$, $\partial \zeta_2=\beta^*\beta-\alpha\alpha^*$, $\partial \zeta_3=\gamma^*\gamma-\beta\beta^*$ and $\partial \zeta_4=\delta^*\delta-\gamma\gamma^*$: 
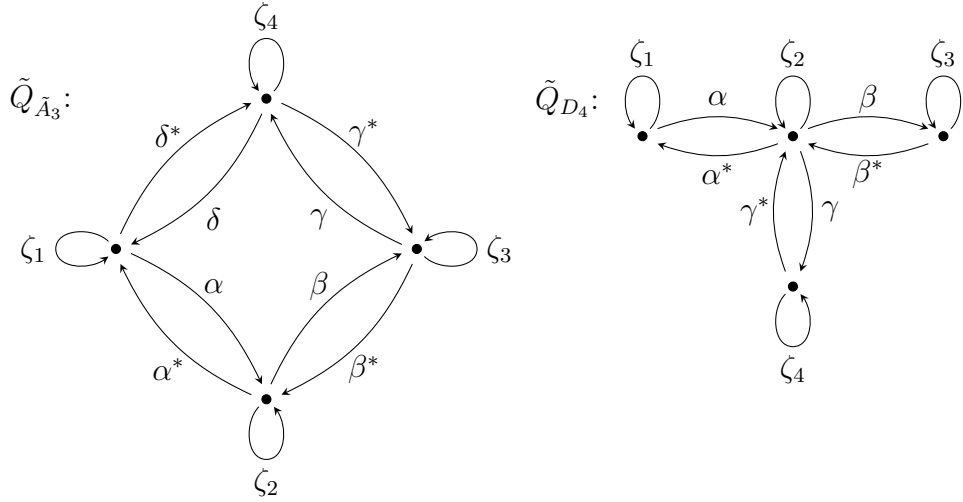
\begin{figure}[H]
\centering
\begin{tikzpicture}
\node at (-1,2) {$\tilde{Q}_{\tilde{A}_3}$:};
\node at (6,2) {$\tilde{Q}_{D_4}$:};

\draw[fill=black] (0,0) circle (0.06cm);
\draw[fill=black] (2,-2) circle (0.06cm);
\draw[fill=black] (4,0) circle (0.06cm);
\draw[fill=black] (2,2) circle (0.06cm);

\draw[black,-stealth] (0.06,0.2) [out=65, in=-155] to (1.8,1.94) ;
\draw[black,stealth-] (0.2,0.06) [out=25, in=-115] to (1.94,1.8) ;

\draw[black,-stealth] (2.06,-1.8) [out=65, in=-155] to (3.8,-0.06) ;
\draw[black,stealth-] (2.2,-1.94) [out=25, in=-115] to (3.94,-0.2) ;

\draw[black,-stealth] (0.2,-0.06) [out=-25, in=115] to (1.94,-1.8) ;
\draw[black,stealth-] (0.06,-0.2) [out=-65, in=155] to (1.8,-1.94) ;

\draw[black,-stealth] (2.2,1.94) [out=-25, in=115] to (3.94,0.2) ;
\draw[black,stealth-] (2.06,1.8) [out=-65, in=155] to (3.8,0.06) ;

\node at (1.3,-0.5) {$\alpha$};
\node at (0.7,-1.6) {$\alpha^*$};

\node at (2.7,-0.5) {$\beta$};
\node at (3.3,-1.6) {$\beta^*$};

\node at (3.3,1.55) {$\gamma^*$};
\node at (2.7,0.4) {$\gamma$};

\node at (0.7,1.5) {$\delta^*$};
\node at (1.3,0.4) {$\delta$};

\draw[fill=black] (7,1.5) circle (0.06cm);
\draw[fill=black] (9,1.5) circle (0.06cm);
\draw[fill=black] (11,1.5) circle (0.06cm);
\draw[fill=black] (9,-0.5) circle (0.06cm);

\draw[black,stealth-] (7.2,1.4) [out=-20, in=-160] to (8.8,1.4) ;
\draw[black,-stealth] (7.2,1.6) [out=20, in=160] to (8.8,1.6) ;

\draw[black,stealth-] (9.2,1.4) [out=-20, in=-160] to (10.8,1.4) ;
\draw[black,-stealth] (9.2,1.6) [out=20, in=160] to (10.8,1.6) ;

\draw[black,stealth-] (8.9,1.3) [out=-110, in=110] to (8.9,-0.3) ;
\draw[black,-stealth] (9.1,1.3) [out=-70, in=70] to (9.1,-0.3) ;

\node at (8,2) {$\alpha$};
\node at (8,1) {$\alpha^*$};
\node at (10,2) {$\beta$};
\node at (10,1) {$\beta^*$};
\node at (8.5,0.5) {$\gamma^*$};
\node at (9.5,0.5) {$\gamma$};

\draw[black,-stealth] (-0.1,0.1) [out=135, in=90] to (-0.8,0) [out=-90, in=-135] to (-0.1,-0.1) ;
\draw[black,-stealth] (4.1,-0.1) [out=-45, in=-90] to (4.8,0) [out=90, in=45] to (4.1,0.1) ;
\draw[black,-stealth] (2.1,2.1) [out=45, in=0] to (2,2.8) [out=180, in=135] to (1.9,2.1) ;
\draw[black,-stealth] (1.9,-2.1) [out=-135, in=180] to (2,-2.8) [out=0, in=-45] to (2.1,-2.1) ;

\node at (-1.1,0) {$\zeta_1$};
\node at (2,3.1) {$\zeta_4$};
\node at (5.1,0) {$\zeta_3$};
\node at (2,-3.1) {$\zeta_2$};

\draw[black,-stealth] (7.1,1.6) [out=45, in=0] to (7,2.3) [out=180, in=135] to (6.9,1.6) ;
\draw[black,-stealth] (9.1,1.6) [out=45, in=0] to (9,2.3) [out=180, in=135] to (8.9,1.6) ;
\draw[black,-stealth] (11.1,1.6) [out=45, in=0] to (11,2.3) [out=180, in=135] to (10.9,1.6) ;
\draw[black,-stealth] (8.9,-0.6) [out=-135, in=180] to (9,-1.3) [out=0, in=-45] to (9.1,-0.6) ;

\node at (7,2.6) {$\zeta_1$};
\node at (9,2.6) {$\zeta_2$};
\node at (11,2.6) {$\zeta_3$};
\node at (9,-1.6) {$\zeta_4$};
\end{tikzpicture}
\caption{Examples of quivers used to define the Ginzburg algebra, obtained from quivers with underlying graphs $\tilde{A}_3$ and $D_4$ respectively; compare with Figure \ref{quiversexample} and Figure \ref{quiversdoubled}.}
\end{figure} 
The preprojective algebra is then the zeroth homology of the Ginzburg algebra, $\Pi(Q)=H_0(\mathcal{G}(Q))$. (There are several different relevant grading conventions for the Ginzburg algebra. We use the grading that makes $\mathcal{G}(Q)$ into a 2-Calabi-Yau algebra.)
\end{example}
For $\lambda_i\in k$, $i\in Q_0$, the \emph{deformed Ginzburg algebra} $\mathcal{G}^{\lambda}(Q)$, of a quiver $Q=(Q_0,Q_1)$, is defined completely analogously to Definition~\ref{defginzb}, except for the differential of $\zeta_i$ being
\[
\partial\zeta_i=\sum_{\substack{\alpha\in Q_1 \\ s(\alpha)=i} }\alpha^*\alpha-\sum_{\substack{\beta\in Q_1 \\ t(\beta)=i} }\beta\beta^*+\lambda_ie_i.
\]

Also the deformed multiplicative preprojective algebra has a dg-version.
\begin{definition}\label{defdermultpreproj}
The \emph{deformed derived multiplicative preprojective algebra}, which we denote $\mathcal{G}_M^q(Q)$, is the dg-algebra defined as follows. Let $\tilde{Q}_M=\tilde{Q}\cup\overline{Q}_M\cup\{\tau_{\alpha}\;|\;\alpha\in Q_1\smallsetminus T_1\}$ (see Definition~\ref{defginzb} and Definition~\ref{defmultpreproj}), where $T\subset Q$ is a spanning tree, and let $q_i\in k^{\times}$, for $i\in Q_0$. Then $\mathcal{G}_M^q(Q)=k\tilde{Q}_M/I$, where $I$ is the set of relations $z_{\alpha}z_{\alpha}^{-1}=z_{\alpha}^{-1}z_{\alpha}=e_{s(\alpha)}$. The grading is defined as above on $\tilde{Q}\cup\overline{Q}_M$, and $|\tau_{\alpha}|=1$. The differential is defined by
\begin{align*}
    \partial \tau_{\alpha}&=e_{s(\alpha)}+\alpha^*\alpha-z_{\alpha}, \\
    \partial \zeta_{i}&=
\prod_{s(\alpha)=i}(q_ie_i+q_i\alpha^*\alpha)-\prod_{t(\alpha)=i}(e_i+\alpha\alpha^*),
\end{align*}
and $0$ on the generators in $\overline{Q}_M$. 
\end{definition}
The homology in degree $0$ is the deformed multiplicative preprojective algebra, and, as for the non-derived version, we call the case $q_i=1$, for all $i\in Q_0$, \emph{non-deformed}. (For an alternative definition of the preprojective algebra and its derived version definition see the paper by Kaplan and Shendler, and in particular \cite[Remark 4.7]{MR4582531} for an explicit quasi-isomorphism between the two versions.)

A \emph{dg-category} is a $k$-linear category $\mathcal{A}$ such that the morphism spaces $\mathcal{A}(X,Y)$ are graded and quipped with a differential $\partial=\partial_{X,Y}$ of degree $-1$, i.e. a $k$-linear map satisfying the graded Leibniz rule:
\[
\partial_{X,Z}(g\circ f)=\partial_{Y,Z}(g)\circ f+(-1)^{|g|}g\circ\partial_{X,Y}(f).
\]

A \emph{dg-module} over a dg-algebra $(A,\partial_A)$ is a graded (right) module $M$ equipped with a differential $\partial_M$ of degree $-1$, such that $\partial_M(ma)=\partial_M(m)a+(-1)^{|m|}m\partial_A(a)$. Finitely generated (as modules) dg-modules over $A$ form a dg-category, denoted by $\text{dg-mod}_A$. The morphism spaces $\text{dg-mod}_{A,\bullet}(M,N)$ are graded sets of module morphisms, with $\text{dg-mod}_{A,n}(M,N)=\{f:M\rightarrow N\,|\,|f|=n\}$, where $|f|=n$ means that $f$ maps homogeneous elements $m\in M$ of degree $d$ to homogeneous elements $f(m)\in N$ of degree $d+n$. The differential is given by $\partial=\partial_N\circ f+(-1)^{|f|+1}f\circ\partial_M$. 

The notion of a functor generalises to the dg-setting to the notion of a dg-functor, see \cite{MR2275593}. A (right) \emph{dg-module} over a dg-category $\mathcal{A}$ is a dg-functor $\mathcal{M}:\mathcal{A}^{\text{op}}\rightarrow\text{Ch}(k)$, where $\text{Ch}(k)$ is the category of chain complexes over $k$. We denote the category of dg-modules over $\mathcal{A}$ by $\mathcal{A}\text{-mod}$.

A family of examples of dg-categories is given by dg-algebras with a complete finite set of orthogonal idempotents, e.g. from dg-algebras $A$ having underlying algebra structure $A=kQ$, for a quiver $Q$. The objects are in one-to-one correspondence with the idempotents (vertices) and the morphisms from $i$ to $j$ are given by span of all paths from $i$ to $j$: $\mathcal{A}(e_i,e_j)=e_jAe_i$. We think of the objects as the right $A$-modules $e_iA$ (of all paths which end in $i$), and the (right) module morphisms are given by multiplication on the left. In this case the two definitions of dg-modules coincides: a dg-module $M$ over $A$ as an algebra can naturally be viewed as a dg-module $\mathcal{M}$ over $A$ as category, and vice versa.

One can easily extend the definition of homotopy to the dg-algebra setting. Doing this one defines $K(\text{dg-mod}_A)$, the homotopy category of finitely generated dg-modules over $A$. The Verdier quotient of this category, inverting all quasi-isomorphisms, defines the \emph{(unbounded) derived category} of finitely generated dg-modules over $A$, which we also denote by $D(A)$, which comes equipped the structure of a triangulated category. However, we will not consider any analogue $D^b(A)$, but another full triangulated subcategory $D^{\pi}(A)\subset D(A)$, called the \emph{perfect derived category}, which is the dg-analogue of $\text{Perf}(A)$. In terms of dg-algebras, wrapped Fukaya categories can always always be represented by the \emph{perfect derived category} $D^{\pi}(A)$ of $A$, for some dg-algebra $A$. It is defined as follows (viewing $A$ as a dg-category as above).

The category of \emph{twisted complexes} $\text{Tw}\,\mathcal{A}$ of a dg-category $\mathcal{A}$ is defined as follows.
 The objects are finite direct sums
 \[
 X=\left(\bigoplus_i X_i[n_i],(f_{i,j})\right),
 \]
where $X_i\in\mathcal{A}$ and $[n]$ is the shift of degree by $n$ functor (which changes the sign on the differential and on the morphisms by $(-1)^n$), and where $(f_{i,j})$ is an upper triangular matrix ($f_{i,j}=0$ if $i> j$) with $f_{i,j}=(-1)^{n_j}g_{i,j}:X_j\rightarrow X_i$, where $g_{i,j}$ is homogeneous of degree $n_i-n_j-1$, and $f_{ii}=\partial_{X_i}[n_i]$, such that the matrix multiplication ($f_{i,j})^2=0$. The last condition is required for $(f_i,j)$ to be a differential, i.e. to make $X$ into a complex and a dg-module over $\mathcal{A}$. A morphism $g=(g_{i,j})\in\text{Tw}\,\mathcal{A}(X,X')$ is a matrix with $g_{i,j}\in\mathcal{A}(X_j,X_i')[n_i'-n_j']$. The differential on the morphisms spaces is by the usual formula: $\partial g=\partial_{X'}\circ g+(-1)^{|g|+1}g\circ\partial_X$, for $g$ homogeneous.

Note that for $f:X_1\rightarrow X_2$ of degree $n_2-n_1-1$, such that $\partial(f)=0$, $\text{cone}(f[n_1]:X_1[n_1]\rightarrow X_2[n_2])=(X_2[n_2+1]\oplus X_1[n_1],f_{i,j})$, with $f_{1,1}=(-1)^{n_2+1}\partial_{X_2}$, $f_{2,2}=(-1)^{n_1}\partial_ {X_1}$ and $f_{1,2}=(-1)^{n_1}f$, is a twisted complex; in fact, twisted complexes is the formal way to add all shifts, cones, direct sums and their  combinations. Note also that $\mathcal{A}$ embeds into $\text{Tw}\,\mathcal{A}$. Cones define the triangulated structure on $D^b(\mathcal{A})$, by saying that distinguished triangles are triangles isomorphic to the ones of the form $X_2\xrightarrow{f_2} X_1\rightarrow\text{cone}(f)\rightarrow X_1[1]$.

Given a dg-category $\mathcal{A}$, there is a Yoneda embedding $\mathcal{A}\rightarrow\mathcal{A}\text{-mod}$. This embedding extends to an embedding $\text{Tw}\,\mathcal{A}\rightarrow\mathcal{A}\text{-mod}$. The smallest subcategory of $\mathcal{A}$-mod which contains $\text{Tw}\,\mathcal{A}$, such that its zeroth homology is a triangulated category, and which is furthermore closed under taking direct summands, is denoted by $\Pi\text{Tw}\,\mathcal{A}$. Then $D^{\pi}(\mathcal{A})=H_0\Pi\text{Tw}\,\mathcal{A}$. If $A$ is quasi-isomorphic to an algebra $B$, then $D^{\pi}(A)=\text{Perf}(B)$.

Given a set of objects $X_1,\dots,X_n$ in $\Pi\text{Tw}\,\mathcal{A}$ (in particular they can all belong to $\mathcal{A}$), let $\langle X_1,\dots,X_n\rangle$ denote the full subcategory with only these objects. If $H_0\text{Tw}\langle X_1,\dots,X_n\rangle\simeq D^{\pi}\mathcal{A}$, then we say that $X_1,\dots,X_n$ \emph{generates} $D^{\pi}(\mathcal{A})$. If $D^{\pi}\langle X_1,\dots,X_n\rangle\simeq D^{\pi}(\mathcal{A})$ we say that $X_1,\dots,X_n$ \emph{split-generates} $D^{\pi}(\mathcal{A})$.

Suppose that $\mathcal{A}$ is semi-free in the sense of \cite[Appendix B]{MR2028075} (otherwise one can take a semi-free replacement $\tilde{\mathcal{A}}$ of $\mathcal{A}$, having the same objects). The \emph{Drinfeld quotient} of $\mathcal{A}$ at a set of objects $X_1,\dots,X_n$ is the dg-category obtained by adding free generators $\kappa_i$ to $\mathcal{A}(X_i,X_i)$ of degree $|\kappa_i|=1$, such that $\partial \kappa_i=1_{X_i}$. 

In the case of a dg-algebra with underlying algebra structure $A=kQ$, for some quiver $Q$, the Drinfeld quotient at a set of objects given by vertices $i_1,\dots,i_m$, is then the dg-algebra obtained by adding loops $\kappa_{i_j}$ at these vertices, with $|\kappa_{i_j}|=1$ and $\partial \kappa_{i_j}=e_{i_j}$.

Next follow two proposition we will use later.
\begin{proposition}\label{quotgen}
Suppose that $X_1,\dots,X_n,Y_1,\dots,Y_m$ are objects of $\mathcal{A}$ that generate $D^{\pi}(\mathcal{A})$. If the Drinfeld quotient at $X_1,\dots,X_n$ is trivial, i.e. quasi-equivalent to the trivial dg-category, then $X_1,\dots, X_n$ split-generate $\mathcal{A}$.
\end{proposition}
\begin{proof}
The derived category of the Drinfeld quotient is the Verdier quotient, hence 
\[
D^{\pi}\langle X_1,\dots,X_n,Y_1,\dots,Y_m\rangle/H_0\text{Tw}\langle X_1,\dots,X_n\rangle=0
\]
(\cite{MR2028075} 3.4). An object $X\in D=D^{\pi}\langle X_1,\dots,X_n,Y_1,\dots,Y_m\rangle$ is sent to $0$ in the Verdier quotient (up to quasi-isomorphism) if and only if there is an $X'\in D$ such that $X\oplus X'\in H_0\text{Tw}\,\langle X_1,\dots,X_n\rangle$.
\end{proof}
\begin{proposition}\label{dgquotalg}
For a dg-algebra $(A,\partial)$ with $A=kQ$, let $B=A\langle h_i,\,i\in I\rangle$ denote the Drinfeld quotient at a set of vertices $I\subset KQ_0$. Let $\tilde{A}=A/Ae_IA$, where $e_I=\sum_{i\in I}e_i$. Then $(\tilde{A},\tilde{\partial})$ is the semi-free dg-algebra with $\tilde{A}=kQ'$, where $Q'$ is the quiver obtained by removing all the vertices in $I$ and all arrows starting or terminating in $I$ (with the induced differential). Moreover, the obvious map (sending $h_i$ to $0$) $B\rightarrow \tilde{A}$ is a quasi-isomorphism.
\end{proposition}
\begin{proof}
This follows from that $B/Be_IB=A/Ae_IA$, together with \cite[Lemma 2.5]{bäcke2023contacthomologycomputationssingular}, which is obtained by a spectral sequence argument.
\end{proof}

\subsubsection{$A_{\infty}$-algebras and $A_{\infty}$-categories}\label{algprelAinf}
An $A_{\infty}\text{-category}$ $\mathcal{A}$ consists of a collection objects $\text{Ob}\,\mathcal{A}$, together with graded vector spaces $\mathcal{A}(A,B)$, for each pair of objects $A,B\in\text{Ob}\,\mathcal{A}$, and (higher) multiplication maps
\[
\mu_m:\mathcal{A}(X_{m},X_{m+1})\otimes\cdots\mathcal{A}(X_1,X_2)\rightarrow\mathcal{A}(X_1,X_{m+1})
\]
of degree $m-2$, which satisfy the $A_{\infty}$-relations:
\[
\sum_{\substack{r+s+t=m \\ s\geq 1}}(-1)^{r+st}\mu_{r+1+t}(1^{\otimes r}\otimes\mu_s\otimes 1^{\otimes t})=0
\]
(with the Koszul sign rule when acting on objects). We say that $\mathcal{A}$ is \emph{strictly unital} if every object $X$ contains an element $e_X$ such that 
\begin{align*}
\mu_1(e_X)&=0 \\
\mu_2(a,e_X)=\mu_2(e_Y,a)&=\begin{cases}
a,&\mbox{ if }a\in\mathcal{A}(X,Y) \\
0&\mbox{ otherwise, } 
\end{cases} \\
\mu_m(\dots,e_X,\dots)&=0,\quad m\geq 3.
\end{align*}
When $\mathcal{A}$ has only one object which is strictly unital, we recover the usual notion of a unital $A_{\infty}$-algebra over $k$. However, just as with dg-categories, we are more interested in $A_{\infty}$-algebras over semi-simple rings $R=ke_1\oplus\cdots\oplus ke_n$. That is, when $\mathcal{A}$ is an $A_{\infty}$-category with $n$ objects $X_1,\dots, X_n$, with units $e_i\in\mathcal{A}(X_i,X_i)$.

If $\mu_m=0$ for $m\geq 3$, in the definitions above, we recover the dg-notions above. 

An $A_{\infty}$-functor $F:\mathcal{A}\rightarrow\mathcal{B}$ is a map on objects together with a collection of maps $F_m:\mathcal{A}(X_{m-1},X_m)\otimes\cdots\otimes\mathcal{A}(X_0,X_1)\rightarrow\mathcal{B}(FX_0,FX_m)$ of degree $m-1$, such that
\[
\sum_{\substack{1\leq r\leq m\\s_1+\dots+s_r=m}}\mu_r^{\mathcal{B}}(F_{s_1}\otimes\cdots\otimes F_{s_r})=\sum_{\substack{r+s+t=m \\ s\geq 1}}F_{r+1+t}(1^{\otimes r}\otimes\mu_s^{\mathcal{A}}\otimes 1^{\otimes t}).
\]
Note that $H_0(\mathcal{A})$ and $H_{\bullet}(\mathcal{A})$ are always categories (composition induced by $\mu_2$ is associative). If $F_1$ induces an equivalence of categories $H_{\bullet}(\mathcal{A})\rightarrow H_{\bullet}(\mathcal{B})$, then we say that $F$ is a \emph{quasi-equivalence}; if $\mathcal{A}$ and $\mathcal{B}$ are $A_{\infty}$-algebras, then $F$ is a \emph{quasi-isomorphism}. 

For any $A_{\infty}$-category $\mathcal{A}$, there exists an $A_{\infty}$-structure on $H_{\bullet}(\mathcal{A})$ (so that $\mu_1=0$, unique up $A_{\infty}$-isomorphism) and an $A_{\infty}$-quasi-isomorphism to $\mathcal{A}$. This is called a \emph{minimal model} of $\mathcal{A}$. We say that $\mathcal{A}$ is \emph{formal}, if $\mu_m=0$ for $m\geq 3$, for the $A_{\infty}$-structure making $H_{\bullet}(\mathcal{A})$ into a minimal model for $\mathcal{A}$.

Sometimes $A_{\infty}$-categories $\mathcal{A}$ geometrically defined do not have units or idempotents, but they do have elements which become units in $H_{\bullet}(\mathcal{A})$. In such cases we say that $\mathcal{A}$ is \emph{homologically unital}. However, every homologically unital $A_{\infty}$-category is quasi-equivalent to a strictly unital one \cite{MR2441780}.

Any $A_{\infty}$-category is quasi-equivalent to a dg-category (as $A_{\infty}$-categories). Furthermore, the definitions of twisted complexes and $D^{\pi}(\mathcal{A})$ generalises to $A_{\infty}$-categories as follows (see also e.g. \cite{MR4297810}, \cite{MR2441780}). Given an $A_{\infty}$-category $\mathcal{A}$, we consider finite formal expressions $A=\bigoplus_i(X_i[n_i])$, $X_i\in\text{Ob}\,\mathcal{A}$, $n_i\in\mathbb{Z}$. Between these objects we define spaces of morphisms as matrices:
\[
\text{Add}_{\mathcal{A}}\left(\bigoplus_iX_i[n_i],\bigoplus_jY_j[m_j]\right)=\bigoplus_{i,j}\mathcal{A}(X_i,Y_j)[m_j-n_i]
\]
for which define $\mu_m^{\text{Add}\,\mathcal{A}}$ by the matrix version of $A_{\infty}$-operations:
\[
(\mu_m^{\text{Add}_{\mathcal{A}}}(g_m,\dots,g_1))_{i,j}=\sum_{i_1,\dots,i_m}\mu_m^{\mathcal{A}}((g_m)_{i,i_1}),(g_{m-1})_{i_1,i_2},\dots,(g_1)_{i_{m-1},j}),
\]
where $g_r=((g_r)_{i,j})\in\bigoplus_{i,j}\mathcal{A}(X_r,X_{r+1})[n_{r+1}-n_r]$. A \emph{twisted complex} over $\mathcal{A}$ is a pair $\left(\bigoplus_iX_i[n_i],f\right)$, where $f=(f_{i,j})$ is an upper triangular matrix, such that $f_{i,j}:X_j\rightarrow X_i$ of degree $n_i-n_j-1$, which satisfies the Maurer--Cartan equation
\[
\sum_{m\geq 1}\mu_m^{\text{Add}_{\mathcal{A}}}(f,\dots,f)=0.
\]
Twisted complexes are the objects of the $A_{\infty}$-category $\text{Tw}\,\mathcal{A}$, with morphisms being $\text{Add}_{\mathcal{A}}$. Its higher operations are defined by
\[
\mu_m^{\text{Tw}\,\mathcal{A}}(g_m,\dots,g_1)=\sum_{i_{m+1},\dots,i_1\geq 0}\mu^{\text{Add}_{\mathcal{A}}}_{m+i_1\cdots+i_{m+1}}(f_{m+1},\dots,f_{m+1},g_m,f_{m},\dots,f_2,g_1,f_1,\dots,f_1),
\]
where $g_r$ is a morphism of twisted complexes from $(A_r,f_r)$ to $(A_{r+1},f_{r+1})$, where we insert $i_{m+1}$ copies of $f_{m+1}$ followed by $g_m$, then $i_m$ copies of $f_{m}$ followed by $g_{m-1}$, etcetera. The sum is finite because $f_r$ is required to be upper triangular.

An $A_{\infty}$-functor $F:\mathcal{A}\rightarrow\mathcal{B}$ defines an $A_{\infty}$-functor $\text{Tw}\,F:\text{Tw}\,\mathcal{A}\rightarrow\text{Tw}\,\mathcal{B}$, defined on objects by sending $(\bigoplus_i X_i[n_i],f)$ to $(\bigoplus_i F(X_i)[n_i],\sum F_i(f,\dots,f))$ and
\[
\left(\text{Tw}\,F\right)_m(g_m,\dots,g_1)=\sum_{i_{m+1,\dots,i_1}\geq 0}\mathcal{F}_{m+i_1+\cdots+i_{m+1}}(f_{m+1},\dots,f_{m+1},g_m,f_m,\dots,f_2,g_1,f_1,\dots,f_1),
\]
where we insert $i_{m+1}$ copies of $f_{m+1}$ followed by $g_m$ followed by $i_m$ copies of $f_m$ etc. \cite[Section 3.3.3]{MR4297810}.

We define the \emph{perfect derived category} of the $A_{\infty}$-category $\mathcal{A}$ as $D^{\pi}(\mathcal{A})=H_0\Pi\text{Tw}\,\mathcal{A}$, where $\Pi$, denotes the idempotent completions (this means that we need to add summands, after embedding $\text{Tw}\,\mathcal{A}$ into the category of $A_{\infty}$-modules over $\mathcal{A}$ using the Yoneda embedding, as is explained in e.g. \cite{MR2441780}). A functor of $A_{\infty}$-categories $F:\mathcal{A}\rightarrow\mathcal{B}$ induces functors $H_0\text{Tw}\,\mathcal{A}\rightarrow H_0\text{Tw}\,\mathcal{B}$ and $D^{\pi}(\mathcal{A})\rightarrow D^{\pi}(\mathcal{B})$, via the $A_{\infty}$-functor $\text{Tw}\,F$ defined above. They are quasi-equivalences if $F$ is a quasi-isomorphism.

The notions of generation and split-generation is generalised to $A_{\infty}$-categories verbatim. The Drinfeld quotient, is also generalised, and Proposition~\ref{quotgen} remain true for $A_{\infty}$-categories. In particular, if $\mathcal{A}$ is an $A_{\infty}$-category and $\mathcal{B}=\langle X_1\dots,X_n\rangle$ is the full sub-$A_{\infty}$-category with objects $X_1,\dots,X_n\in\text{Tw}\,\mathcal{A}$, then the morphism spaces in quotient $A_{\infty}$-category $\mathcal{A}/\mathcal{B}$ are given by
\[
\mathcal{A}/\mathcal{B} (U,V)=\bigoplus_{\substack{k\geq 0 \\
X_{i_1},\dots,X_{i_k}\in\mathcal{B}}}\text{Tw}\,\mathcal{A}(X_{i_k},V)\otimes\cdots\otimes\text{Tw}\,\mathcal{A}(X_{i_1},X_{i_2})\otimes\text{Tw}\,\mathcal{A}(U,X_{i_1}),
\]
and the differential is given by the bar differential. See \cite{MR2259271} for the complete description of the $A_{\infty}$-structure. 

\subsection{Background on the McKay correspondence}
The McKay correspondence is a remarkable correspondence between classification results in different areas of mathematics.

\subsubsection{The classical McKay correspondence}

\noindent
We start with the graphs and the quivers. The following graphs (Figure~\ref{dynkingraphs}) without the extra vertices and dashed edges are the Dynkind diagrams of type $A_n$, $D_n$, $E_6$, $E_7$ and $E_8$. If one includes the dashed part one obtaines the extended Dynkin diagrams: $\tilde{A}_n$, $\tilde{D}_n$, $\tilde{E}_6$, $\tilde{E}_7$ and $\tilde{E}_8$.
\begin{figure}[H]
\begin{tikzpicture}[scale=0.5]
\node at (-2,1.5) {$A_n/\tilde{A}_n$:};
\node at (-2,-2) {$D_n/\tilde{D}_n$:};
\node at (-2,-6) {$E_6/\tilde{E}_6$:};
\node at (14,0) {$E_7/\tilde{E}_7$:};
\node at (14,-4) {$E_8/\tilde{E}_8$:};

\draw[fill=black] (0,0) circle (0.1cm);
\draw[fill=black] (2,0) circle (0.1cm);
\draw[fill=black] (6,0) circle (0.1cm);
\draw[fill=black] (8,0) circle (0.1cm);

\draw[fill=black] (4,1.5) circle (0.1cm);
\draw[thick] (0,0) to (3.75,0);
\draw[fill=black] (3.9,0) circle (0.008cm);
\draw[fill=black] (4.0,0) circle (0.008cm);
\draw[fill=black] (4.1,0) circle (0.008cm);
\draw[thick] (4.25,0) to (8,0);

\draw[dashed, thick] (0,0) to (4,1.5);
\draw[dashed, thick] (8,0) to (4,1.5);

\node[scale=0.5,thick] at (0,-0.3) {$1$};
\node[scale=0.5,thick] at (2,-0.3) {$2$};
\node[scale=0.5,thick] at (6,-0.3) {$n-1$};
\node[scale=0.5,thick] at (8,-0.3) {$n$};
\node[scale=0.5,thick] at (4,1.8) {$n+1$};

\draw[fill=black] (0,-2) circle (0.1cm);
\draw[fill=black] (2,-2) circle (0.1cm);
\draw[fill=black] (4,-2) circle (0.1cm);
\draw[fill=black] (8,-2) circle (0.1cm);
\draw[fill=black] (10,-2) circle (0.1cm);
\draw[fill=black] (2,-4) circle (0.1cm);
\draw[fill=black] (8,-4) circle (0.1cm);

\draw[fill=black] (5.9,-2) circle (0.008cm);
\draw[fill=black] (6.0,-2) circle (0.008cm);
\draw[fill=black] (6.1,-2) circle (0.008cm);
\node[scale=0.5,thick] at (0,-1.7) {$1$};
\node[scale=0.5,thick] at (2,-1.7) {$2$};
\node[scale=0.5,thick] at (4,-1.7) {$3$};
\node[scale=0.5,thick] at (8,-1.7) {$n-2$};
\node[scale=0.5,thick] at (10,-1.7) {$n-1$};
\node[scale=0.5,thick] at (2,-4.3) {$n$};
\node[scale=0.5,thick] at (8,-4.3) {$n+1$};

\draw[thick] (0,-2) to (5.75,-2);
\draw[thick] (6.25,-2) to (10,-2);
\draw[thick] (2,-2) to (2,-4);
\draw[thick, dashed] (8,-2) to (8,-4);

\draw[fill=black] (0,-6) circle (0.1cm);
\draw[fill=black] (2,-6) circle (0.1cm);
\draw[fill=black] (4,-6) circle (0.1cm);
\draw[fill=black] (6,-6) circle (0.1cm);
\draw[fill=black] (8,-6) circle (0.1cm);
\draw[fill=black] (4,-8) circle (0.1cm);
\draw[fill=black] (4,-10) circle (0.1cm);
\draw[thick] (0,-6) to (8,-6);
\draw[thick] (4,-6) to (4,-8);
\draw[thick,dashed] (4,-8) to (4,-10);

\node[scale=0.5,thick] at (0,-5.7) {$1$};
\node[scale=0.5,thick] at (2,-5.7) {$2$};
\node[scale=0.5,thick] at (4,-5.7) {$3$};
\node[scale=0.5,thick] at (6,-5.7) {$4$};
\node[scale=0.5,thick] at (8,-5.7) {$5$};
\node[scale=0.5,thick] at (4.3,-8) {$6$};
\node[scale=0.5,thick] at (4.3,-10) {$7$};

\draw[fill=black] (16,0) circle (0.1cm);
\draw[fill=black] (18,0) circle (0.1cm);
\draw[fill=black] (20,0) circle (0.1cm);
\draw[fill=black] (22,0) circle (0.1cm);
\draw[fill=black] (24,0) circle (0.1cm);
\draw[fill=black] (26,0) circle (0.1cm);
\draw[fill=black] (28,0) circle (0.1cm);
\draw[fill=black] (22,-2) circle (0.1cm);

\node[scale=0.5,thick] at (16,0.3) {$8$};
\node[scale=0.5,thick] at (18,0.3) {$1$};
\node[scale=0.5,thick] at (20,0.3) {$2$};
\node[scale=0.5,thick] at (22,0.3) {$3$};
\node[scale=0.5,thick] at (24,0.3) {$4$};
\node[scale=0.5,thick] at (26,0.3) {$5$};
\node[scale=0.5,thick] at (28,0.3) {$6$};
\node[scale=0.5,thick] at (22,-2.3) {$7$};

\draw[thick,dashed] (16,0) to (18,0);
\draw[thick] (18,0) to (28,0);
\draw[thick] (22,0) to (22,-2);

\draw[fill=black] (16,-4) circle (0.1cm);
\draw[fill=black] (18,-4) circle (0.1cm);
\draw[fill=black] (20,-4) circle (0.1cm);
\draw[fill=black] (22,-4) circle (0.1cm);
\draw[fill=black] (24,-4) circle (0.1cm);
\draw[fill=black] (26,-4) circle (0.1cm);
\draw[fill=black] (28,-4) circle (0.1cm);
\draw[fill=black] (30,-4) circle (0.1cm);
\draw[fill=black] (20,-6) circle (0.1cm);

\draw[thick,dashed] (28,-4) to (30,-4);
\draw[thick] (16,-4) to (28,-4);
\draw[thick] (20,-4) to (20,-6);

\node[scale=0.5,thick] at (16,-3.7) {$1$};
\node[scale=0.5,thick] at (18,-3.7) {$2$};
\node[scale=0.5,thick] at (20,-3.7) {$3$};
\node[scale=0.5,thick] at (22,-3.7) {$4$};
\node[scale=0.5,thick] at (24,-3.7) {$5$};
\node[scale=0.5,thick] at (26,-3.7) {$6$};
\node[scale=0.5,thick] at (28,-3.7) {$7$};
\node[scale=0.5,thick] at (30,-3.7) {$9$};
\node[scale=0.5,thick] at (20,-6.3) {$8$};
\end{tikzpicture}
\caption{The classical Dynkin diagrams of type $A_n$, $D_n$ and $E_{6,7,8}$ shown in filled edges. If we add the dashed edges we get the extended Dynkin diagrams $\tilde{A}_n$, $\tilde{D}_n$ and $\tilde{E}_{6,7,8}$.}
\label{dynkingraphs}
\end{figure}
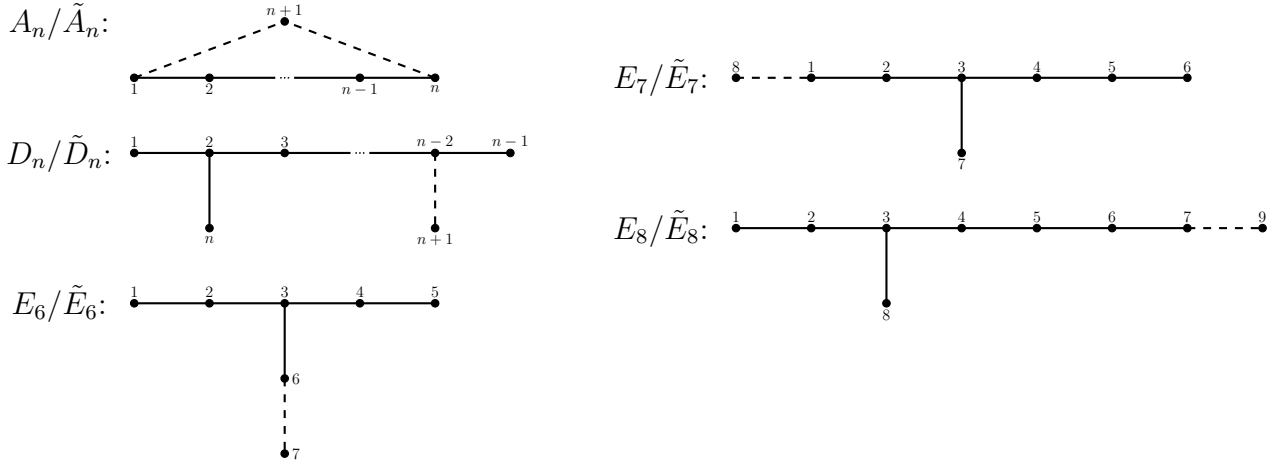
A quiver $Q$ is said to be of \emph{finite representation type} if there are finitely many isomorphism classes of indecomposable $kQ$-modules. Gabriel's theorem states that a connected quiver is of finite type if and only if the underlying graph is one of Dynkin type $ADE$ graphs above. Furthermore, there is notion of \emph{tame representation type}, which occur precisely when the underlying graph is one of the extended Dynkin type $\tilde{A}\tilde{D}\tilde{E}$ above. (All other quivers are of \emph{wild representation type}, which more or less says that it is impossible to have good understanding of all their indecomposable modules.)

The Dynkin diagrams correspond to finite subgroups $G\subset\text{SU}(2)$. The finite subgroups of $\text{SU}(2)$ are classified (up to conjugacy) by the following list:
\begin{enumerate}[\qquad$\bullet$]
\item The cyclic groups $\mathbb{Z}_n\cong\text{C}_n=\langle x\,|\,x^n=1\rangle$.
\item The binary dihedral groups, also known as the dicyclic groups, $\mathbb{B}\text{D}_{4n}=\langle x,y\,|\,x^{2n}=1,\, x^n=y^2, \,y^{-1}xy=x^{-1}\rangle$.
\item The binary tetrahedral group $\mathbb{B}\text{T}=\langle x,y,z\,|\,x^2=y^3=z^3=xyz\rangle$.
\item The binary octrahedral group $\mathbb{B}\text{O}=\langle x,y,z\,|\,x^2=y^3=z^4=xyz\rangle$.
\item The binary icosahedral group $\mathbb{B}\text{I}=\langle x,y,z\,|\,x^2=y^3=z^5=xyz\rangle$.
\end{enumerate}
Given a finite group $G$ and a finite dimensional $\mathbb{C}$-representation $V$ of $G$, we can define the \emph{McKay quiver} of $V$ as follows. For each (isomorphism class of) irreducible finite dimensional representation $\rho:G\rightarrow V$ we take a corresponding vertex $v_{\rho}$. We put an arrow $v_{\rho_1}\rightarrow v_{\rho_2}$ if and only if $\rho_2$ is a direct summand of $V\otimes\rho_1$. 

In the case $G\subset\text{SU}(2)$ we have a canonical representation on $V=\mathbb{C}^2$. It turns out the underlying graphs of the McKay quivers in these cases are the corresponding extended Dynkin diagrams above, where the extra vertex corresponds to the trivial one-dimensional representation, according to:
\begin{enumerate}[\qquad$\bullet$]
    \item $G=\mathbb{Z}_{n+1}$ yields McKay graph $\tilde{A}_{n}$.
    \item $G=\mathbb{B}\text{D}_{4(n-2)}$ yields McKay graph $\tilde{D}_n$.
    \item $G=\mathbb{B}\text{T}$ yields McKay graph $\tilde{E}_6$.
    \item $G=\mathbb{B}\text{O}$ yields McKay graph $\tilde{E}_7$.
    \item $G=\mathbb{B}\text{I}$ yields McKay graph $\tilde{E}_8$.
\end{enumerate}
Finally, in the classical McKay we a have a correspondence also with the \emph{du Val singularities}, given at the origin of $\mathbb{C}^3$ by:
\begin{enumerate}[\qquad$\bullet$]
    \item  The $A_n$ singularities given by $X^2+Y^2+Z^{n+1}=0$.
    \item The $D_n$ singularities $X^2+ZY^2+Z^{n-1}=0$.
    \item The $E_6$ singularity $X^2+Y^3+Z^4=0$.
    \item The $E_7$ singularity $X^2+Y^3+YZ^3=0$. 
    \item The $E_8$ singularity $X^2+Y^3+Z^5=0$.
\end{enumerate}
These singularities can equivalently be defined by $\mathbb{C}^2/G$, where $G$ is the corresponding group. They also classify certain types of singularities, called simple. One finds the crepant resolution of singularities by repeated blow-up. The exceptional divisor in the resolution of singularities is then a union $\mathbb{P}_{\mathbb{C}}^1$'s with transverse intersections. One defines the \emph{intersection graph} of the by taking one vertex for each copy $\mathbb{P}_{\mathbb{C}}^1$, and we put an edge between any two vertices if the corresponding $\mathbb{P}_{\mathbb{C}}^1$'s intersect. The intersection graph of the du Val singularities is then the corresponding Dynking graph.

\subsubsection{The derived McKay correspondence}\label{secmckay}
The preprojective algebras enter the picture because of the equivalence of categories,
\[
\Pi(Q)\text{-mod}\simeq (k[x,y]*G)\text{-mod}
\]
(see \cite{MR1620538}). Here, $G\subset\text{SU}(2)$ is a finite subgroup and $Q$ is a corresponding extended Dynkin quiver. The algebra $k[x,y]*G$ appearing on the right-hand side, called the \emph{skew-symmetric algebra}, is the set $G\times k[x,y]$ with the twisted product $(p(x,y),g)\cdot(q(x,y),h)=(p(x,y)\cdot g(q(x,y)),gh)$. For cyclic groups the proof of the Morita equivalence even yields an isomorphism of algebras.

The derived McKay correspondence below relates this to the resolutions of singularities. It is a basic fact that $\text{Coh}_G(\mathbb{C}^2)$ is equivalent to $(\mathbb{C}[x,y]*G)\text{-mod}$, where $\text{Coh}_G(\mathbb{C}^2)$ is the categories of $G$-equivariant coherent sheaves on $\mathbb{C}^2$.

\begin{theorem}[The derived McKay correspondence {\cite[Theorem 1.4]{MR1752785}}]\label{derivedmckay}
Let $G\subset\text{SU}(2)$ be a finite subgroup and let $X_{\textup{res}}$ be the crepant resolution of singularities of $\mathbb{C}^2/G$. Then we have an equivalence of derived categories
\[
D^b(\textup{Coh}(X_{\textup{res}}))\simeq D^b(\textup{Coh}_G(\mathbb{C}^2)).
\]
\end{theorem}
Thus the bounded derived category of coherent sheaves on resolution of singularities is derived Morita equivalent to the preprojective algebra of the associated extended Dynkin quiver.

\subsection{Formality of the type $\tilde{A}_n$ Ginzburg algebra}\label{formalityginzburgsec}
We will later realise the Ginzburg dg-algebra of type $\tilde{A}_n$ as a Chekanov--Eliashberg algebra of a Legendrian attaching link for a Weinstein handle-decomposition of the $A_n$-Milnor fibre with a divisor removed and suitable stop. Here we prove the important fact that this dg-algebra is formal.

Given a quiver $Q$, when the original quiver $Q$ is acyclic the relation between the full homology of $\mathcal{G}(Q)$ and $\Pi(Q)$ is worked out in \cite{MR3503979}:
\begin{theorem}[{\cite[Theorem A]{MR3503979}}]\label{Hermes}
Suppose that $Q$ is acyclic. Then $\mathcal{G}(Q)$ is formal with $H_{\bullet}(\mathcal{G}(Q))=\Pi(Q)$, i.e., the homology of the Ginzburg algebra is concentrated in degree $0$, if and only if the underlying graph of $Q$ is \emph{not} of Dynkin type $ADE$.
\end{theorem}
For example, according to Theorem~\ref{Hermes}, $H_{\bullet}(\mathcal{G}({Q_{D_4}}))\not=\Pi({Q_{D_4}})$ but $H_{\bullet}(\mathcal{G}({Q_{\tilde{D}_4}}))=\Pi({Q_{\tilde{D}_4}})$. This theorem does however not include the $\tilde{A}_n$ cases with the cyclic orientation of the arrows. One could fix this by changing the orientation and applying an isomorphism changing some of the signs. However we choose to give a direct independent proof for the claim that the full homology of the Ginzburg algebra is given by the preprojective algebra in the $\tilde{A}_n$ cases, independent of Theorem~\ref{Hermes} (see also \cite[Lemma 9.1]{kalck2018relativesingularitycategoriesii}):
\begin{theorem}\label{formal}
Let $\mathcal{G}(\tilde{A}_n)$ denote the Ginzburg algebra and $\Pi(\tilde{A}_n)$ the preprojective algebra of the quiver in Figure~\ref{Anquiv}. Then $H_{\bullet}(\mathcal{G}(\tilde{A}_n))=H_0(\mathcal{G}(\tilde{A}_n))=\Pi(\tilde{A}_n)$. Consequently, the map $\mathcal{G}(\tilde{A}_n)\rightarrow\Pi(\tilde{A}_n)$ given by $\alpha_i\mapsto \alpha_i$, $\alpha_i^*\mapsto\alpha_i^*$ and $\zeta_i\mapsto 0$ defines a quasi-isomorphism.
\end{theorem}
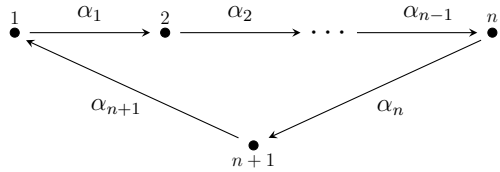
\begin{figure}[H]
\centering
\begin{tikzpicture}
\draw[fill=black] (0,0) circle (0.06cm);
\draw[fill=black] (2,0) circle (0.06cm);
\draw[fill=black] (6.35,0) circle (0.06cm);
\draw[black,-stealth] (0.2,0) to (1.8,0);
\draw[black,-stealth] (2.2,0) to (3.8,0);
\node at (4.2,0) {$\cdots$};
\draw[black,-stealth] (4.55,0) to (6.15,0);
\draw[fill=black] (3.175,-1.5) circle (0.06cm);
\draw[black,-stealth] (6.2,-0.1) to (3.375,-1.4);
\draw[black,stealth-] (0.15,-0.1) to (2.975,-1.4);

\node[scale=0.6,thick] at (0,0.2) {$1$};
\node[scale=0.6,thick] at (2,0.2) {$2$};
\node[scale=0.6,thick] at (6.35,0.2) {$n$};
\node[scale=0.6,thick] at (3.175,-1.7) {$n+1$};

\node[scale=0.8] at (1,0.25) {$\alpha_1$};
\node[scale=0.8] at (3,0.25) {$\alpha_2$};
\node[scale=0.8] at (5.5,0.25) {$\alpha_{n-1}$};
\node[scale=0.8] at (5,-1) {$\alpha_{n}$};
\node[scale=0.8] at (1.35,0-1) {$\alpha_{n+1}$};
\end{tikzpicture}
\caption{Cyclically oriented $\tilde{A}_{n}$.}
\label{Anquiv}
\end{figure}
\begin{proof}
Let $A$ denote the algebra $R$-algebra obtained by dualising the quadratic differential of the Ginzburg algebra to a multiplication (see Figure~\ref{dualisingfigure}), the Koszul dual. In other words, $A=k\overline{Q}/I$, where $I$ is the ideal generated by relations,
\[
I=(a_{i+1}a_i,\,a_i^*a_{i+1}^*,\,a_i^*a_i-a_{i-1}a_{i-1}^*\;|\;i=1,\dots,n+1)
\]
(the last relation should be read with indices modulo $n+1$). We use the notations $\overline{A}=A/R$ and replace $\alpha$ by $a$ to distinguish between elements in $A$ and in $\mathcal{G}(Q)$. We turn $A$ into a graded algebra by letting $|a_i|=|a_i^*|=-1$. 

We denote by $BA$ the bar complex of $A$:
\[
\begin{tikzcd}
R&\arrow{l}[swap]{\epsilon} \overline{A}[-1]&\arrow{l} (\overline{A}[-1])^{\otimes_R2}&\arrow{l}\cdots&\arrow{l} (\overline{A}[-1])^{\otimes_Rm}&\arrow{l}\cdots\quad=\quad BA
\end{tikzcd}
\]
where $\epsilon$ is the obvious augmentation. Its differential is given on pure tensors by
\[
\partial_{BA}(sx_1\otimes\cdots\otimes sx_n)=\sum_{i=1}^{n-1}(-1)^{i-1+|x_1|+\dots+|x_{i-1}|}sx_1\otimes\cdots\otimes s(x_ix_{i+1})\otimes\cdots\otimes sx_n,
\]
where we use the notation $sx=x[-1]$.

We get a projective resolution of $R$ over $A$ (as right $A$-modules) by tensoring $BA$ with $A$:
\[
\begin{tikzcd}
R&\arrow{l}BA\otimes_R A
\end{tikzcd}
\]
Applying $\text{Hom}_A^{\bullet}(-,R)$ (graded dual with respect to the total grading) to this resolution, followed by taking cohomology, calculates $\text{Ext}_A(R,R)$ (considering $A$ as a dg-algebra with trivial differential). Moreover, we have isomorphisms
\begin{align*}
\text{Hom}_A^{\bullet}(BA\otimes_RA,R)&\cong\text{Hom}_R^{\bullet}(BA,\text{Hom}_A^{\bullet}(A,R)) \\
&\cong \text{Hom}_R^{\bullet}(BA,R),
\end{align*}
using the tensor-hom adjunction. Moreover
\[
\text{Hom}_R^{\bullet}(BA,R)\cong\mathcal{G}(Q),
\]
as dg-algebras (with respect to the total grading), by the $R$-isomorphism $(s\alpha_i)^{\vee}\mapsto\alpha_i$, $(sa_i^*)^{\vee}\mapsto\alpha_i^*$, $(sz_i)^{\vee}\mapsto\zeta_i$, where $z_i=a_i^*a_i=a_{i-1}a_{i-1}^*$.

Thus we can compute $H_{\bullet}(\mathcal{G}(Q))$ by computing $\text{Ext}_A(R,R)$. We do this by the following simpler free resolution:
\[
\begin{tikzcd}
R&\arrow{l}[swap]{\epsilon}A&\arrow{l}[swap]{\Phi_1}A^{\oplus 2}&\arrow{l}[swap]{\Phi_2}A^{\oplus 3}&\arrow{l}[swap]{\Phi_3}\cdots
\end{tikzcd}
\]
The maps are given by
\[
\Phi_m=\begin{pmatrix}
\phi_1&(-1)^m\phi_2&0&0&0&\cdots &0&0 \\
0&\phi_1&(-1)^m\phi_2&0&0&\cdots&0&0 \\
0&0&\phi_1&(-1)^m\phi_2&0&\cdots&0&0 \\
\vdots&\vdots&\ddots&\ddots&\ddots&\ddots&\vdots&\vdots\\
0&0&0&0&0&\cdots&\phi_1&(-1)^m\phi_2
\end{pmatrix},
\]
where
\begin{align*}
\phi_1(x)&=(a_1+\cdots+a_n)\cdot x,\\
\phi_2(x)&=(a_1^*+\cdots+a_n^*)\cdot x
\end{align*}
(which are right-module morphisms $A\rightarrow A$). This indeed defines a resolution. First of all,
\[
\Phi_m\circ\Phi_{m+1}=0
\]
because $\phi_1^2=\phi_2^2=\phi_1\phi_2-\phi_2\phi_1=0$. Secondly, the sequence is necessarily exact by dimension reasons. We have that $\text{dim}(A^{\oplus m})=4n m$ and $\text{rank}(\Phi_m)=2mn+n$. But then $4(m+1)n=2mn+n+\text{dim}(\text{ker}(\Phi_m))$, by the dimension formula of linear algebra, hence $\text{dim}(\text{ker}(\Phi_n))=2mn+3n=2(m+1)n+n=\text{rank}(\Phi_{m+1})$. Since $\text{ker}(\Phi_m)\supset\text{Im}(\Phi_{m+1})$, and they have the same dimension, we conclude that $\text{ker}(\Phi_m)=\text{Im}(\Phi_{m+1})$. (Similarly, $\text{ker}(\epsilon)=\text{Im}(\Phi_1)$ is easily checked.) 

Note that since we want a complex with differential of degree $-1$, in particular $\Phi_m$ should be of degree $-1$, we do not shift the grading on $A$. 

Applying $\text{Hom}_A(-,R)$ to the resolution we get the complex with vanishing differential
\[
\begin{tikzcd}
R\arrow{r}{0}&R^{\oplus 2}\arrow{r}{0}&R^{\oplus 3}\arrow{r}{0}&\dots
\end{tikzcd}
\]
Taking homology thus gives back the same algebra. Therefore the homology is concentrated in degree $0$.
\end{proof}
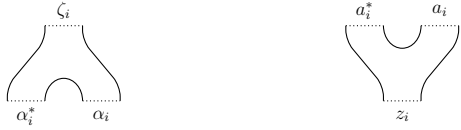
\begin{figure}[H]
\centering
\begin{tikzpicture}
\draw[densely dotted] (2,4) to (2.5,4);
\draw[rounded corners=4pt] (2,4) to (2,3.8) to (1.5,3.2) to (1.5,3);
\draw[densely dotted] (1.5, 3) to (2,3);
\draw (2.5,3) arc (0:180:0.25 and 0.30);
\draw[densely dotted] (2.5, 3) to (3,3);
\draw[rounded corners=4pt] (2.5,4) to (2.5,3.8) to (3,3.2) to (3,3);

\node[scale=0.6] at (2.25,4.2) {$\zeta_i$};
\node[scale=0.6] at (1.75,2.8) {$\alpha_i^*$};
\node[scale=0.6] at (2.75,2.8) {$\alpha_i$};

\draw[densely dotted] (6,4) to (6.5,4);
\draw[densely dotted] (7,4) to (7.5,4);
\draw (7,4) arc (0:-180:0.25 and 0.30);

\draw[densely dotted] (6.5,3) to (7,3);
\draw[rounded corners=4pt] (6,4) to (6,3.8) to (6.5,3.2) to (6.5,3);
\draw[rounded corners=4pt] (7.5,4) to (7.5,3.8) to (7,3.2) to (7,3);

\node[scale=0.6] at (6.755,2.8) {$z_i$};
\node[scale=0.6] at (6.25,4.2) {$a_i^*$};
\node[scale=0.6] at (7.25,4.2) {$a_i$};
\end{tikzpicture}
\caption{"Discs" contributing to $\partial\zeta_i=\alpha_i^*\alpha_i$, respectively to $a_i^*a_i=z_i$}
\end{figure}
\label{dualisingfigure}
\subsection{Homological mirror symmetry}

In string theory, the extra dimensions required are described by a compact Calabi--Yau manifold $X$. There are two different theories which to describe the physics: the $A$-side which uses the symplectic geometry of $X$ and the $B$-side which uses the algebraic geometry of $X$. These theories are not equivalent in general. However, mirror symmetry is a conjecture that for any Calabi--Yau manifold $X$ there exists a mirror Calabi--Yau manifold $\check{X}$, such that
\begin{align*}
A\text{-side}(X)&\simeq B\text{-side}(\check{X}), \\
B\text{-side}(X)&\simeq A\text{-side}(\check{X}). 
\end{align*}
Based upon these ideas Kontsevich \cite{MR1403918} formulated his homological mirror symmetry conjecture, which predicts connections between invariants of symplectic and algebraic varieties, and in particular, between different Kähler manifolds, which have both the structure of an algebraic variety and that of a symplectic manifold. According to this conjecture, the symplectic invariant given by the derived Fukaya category on the $A$-side on $X$ should be equivalent to the algebraic invariant given by the derived category of coherent sheaves on the $B$-side on the mirror $\check{X}$. And furthermore, another equivalence given by swapping the $A$-side and $B$-side on $X$ and $\check{X}$. The equivalence should be as triangulated categories.

This conjecture has been extended beyond the scope of Calabi--Yau manifolds (or more generally, compact Kähler manifolds). In this paper we will primarily be concerned with the equivalences of the (partially) wrapped Fukaya categories of the Weinstein manifolds (with stops) that appear in Example~\ref{loghkA}, and the derived categories of coherent sheaves (or singularity categories of the Landau--Ginzburg models) of the same manifolds considered as algebraic varieties. Of central importance is the notion of log-Calabi Yau manifolds, which behave as the open analogues of closed Calabi-Yau manifolds.

Below is a famous example of homological mirror symmetry, with variations corresponding to the different mirror statements (Section~\ref{mirrorsec}) of the paper. The part which uses the surgery formula (Example~\ref{surgeryex}), illustrates the strategy of the main part of this article.

\subsubsection{Examples of homological mirror symmetry derived from the case $T^*S^1$}\label{examplessec}
Here we discuss some well-known and basic examples of homological mirror symmetry.
\begin{example}[$T^*S^1$]\label{examplecylinder}
An example of a manifold which is its own mirror is $T^*S^1=\mathbb{C}^*$. We have that $D^b(\text{Coh}(\mathbb{C}^*))=D^b(\mathbb{C}[x,x^{-1}])$, as coherent sheaves on $\mathbb{C}$ are just finitely generated $\mathbb{C}[x,x^{-1}]$-modules.

The (derived) wrapped Fukaya category (see Section~\ref{wrappedsec}) of the cylinder $D^{\pi}\mathcal{W}(T^*S^1)$ is generated by a cotangent fibre $L=T_p^*S^1\subset T^*S^1$. Thus it suffices to compute the endomorphism algebra $CW_{T^*S^1}(L,L)$, i.e. the wrapped Floer complex of the fibre. We show that this $A_{\infty}$-algebra is equal to $\mathbb{C}[x,x^{-1}]$ for suitable choices. As a vector space $CW_{T^*S^1}(L,L)$ is generated by the intersection points of $L$ and $\phi_H^1(L)$,"the wrapping" of $L$, where $\phi_H^1$ is the time $1$-flow of a quadratic Hamiltonian, say $H(q,p)=p^2$ (on $T^*S^1=S^1_q\times\mathbb{R}_p$). This gives countably many generators $1,x^{\pm 1},x^{\pm 2},\dots$ (left-hand side of Figure~\ref{wfukdisc}). 

The degree of all generators are $0$ (in the standard grading, we omit details). The endomorphism algebra $CW_{T^*S^1}(L,L)$ is a priori an $A_{\infty}$-algebra, but since it is concentrated in degree $0$ the only non-trivial multiplication is $\mu^2$, and it is thus an ordinary associative algebra.

To see that the product structure is right, i.e. that $x^i\cdot x^j=x^{i+j}$ we must perturb $L$ again, say $\phi_G^1(F)$ with $G=2H$, and count certain immersed holomorphic discs with punctures on the boundary. The punctures should go to the intersection points and the boundary should go the fibre and its two perturbations. The generators need to be identified. Such discs exists, and can be drawn in the universal cover (see the right-hand side of Figure~\ref{wfukdisc}). They are the only possible discs, which can be seen by that the winding numbers must be $0$.
\begin{figure}[H]
\centering
\begin{tikzpicture}
\draw[thick] (0,0) rectangle (3,4);
\draw (0,2) -- (3,2);
\node at (3.3,2) {$S^1$};

\draw[thick, blue!80!black] (1.5,0) to (1.5,4);
\node[blue!80!black] at (1.5,-0.3) {$L$};
\draw[purple!80!black, thick] (0,1.7) to (3,2.3);
\draw[purple!80!black, thick] (0,2.3) to (3,2.9);
\draw[purple!80!black, thick] (0,2.9) to (3,3.5);
\draw[purple!80!black, thick] (0,3.5) to (2.5,4);
\draw[purple!80!black, thick] (0,1.1) to (3,1.7);
\draw[purple!80!black, thick] (0,0.5) to (3,1.1);
\draw[purple!80!black, thick] (0.5,0) to (3,0.5);
\node[purple!80!black] at (0.5,-0.3) {$\phi_H^1(L)$};
\node[scale=0.7] at (1.65,1.8) {$1$};
\node[scale=0.7] at (1.75,1.25) {$x^{-1}$};
\node[scale=0.7] at (1.65,2.4) {$x$};
\node[scale=0.7] at (1.7,3.05) {$x^2$};
\node[scale=0.7] at (1.7,3.65) {$x^3$};
\node[scale=0.7] at (1.75,0.65) {$x^{-2}$};
\node[scale=0.7] at (1.84,0.15) {$x^{-3}$};
\end{tikzpicture}
\qquad
\begin{tikzpicture}
\draw[thick] (0,0) rectangle (5,5);
\draw (0,1) to (5,1);
\node at (5.9,1) {$\mathbb{R}\twoheadrightarrow S^1$};
\node[blue!80!black] at (5.5,1.7) {$L$};
\node[cyan!80!black] at (5.8,2.4) {$\phi_G^1(L)$};
\node[purple!80!black] at (5.8,3.1) {$\phi_H^1(L)$};
\draw[blue!80!black, thick] (0.5,0) to (0.5,5);
\draw[blue!80!black, thick] (2.5,0) to (2.5,5);
\draw[blue!80!black, thick] (4.5,0) to (4.5,5);
\draw[purple!80!black, thick] (0,0.75) to (5,3.25);
\draw[purple!80!black, thick] (0,1.75) to (5,4.25);
\draw[purple!80!black, thick] (0,2.75) to (4.5,5);
\draw[purple!80!black, thick] (0,3.75) to (2.5,5);
\draw[purple!80!black, thick] (0,4.75) to (0.5,5);
\draw[purple!80!black, thick] (0.5,0) to (5,2.225);
\draw[purple!80!black, thick] (2.5,0) to (5,1.25);
\draw[purple!80!black, thick] (4.5,0) to (5,0.25);
\draw[cyan!80!black, thick] (0,0) to (2.5,5);
\draw[cyan!80!black, thick] (2,0) to (4.5,5);
\draw[cyan!80!black, thick] (4,0) to (5,2);
\draw[cyan!80!black, thick] (0,4) to (0.5,5);
\draw[fill, opacity=0.5] (0.5,1) -- (0.5,2) -- (1.166,2.333) -- (0.5,1);

\draw[fill, opacity=0.5] (4.5,5) -- (4.5,3) -- (3.166,2.333) -- (4.5,5);

\draw[fill, opacity=0.5] (2.5,5) -- (2.5,4) -- (1.833,3.667) -- (2.5,5);
\node[scale=0.7] at (0.3,1.2) {$1$};
\node[scale=0.7] at (0.3,2.15) {$x$};
\node[scale=0.7] at (1.3,2.15) {$x$};

\node[scale=0.7] at (2.7,4.8) {$x$};
\node[scale=0.7] at (2.7, 3.85) {$x^3$};
\node[scale=0.7] at (1.65,3.9) {$x^2$};

\node[scale=0.7] at (4.7,4.8) {$x$};
\node[scale=0.7] at (3.4,2.25) {$x$};
\node[scale=0.7] at (4.7,2.85) {$x^2$};
\end{tikzpicture}
\caption{$T^*S^1$ with $L$ and the wrapping $\phi^1_H(L)$ of $L$ on the left-hand side. Discs shown in the universal cover, contributing to $\mu^2(x,1)=x$, $\mu^2(x^2,x)=x^3$ and $\mu^2(x,x)=x^2$, respectively, on the right-hand side. The multiplication should be read as a map $\mu^2:CW_{T^*S^1}(\phi_G^1(L),\phi_H^1(L))\otimes CW_{T^*S^1}(L,\phi_G^1(L))\rightarrow CW_{T^*S^1}(L,\phi_H^1(L))$}.
\label{wfukdisc}
\end{figure}
\end{example}
\begin{example}[Stops and $\mathbb{C}P^1$]
Returning to $\mathbb{C}$; we have that $D^b(\text{Coh}(\mathbb{C}))=D^b(\mathbb{C}[X])$, and $D^b(\text{Coh}(\mathbb{C}^*))=D^b(\mathbb{C}[x,x^{-1}])$ is obtained by localising $x$. This localisation occurs when removing a divisor, which is a general pattern. Furthermore, following this pattern, $\mathbb{C}$ is obtained by removing a divisor from $\mathbb{C}P^1$, and the category $D^b(\mathbb{C}[X])$ is obtained by localising the category
\[
D^b(\text{Coh}(\mathbb{C}P^1))=D^b(kQ),\qquad Q=
\begin{tikzcd}
1\arrow[shift left, bend left]{r}\arrow[shift right, bend right]{r}&2,
\end{tikzcd}
\]
that is, the derived category of finitely generated modules of the path algebra of the Kronecker quiver (see e.g. \cite{MR4297810}).

Removing/adding a divisor (a point in the examples above) is supposed to correspond to removing/adding a so called stop in symplectic geometry. For now we can think of this as a neighbourhood of a Legendrian submanifold in the boundary at infinity, where we stop wrapping. In the example with $T^*S^1$, the boundary at infinity is two copies of $S^1$, and a Legendrian is just a collection of points on these two copies.

If we stop at one point, $D^{\pi}\mathcal{W}(TS^1,\text{pt})$ is known to be generated by a fibre (cocore) and a linking disc of the stop. However, the linking disc is actually not needed as a generator. Thus we need to compute $CW_{(T^*S^1,\text{pt})}(L,L)$. Now we the wrapping is stopped on, say the part with negative $p$-coordinate. The count of discs between these intersection points remains unchanged and we conclude that $CW_{(T^*S^1,\text{pt})}(L,L)=\mathbb{C}[X]$.

Adding two stops, one on each copy of $S^1$, changes the generators: $D^{\pi}\mathcal{W}(T^*S^1,\{\text{pt}_1,\text{pt}_2\})$ is no longer generated by just a fibre. It is known that it is generated by a fibre together with two linking discs, one for each stop. However one can show from this that it is also generated by two fibres, as in Figure~\ref{twostops} below (geometrically cones corresponds to "surgery" operation on Lagrangian submanifolds which allows to find new generators, see e.g. \cite[Figure 9]{MR3220941}). Here the wrapping stops after one intersection, and thus we obtain the Kronecker quiver.
\begin{figure}[H]
\centering
\begin{tikzpicture}
\draw[thick] (0,0) rectangle (3,4);
\draw[red!80!black, fill=red!80!black] (1.5,0) circle (0.1cm);
\draw[red!80!black, fill=red!80!black] (1.5,4) circle (0.1cm);
\draw (0,2) -- (3,2);
\node at (3.3,2) {$S^1$};

\draw[purple!80!black,thick] (1,0) to[out=90,in=-135] (1.5,2) to[out=45,in=-90] (2,4);
\draw[blue!80!black, thick] (0.5,0) to (0.5,4);
\node[blue!80!black] at (0.5,-0.3) {$L_1$};
\node[purple!80!black] at (1,-0.3) {$L_2$};

\draw[thick] (6,0) rectangle (9,4);
\draw[red!80!black, fill=red!80!black] (7.5,0) circle (0.1cm);
\draw[red!80!black, fill=red!80!black] (7.5,4) circle (0.1cm);
\draw (6,2) -- (9,2);
\node at (9.3,2) {$S^1$};

\draw[blue!80!black, thick] (6.5,0) to (6.5,4);
\node[blue!80!black] at (6.5,-0.3) {$L_1$};
\node[purple!80!black] at (8.25,-0.3) {$\phi^1_H(L_2)$};
\draw[purple!80!black,thick] (6,1) to (9,3);
\draw[purple!80!black,thick] (8,0) to[out=90,in=-135] (9,1);
\draw[purple!80!black, thick] (6,3) to[out=45,in=-90] (7,4);

\node[scale=0.8] at (6.7,3.2) {$x$};
\node[scale=0.8] at (6.7,1.2) {$y$};
\end{tikzpicture}
\caption{Generators $L_1$ and $L_2$ of $D^{\pi}\mathcal{W}(T^*S^1,\{\text{pt}_1,\text{pt}_2\})$ on the left-hand side. Calculation of $CW_{(T^*S^1,\{\text{pt}_1,\text{pt}_2\})}(L_1,L_2)=L_1\cap\phi^1_H(L_2)$ on the right-hand side. The stops are marked with red points. Thus the full category can be described by the algebra $
\bigoplus_{1\leq i,j\leq 2}CW_{(T^*S^1,\{\text{pt}_1,\text{pt}_2\})}(L_i,L_j)=\mathbb{C}\left(
1\,\substack{{\xlongrightarrow{x}}\\{\xlongrightarrow[y]{\,}}}\,2\right)$
.}
\label{twostops}
\end{figure}
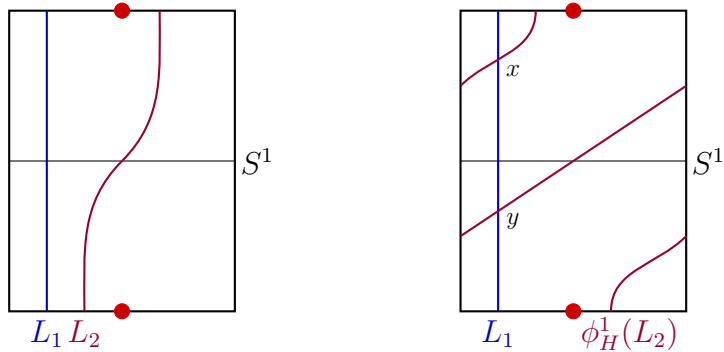
\end{example}
\begin{example}[Using the Chekanov--Eliashberg algebra and the surgery formula]\label{surgeryex}
Finally, we can also compute the wrapped Fukaya category, by computing the Chekanov--Eliashberg algebra (see Section~\ref{wrappedsec}), via the surgery formula (\cite{MR2916289}, \cite{MR4634745} and for this dimension \cite{bäcke2023contacthomologycomputationssingular}). In the example above, $T^*S^1$ is obtained by attaching one critical Weinstein handle to the subcritical domain $D^2$. By the surgery formula $CW_{T^*S^1}(L,L)=CE(\Lambda)$, where $\Lambda\subset S^1=\partial D^2$ is the Legendrian attaching link in the boundary at infinity consisting of two points. The Reeb flow here is given by rotation counterclockwise, thus the Reeb chords are given by $c^0_{01}$, $c^p_{ij}$, $0\leq i,j\leq 1$, $p\geq 1$. The grading is given by $|c_{i,j}^p|=2p-1+\delta_{j,1}-\delta_{i,1}$, where $\delta_{i,j}$ is the Kronecker delta. The differential can be shown to be (see \cite{bäcke2023contacthomologycomputationssingular})
\begin{align*}
&\partial c^0_{0,1}=0,\quad &&\partial c^1_{1,0}=0,\quad &&\partial c^1_{0,0}=1-c^1_{1,0}c^0_{0,1}, \\ &\partial c^1_{1,1}=1-c^0_{0,1}c^1_{1,0},\quad &&\partial c^1_{0,1}=-c^0_{0,1}c^1_{0,0}+c^1_{1,1}c^0_{0,1},\quad &&\partial c^p_{i,j}=\sum_{l=0}^p(-1)^{\delta_{l,1}+\delta_{j,1}}c^l_{lj}c^{p-l}_{il} 
\end{align*}
(where we use the convention $c^0_{00}=c^0_{11}=c^0_{10}=0$ in the last equation above). This DGA can be shown to formal, and the homology is generated by $x=c^0_{0,1}$ and $x^{-1}=c^1_{1,0}$; the relations $xx^{-1}=1=x^{-1}x$ follows from $\partial c^1_{1,1}$ and $\partial c^1_{0,0}$ (see also \cite{MR4033516}). 

Adding a stop yields the Chekanov--Eliashberg algebra of three points on $S^1$, with one idempotent corresponding to the attaching sphere $\Lambda=S^0$ and one idempotent to the stop. The algebra formed by Reeb chords starting and ending at these points is quasi-isomorphic to the partially wrapped cochains of the cocore and the linking disc. However, we know that the category is generated by just the cocore of the handle, and thus by the surgery formula we only need to compute the sub-dg-algebra consisting of chords starting and ending at $S^0$. This dg-algebra can also be computed in the boundary of the Weinstein sector which is obtained by removing the stop; this boundary is just an interval. The DGA of $S^0$ on an interval clearly is $k[x]$ (see Figure~\ref{surgone}). 
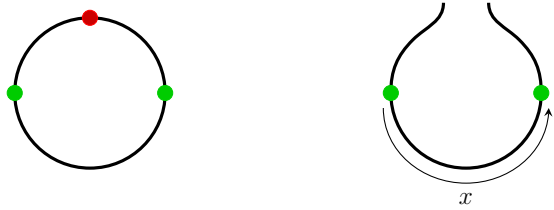
\begin{figure}[H]
\centering
\begin{tikzpicture}
\draw[very thick] (0,0) circle (1cm);
\draw[red, fill=red!80!black] (0,1) circle (0.1cm);
\draw[green!80!black, fill= green!80!black] (1,0) circle (0.1cm);
\draw[green!80!black, fill=green!80!black] (-1,0) circle (0.1cm);

\draw[very thick] (4.7,1.2) to[out=-90,in=45] (4.293,0.707) to[out=-135,in=90] (4,0) to[out=-90,in=180] (5,-1) to[out=0,in=-90] (6,0) to[out=90,in=-45] (5.707,0.707) to[out=135,in=-90] (5.3,1.2);
\draw[green!80!black, fill= green!80!black] (6,0) circle (0.1cm);
\draw[green!80!black, fill=green!80!black] (4,0) circle (0.1cm);
\draw[-stealth] (3.9,-0.2) to[out=-90,in=180] (5,-1.2) to[out=0,in=-90] (6.1,-0.2);
\node[scale=0.8] at (5,-1.4) {$x$};
\end{tikzpicture}
\caption{On the left-hand side is the attaching sphere in green and the stop in red. The chords starting and ending at the green points yields the sub-DGA of the cocore, which generates. On the right-hand side is also the DGA of the cocore, but from the sector point of view, together with the only Reeb chord $\Lambda\rightarrow\Lambda$.}
\label{surgone}
\end{figure}

Finally, adding two stops as in the cylinder above yields the following pictures, adding two $1$-handles to the boundaries of two discs, and a stop on the boundary on each disc. Again in the sector picture it is easy to compute the Chekanov--Eliashberg algebra, yielding the Kronecker quiver. 
\begin{figure}[H]
\centering
\begin{tikzpicture}
\draw[very thick] (-1,0) circle (1cm);
\draw[red!80!black, fill=red!80!black] (-1,1) circle (0.1cm);
\draw[green!80!black, fill= green!80!black] (0,0) circle (0.1cm);
\draw[green!80!black, fill=green!80!black] (-2,0) circle (0.1cm);

\draw[very thick] (2,0) circle (1cm);
\draw[red!80!black, fill=red!80!black] (2,-1) circle (0.1cm);
\draw[green!80!black, fill= green!80!black] (3,0) circle (0.1cm);
\draw[green!80!black, fill=green!80!black] (1,0) circle (0.1cm);

\node[green!80!black, scale=0.8] at (0.25,0) {$2$};
\node[green!80!black, scale=0.8] at (0.75,0) {$2$};
\node[green!80!black, scale=0.8] at (-2.25,0) {$1$};
\node[green!80!black, scale=0.8] at (3.25,0) {$1$};

\draw[very thick] (6.7,1.2) to[out=-90,in=45] (6.293,0.707) to[out=-135,in=90] (6,0) to[out=-90,in=180] (7,-1) to[out=0,in=-90] (8,0) to[out=90,in=-45] (7.707,0.707) to[out=135,in=-90] (7.3,1.2);
\draw[green!80!black, fill= green!80!black] (6,0) circle (0.1cm);
\draw[green!80!black, fill=green!80!black] (8,0) circle (0.1cm);
\draw[-stealth] (5.9,-0.2) to[out=-90,in=180] (7,-1.2) to[out=0,in=-90] (8.1,-0.2);

\draw[very thick] (10.3,-1.2) to[out=90,in=-135] (10.707,-0.707) to[out=45,in=-90] (11,0) to[out=90,in=0] (10,1) to[out=180, in=90] (9,0) to[out=-90,in=135] (9.293,-0.707) to[out=-45,in=90] (9.7,-1.2);
\draw[green!80!black, fill= green!80!black] (11,0) circle (0.1cm);
\draw[green!80!black, fill=green!80!black] (9,0) circle (0.1cm);
\draw[-stealth] (11.1,0.2) to[out=90,in=0] (10,1.2) to[out=180,in=90] (8.9,0.2);

\node[green!80!black, scale=0.8] at (8.25,0) {$2$};
\node[green!80!black, scale=0.8] at (8.75,0) {$2$};
\node[green!80!black, scale=0.8] at (5.75,0) {$1$};
\node[green!80!black, scale=0.8] at (11.25,0) {$1$};

\node[scale=0.8] at (7,-1.4) {$x$};
\node[scale=0.8] at (10,1.4) {$y$};
\end{tikzpicture}
\caption{The stops and the attaching spheres of the $1$-handles on the boundary of the Weinstein domain on the left-hand side. The points on the same $S^0$ are marked by the same number. The attaching spheres on the boundary of the Weinstein sector on the right-hand side.}
\end{figure}
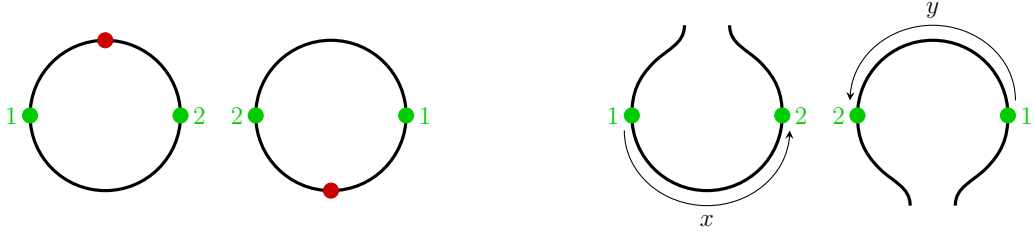
\end{example} 
\begin{example}[Landau--Ginzburg models]
We could also consider the $A$-side on $\mathbb{C}$ or $\mathbb{C}P^1$. Then the mirror is not given by a manifold, and instead of attaching stops we attach a (Landau--Ginzburg) potential $w:\mathbb{C}^*\rightarrow\mathbb{C}$ defined by a certain count of index $2$ Maslov disc. The $B$-side category is given by the product of the triangulated categories of singularities corresponding to the singular values of $w$ (or, equivalently, by matrix factorisation categories)
\[
DB(w)=\prod_{\lambda\in\mathbb{C}}D_{\text{sg}}(w^{-1}(\lambda))=\prod_{\lambda\in\mathbb{C}}\frac{D^b(\text{Coh}(w^{-1}(\lambda)))}{D^{\pi}(\text{Coh}(w^{-1}(\lambda)))}.
\]

In the example $\mathbb{C}$, $D^{\pi}\mathcal{W}(\mathbb{C})=0$. The potential is given by $w=z$, and hence $DB(w)=0$, as $w$ has no critical points. 

In the example $\mathbb{C}P^1$, the relevant $A$-side is not the wrapped but a compact Fukaya category (this is not a Liouville manifold, also note that this case is $\mathbb{Z}_2$-graded). Mirror symmetry is here known to be $\mathcal{F}(\mathbb{C}P^1)=DB(w)$, for 
 the potential $w=z+z^{-1}$. Note that $w$ has two critical points $z=\pm 1$, and thus $DB(w)= D_{\text{sg}}(1)\times D_{\text{sg}}(-1)$. 
(See e.g. \cite{MR2101296} for this example and more about Landau--Ginzburg models.)
\end{example}
\subsection{Log Calabi--Yau property and hyper-Kähler structures}
The result of this paper is motivated by the anticipation that manifolds in log Calabi--Yau hyper-Kähler families should have mirror given by a hyper-Kähler rotation (analogous to the case of certain closed hyper-Kähler K3 surfaces, see e.g. \cite{bruzzo1998mirrorsymmetryk3surfaces}). An elaboration of what this means and the relations to the $ADE$ singularities follows below.

A \emph{log Calabi--Yau manifold} is a manifold $U=X\smallsetminus D$, where $X$ is a smooth projective variety and $D$ is a normal crossing and anti-canonical divisor (i.e. $D\in |-\mathcal{K}_X|$). The pair $(X,D)$ defining $U$ is not unique, in general, and is known as a \emph{Looijenga pair}. Log Calabi--Yau manifolds are central in the study of cluster varieties and algebras, and Homological mirror symmetry for log Calabi--Yau surfaces has been studied quite extensively (see e.g. \cite{MR3415066}, \cite{MR4562569} and \cite{MR3966785}).

An equivalent definition for $U$ to be log Calabi--Yau, used below, is that there is a smooth and projective complex manifold $X$ containing $U$, such that $D=X\smallsetminus U$ is a normal crossing divisor and there exists a non-vanishing holomorphic volume form $\Omega$ on $U$ which as simple poles along $D$. 
\begin{example}\label{logcyexone}
For $\mathbb{C}P^n$, in a local chart $\mathbb{C}^n=\{(1:z_1:\cdots:z_n)\}$ the canonical bundle is given by $dz_1\wedge\cdots\wedge dz_n$. This has poles of order $n+1$ at $\mathbb{C}P^{n-1}_{\infty}=\{(0:z_1:\cdots:z_n)\}$. Thus an anti-canonical divisor is given by removing a degree $n$ normal crossing divisor from $\mathbb{C}^n=\mathbb{C}P^n\smallsetminus\mathbb{C}P^{n-1}_{\infty}$, e.g. $D=V(f)$, for a smooth regular function $f:\mathbb{C}^n\rightarrow\mathbb{C}$ of degree $n$, intersecting $\mathbb{C}P^{-1}_{\infty}$ transversely in a smooth subvariety.

For $n=1$, we can remove any point from $\mathbb{C}$, thus e.g. $\mathbb{C}^*$ is a log Calabi--Yau curve. For $n=2$, we can remove any smooth conic $C$ from $\mathbb{C}^2$. Thus, e.g. $\mathbb{C}^2\smallsetminus\{z_1z_2=-1\}$ is a log Calabi--Yau surface. Alternatively, one can see this by the fact that
\[
\Omega=\frac{dz_1\wedge dz_2}{z_1z_2+1}
\] 
is a holomorphic volume form on $\mathbb{C}^2\smallsetminus\{z_1z_2=-1\}$ which is nowhere vanishing and has simple poles along the normal crossing divisor $D=\{z_1z_2=-1\}\cup\mathbb{C}P^1_{\infty}$.
\end{example}
A \emph{hyper-Kähler} manifold $(M,g,I,J,K)$ is Riemannian manifold $M$ with metric $g$ and three covariantly constant integrable complex structures $I$, $J$ and $K$, satisfying the quaternionic relation $IJK=-1$. Each of the complex structures defines a symplectic form by the formulas $\omega_I(-,-)=g(I(-),-)$, $\omega_J(-,-)=g(J(-),-)$ and $\omega_K(-,-)=g(K(-),-)$, making it a Kähler manifold; in fact, for any $u=(u_I,u_j,u_K)\in S^2$, $I_u=u_II+u_JJ+u_KK$ is a complex structure for which $M$ together with $\omega_{u}(-,-)=g(I_u(-),-)=u_I\omega_I+u_J\omega_J+u_K\omega_K$ is a Kähler manifold. 

\begin{example}\label{basicHK}
The standard example of a hyper-Kähler manifold is $\mathbb{H}=\mathbb{C}\times\mathbb{C}$, with $g$ the Euclidean metric and with 
\[
I=\begin{pmatrix}
    i&0\\0&-i
\end{pmatrix},\qquad J=\begin{pmatrix}
    0&-1\\1&0
\end{pmatrix},\qquad K=\begin{pmatrix}
    0&-i\\-i&0
\end{pmatrix}.
\]
\end{example}

\begin{example}[ALE spaces as hyper-Kähler quotients \cite{MR992334}]\label{HKquotient}
By \cite{MR992334} (see also \cite[Section 7]{MR2720232}), given a simply connected hyper-Kähler manifold $X$ with the action of a compact Lie group $F$, there are three moment maps $\mu_I,\mu_J$ and $\mu_K$ (obtained with respect to $\omega_I$, $\omega_j$ and $\omega_K$ respectively), which put together yields a map $\mu=(\mu_I,\mu_J,\mu_K):X\rightarrow\mathbb{R}^3\otimes\mathfrak{f}^{\vee}$ (where $\mathfrak{f}$ denotes the Lie algebra of $F$). Given $\zeta\in\mathbb{R}^3\otimes Z$ , where $Z=\text{centre}(\mathfrak{f}^{\vee})$, the \emph{hyper-Kähler quotient} is defined to be the space
\[
X_{\zeta}=\mu^{-1}(\zeta)/F.
\] 
If $F$ acts freely, then $X_{\zeta}$ is a manifold of (real) dimension $\text{dim}(X)-4\text{dim}(F)$ \cite[Proposition 2.1]{MR992334}. It inherits a hyper-Kähler structure from $X$.

For $\mathbb{C}^2$ with $G\subset\text{SU}(2)$ finite, we put $X=(Q\otimes_{\mathbb{C}}\text{End}(R))^G$, where $Q$ is the canonical $2$-dimensional representation and $R=\mathbb{C}[G]$ is the regular representation (the group algebra with action induced by group multiplication) endowed with the natural Euclidean metric. We make the space $Q\otimes_{\mathbb{C}}\text{End}(R)$, and therefore also $X$, into an $\mathbb{H}$-module, as follows. Elements of the tensor product $Q\otimes\text{End}(R)$ can be viewed as pairs of morphisms $(\alpha,\beta)$ in $\text{End}(R)$. The action of $I$ is defined by $I(\alpha,\beta)=(i\alpha,i\beta)$ and the action of $J$ by $J(\alpha,\beta)=(-\alpha^*,\beta^*)$, where the metric is used to define the real structure (anti-linear involution) $(-)^*$ on the complex vector space $\text{End}(R)$. Using this, $X$ is given the structure of a hyper-Kähler manifold.

We next describe the Lie group $F$ and its action on $X$. Note that $X$ comes equipped with the diagonal action of $U(R)$, acting by conjugation in each factor: $f\cdot(\alpha,\beta)=(f\alpha f^{-1},f\beta f^{-1})$. We let $F'$ be the elements of $U(R)$ which commute with $G$, and then we put $F=F'/T$, where $T$ is the subgroup of unit scalars. 

For $\zeta=0$, the hyper-Kähler quotient is $X_0\cong \mathbb{C}^2/G$, and for generic choices of $\zeta$ (see \cite[Definition 38]{MR2720232} or \cite{MR992334}), $X_{\zeta}$ is a smooth manifold diffeomorphic to the resolution of singularities of $\mathbb{C}^2/G$; varying $\zeta\in Z\otimes\mathbb{R}^3$ (or more precisely an open subset) classifies the so called $ALE$ spaces.

We will restrict interest to the spaces $X_{\zeta}$ for $\zeta=(\zeta_1,0,0)$. Then, via the isomorphism $Z\cong H^2(X_{\zeta};\mathbb{R})$, the cohomology classes of the Kähler forms are: $[\omega_I]=\zeta_1$ and $[\omega_J]=[\omega_K]=0$. For each of these choices of $\zeta$, we have a family of Kähler manifolds $(X_{\zeta},I_u,\omega_u)$, $u=(u_I,u_J,u_K)\in S^2$. For $u_I=\pm 1$ (i.e. $u_J=u_K=0$), $(X_{\zeta},\pm I)$ is the minimal resolution of singularities $X_{\text{res}}$, obtained by repeatedly blowing up the singularity. For $u_I=0$, we obtain an $S^1$-family of symplectic manifolds $(X_{\zeta},\omega_u)$ which are exact; they are all equivalent and Weinstein, and, more precisely, they are given by the Milnor fibres $X_{\text{Mil}}$ (see Section~\ref{geomeq}). 

We can even say more: With respect to the symplectic structures $\omega_u$, for all $u\in S^2$ with $u_I\not=0$, $(X_{\zeta},\omega_u)$ is a non-exact symplectic manifold. They can be thought of as the resolution of singularities with area on the exceptional spheres; however changing the $\zeta_1$-parameter changes the areas of these $(-2)$-spheres and may result in non-equivalent symplectic manifolds. For a fixed $\zeta_1$, changing $u\in S_2\smallsetminus\{u_I=0\}$ rescales the symplectic area of the exceptional $(-2)$-spheres by a common factor (see Remark~\ref{moreonfamily} below for further details, in the case $G=\mathbb{Z}_{n+1}$). 

With respect to the complex structures, for $u\in S^2\smallsetminus\{(\pm1,0,0)\}$, i.e. away from these poles, $(X_{\zeta},I_u)$ is a Milnor fibre considered as an affine variety. In the type $A_n$ case, that is when $G=\mathbb{Z}_{n+1}$, it can be given as a hypersurface of the form $\{P(Z)=XY\}\subset\mathbb{C}^3$, for a polynomial $P(Z)$ in one variable degree $n+1$, having degree $n+1$ and distinct roots. As affine varieties, such hypersurfaces are not necessarily isomorphic, analogously to the symplectic structures on the resolutions of singularities (see however Remark~\ref{moreonfamily}).
\end{example}
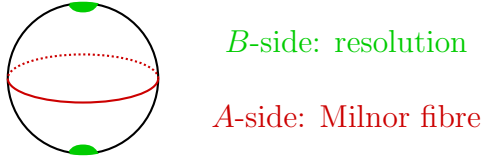
\begin{figure}[H]
\centering
\begin{tikzpicture}
\draw[thick] (-1,0) circle (1cm);
\draw[densely dotted, thick, red!80!black] (0,0) arc (0:180:1 and 0.32);
\draw[thick,red!80!black] (-2,0) arc (0:180:-1 and -0.32);

\draw[green!80!black,fill, very thick] (-1.174,0.985) to[out=-60,in=180] (-1,0.9) to[out=0,in=-120] (-0.826,0.985) to[out=170,in=0] (-1,1) to[out=180,in=10] (-1.174,0.985);

\draw[green!80!black,fill, very thick] (-1.174,-0.985) to[out=60,in=180] (-1,-0.9) to[out=0,in=120] (-0.826,-0.985) to[out=-170,in=0] (-1,-1) to[out=180,in=-10] (-1.174,-0.985);

\node[green!80!black] at (2.5,0.5) {$B$-side: resolution};
\node[red!80!black] at (2.5,-0.5) {$A$-side: Milnor fibre};
\end{tikzpicture}
\caption{The $S^2$-family of Kähler structures shown. On the equator the $u_I=0$ the symplectic manifold $(X_{\zeta},\omega_u)$ is exact symplectic and equivalent to the Milnor fibre. For the other values the manifolds are determined by the area put on the exceptional spheres. On the poles the varieties $(X_{\zeta},I_u)$ are the resolutions of singularities, whereas for the other values they are affine varieties. In the type $A_n$ case, when $G=\mathbb{Z}_{n+1}$, the affine varieties are given by the hypersurfaces defined by the equations $P(Z)=XY$, where $P(Z)$ is a polynomial of degree $n+1$ without repeated roots.}
\label{hkfamilypic}
\end{figure}

Next we give the definition for a family of manifolds to be log Calabi--Yau hyper-Kähler, followed by two examples, relating this concept to Example~\ref{basicHK} an Example~\ref{HKquotient} above. 
\begin{definition}\label{logCYHKdef}
We say that a family $(U_u,I_u,\omega_u)$, $u\in O\subset S^2$, where $O$ is an open subset, of open Kähler manifolds is \emph{log Calabi--Yau hyper-Kähler} if the the following conditions are satisfied:
\begin{enumerate}[(i)]
    \item There exists a hyper-Kähler manifold $(U,g,I,J,K)$ such that $U_u=U\smallsetminus D_u$, where $D_u$ is a continuous family of $I_u$-holomorphic divisor, and such that the Kähler structures on $(U_u,I_u,\omega_u)$ are compatible with the hyper-Kähler structure on $U$ in the sense that $I_u=u_II+u_JJ+u_KK$ and $\omega_u=u_I\omega_I+u_J\omega_J+u_K\omega_K$, for $u=(u_I,u_J,u_K)\in S^2$.
    \item For each $u\in O\subset S^2$, there is a compactification of $U=U_u\cup D_u$ by a complex manifold $X_u$, such that $D_u'=X_u\smallsetminus U_u$ is a normal crossing anti-canonical divisor with respect to (the extension to the compactification of) $I_u$. 
\end{enumerate}
\end{definition}
In particular each $U_u$ is log Calabi--Yau. Note also that $D_u\subset D_u'$, but we do not require equality. 

\begin{remark}
We expect mirrors of the Kähler manifolds in such families to be obtained by hyper-Kähler rotation, i.e. by rotating to another Kähler manifold in the family. Note however, that this definition is never used in proving any example of homological mirror symmetry. It is only a tentative proposal, formulated in the hope that it encapsulates a more general phenomenon. 
\end{remark}
\begin{remark}
One would naturally look for a global family of divisors $D_u$, i.e. $O=S^2$, as as is the case in Example~\ref{logCYHKexone} below. However, it is not possible to extend the choice of divisors in Example~\ref{loghkA} to a global family without altering the topology of the divisors; such a choice will necessarily become nodal over one of the poles. 
\end{remark}

\begin{example}\label{logCYHKexone}
Consider the hyper-Kähler manifold $(\mathbb{H},g,I,J,K)$ from Example~\ref{basicHK}. Each $I_u=u_II+u_JJ+u_KK$, $u=(u_I,u_J,u_K)\in S^2$, lie in $\text{SU}(2)$ and has eigenvalues $\pm i$. Thus we can find unitary matrices $M_u$ such that $I_u=M_u^{-1}IM_u$ (where we take $M_{(1,0,0)}=\text{Id}$). Moreover, as a family, we can choose $M_u$ to vary smoothly with $u\in S^2$. Using this, we define $f_u:\mathbb{H}\rightarrow\mathbb{C}$ in two steps: we put $f_{(1,0,0)}=z_1\overline{z}_2$, and then $f_u=f_{(1,0,0)}\circ M_u$, for $u\in S^2$. A simple calculation shows that $df_{(1,0,0)}\circ I=i\,df_{(1,0,0)}$. Therefore, 
\[
df_u\circ I_u=df_{(1,0,0)}\circ M_u\circ M_u^{-1}IM_u=i\,df_{(1,0,0)}\circ M_u=i\,df_u.
\]
Hence $f_u$ is $I_u$-holomorphic and $D_u=\{f_u=0\}$ is an $I_u$-holomorphic divisor.

Put $U_u=\mathbb{H}\smallsetminus D_u$. We claim that $(U_u,I_u,\omega_u)$ is a log Calabi--Yau hyper-Kähler family  (where $I_u$ and $\omega_u$ are inherited from the hyper-Kähler structure on $\mathbb{H}$).

Condition (i) of Definition~\ref{logCYHKdef} is satisfied by construction, with $U=\mathbb{H}$. For condition (ii), note that $M_u^{-1}$ defines a biholomorphism $M_u^{-1}:(U_{(1,0,0)},I_{(1,0,0)})\rightarrow (U_u,I_u)$. Therefore, it suffices to find a compactification $X$ of $(U_{(1,0,0)},I_{(1,0,0)})=((\mathbb{C}^*)^2,I)$, such that $D'=X\smallsetminus (\mathbb{C}^*)^2$ is a normal crossing divisor with respect to $I$. This is of course accomplished by $X=\mathbb{C}P^1\times\mathbb{C}P^1$ (where use the extension of the complex structure $-i$ in the second factor). We conclude that each $(U_u,I_u)$, is a log Calabi--Yau surface, and hence that $(U_u,I_u,\omega_u)$, $u\in S^2$, is a log Calabi--Yau hyper-Kähler family. 

We next verify that homological mirror symmetry for this family can be obtained by hyper-Kähler rotation. Firstly, $\text{Coh}(U_u,I_u)\cong\text{Coh}((\mathbb{C}^*)^2)\cong\mathbb{C}[x_1^{\pm 1},x_2^{\pm 1}]\text{-mod}$ and hence 
\[
D^b(\text{Coh}(U_u,I_u))=D^b([x_1^{\pm 1},x_2^{\pm 1}]\text{-mod}).
\]
Secondly, since $M_u^{-1}$ is unitary (and linear), it is a symplectomorphism:
\[
(M_u^{-1})^*\omega_u=g(I_uM_u^{-1}(-),M_u^{-1}(-))=g(M_u^{-1}I(-),M_u^{-1}(-))=g(I(-),-)=\omega_I.
\]
We have already seen that $D^{\pi}\mathcal{W}(\mathbb{C}^*,\omega_{\text{std}})=D^{\pi}\mathcal{W}(T^*S^1,\omega_{\text{std}})$ is triangulated equivalent to 
\[
D^{\pi}(\mathbb{C}[x,x^{-1}]\text{-mod})=D^b(\mathbb{C}[x,x^{-1}]\text{-mod})
\]
(Example~\ref{examplecylinder}). Thus, using $(\mathbb{C}^*,\omega_{\text{std}})\cong(\mathbb{C}^*,-\omega_{\text{std}})$ together with the Künneth formula \cite[Corollary 1.18]{MR4695507}, we conclude that we have equivalences 
\[
D^{\pi}\mathcal{W}(U_u,\omega_u)\simeq D^{\pi}\mathcal{W}((\mathbb{C}^*)^2,\omega_I)\simeq D^{\pi}\mathbb{C}[x_1^{\pm 1},x_2^{\pm 1}]\text{-mod}.
\]

Therefore, in this example, mirror symmetry is obtained by any hyper-Kähler rotation, for example by self-mirroring, i.e. the trivial hyper-Kähler rotation. 
\end{example}

The following example is of central importance. We study the hyper-Kähler manifolds $X_{\zeta}$ from Example~\ref{HKquotient}, restricting to the case $G=\mathbb{Z}_{n+1}$. We remove an appropriate family of $I_u$-divisors $D_u$. We show in Proposition~\ref{logCYHKprop} below that this turns it into a log Calabi--Yau hyper-Kähler family in the sense of Definition~\ref{logCYHKdef}.

\begin{example}[The underlying log Calabi--Yau hyper-Kähler families of $X_{\text{res}}\smallsetminus D_r$ and $X_{\text{Mil}}\smallsetminus D_m$]\label{loghkA}
We consider $U=X_{\zeta}$ from Example~\ref{HKquotient}, now only in the case $G=\mathbb{Z}_{n+1}$. We let $O=\{u\in S^2\;|\;I_u\geq 0\}$ be the upper hemisphere in $S^2$. We show that we can remove $I_u$-divisors $D_u$, for $u\in O$, in a way that makes this family log Calabi--Yau hyper-Kähler. 

For $u\in O\smallsetminus\{(1,0,0)\}$, as is explained in Remark~\ref{moreonfamily}, the affine variety $(X_{\zeta},I_u)$ can be concretely realised as the hypersurface defined by
\[
XY=P_{\tilde{u}}(Z),
\]
where
\[
P_{\tilde{u}}(Z)=(Z-\tilde{u}b_1)\cdots(Z-\tilde{u}b_{n+1}).
\]
Here $0=b_{n+1}, b_1,\dots,b_n$ are real and distinct, depending on $\zeta$, and $\tilde{u}\in \{z\in\mathbb{C}\;|\;0<|z|\leq 1\}\subset\mathbb{C}P^1$ is determined from $u$ by stereographic projection (sending $\tilde{u}=0$, $\tilde{u}=1$ and $\tilde{u}=\infty$ to $u=(1,0,0)$, $u=(0,1,0)$ and $u=(-1,0,0)$ respectively).

We put $D_u=\{Z=-C\}$, for $C$ some constant such that $C>|b_i|$, for $i=1,\dots,n+1$. Note that the affine varieties $(X_{\zeta},I_u)=\{XY=P_{\tilde{u}}(Z)\}$ are all isomorphic (here $u\in O\smallsetminus\{(1,0,0)\}$). An explicit isomorphism 
\[
(X_{\zeta},I_{u})=\{XY=P_{\tilde{u}}(Z)\}\xrightarrow{\sim}(X_{\zeta},I_{(0,1,0)})=\{XY=P_1(Z)\}
\]
is given by $X\mapsto \tilde{u}^{n+1}X$, $Y\mapsto Y$ and $Z\mapsto \tilde{u}Z$. However, this isomorphism does not identify $D_u$ and $D_{u'}$ when $\tilde{u}\not=\tilde{u}'$. Thus, the affine varieties after removing divisors $(X_{\zeta}\smallsetminus D_u,I_u)$ are not necessarily isomorphic. 

Next, we put $U_u=U\smallsetminus D_u$, $u\in O$. Then, this is a family of Kähler manifolds (see Example~\ref{HKquotient} for the definition of $I_u$ and $\omega_u$). We show in Proposition~\ref{logCYHKprop} below that this is a log Calabi--Yau hyper-Kähler family. 

Therefore we expect that homological mirror symmetry can be obtained by hyper-Kähler rotation, i.e. we anticipate that $X_{\zeta}\smallsetminus D_u$, for any $u\in S^2$, is mirror to a hyper-Kähler rotation $X_{\zeta}\smallsetminus D_{\check{u}}$, for another $\check{u}\in S^2$. Of particular interest are the pole $u=(1,0,0)$ and the equator $\check{u}=(0,\check{u}_J,\check{u}_K)$. More specifically, here we expect (later we prove this in one direction and partly in the other) that $(X_{\zeta}\smallsetminus D_{(1,0,0)},I,\omega_I)$ is mirror to the $S^1$-family $(X_{\zeta}\smallsetminus D_u,I_u,\omega_u)$, for $u\in\{u_I=0\}\subset S^2$. 

For $u=(1,0,0)$, when $(X_{\zeta},I_{(1,0,0)})$ is the crepant resolution $X_{\text{res}}$ of the $A_n$-singularity, we let $D_{(1,0,0)}=\{Z=-C\}\subset\{XY=Z^{n+1}\}\smallsetminus\{(0,0,0)\}=X_{\text{res}}\smallsetminus E$, where $E$ denotes the exceptional divisor. Alternatively, this divisor is obtained by the $G$-invariant (this is one reason why we restrict to the cyclic groups) divisor $\tilde{D}_{(1,0,0)}=\{z_1z_2=-C\}\subset\mathbb{C}^2$. This is a smooth conic disjoint from the origin (compare with Example~\ref{logcyexone}). We project down to a divisor in $\mathbb{C}^2/G$ which we thereafter pull back to the divisor $D_{(1,0,0)}$ in $U$. This a family of divisors smoothly varying with $u$ in the upper hemisphere $O$. Note that by a change of coordinates, at the pole, all choices of the constant $C$ define isomorphic varieties; cf. Proposition~\ref{bsidecycplumb} and Remark~\ref{bsidecycplumbremark}.

We will sometimes write $D_m$ for a divisor $D_u$ with $u=(0,u_J,u_K)$, and $D_r$, when $u=(1,0,0)$.
\end{example}

\begin{proposition}\label{logCYHKprop}
The family of open Kähler manifolds $(U_u,I_u,\omega_u)$ from Example~\ref{loghkA} is a log Calabi--Yau hyper-Kähler family.
\end{proposition}
\begin{proof}
It remains to see that $U_u=X_{\zeta}\smallsetminus D_u$ in Example~\ref{loghkA} is log Calabi--Yau, i.e. we need to make sure that we can compactify $X_{\zeta}$ at infinity for each complex structure $I_u$, adding $D_u^{\infty}$ so that $D_u'=D_u\cup D_u^{\infty}$ is anti-canonical and has nodal singularities.

We start with $u=(1,0,0)$, i.e. the resolution of singularities (this is of course equivalent to taking $u=(-1,0,0)$), and we can assume that $D_r=\{z_1z_2=-1\}$, without loss of generality, since removing the divisor $\{z_1z_2=-C\}$ from the resolution yields isomorphic varieties for all $C$. We consider the nowhere vanishing holomorphic volume form
\[
\Omega=\frac{dz_1\wedge dz_2}{z_1z_2+1}
\]
on $\mathbb{C}^2\smallsetminus\{z_1z_2=-1\}$. It has simple poles along $\tilde{D}_r=\{z_1z_2=-1\}$ and at infinity $\mathbb{C}P^1$ (note that $dz_1\wedge dz_2$ has poles of order $3$ at $\mathbb{C}P^1_{\infty}$). The form is invariant under the $G$-action, hence it defines a nowhere vanishing holomorphic volume form on $\mathbb{C}^2/G\smallsetminus D_r$, with simple poles along $D_r=\pi(\{z_1z_2=-1\})$, where $\pi:\mathbb{C}^2\rightarrow\mathbb{C}^2/G$ is the projection. Since the resolution is crepant, this form extends to a nowhere vanishing $I$-holomorphic volume form on $X_{\zeta}\smallsetminus D_r$; we denote this form also by $\Omega$, with no risk of confusion.

Note that the action of $G$ on $\mathbb{C}^2$ extends to an action on $\mathbb{C}P^2$ as follows. Using homogeneous coordinates on $\mathbb{C}P^2$, we express our standard neighbourhood $U_0=\mathbb{C}^2\subset\mathbb{C}P^2$ as the subset $\{[1:z_1:z_2]\}$. Then $\mathbb{C}P^1_{\infty}=\mathbb{C}P^2\smallsetminus U_0=\{[0:z_1:z_2]\}$. The complex projective plane $\mathbb{C}P^2$ is covered by three affine charts: $U_0$ together with $U_1=\{[w'_1:1:w'_2]\}\cong \mathbb{C}^2$ and $U_2=\{[w_1:w_2:1]\}\cong\mathbb{C}^2$. By a simple change of coordinates we see that the action extends to $\mathbb{C}^2\cong U_1$ by $\xi\cdot(w'_1,w'_2)=(\xi^{-1}w'_1,\xi^{-2}w'_2)$ and to $U_2$ by $\xi\cdot (w_1,w_2)=(\xi w_1,\xi^2w_2)$.

Observe that, besides from the singularity at the origin in $U_0$ which is already dealt with, $\mathbb{C}P^2/G$ has two more singular points: one in the image of the origin in the chart $U_1$ and one in the image of the origin in the chart $U_2$. We construct $X_u=X_{(\pm 1,0,0)}$ by resolving these singularities. Writing $f_u:X_u\rightarrow\mathbb{C}P^2/G$, our aim is then to show that $(f_u)^*\Omega$ has simple poles along $\tilde{D}_r=f_u^{-1}(\pi(\mathbb{C}P^1_{\infty}))$, and that $D_r'=D_r\cup \tilde{D}_r$ has nodal singularities.

We prove this in $U_2/G$; we note that by the group automorphism $\xi\mapsto\xi^{-1}$ it must then also hold in $U_1/G$. Note that, by the coordinate change $w_1=1/z_2$, $w_2=z_1/z_2$, we can write
\[
\Omega=\frac{dw_1\wedge dw_2}{w_1(w_2+w_1^2)}.
\]

Following \cite[Theorem 3.2]{reid93}, starting with the case when $n+1$ is odd. In this case, the quotient $U_2/G$ can be written as the singular hypersurface 
\[
\{(u_0,\dots,u_{_{n/2}})\;|\;u_{i-1}u_{i+1}=u_i^2\}\subset\mathbb{C}^{n/2+1}
\]
(here $u_i$ is identified with the monomial $w_1^{n+1-2i}w_2^i\in \mathbb{C}[w_1,w_2]^G$). In particular, $\pi(\mathbb{C}P^1_{\infty})=U_2/G\cap\{u_1=0\}$. Furthermore, we note that
\[
\frac{1}{n+1}\cdot\frac{du_0\wedge du_1}{u_0(u_1+u_0)}=\frac{dw_1\wedge dw_2}{w_1(w_2+w_1^2)}.
\]

The resolution can be described by
$f^{-1}_{u}(U_2/G)=Y_0\cup Y_1\cup Y_2$, with $Y_i\cong\mathbb{C}^2$, where we write $(\xi_i,\eta_i)$ for the coordinates in $Y_i$; the gluing is given by
\[
Y_i\smallsetminus\{\eta_i=0\}\xrightarrow{\sim}Y_{i+1}\smallsetminus\{\xi_{i+1}=0\},\qquad \xi_{i+1}=\frac{1}{\eta_i},\;\eta_{i+1}=\xi_i\eta_i^{b_i},
\]
where $b_0=n/2+1$ and $b_1=2$ (where these numbers are obtained by Hirzebruch--Jung continued fraction of $(n+1)/2$). The birational map $Y_0\rightarrow U_2/G$ is given by
\[
(\xi_0,\eta_0)\mapsto(u_0,u_1,u_2,\dots,u_{n/2})=\left(\xi_0,\xi_0\eta_0,\xi_0\eta_0^2,\dots,\xi_0\eta_0^{n/2}\right)
\]
(i.e. $\xi_0=u_0=w_1^{n+1}$ and $\eta_0=u_1/u_0=w_2/w_1^{2}$). Therefore, 
\[
((f_u)^*\Omega)_{|Y_0}=\frac{1}{n+1}\cdot\frac{d\xi_0\wedge d\eta_0}{\xi_0(\eta_0+1)}.
\]
We note in particular that this form has only simple poles, which are along a divisor with nodal crossings. Here, the part $\{\xi_0=0\}$ is the intersection of $Y_0$ with $f_u^{-1}(\pi(\mathbb{C}P^1_{\infty}))$ and the part $\{\eta_0=-1\}$ is the intersection with $D_r$.

Similarly (using the coordinate transformations) we deduce that
\begin{align*}
((f_u)^*\Omega)_{|Y_1}&=\frac{1}{n+1}\cdot\frac{d\xi_1\wedge d\eta_1}{\xi_1\eta_1(1+\xi_1)} ,\\
((f_u)^*\Omega)_{|Y_2}&=\frac{1}{n+1}\cdot\frac{d\xi_2\wedge d\eta_2}{\xi_2\eta_2(1+\xi_2^2\eta_2)}.
\end{align*}
Again we notice that this form has simple poles along a normal crossing divisor. The union of the divisors $\{\xi_1=0\}$ and $\{\eta_1=0\}$ are the intersection of $Y_1$ with $f_u^{-1}(\pi(\mathbb{C}P^1_{\infty}))$ and the part $\{\xi_1=-1\}$ is the intersection of $Y_1$ with $D_r$; the parts $\{\xi_2=0\}$ and $\{\eta_2=0\}$ are the intersection of $Y_2$ with $f_u^{-1}(\pi(\mathbb{C}P^1_{\infty}))$ and the divisor $\{\xi_2^2\eta_2=-1\}$ is the intersection of $Y_2$ with $D_r$.

We conclude that $(X_{\zeta}\smallsetminus D_r,\pm I)$ is log Calabi--Yau when $n+1$ is odd. Next we assume that $n+1$ is even. In this case we note that the map 
\[(w_1,w_2)\mapsto(\tilde{w}_1,\tilde{w}_2)=(w_1^2,w_2)
\]
induces as isomorphism $\mathbb{C}^2/G\rightarrow\mathbb{C}^2/\tilde{G}$, where $\tilde{G}=\mathbb{Z}/(\frac{n+1}{2})$, where the action of $\tilde{G}$ on $\mathbb{C}^2$ by $\tilde{\xi}\cdot(\tilde{w}_1,\tilde{w}_2)=(\tilde{\xi}\tilde{w}_1,\tilde{\xi}\tilde{w}_2)$, where $\tilde{\xi}=\xi^2=e^{2\pi i/(\frac{n+1}{2})}$ is a generator of the cyclic group $\tilde{G}$. (See e.g. \cite[Section 1.1]{MR385162} for the more general rationale behind this.)

In the this case, we can write $U_2/G$ as 
\[
\{(u_0,\dots,u_{(n+1)/2)})\;|\;u_iu_{j+1}=u_{i+1}u_j,\,0\leq i<j<(n+1)/2\}\subset\mathbb{C}^{(n+1)/2+1},
\]
(where we identify $u_i$ with $\tilde{w}_1^{(n+1)/2-i}\tilde{w}_2^{i}=w_1^{n+1-2i}w_2^{i}$). In terms of these, just as above, we can write the holomorphic volume form as
\[
\frac{1}{n+1}\cdot\frac{du_0\wedge du_1}{u_0(u_1+u_0)}=\frac{dw_1\wedge dw_2}{w_1(w_2+w_1^2)}.
\]

The singularity can be resolved by
\[
f_u^{-1}(U_2/G)=Y_0\cup Y_1,
\]
where $Y_i\cong\mathbb{C}^2$ are glued along $(\xi_1,\eta_1)=\left(1/\eta_0,\xi_0\eta_0^{(n+1)/2}\right)$ and the map $Y_0\rightarrow U_0/G$ is given by
\[
(\xi_0,\eta_0)\mapsto(u_0,u_1,u_2,\dots,u_{(n+1)/2})=\left(\xi_0,\xi\eta_0,\xi_0\eta_0^2,\dots,\xi_0\eta_0^{(n+1)/2}\right)
\]
(i.e. $\xi_0=u_0=\tilde{w}_1^{(n+1)/2}=w_1^{n+1}$ and $\eta_0=u_1/u_0=\tilde{w}_2/\tilde{w}_1=w_2/w_1^{2}$).
The holomorphic volume form can be expressed in these charts as
\begin{align*}
((f_u)^*\Omega)_{|Y_0}&=\frac{1}{n+1}\cdot\frac{d\xi_0\wedge d\eta_0}{\xi_0(\eta_0+1)}, \\
((f_u)^*\Omega)_{|Y_1}&=\frac{1}{n+1}\cdot\frac{d\xi_1\wedge d\eta_1}{\xi_1\eta_1(1+\xi_1)}.
\end{align*}
Again we conclude that this form has with simple poles along a normal crossing divisor. The translation of the descriptions of the divisors in different charts is identical to the case above, $n+1$ being odd.

Outside of the poles, i.e. $u=(u_I,u_J,u_K)\in S^2$, with $u_I\not=\pm 1$, the complex manifold $(X_{\zeta},I_u)$ is a Milnor fibre and can be described as a hypersurface given by $P(Z)=XY$, for $P(Z)$ a polynomial in one variable of degree $n+1$ without repeated roots. When $X_{\zeta}$ is defined by $P(Z)=Z^{n+1}-1$, the complex manifold $(X_{\zeta}\smallsetminus D_u,I_u)$ have been shown to be log Calabi--Yau in \cite[Example 3.2.4]{MR4039179}. In the general case, denote the roots of $P(Z)$ by $Q_1,\dots,Q_{n+1}$. Since the construction and proof in \cite[Section 7.1, Lemma 7.1]{MR2787361}, carry over to this case (it is however crucial for the proof that the roots are distinct), replacing $\xi_k$ with $Q_k$, we conclude that $(X_{\zeta}\smallsetminus D_u,I_u)$ is log Calabi--Yau.
\end{proof}

\begin{remark}\label{moreonfamily}
For $X_{\zeta}$ in Example~\ref{HKquotient} and Example~\ref{loghkA}, in the case $G=\mathbb{Z}_{n+1}$, we expand our exposition slightly on the algebraic structure and its relation to the symplectic structures for $u\in O\smallsetminus\{(1,0,0)\}$.

The symplectic areas of the exceptional spheres $E_i$ with respect to $\omega_I$ (that is, the values $\zeta_1^i$ in the $n$-tuple $\zeta_1=(\zeta_1^1,\dots,\zeta_1^n)$) are proportional to $b_{i}-b_{i-1}$, for $i=1,\dots,n$, by a common factor, where we put $b_0=b_{n+1}=0$ (see \cite[page 475]{MR520463}, and also \cite[Section 2]{MR4865839} and \cite[Section 2, Section 4]{MR4271387}). Here we use the assumption $\zeta=(\zeta_1,0,0)$ (which means that the symplectic area of the exceptional spheres $E_i$ with respect to $\omega_J$ and $\omega_K$ is $0$). For a fixed complex structure away from the poles, that is, for a fixed $\tilde{u}\in\mathbb{C}^*$, we let $\Lambda_i=\tilde{u}b_i$ denote the roots of $P_{\tilde{u}}(Z)$ (where thus $\Lambda_0=\Lambda_{n+1}=0$). We put $\lambda_i=\Lambda_{i}-\Lambda_{i-1}$, for $i=1,\dots,n+1$. Then $\Lambda_i=\sum_{j=1}^i\lambda_j$, for $i=1,\dots,n+1$, and the following relation is satisfied: $\lambda_1+\cdots+\lambda_{n+1}=0$. Compare this to Proposition~\ref{nonlocmil} (and also Proposition~\ref{parameters} and Proposition~\ref{multloc}, via Remark~\ref{substrmk}). 

In terms of the genericity condition for $\zeta$ \cite[Definition 38]{MR2720232}, we require all areas $\zeta_1^i$ to satisfy $\sum_{l=i}^j\zeta_l\not=0$. This is equivalent to $\sum_{l=i}^j\lambda_l=\Lambda_j-\Lambda_{i}\not=0$, $i=1,\dots, n+1$. In particular, all roots $\Lambda_i$ are distinct, which is equivalent to $X_{\zeta}$ being smooth; compare again with Proposition~\ref{nonlocmil}.

Furthermore, the sufficiency of the genericity condition also implies that all choices of $0=b_{n+1}=b_0,b_1,\dots,b_n$ are realised in some $X_{\zeta}$.

Lastly, in the definition $D_u=\{Z=-C\}$, note that varying the constant $C$ is equivalent to keeping the divisor fixed, say at $\{Z=-1\}$ while varying the complex structure. Indeed, this is seen by the coordinate change $Z\mapsto ZC$, $X\mapsto C^{n+1}X$.
\end{remark}

\section{The symplectic geometry of the $A_n$-Milnor fibres}\label{geometry}
We describe the symplectic geometry of relevance in our setting. We give a short preliminary on Weinstein manifolds and sectors. We describe the Milnor fibres with a fibre removed as Weinstein manifolds obtained by attaching critical handles to $T^*\mathbb{T}^2$. A covering argument together with a Legendrian Kirby calculation shows that this manifold is equivalent to the Weinstein manifold obtained by cyclic plumbing of $T^*S^2$.

Finally we identify a Legendrian $S^1\simeq\Lambda_{\text{stop}}\subset S^1\times S^2$ at which we put our stop, and we show that the complement of a neighbourhood of the stop can be identified with a Weinstein sector $W_{\sigma}$ with boundary at infinity $\partial_{\infty}W_{\sigma}=J^1S^1$, with its standard contact structure. We use this in the following section to compute the partially wrapped Fukaya category.
\subsection{Preliminaries: Weinstein sectors}\label{geomprel}

\subsubsection{Weinstein Manifolds}

A \emph{Liouville domain} is an exact compact symplectic manifold $(\overline{W},\omega=d\lambda)$ with boundary $M=\partial \overline{W}$, such that the Liouville vector field $Z$, defined by $\iota_Z\omega=\lambda$, is an outwards pointing normal to the boundary. This implies that $(M,\xi)$, with $\xi=\text{ker}\,\lambda_{|M}$, is a (cooriented) contact manifold.

A \emph{Liouville manifold} is an exact symplectic manifold $(W,\lambda)$ which is cylindrical and convex at infinity, meaning that: There is a Liouville domain $\overline{W}\subset W$ such that the positive Liouville flow $\phi_Z^t:\mathbb{R}_t\times\partial\overline{W}\rightarrow W$ is defined for all $t$ and the induced map $\overline W\cup_{\partial \overline{W}}(\mathbb{R}_{t\geq 0}\times\partial\overline{W})\rightarrow W$ is a symplectomorphism, where the symplectic form on the \emph{cylindrical end} $\mathbb{R}_t\times\partial\overline{W}$ is given by the symplectisation form $d(e^t\lambda)$.

The contact manifold $(M,\xi=\text{ker}\,\alpha)=(\partial\overline{W},\text{ker}\,\lambda_{|\partial\overline{W}})$ is unique up to contactomorphism (by Liouville flow), in the sense that, for any other choice of Liouville domain $\overline{W}'$ such that the induced map as above is a symplectomorphism, $(\partial\overline{W},\text{ker}\,\lambda_{|\partial\overline{W}})\simeq (\partial\overline{W}',\text{ker}\,\lambda_{|\partial\overline{W}'})$, are contactomorphic. In particular, any contact type hypersurface in the cylindrical end which is everywhere transverse to the Liouville vector field is contactomorphic to $(M,\xi)$. We denote this contact manifold by $(\partial_{\infty} W,\xi)$ and call it the \emph{boundary at infinity}.

Any Liouville domain $(\overline{W},\lambda)$ can be completed to a Liouville manifold by attaching a cylindrical end as follows. Flowing along $-Z$ gives a collar neighbourhood of the boundary $(\overline{W},\lambda)\supset (U,\lambda)\cong (-\infty,0]\times M$, where $M=\partial \overline{W}$. Attaching the rest of the symplectisation of $M$ to $W$, called \emph{attaching the cylindrical end}, gives an open manifold without boundary $W=\overline{W}\cup_M([0,\infty)\times M)$. The Liouville form $\lambda$ extends to a Liouville form which at the cylindrical end is $\lambda=e^t\lambda_{|M}$.

The \emph{skeleton} of a Liouville domain is the set of points which do flow to the boundary with the Liouville flow:
\[
\text{skel}(\overline{W})=\bigcap_{t<0} \phi_Z^{t}(\overline{W}).
\]
For general Liouville domains the skeleton can be "big" and can have codimension $1$, but for Weinsetin domains (defined below) the skeleton is always (possibly singular) isotropic, and thus "at most" Lagrangian.

We consider two Liouville manifolds equivalent if they can be homotoped to each other through a smooth family of Liouville manifold structures, which are the completions of a smooth family of Liouville sub-domains. 

A Liouville domain is a \emph{Weinsetin domain} if it has a Morse (or Morse\--Smale) function $f$ such that the Liouville vector field $Z$ is a pseudo-gradient vector field for $f$ (in particular the critical points are finitely many and away from the boundary). A Liouville manifold is called a \emph{Weinstein manifold} if is the completion of a Weinstein domain $\overline{W}\subset W$, such that the Morse function $f$ can be extended to the whole manifold with no critical points outside $\overline{W}$.

The property that the Liouville flow expands the symplectic form implies that the critical points of $f$ must have index less than or equal to $n$, half the dimension of $W$, and $W$ is built by symplectic handle attachments, as follows. We define the \emph{Weinstein standard handle} of index $0\leq k\leq n$, of size $\epsilon>0$, to be the subset
\[
H_k=\left\{\sum_{i=1}^nx_i^2+\sum_{i=1}^{n-k}y_i^2\leq\epsilon\right\}\cap\left\{\sum_{i=n-k+1}^ny_i^2\leq 1\right\}\cong D^{2n-k}\times D^k
\]
of $(\mathbb{R}^{2n},\omega_{\text{std}})$, where $\omega_{\text{std}}=\sum_idy_i\wedge dx_i$. It comes with a Weinstein structure defined by
\begin{align*}
\lambda_k&=\frac{1}{2}\sum_{i=1}^{n-k}\Big(y_idx_i-x_idy_i\Big)-\sum_{i=n-k+1}^n\Big(y_idx_i+2x_idy_i\Big), \\
Z_k&=\frac{1}{2}\sum_{i=1}^{n-k}\Big(x_i\partial_{x_i}+y_i\partial_{y_i}\Big)+\sum_{i=n-k+1}^n\Big(2x_i\partial_{x_i}-y_i\partial_{y_i}\Big), \\
f_k&=\frac{1}{4}\sum_{i=1}^{n-k}\Big(x_i^2+y_i^2\Big)+\sum_{i=n-k+1}^n\Big(x_i^2-\frac{1}{2}y_i^2\Big).
\end{align*}
The disc $\{x_1=\dots=x_n=y_1=\dots y_{n-k}=0\}\cap H_k\cong D^{k}$ is called the \emph{core}, its boundary ($\cong S^{k-1}$) is called the \emph{attaching sphere}, and the disc $\{y_{n-k+1}=\dots=y_n=0\}\cap H_k\cong D^{2n-k}$ is called the cocore. Note that the core is isotropic and the cocore is coistropic, and that the attaching sphere is isotropic in the contact boundary. In particular, if $k=n$, both the core and the cocore are Lagrangian submanifolds, and the attaching sphere is a Legendrian sphere in the boundary. 

We denote by $\partial_h H_k=\left\{\sum_{i=n-k+1}^ny_i^2=1\right\}\cong D^{2n-k}\times S^{k-1}\subset D^{2n-k}\times D^k$, the attaching region. An attachment of a handle $H_k$ to a Weinstein domain $\overline{W}$ is given by $\overline{W}\cup_{\phi} H_k$. By this we mean the disjoint union of $\overline{W}$ and $H_k$, modulo an identification of the attaching region $\partial_h H$ with its image via an embedding, the \emph{attaching map}, $\phi:\partial_h H_k\rightarrow \partial \overline{W}$. This results in a manifold with boundary with corners, which then is smoothened while making sure the symplectic structures are compatible and that the Liouville vector fields still points out of the boundary; thus we obtain a new Weinstein domain. The handle attachment is uniquely determined by an embedding of the attaching sphere $S^{k-1}\xrightarrow{\sim} L\subset\partial \overline{W}$ and a trivialisation of its conformal symplectic normal bundle $\text{CSN}(L)=(TL)^{\perp}/TL$. In particular, the attaching of a standard Weinstein handle of index $n$ is uniquely determined by the embedding of the attaching Legendrian \cite{MR1114405}. 

Given a Weinstein domain $\overline{W}$, one can deform the Weinstein structure (in particular changing the Morse function) so that it is built by first attaching handles of index $1\leq k\leq n-1$ to the the standard ball $B^{2n}$, one for each critical point of $f$ with the corresponding index, which yields the \emph{subcritical} part $\overline{W}_\text{sub}$, and then attaching the handles of index $n$, called the \emph{critical handles}. A subcritical manifold is symplectically uninteresting in the sense that it has trivial wrapped Fukaya category, and all related related symplectic invariants vanish. 

In this way Weinstein manifolds are up to deformations of Weinstein structures built by first attaching handles and then attaching the cylindrical end.

\subsubsection{Weinstein Kirby calculus}

To describe Weinstein $4$-manifolds that are constructed by as sequence of handle attachments we will use Weinstein handlebody diagrams and Legendrian Kirby calculus (see \cite{MR1668563}). This is a calculus which can be used to determine when two Weinstein manifold are equivalent. Namely, two Weinstein manifolds are equivalent if and only if their Weinstein structures can be related by a so-called Weinstein homotopy, which involves isotopies of the structures, handle-slides, and birth/death moments. When the Weinstein manifold is $4$-dimensional, these moves can all be described by using the $2$-dimensional front projection of the attaching spheres.

To construct a \emph{Weinstein handlebody diagram}, we can always start with one $0$-handle $B^4$, and attach first $1$-handles, to obtain the subctritical Weinstein domain with boundary $M=\#^m(S^1\times S^2)$ (with its standard contact form), and then we attach the $2$-handles along Legendrian knots in $M$. In this case we only attach $1$-handles, with attaching spheres $S^0$, and $2$-handles, which are critical. There is a standard choice for the attaching regions of the $1$-handles (see \cite{MR1668563}), and the attaching maps of the critical handles are determined by the the attaching spheres.

We will draw a front diagram representing the handle attachments which make up the Weinstein manifold. Thus this picture completely specifies the Weinstein manifold.

The $0$-handle is just the standard Liouville domain $(D^4,\lambda_0)$. It has boundary standard contact $S^3$. All handle attachments must happen in the complement of some point $\text{pt}\in S^3$, and so we can think of the handle attachment as happening in standard contact $\mathbb{R}^3=S^3\smallsetminus\{\text{pt}\}$ (with $\xi=\text{ker}(dz-ydx)$). This is "pictured" as the background, via the front projection $\pi_F:\mathbb{R}_{x,y,z}^3\rightarrow\mathbb{R}^2_{x,z}$. The $1$-handles are pictured as pairs of balls aligned horizontally (the attaching region is $S^0\times D^3$), and the $2$-handles by their attaching spheres. See Figure~\ref{plumbingpic} and Figure~\ref{eqpic} for examples. We say that the diagram is in \emph{Gompf standard form}.

Two Legendrian links are Legendrian isotopic in $M$ if and only if they can be related by Legendrian Reidemeister moves and the Gompf moves, together also called Legendrian Kirby moves. This induces equivalences of the Weinstein manifolds built by the handle attachments (and quasi-isomorphic Chekanov--Eliashberg dg-algebras, see Section~\ref{invprel}). Weinstein manifolds are also equivalent if their Weinstein handlebody diagrams can be related by handle slides and handle cancellations (which induce quasi-equivalences of the wrapped Fukaya categories, see Section~\ref{invprel}).

\subsubsection{Stops and sectors}

A Liouville/Weinstein sector is a kind of Liouville/Weinstein manifold with boundary. The boundary is assumed to have collar neighbourhood of a certain type which allows the sector to be doubled to a Liouville/Weinstein manifold without boundary, analogously to doubling of smooth manifolds with boundary. The Floer theories on these manifolds are well defined if one makes an appropriate choice of almost complex structure near the boundary (see \cite{MR4106794}). 

A \emph{Liouville sector} as defined in \cite{MR4106794} is a Liouville manifold $W$ with boundary (so that a corresponding Liouville domain $\overline{W}$, as above, is a manifold with corners), such that for some (or all) positive number $a>0$, there exists a function $I:\partial W\rightarrow \mathbb{R}$ such that $ZI=a I$ near infinity (for $t$ large enough, where $t$ is the $\mathbb{R}$-coordinate in the cylindrical end) and $dI_{|\mathcal{F}_{\text{Ch}}}>0$, where $\mathcal{F}_{\text{Ch}}$ is the characteristic foliation. 

A consequence of the existence of I is that $\partial_{\infty}W$ is contact manifold with convex boundary $\partial\partial_{\infty}W$ and that there exists a diffeomorphism $\partial W\xrightarrow{\sim} F\times\mathbb{R}$ sending the characteristic foliation of $\partial_{\infty} W$ to the leaves of the foliation $\mathbb{R}\times F$ by $\mathbb{R}\times\{\text{pt}\}$.

The typical example is $T^*M$, where $M$ is a manifold with boundary.

Next we discuss the relations to stops, due to \cite{MR3427304}. We recall the definition from \cite{MR4106794}:
\begin{definition}\label{stop}
Let $(F^{2n-2},\lambda_F)$ and $(W^{2n},\lambda_W)$ be Liouville manifolds, and let $\epsilon>0$. A \emph{stop with fibre} $F$ is a proper embedding
\[
\sigma:\mathbb{C}_{\text{Re}\leq\epsilon}\times F\rightarrow W
\]
satisfying $\sigma^*\lambda_W=\lambda_F+\lambda_{\mathbb{C}}+df$, for some compactly supported function $f$. Here $\lambda_{\mathbb{C}}=\frac{1}{2}\left(ydx-xdy\right)$.
\end{definition}

A Liouville manifold with stop $(W,\sigma)$ defines a Liouville sector, which we denote $W_{\sigma}$, by removing $\sigma(\mathbb{C}_{\text{Re}<0}\times F)$ from $W$. On the other hand, there is a way to go from a Liouville sector to Liouville manifold (which is unique up to deformation), roughly remembering the boundary as a stop. Thus the notions of stopped Liouville manifolds and of Liouville sectors are essentially the same (see also Theorem~\ref{stopsecteq}).

By the following example we will talk about putting a stop at a Legendrian submanifold (similarly to attaching Weinstein handles along them), and these are the stops of interest here:
\begin{example}
Let $\Lambda\subset M=\partial_{\infty} W$ be a Legendrian submanifold. By the Legendrian tubular neigbourhood theorem, $\Lambda$ has a tubular neighbourhood $U$ contactormorphic to an open convex neighbourhood of the zero-section in the $1$-jet bundle $J^1\Lambda=T^*\Lambda\times\mathbb{R}$. Identifying these two neighbourhoods we obtain a stop $\sigma$ with fibre $T^*\Lambda$ in $W$ (see \cite[Example 2.8]{MR3427304} for details).
\end{example}

Given any point $p\in \Lambda$, one can define a certain sphere $S_p$ in a Darboux chart around $\Lambda$, which links with $\Lambda$. This sphere bounds a Lagrangian disc $D_p$ in the symplectisation of $M$ called a \emph{linking disc}, which is roughly given by the Cartesian product of a cotangent fibre in $T^*\Lambda$ and a curve in the $\mathbb{C}$-factor. This can then be viewed as a Lagrangian in the Weinsetin manifold $W$ (the completion). (See \cite[Section 5.3]{MR4695507}.)

Together with the cocores, the Linking discs generates the partially wrapped Fukaya categories of Weinstein sectors (see Theorem~\ref{generation}). 

\subsubsection{Cutoff Reeb dynamics}

In order to define the wrapped Fukaya category on sectors one needs to control the Reeb dynamics in the boundary at infinity. 

Let $W$ be a Liouville sector with boundary at infinity $(\partial_{\infty}W,\xi)=(M,\xi=\text{ker}\,\alpha)$, where $M$ is a contact manifold with convex boundary. Recall that the \emph{Reeb vector field} on $M$ defined by the contact form $\alpha$ is the unique vector field satisfying
\begin{align*}
    \alpha(R_{\alpha})&=1, \\
    \iota_{R_{\alpha}}d\alpha&=0.
\end{align*}

We will assume that the contact forms $\alpha$ on $\partial_{\infty}W=M$ define cutoff Reeb vector fields, in sense described below. Recall that a \emph{contact vector field} is a vector field $X$ whose flow preserve the contact structure or, equivalently, $\mathcal{L}_V\alpha=f\cdot\alpha$, for some function $f:M\rightarrow \mathbb{R}$, where $\mathcal{L}$ denotes the Lie derivative. Contact vector fields are in one-to-one correspondence with contact Hamiltonians, i.e. sections $H:M\rightarrow TM/\xi$. Given a fixed contact form $\alpha$, contact Hamiltonians are equivalently defined as functions $H:M\rightarrow\mathbb{R}$, and contact vector fields are equivalently defined as Hamiltonian vector fields, i.e. vector fields $X_H$ defined by requiring
\begin{align*}
    \alpha(X_H)&=H,\\
     \iota_{X_H}d\alpha&=dH(R_{\alpha})\alpha-dH,
\end{align*}
to hold. (In particular $H=1$ defines the Reeb vector field.) 

A contact vector field is said to be a \emph{cutoff contact vector field} if it is generated by a Hamiltonian of the form $H=f^2G$, for a function $f:M\rightarrow\mathbb{R}$ vanishing transversely on $\partial M$ and a section $G>0$ on all of $M$ (including $\partial M$).
\subsection{Equivalence of Milnor fibres and plumbings}\label{geomeq}
We show that the smoothening of the $A_n$-singularity, which we denote by $X_{\text{Mil}}$, with a fibre removed, i.e. $X_{\text{Mil}}\smallsetminus D_m$, has a nice geometric description in terms of cyclic plumbings of $T^*S^2$. In particular, the cyclic plumbing has a well known simple handle decomposition which we will use.

\subsubsection{Milnor fibres}

In \cite{MR239612}, Milnor analysed the topology of singular hypersurfaces given by $\{f=0\}$, for analytic maps $f:\mathbb{C}^m\rightarrow\mathbb{C}$, locally around the origin.

To summarise shortly, we let $f$ be a non-constant polynomial function (or, more genreally, an analytic one) $f:\mathbb{C}^{m+1}\rightarrow\mathbb{C}$, $m\geq 1$, with an isolated singularity at the origin. For small enough $\epsilon>0$, let $S_{\epsilon}^{2m+1}$ denote the sphere of radius $\epsilon$. It intersects $f^{-1}(0)$ transversely, and we obtain a smooth fibre bundle over $S^1$,
\[
\phi:S_{\epsilon}^{2m+1}\smallsetminus K\rightarrow S^1,\qquad z\mapsto \frac{f(z)}{|f(z)|}.
\]
There exists $\delta=\delta(\epsilon)$ such that for all $c$ with $0<|c|<\delta$ we have that
\[
f^{-1}(c)\cap\overline{B}_{\epsilon}\simeq\phi^{-1}(\text{arg}(c)),
\]
are diffeomorphic smooth manifolds with boundary, where $B_{\epsilon}$ denotes the open ball of radius $\epsilon$. 

For fixed $c$ and $\epsilon$, we denote the \emph{Milnor fibre} of $f$ by $\overline{F}_{c,\epsilon}=f^{-1}(c)\cap\overline{B}_{\epsilon}$. The Milnor fibre admits the structure of a Weinstein domain given by restricting the radial structure:
\begin{align*}
f_{\text{rad}}&=\sum_i|z_i|^2, \\
\lambda_{\text{rad}}&=\frac{i}{4}\sum_i(z_id\overline{z}_i-\overline{z}_idz_i).
\end{align*}
Indeed, one can show that the induced Liouville vector field is pointing outwards along the boundary. The domain $\overline{F}_{c,\epsilon}$ is unique up to exact symplectomorphism \cite[Section 3.1]{MR3279026}. The completions obtained by adding cylindrical ends yield equivalent Liouville manifolds. Furthermore, if $g$ is obtianed form $f$ by holomorphic reparametrisation, then the Milnor fibres of $f$ and $g$ are exact symplectomorphic. 

The homotopy type of the Milnor fibre is a bouquet of $\mu$ $m$-dimensional spheres, where $\mu$ is the Milnor number.

We are interesed in the case $f=x^2+y^2+z^{n+1}$.  (Equivalently, we could consider any polynomial on the form $f=x^2+y^2+p(z)$, or $f=xy+p(z)$, for $p(z)$ a polynomial of degree $n+1$ with distinct roots.) We obtain the corresponding Milnor fibre $\overline{F}_{A_n}\subset\mathbb{C}^3$, by e.g. letting $c=1$. In this case we obtain a Lefschetz fibration $\pi_z:\overline{F}_{A_n}\rightarrow\mathbb{C}$ by projection onto the $z$ coordinate. This means that except for at finitely many points the fibres are given by smooth conics $\pi_z^{-1}(z_0)=\{x^2+y^2=z_0\}\simeq\mathbb{C}^*\simeq T^*S^1$. The singular fibres are the sets $\{(x,y,\xi_{n+1})\,|\,x^2+y^2=0\}$, where $\xi_{n+1}$ is an $n+1$:th root of unity. We denote the completed Weinstein manifold by $X_{\text{Mil}}$.

We remove a generic fibre, say $D_m=\pi_z^{-1}(0)$, which is a divisor. Now the above vector field is no longer transversely pointing outwards along all of the boundary of the corresponding domain. This could be amended, but we will instead understand this Weinstein manifold by a description of a handle decomposition.

\begin{figure}[H]
\centering
\begin{tikzpicture}
\draw (0,0) circle (2cm);
\draw[fill=white] (0,0) circle (0.13cm);

\draw[black, line width = 1.4] (1.9,-0.1) to (2.1,0.1);
\draw[black, line width = 1.4] (1.9,0.1) to (2.1,-0.1);

\draw[black, line width = 1.4] (-1.1,1.832) to (-0.9,1.632);
\draw[black, line width = 1.4] (-1.1,1.632) to (-0.9,1.832);

\draw[black, line width = 1.4] (-1.1,-1.832) to (-0.9,-1.632);
\draw[black, line width = 1.4] (-1.1,-1.632) to (-0.9,-1.832);

\draw (0,0.6) to (0,2.5);

\draw[densely dotted] (0.5,3) arc (0:180:0.5 and 0.16);
\draw (-0.5,3) arc (0:180:-0.5 and -0.16);

\draw (0.5,5) arc (0:180:0.5 and 0.16);
\draw (-0.5,5) arc (0:180:-0.5 and -0.16);
\draw[rounded corners=8pt] (-0.5,5) to (-0.2,4) to (-0.5,3);
\draw[rounded corners=8pt] (0.5,5) to (0.2,4) to (0.5,3);

\draw (2,0.6) to (2,2.5);

\draw[densely dotted] (2.5,3) arc (0:180:0.5 and 0.16);
\draw (1.5,3) arc (0:180:-0.5 and -0.16);

\draw (2.5,5) arc (0:180:0.5 and 0.16);
\draw (1.5,5) arc (0:180:-0.5 and -0.16);

\draw (2.5,5) to (1.5,3);
\draw (1.5,5) to (2.5,3);
\end{tikzpicture}
    \caption{Picture of the Lefschetz fibration $\{x^2+y^2+z^3=1\}\xrightarrow{\pi_z}\mathbb{C}_z$, showing the critical fibres (the crosses) and one generic fibre at the origin.}
\end{figure}
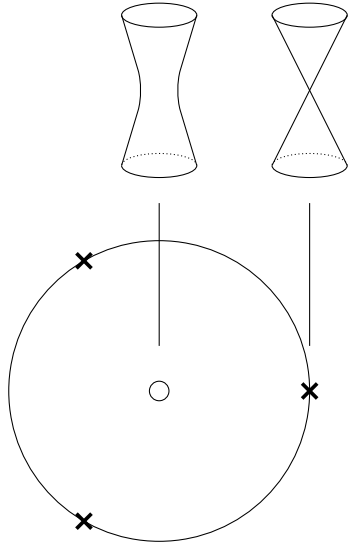

\begin{proposition}
The Milnor fibre minus a fibre $X_{\textup{Mil}}\smallsetminus D_m$ can be obtained from $T^*\mathbb{T}^2$ by attaching Weinstein $2$-handles along the conormal lifts of $n+1$ parallel embedded curves on $\mathbb{T}^2$, as in Figure~\ref{milnorpic} (where $\mathbb{T}^2=S^1\times S^1$).
\end{proposition}
\begin{proof}
We note that for small $\epsilon$, that is, such that $\pi^{-1}_z(D_{\epsilon})$ does not contain any singular fibres, where $D_{\epsilon}$ is the disc of radius $\epsilon$, the completion of $\pi_z^{-1}(D_{\epsilon}\smallsetminus\{0\})$ is equivalent to the trivial fibration $\mathbb{C}^*\times\mathbb{C}^*=T^*\mathbb{T}^2$. When passing to the critical level, i.e. when passing to $\pi^{-1}_z(D_{\epsilon})$ for $\epsilon$ large enough to contain the singular fibres, we attach to it $n+1$ critical handles, one for each singular fibre. How this is done is shown in \cite[Section 3]{MR4377932}, using a slightly different picture, from which we deduce that the handle attachment can be understood from Figure~\ref{milnorpic}. The conormal lifts of the blue lines are the Legendrian attaching spheres in the unit cotangent bundle $ST^*\mathbb{T}^2$. It is shown in \cite{MR4913463} how these diagrams can be translated into a Weinstein handlebody diagram (see Figure~\ref{eqpic}). (See also Figure~\ref{skeleton}, for the skeleton.) 
\end{proof}

\begin{figure}[H]
   \centering
    \begin{tikzpicture}
        \draw[blue!80!black,thick] (0,3.6) to (4,3.6);
        \draw[blue!80!black,thick] (0,3.4) to (4,3.4);
        \draw[blue!80!black,thick] (0,3.2) to (4,3.2);
        \draw[blue!80!black,thick] (0,2.4) to (4,2.4);
        \draw[black, line width =1.5 pt] (0,0) rectangle (4,4);
        \node[blue!80!black, very thick] at (2,2.9) {$\vdots$};
        \node[scale=1.5] at (2,0) {\midarrow};
        \node[scale=1.5] at (2,4) {\midarrow};
        \node[scale=1.5, rotate=90] at (0,1.9) {\midarrow};
        \node[scale=1.5, rotate=90] at (0,2.2) {\midarrow};
        \node[scale=1.5, rotate=90] at (4,1.9) {\midarrow};
        \node[scale=1.5, rotate=90] at (4,2.2) {\midarrow};
    \end{tikzpicture}
\caption{The projection onto $\mathbb{T}^2$ of the Legendrian attaching spheres (blue) for $X_{\text{Mil}}\smallsetminus D_m$ in $ST^*\mathbb{T}^2$. The attaching spheres are the conormal lifs of the parallel curves.}
\label{milnorpic}
\end{figure}
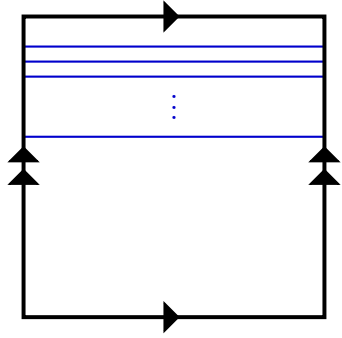

\subsubsection{Cyclic plumbings of $T^*S^2$}

We recall the basic properties of plumbings of cotangent bundles, which is a kind of symplectic connected sum. In particular we introduce the cyclic plumbing of cotangent bundles of spheres.
\begin{definition}
Let $M$ and $N$ be $n$-dimensional manifolds and let $X=T^*M$ and $Y=T^*N$. A \emph{plumbing} of $X$ and $Y$ is constructed as follows. Pick $q_0\in M$ and $q_1\in N$, and open neighbourhoods $B_0\subset M$ and $B_1\subset N$ of $q_0$ and $q_1$ respectively. Let the plumbing of $X$ and $Y$, denoted $X\#_{\text{pl}}Y$, be defined as a smoothing of $X\sqcup Y/\sim$,  where $\sim$ denotes the identification of $DT^*B_1$ and $DT^*B_2$ along the symplectomorphism $(q,p)\mapsto (-p,q)$. (Here $DT^*B_i$ denotes the disc cotangent bundle consisting of $(q,p)$ with $\|p\|\leq 1$ after some choice of metric.)

A \emph{self-plumbing} of $X=T^*M$ is defined choosing $q_0,q_1$ in $M$ and $B_0$ and $B_1$ disjoint neighbourhoods of $q_0$ and $q_1$ respectively. Then $\#_{\text{pl}}X$ is defined as a smoothening of $X/\sim$, where $\sim$ denotes denotes the same identification as above. There is an additional orientation data needed in order to make this operation well-defined. In the case when the dimension $n$ of $M$ is even, this is the choice of a local intersection number $+1$ or $-1$ at the unique self-intersection of the immersion of the zero-section.
\end{definition}
\begin{example}[Plumbing graphs of $T^*S^2$]
Given a (finite) graph $\Gamma=(V,E)$, the plumbing of copies of $T^*S^2$ according to $\Gamma$ consists taking one copy $S_i=T^*S^2$ for each vertex $i\in V$ and plumbing along them according to the edges. More precisely, for each edge $e:i\leftrightarrow j$, we plumb $S_i$ and $S_j$, where we choose the basepoints for the plumbings distinct for each edge, and where the neighbourhoods chosen for different plumbings are disjoint. 

In order for this to be well-defined, an additional choice of orientation data for each cycle in the graph $\Gamma$ is needed. Note that the local intersection number at the unique point of intersection corresponding to each edge is either $+1$ or $-1$.  For each cycle in $\Gamma$, the additional data needed is a choice of $+1$ or $-1$ of the products of the intersection numbers corresponding to the edges of the cycle. Here we will always assume that this multiplication is $1$. This is equivalent to the condition that the Lagrangian immersion obtained by taking the spheres corresponding to the nodes in the cycle and resolving the double points corresponding to the edges in the cycle is orientable.

With these choices, given a graph $\Gamma$, we denote the resulting symplectic manifold by $\#_{\Gamma}T^*S^2$. In particular, the plumbing of two copies $T^*S^2$ is $T^*S^2\#_{\text{pl}}T^*S^2=\#_{A_2}T^*S^2$, where $A_2$ denotes the Dynkin graph of type $A$ with two nodes. The self-plumbing is $\#_{\text{pl}}T^*S^2=\#_{\tilde{A}_0}T^*S^2$, i.e. plumbing accoriding to the extended Dynkin diagram $\tilde{A_0}$. 
\end{example}
Plumbings $\#_{\Gamma}T^*S^2$ are Weinstein manifolds. Their wrapped Fukaya categories were calculated in \cite{MR4033516}.

Cyclic plumbings of $T^*S^2$, i.e. $\#_{\tilde{A}_n}T^*S^2$, have a well known handle decomposition into one $1$-handle and $n+1$ $2$-handles, as shown in Figure~\ref{plumbingpic} below.
\begin{figure}[H]
\centering
\begin{tikzpicture}
\draw[thick] (-1,0) circle (1cm);
\draw[densely dotted] (0,0) arc (0:180:1 and 0.32);
\draw[thick] (-2,0) arc (0:180:-1 and -0.32);

\draw[thick] (8,0) circle (1cm);
\draw[densely dotted] (9,0) arc (0:180:1 and 0.32);
\draw[thick] (7,0) arc (0:180:-1 and -0.32);

\begin{knot}[clip width=5, flip crossing=2,flip crossing =4]
\strand[thick] (-0.293,0.7) to[out=0,in=180] (0.7,0.7) to[out=0,in=180]  (1.7,0);
\strand[thick] (-0.293,-0.7) to[out=0,in=180] (0.7,-0.7) to[out=0,in=180]  (1.7,0);

\strand[thick] (0.8,0) to[out=0,in=180] (1.8,0.7) to[out=0,in=180] (2.0,0.7);
\strand[thick] (2.0,0.7) to[out=0,in=180] (2.7,0.7) to[out=0,in=180] (3.7,0);
\strand[thick] (2.0,-0.7) to[out=0,in=180] (2.7,-0.7) to[out=0,in=180] (3.7,0);
\strand[thick] (0.8,0) to[out=0,in=180] (1.8,-0.7) to[out=0,in=180] (2.0,-0.7);

\strand[thick] (2.8,0) to[out=0,in=180] (3.8,0.7) to[out=0,in=180] (4,0.7);
\strand[thick] (2.8,0) to[out=0,in=180] (3.8,-0.7) to[out=0,in=180] (4,-0.7);
\end{knot}
\draw[fill=black] (4.4,0) circle (0.01cm);
\draw[fill=black] (4.5,0) circle (0.01cm);
\draw[fill=black] (4.6,0) circle (0.01cm);

\begin{knot}[clip width=5, flip crossing =2]
\strand[thick] (5.0,0.7) to[out=0,in=180] (5.3,0.7) to[out=0,in=180] (6.2,0);
\strand[thick] (5.0,-0.7) to[out=0,in=180] (5.3,-0.7) to[out=0,in=180] (6.2,0);
\strand[thick] (5.3,0) to[out=0,in=180] (6.3,0.7) to[out=0,in=180] (7.293,0.7);
\strand[thick] (5.3,0) to[out=0,in=180] (6.3,-0.7) to[out=0,in=180] (7.293,-0.7);
\end{knot}
\end{tikzpicture}
\caption{Weinstein handlebody diagram for cyclic plumbing of copies of $T^*S^2$, showing the Legendrian attaching link going through a $1$-handle.}
\label{plumbingpic}
\end{figure}
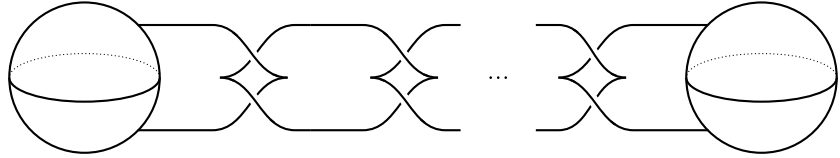

\subsubsection{Proof of the equivalence of $X_{\text{Mil}}\smallsetminus D_m$ and cyclic plumbing of $T^*S^2$}
The main purpose of Section~\ref{geomeq} is the following proposition:
\begin{proposition}\label{geometriceq}
The Milnor fibre with a fibre removed $X_{\textup{Mil}}\smallsetminus D_m$ is equivalent to the cyclic plumbing $\#_{\tilde{A}_n}T^*S^2$ as Weinstein manifolds.
\end{proposition}
\begin{proof}
From Figure~\ref{milnorpic} we see that the Weinstein manifold $X_{\text{Mil}}\smallsetminus D_m$ is the $n+1$-fold covering space of Weinstein manifold obtained by attaching one critical handle to $\mathbb{C}^*\times\mathbb{C}^*$. The following example (see \cite[Figure 19]{MR4417717}, see also \cite[Example 10.6]{MR4913463}) shows that attaching one critical handle in this way yields a Weinstein manifold equivalent to self-plumbing of $T^*S^2$. Using that the Weinstein manifold obtained represented by the diagram in Figure~\ref{plumbingpic} is the $n+1$-fold covering space of self-plumbing of $T^*S^2$, we conclude that $\#_{\tilde{A}_n}T^*S^2$ is equivalent to $X_{\text{Mil}}\smallsetminus D_m$.
\end{proof}
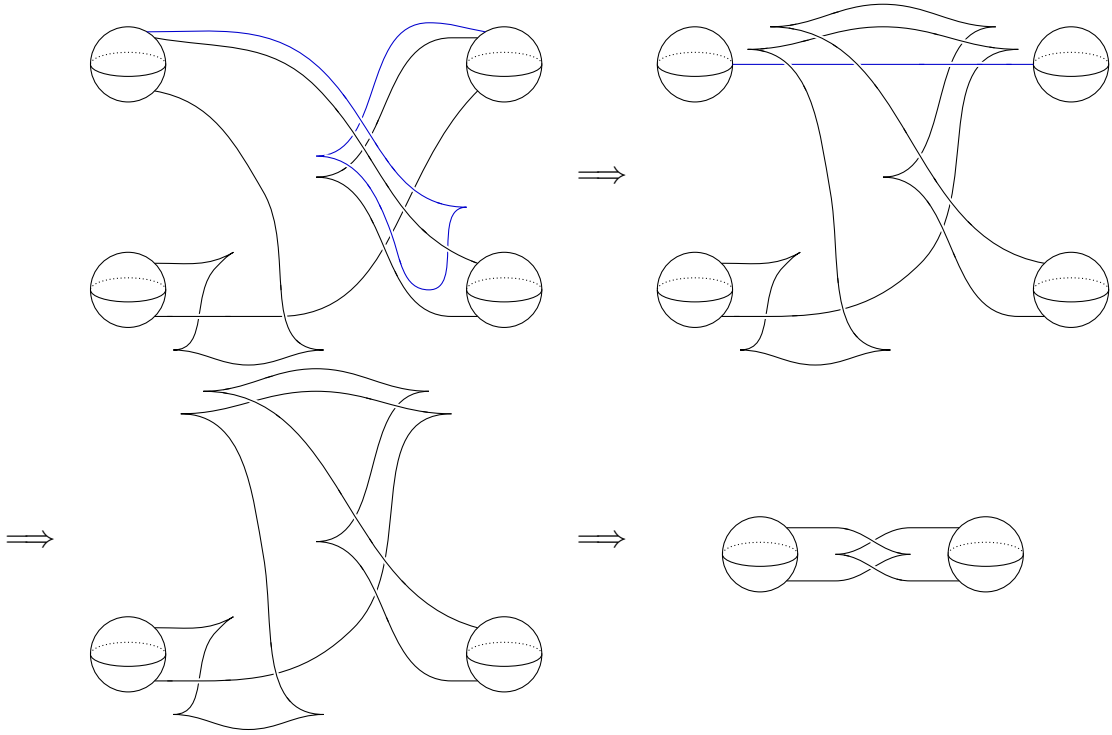
\begin{figure}[H]
    
\begin{tikzpicture}
\node[white] at (-1.3,0) {$\implies$};
\draw (0,0) circle (0.5cm);
\draw[densely dotted] (0.5,0) arc (0:180:0.5 and 0.16);
\draw (-0.5,0) arc (0:180:-0.5 and -0.16);

\draw (5,0) circle (0.5cm);
\draw[densely dotted] (5.5,0) arc (0:180:0.5 and 0.16);
\draw (4.5,0) arc (0:180:-0.5 and -0.16);

\draw (0,3) circle (0.5cm);
\draw[densely dotted] (0.5,3) arc (0:180:0.5 and 0.16);
\draw (-0.5,3) arc (0:180:-0.5 and -0.16);

\draw (5,3) circle (0.5cm);
\draw[densely dotted] (5.5,3) arc (0:180:0.5 and 0.16);
\draw (4.5,3) arc (0:180:-0.5 and -0.16);

\begin{knot}[clip width=2.2,flip crossing=1,flip crossing=4,flip crossing=5,flip crossing=6,flip crossing=7,flip crossing=10, flip crossing=11, flip crossing =3]
\strand (0.353553,0.353553) to[out=0,in=180] (0.4,0.353553) to[out=0,in=-150] (1.4,0.5)
to[out=-150,in=45] (1.17,0.3) to[out=-135,in=0] (0.6,-0.8) to[out=0,in=180] (1.6,-1) to[out=0,in=180] (2.6,-0.8) to[out=180, in=-60] (1.8,1.3) to[out=120,in=-10](0.353553,2.646447);

\strand (0.353553,-0.353553) to[out=0, in=180] (2,-0.353553)to[out=0,in=-135] (4.646447,2.646447);

\strand (0.353553,3.353553) to[out=-10,in=135] (2.5,2.7) to[out=-45,in=160] (4.646447,0.353553);

\strand (4.646447,-0.353553) to[out=180,in=0] (4.3,-0.353553) to[out=180,in=0] (2.5,1.5) to[out=0,in=180] (4.3,3.353553) to[out=0,in=180] (4.646447,3.353553);

\strand[blue!80!black] (0.25,3.433013) to[out=0,in=135] (2.6,2.9) to[out=-45,in=180] (4.5,1.1)
to[out=180, in=0] (4,0) to[out=180,in=-70] (3.5,0.75) to[out=110,in=0] (2.5,1.78);
\strand[blue!80!black] (2.5,1.78) to[out=0,in=180] (4.0,3.553553) to[out=0,in=170] (4.75,3.433013);
\end{knot}

\node at (6.3,1.5) {$\implies$};
\end{tikzpicture}
\begin{tikzpicture}
\draw (0,0) circle (0.5cm);
\draw[densely dotted] (0.5,0) arc (0:180:0.5 and 0.16);
\draw (-0.5,0) arc (0:180:-0.5 and -0.16);

\draw (5,0) circle (0.5cm);
\draw[densely dotted] (5.5,0) arc (0:180:0.5 and 0.16);
\draw (4.5,0) arc (0:180:-0.5 and -0.16);

\draw (0,3) circle (0.5cm);
\draw[densely dotted] (0.5,3) arc (0:180:0.5 and 0.16);
\draw (-0.5,3) arc (0:180:-0.5 and -0.16);

\draw (5,3) circle (0.5cm);
\draw[densely dotted] (5.5,3) arc (0:180:0.5 and 0.16);
\draw (4.5,3) arc (0:180:-0.5 and -0.16);

\begin{knot}[clip width=2.2, flip crossing=1, flip crossing =6, flip crossing=7, flip crossing=8, flip crossing=11, flip crossing=3]
\strand (0.353553,0.353553) to[out=0,in=180] (0.4,0.353553) to[out=0,in=-150] (1.4,0.5)
to[out=-150,in=45] (1.17,0.3) to[out=-135,in=0] (0.6,-0.8) to[out=0,in=180] (1.6,-1) to[out=0,in=180] (2.6,-0.8) to[out=180, in=-80] (1.8,1.3) to[out=100,in=0] (0.7,3.2) to[out=0,in=180] (2.5,3.5) to[out=0,in=180] (4.3,3.2);
\strand (4.3,3.2) to[out=180,in=45] (3,0.353553) to[out=-135, in=0] (1,-0.353553) to[out=180,in=0]
(0.353553,-0.353553);

\strand (4.646447,0.353553) to[out=170, in=0] (1,3.5) to[out=0,in=180] (2.5,3.8) to[out=0,in=180] (4,3.5);
\strand (4,3.5) to[out=180,in=0] (2.5,1.5) to[out=0,in=180](4.3,-0.353553) to[out=0,in=180] (4.646447,-0.353553);

\strand[blue!80!black] (0.5,3) to (4.5,3);
\end{knot}
\end{tikzpicture}
\\
\begin{tikzpicture}
\node at (-1.3,1.5) {$\implies$};

\draw (0,0) circle (0.5cm);
\draw[densely dotted] (0.5,0) arc (0:180:0.5 and 0.16);
\draw (-0.5,0) arc (0:180:-0.5 and -0.16);

\draw (5,0) circle (0.5cm);
\draw[densely dotted] (5.5,0) arc (0:180:0.5 and 0.16);
\draw (4.5,0) arc (0:180:-0.5 and -0.16);
\begin{knot}[clip width=2.2,flip crossing=1, flip crossing =3, flip crossing =5, flip crossing =6]
\strand (0.353553,0.353553) to[out=0,in=180] (0.4,0.353553) to[out=0,in=-150] (1.4,0.5)
to[out=-150,in=45] (1.17,0.3) to[out=-135,in=0] (0.6,-0.8) to[out=0,in=180] (1.6,-1) to[out=0,in=180] (2.6,-0.8) to[out=180, in=-80] (1.8,1.3) to[out=100,in=0] (0.7,3.2) to[out=0,in=180] (2.5,3.5) to[out=0,in=180] (4.3,3.2);
\strand (4.3,3.2) to[out=180,in=45] (3,0.353553) to[out=-135, in=0] (1,-0.353553) to[out=180,in=0]
(0.353553,-0.353553);

\strand (4.646447,0.353553) to[out=165, in=0] (1,3.5) to[out=0,in=180] (2.5,3.8) to[out=0,in=180] (4,3.5);
\strand (4,3.5) to[out=180,in=0] (2.5,1.5) to[out=0,in=180](4.3,-0.353553) to[out=0,in=180] (4.646447,-0.353553);
\end{knot}
\node at (6.3,1.5) {$\implies$};
\end{tikzpicture}
\begin{tikzpicture}
\node at (-1.2,-2.2) {$\,$};
\draw (0,0) circle (0.5cm);
\draw[densely dotted] (0.5,0) arc (0:180:0.5 and 0.16);
\draw (-0.5,0) arc (0:180:-0.5 and -0.16);

\draw (3,0) circle (0.5cm);
\draw[densely dotted] (3.5,0) arc (0:180:0.5 and 0.16);
\draw (2.5,0) arc (0:180:-0.5 and -0.16);

\begin{knot}[clip width=2.2, flip crossing=2]
\strand (0.353553,0.353553) to (1,0.353553) to[out=0,in=180] (2,0);
\strand (0.353553,-0.353553) to (1,-0.353553) to[out=0,in=180] (2,0);

\strand (2.646447,0.353553) to (2,0.353553) to[out=180,in=0] (1,0);
\strand (2.646447,-0.353553) to (2,-0.353553) to[out=180,in=0] (1,0);
\end{knot}
\end{tikzpicture}
\caption{A calculation using Legendrian Kirby moves showing that self-plumbing of $T^*S^2$ is equivalent to attaching one critical handle to $\mathbb{C}^*\times\mathbb{C}^*$. (From \cite[Figure 19]{MR4417717}, which contains the missing steps). The two $1$-handles together with the Legendrian knot in black is a handle decomposition for $T^*\mathbb{T}^2$; the blue Legendrian knot is obtained from the conormal lift of a blue curve, as in Figure~\ref{milnorpic}.}
\label{eqpic}
\end{figure}
\begin{figure}[H]
\centering
\begin{tikzpicture}[scale=0.6]
    \draw[thick] (0,0) circle (1cm);
    \draw[thick, blue!80!black] (1,0) to (2,0);
    \draw[thick, blue!80!black] (-0.5,0.866) to (-1,1.73);
    \draw[thick, blue!80!black] (-0.5,-0.866) to (-1,-1.73);
    \draw (0,0) circle (0.08cm);
    
    \draw[thick, blue!80!black] (6,0) circle (2cm);
    \draw (6,0) circle (0.08cm);

\draw[black, line width = 1.4] (1.9,-0.1) to (2.1,0.1);
\draw[black, line width = 1.4] (1.9,0.1) to (2.1,-0.1);

\draw[black, line width = 1.4] (-1.1,1.832) to (-0.9,1.632);
\draw[black, line width = 1.4] (-1.1,1.632) to (-0.9,1.832);

\draw[black, line width = 1.4] (-1.1,-1.832) to (-0.9,-1.632);
\draw[black, line width = 1.4] (-1.1,-1.632) to (-0.9,-1.832);

\draw[black, line width = 1.4] (7.9,-0.1) to (8.1,0.1);
\draw[black, line width = 1.4] (7.9,0.1) to (8.1,-0.1);

\draw[black, line width = 1.4] (4.9,1.832) to (5.1,1.632);
\draw[black, line width = 1.4] (4.9,1.632) to (5.1,1.832);

\draw[black, line width = 1.4] (4.9,-1.832) to (5.1,-1.632);
\draw[black, line width = 1.4] (4.9,-1.632) to (5.1,-1.832);
\end{tikzpicture}
\caption{Both diagram shows the projection onto the $z$-coordinate, in the case $n=2$. On the left-hand side is the skeleton obtained when attaching handles in $ST^*\mathbb{T}^2$, showing four Lagrangian $S^2$. On the right-hand side is the skeleton, three Lagrangian spheres intersecting in double points, of the equivalent manifold obtained by cyclic plumbing of $T^*S^2$.}
\label{skeleton}
\end{figure}
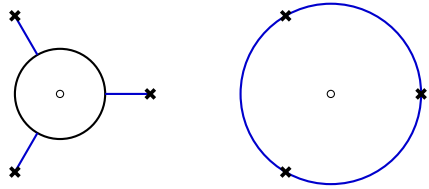
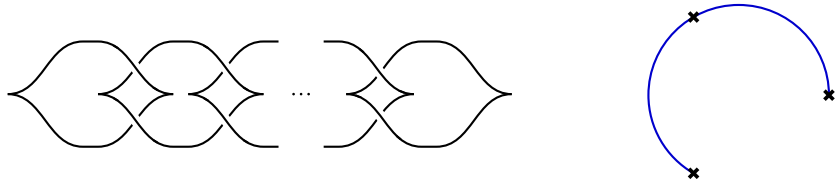
\begin{figure}[H]
    \centering
    \begin{tikzpicture}
\begin{knot}[clip width = 5, flip crossing =2, flip crossing=4]
\strand[thick] (0,0) to[out=0,in=180] (1,0.7) to (1.2,0.7) to[out=0,in=180] (2.2,0);   
\strand[thick] (0,0) to[out=0,in=180] (1,-0.7) to (1.2,-0.7) to[out=0,in=180] (2.2,0);   
\strand[thick] (1.2,0) to[out=0,in=180] (2.2,0.7) to (2.4,0.7) to[out=0,in=180] (3.4,0);
\strand[thick] (1.2,0) to[out=0,in=180] (2.2,-0.7) to (2.4,-0.7) to[out=0,in=180] (3.4,0);

\strand[thick] (2.4,0) to[out=0,in=180] (3.4,0.7) to (3.6,0.7);
\strand[thick] (2.4,0) to[out=0,in=180] (3.4,-0.7) to (3.6,-0.7);
\end{knot}
\draw[fill=black] (3.8,0) circle (0.01cm);
\draw[fill=black] (3.9,0) circle (0.01cm);
\draw[fill=black] (4.0,0) circle (0.01cm);
\begin{knot}[clip width = 5, flip crossing=2]
\strand[thick] (4.2,0.7) to (4.4,0.7) to[out=0,in=180] (5.4,0);
\strand[thick] (4.2,-0.7) to (4.4,-0.7) to[out=0,in=180] (5.4,0);
\strand[thick] (4.5,0) to[out=0,in=180] (5.5,0.7) to (5.7,0.7) to[out=0,in=180] (6.7,0);
\strand[thick] (4.5,0) to[out=0,in=180] (5.5,-0.7) to (5.7,-0.7) to[out=0,in=180] (6.7,0);
\end{knot}
\node at (1,-1) {$\;$};
\end{tikzpicture}
\qquad\qquad
    \begin{tikzpicture}[scale=0.6]
        \draw[thick, blue!80!black] (2,0) arc (0:240:2 and 2);
        
\draw[black, line width = 1.4] (1.9,-0.1) to (2.1,0.1);
\draw[black, line width = 1.4] (1.9,0.1) to (2.1,-0.1);

\draw[black, line width = 1.4] (-1.1,1.832) to (-0.9,1.632);
\draw[black, line width = 1.4] (-1.1,1.632) to (-0.9,1.832);

\draw[black, line width = 1.4] (-1.1,-1.832) to (-0.9,-1.632);
\draw[black, line width = 1.4] (-1.1,-1.632) to (-0.9,-1.832);
\end{tikzpicture}

\caption{Compare the figures above with the case $x^2+y^2+z^{n+1}=1$, without the fibre removed. Then one obtains the linear plumbing of $n$ copies of $T^*S^n$, i.e. the $A_n$-plumbing. The Legendrian attaching link shown on the left-hand side, and the skeleton shown on the right-hand side in the case $n=2$.}
\label{linear}
\end{figure}
\subsection{Geometry of the stop}\label{stopsection}
We here describe the stop put on the cyclic plumbings of $T^*S$. We put the stop on the Legendrian submanifold $\Lambda_{\text{stop}}\subset S^1\times S^2$, as in \cite[Example 2.8]{MR3427304}, which has Front diagram shown in Figure~\ref{stopfigure}. More precisely, we first show that $S^1\times S^2$ can be described as two copies of tubular neighbourhoods of the zero-section in $(J^1S^1,\xi_{\text{std}})$ glued together. This gives a nice geometric description of the sector obtained by cutting out a neighbourhood of the stop. Moreover, we show that the Reeb chords and holomorphic curves defining the contact invariants of the Legendrian attaching spheres can be computed inside $(J^1S^1,\xi_{\text{std}})$.
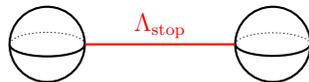
\begin{figure}[H]
    \centering
\begin{tikzpicture}
\draw[thick] (0,0) circle (0.5cm);
\draw[densely dotted] (0.5,0) arc (0:180:0.5 and 0.16);
\draw[thick] (-0.5,0) arc (0:180:-0.5 and -0.16);

\draw[thick] (3,0) circle (0.5cm);
\draw[densely dotted] (3.5,0) arc (2:180:0.5 and 0.16);
\draw[thick] (2.5,0) arc (2.5:180:-0.5 and -0.16);

\draw[thick,red] (0.5,0) to (2.5,0);
\node[red, scale=0.8] at (1.5,0.25) {$\Lambda_{\text{stop}}$};
\end{tikzpicture}
\caption{The Legendrian $\Lambda_{\text{stop}}\subset S^1\times S^2$ shown in the front projection. Carving out a neighbourhood of $\Lambda_{\text{stop}}$ results in a manifold contactomorphic to a convex neighbourhood of the zero-section in $J^1S^1$.}
\label{stopfigure}
\end{figure}

\subsubsection{Gluing construction}

We consider the Weinstein manifold given by the Lefschetz fibration on the Milnor fibre minus a fibre: we start with
\[
\{x^2+y^2+z^{n+1}=1\}\xrightarrow{\pi_z}\mathbb{C},
\]
and we consider $\{x^2+y^2+z^{n+1}=1\}\smallsetminus\pi^{-1}(0)$. We have already seen that this is a cyclic plumbing of $n+1$ copies of $T^*S^2$. We now want to describe the stop, and to make the geometry easier we will think of the stop as being attached before we attach the handles. Thus we restrict to $\pi^{-1}(D_{\epsilon}\smallsetminus 0)$, where $D_{\epsilon}$ is a disc with radius $\epsilon$ small enough so that the inverse image is subcritical. Then $\pi^{-1}(D_{\epsilon}\smallsetminus 0)$ is the Weinstein domain given by ataching a $1$-handle to the standard ball $B^4$. We think of the completion as the Weinstein manifold being the total space of the trivial Lefschetz fibration $\mathbb{C}\times\mathbb{C}^*=\mathbb{C}\times T^*S^1\rightarrow\mathbb{C}$. 

We are going to describe the boundary at infinity $M=\partial_{\infty}(\mathbb{C}\times\mathbb{C}^*)=S^1\times S^2$, as two copies of convex neighbourhoods of the zero-section in $J^1S^1$ glued together. The stop will be attached to the zero-section of one of the copies. The Weinstein handles will be attached along the zero-section of the other copy. Looseness of the Legendrian at which we attach the stop will imply that we can we can calculate the wrapped Fukaya category in the other copy of $J^1S^1$. The division of the contact boundary into two parts will be done by viewing the fibration as an open book decomposition (defined below, following \cite{MR2249250}) and the two halves will have common boundary given by a convex surface constructed as the union of two pages.

\begin{definition}
An \emph{open book decomposition} of a closed $3$-manifold $M$ consists of a pair $(B,\pi)$, where the \emph{binding} $B\subset M$ is an oriented link, and where $\pi:M\smallsetminus B\rightarrow S^1$ is a fibration such that $\pi^{-1}(\theta)$ is the interior of a compact surface $\Sigma_{\theta}$ with $\partial\Sigma_{\theta}=B$, for all $\theta\in S^1$. The surface $\Sigma\cong\Sigma_{\theta}$ is called the \emph{page}.

A contact structure $\xi$ on $M$ is \emph{supported by the open book decomposition} $(B,\pi)$ if $\xi$ can be isotoped through contact structures to a contact structures with a contact form $\alpha$ such that
\begin{enumerate}[\qquad$\bullet$]
    \item $d\alpha$ is a positive symplectic form on the interior of the pages, i.e. on $\pi^{-1}(\theta)$, with respect to the orientation of the pages, which is induced from the orientation of $B$, and
    \item $\alpha>0$ on $B$.
\end{enumerate}
\end{definition}
By isotopy we mean a smooth $1$-parameter family of contact structures $\xi_t$, $t\in[0,1]$. By Gray's stability theorem this implies that there is a $1$-parameter family of diffeomorphisms $\phi_t:M\rightarrow M$ such that $d\phi_t(\xi_0)=\xi_t$ (see e.g. \cite[Theorem 2.2.2]{MR2397738}).

Given a Lefschetz fibration $f:(W,\lambda)\rightarrow \mathbb{C}$, with total space $W$ a Weinstein manifold, we get an open book decomposition on $M=\partial_{\infty} W$ which supports $\xi=\text{ker}\,\lambda_{|M}$ (see \cite[Construction 2.9]{MR4687585}). It has binding $B=f^{-1}(0)$ and the fibration $\pi:M\smallsetminus B\rightarrow S^1$ is given by composing $f$ with the angle coordinate $\theta$ on $D=\{|z|\leq1\}\subset\mathbb{C}$. In the case of $W=\mathbb{C}\times\mathbb{C}^*\rightarrow \mathbb{C}$, the binding is then $B=f^{-1}(0)\cap M=S^1\sqcup S^1$. Equivalently we can think of the fibration as $D\times DT^*S^1\rightarrow D$, and smoothing happens near the "corners" away from $B$. The pages are then seen to be given by $\Sigma_{\theta}=DT^*S^1$, with boundary $F_{\theta}=B$, as shown in Figure~\ref{OBD}.

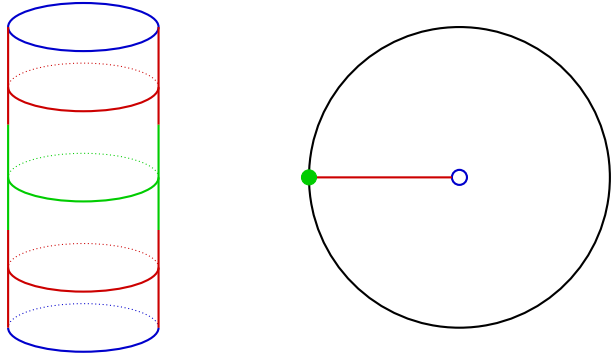
\begin{figure}[H]
\centering
\begin{tikzpicture}
\draw[thick] (0,0) circle (2cm);
\draw[blue!80!black, thick] (0,0) circle (0.1cm);
\draw[red!80!black, thick] (-1.9,0) to (-0.1,0);
\draw[green!80!black,fill=green!80!black] (-2,0) circle (0.1cm);

\draw[blue!80!black,thick] (-4,2) arc (0:180:1 and 0.32);
\draw[blue!80!black,thick] (-6,2) arc (0:180:-1 and -0.32);

\draw[blue!80!black,densely dotted] (-4,-2) arc (0:180:1 and 0.32);
\draw[blue!80!black,thick] (-6,-2) arc (0:180:-1 and -0.32);

\draw[red!80!black,densely dotted] (-4,1.2) arc (0:180:1 and 0.32);
\draw[red!80!black,thick] (-6,1.2) arc (0:180:-1 and -0.32);
\draw[red!80!black,densely dotted] (-4,-1.2) arc (0:180:1 and 0.32);
\draw[red!80!black,thick] (-6,-1.2) arc (0:180:-1 and -0.32);

\draw[green!80!black,densely dotted] (-4,0) arc (0:180:1 and 0.32);
\draw[green!80!black,thick] (-6,0) arc (0:180:-1 and -0.32);

\draw[red!80!black,thick] (-6,-2) to (-6,-0.7);
\draw[red!80!black,thick] (-6,2) to (-6,0.7);
\draw[red!80!black,thick] (-4,-2) to (-4,-0.7);
\draw[red!80!black,thick] (-4,2) to (-4,0.7);

\draw[green!80!black,thick] (-6,-0.7) to (-6,0.7);
\draw[green!80!black,thick] (-4,-0.7) to (-4,+.7);
\end{tikzpicture}
\caption{The picture shows the page $DT^*S^1$ of the open book decomposition $(B,\pi)$ of $M=\partial_{\infty}W$ and how it relates to the Lefschetz fibration $W=\mathbb{C}\times\mathbb{C}^*$. The binding $B$ is in blue, the part of the fibre of $\pi$ which lives above $\partial D$ is in green, and the part which lives above the interior of $D$ is in red.}
\label{OBD}
\end{figure}

The inverse image of the imaginary axis $\pi^{-1}(\{\text{Re}\,z=0\}\cap D)$ is two pages glued smoothly together along $B$, that is, a torus $S^1\times S^1$ which we denote by $S$. After an isotopy on $\xi$, $S$ is a convex surface where we can take $B$ as the dividing set, in sense of Definition~\ref{dividing} below, according to \cite[Lemma 4.1]{MR2249250}. Moreover, $\xi$ is tight on both components of $M\smallsetminus S$. The following definition is due to \cite{MR1129802}.

\begin{definition}\label{dividing}
A closed surface $S\subset M$ is called \emph{convex} if there is a contact vector field $V$ which is transverse to $S$.

Given a convex surface and a transverse contact vector field $V$, the \emph{dividing set} is the multi-curve $\Gamma_S=\{x\in S\;|\;V(x)\in\xi_x\}$.
\end{definition}

Next we consider $J^1S^1=S^1_{\theta}\times\mathbb{R}_y\times\mathbb{R}_z$, with its standard contact structured defined as the kernel of $\alpha=dz-yd\theta$ with standard tubular neighbourhood of the zero-section $U_{\epsilon}=\{(\theta,y,z)\;|\;y^2+z^2\leq \epsilon^2\}$. This is a solid torus. Let $V_{\text{rad}}=y\partial_y+z\partial_z$ --- this is a vector field which is pointing radially outwards from the zero-section, and thus it is transverse to the boundary of $U_{\epsilon}$. It is a contact vector field, as
\begin{align*}
\mathcal{L}_{V_{\text{rad}}}\alpha&=d\iota_{V_{\text{rad}}}\alpha+\iota_{V_{\text{rad}}}d\alpha \\
&=dz-\iota_{V_{\text{rad}}}(dy\wedge d\theta) \\
&=dz-yd\theta \\
&=\alpha.
\end{align*}
The dividing set of the boundary of $U_{\epsilon}$ is given by the points at which $V\in\xi$, i.e. $\alpha(V)=z=0$. Thus the dividing set consists of the union of the two longitudes $\ell_{\pm}=\{(\theta,\pm \epsilon,0)\}$. 

Denote the two components of $M\smallsetminus S$ by $T_0$ and $T_1$, and note that they are the interior of solid tori (recall that $S=\pi^{-1}(\{\text{Re}\,z=0\}\cap D)$ is their common convex boundary, which is homeomorphic to a torus $S^1\times S^1$). There are two Legendrian copies of $S^1$ living as zero-sections in the fibres of the two points $\partial D\cap\{\text{Im}\,z=0\}$. Denote them by $\Lambda_0$ and $\Lambda_1$, and say that $\Lambda_0$ lives above $-1$ in $T_0$ and $\Lambda_{1}$ lives above $1$ in $T_1$. We can find neighbourhoods of both Legendrians, $U^0$ respectively $U^1$, which are contactomorphic to the tubular neighbourhoods $U_{\epsilon}$ of the zero-section in $J^1S^1$. In fact, as the following lemma shows we can assume $U^i=T^i$:
\begin{lemma}
The closed convex neighbourhoods $U^i$ and $\overline{T}_i=T_i\cup S$ are contactomorphic.
\end{lemma}
\begin{proof}
This is a consequence of \cite[Theorem 8.2]{MR1487723}: there is a unique contact structure on the solid torus which is tight and for which the dividing set is the longitudes. 
\end{proof}

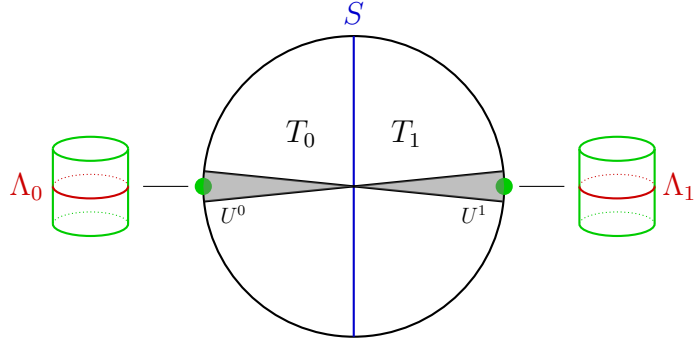
\begin{figure}[H]
\centering
\begin{tikzpicture}
\draw[thick] (0,0) circle (2cm);
\draw[blue!80!black,thick] (0,-2) -- (0,2);
\node[blue!80!black] at (0,2.3) {$S$};
\draw[thick, green!80!black,fill=green!80!black] (2,0) circle (0.1cm);
\draw[thick, green!80!black,fill=green!80!black] (-2,0) circle (0.1cm);

\draw[green!80!black,thick] (-3,0.5) arc (0:180:0.5 and 0.16);
\draw[green!80!black,thick] (-4,0.5) arc (0:180:-0.5 and -0.16);

\draw[green!80!black,densely dotted] (-3,-0.5) arc (0:180:0.5 and 0.16);
\draw[green!80!black,thick] (-4,-0.5) arc (0:180:-0.5 and -0.16);

\draw[green!80!black,thick] (-4,-0.5) -- (-4,0.5);
\draw[green!80!black,thick] (-3,-0.5) -- (-3,0.5);

\draw[red!80!black,densely dotted] (-3,0) arc (0:180:0.5 and 0.16);
\draw[red!80!black,thick] (-4,0) arc (0:180:-0.5 and -0.16);

\node[red!80!black] at (-4.35,0) {$\Lambda_0$};

\draw[green!80!black,thick] (4,0.5) arc (0:180:0.5 and 0.16);
\draw[green!80!black,thick] (3,0.5) arc (0:180:-0.5 and -0.16);

\draw[green!80!black,densely dotted] (4,-0.5) arc (0:180:0.5 and 0.16);
\draw[green!80!black,thick] (3,-0.5) arc (0:180:-0.5 and -0.16);

\draw[green!80!black,thick] (3,-0.5) -- (3,0.5);
\draw[green!80!black,thick] (4,-0.5) -- (4,0.5);

\draw[red!80!black,densely dotted] (4,0) arc (0:180:0.5 and 0.16);
\draw[red!80!black,thick] (3,0) arc (0:180:-0.5 and -0.16);

\node[red!80!black] at (4.33,0) {$\Lambda_1$};

\draw (-2.8,0) to (-2.2,0);
\draw (2.8,0) to (2.2,0);

\node at (-0.7,0.7) {$T_0$};
\node at (0.7,0.7) {$T_1$};
\draw[thick] (-1.98,0.2) to (0,0);
\draw[thick] (1.98,0.2) to (0,0);
\draw[thick] (-1.98,-0.2) to (0,0);
\draw[thick] (1.98,-0.2) to (0,0);

\draw[draw=none,fill,gray,opacity=0.5] (1.98,0.2) to (2,0) to (1.98,-0.2) to (0,0) to (1.98,0.2);
\draw[draw=none,fill,gray,opacity=0.5] (-1.98,0.2) to (-2,0) to (-1.98,-0.2) to (0
,0) to (-1.98,0.2);

\node[scale=0.7] at (-1.6,-0.37) {$U^0$};
\node[scale=0.7] at (1.6,-0.37) {$U^1$};
\end{tikzpicture}
\caption{Showing the different sets projection to the base space $D$, together with the fibres over the points $\pm1$ with the Legendrian $S^1$'s forming the zero-sections in red.}
\label{neighbourhoods}
\end{figure}

Therefore we conclude:
\begin{theorem}\label{stopsect}
Removing a neighbourhood $U^0$, which is contactomorphic to the interior of a solid torus with convex boundary, of $\Lambda_0$ from $(S^1\times S^2,\xi_{\textup{std}})$ yields a contact manifold contactomorphic to a convex closed neighbourhood of the zero-section, identified with $\Lambda_1$ (or $\Lambda_0$), in $(J^1S^1,\xi_{\textup{std}})$. Thus removing a stop put at $\Lambda_0$ from the Weinstein domain $S^1\times D^3$ yields a Weinstein sector with boundary at infinity a convex closed neighbourhood of the zero-section in $(J^1S^1,\xi_{\textup{std}})$.
\end{theorem}

\subsubsection{Reeb chords and holomorphic curves in the complement of the stop}

For the wrapped Fukaya category of the sector obtained by removing a neighbourhood $U^i$ of one of the Legendrians $\Lambda_i$, as above, to be well-defined, we will need to adjust the contact form in the neighbourhood to have a cutoff Reeb vector field. Moreover, to be able say that the contact invariants in the boundary at infinity is the same as those computed in $(J^1S^1,\xi_{\text{std}})$, we need to make sure that there are no periodic Reeb orbits and that Reeb chord which leaves a smaller tubular neighbourhood never returns to this neighbourhood. Finally, we have to make sure that all relevant pseudoholomorphic discs are contained in the samller neighbourhood.

Assume now that we have removed $U^0$ and that we are working in $U^1$. By possibly taking a smaller neighbourhood and rescaling  we can assume that 
\[
U^1=U^1_2=\{(\theta,y,z)\,|\,\rho=y^2+z^2<2\}.
\]
We define $f(\rho)$ to be a function $[0,2)$ such that
\begin{align*}
    &\bullet\; f(\rho)=1,\quad 0\leq \rho\leq 1,\\
    &\bullet\; f(\rho)=\frac{1}{H(\rho)},\text{ for a smooth function } H\text{ such that } H(2)=H'(2)=0,\\
    &\bullet\; f'(\rho)\geq 0.
\end{align*}
Then we let $\beta=f\alpha$, so that the Reeb dynamics on $(U^1,\beta)$ coincides with the one in $J^1S^1$ in the neighbourhood $U_1^1=\{\rho<1\}$, and the Reeb vector field $R_{\beta}$ is a cutoff contact vector field \cite[Section 2.9]{MR4106794}.
\begin{lemma}\label{chords}
For the contact form $\beta$ defined above, the Reeb vector field $R_{\beta}$ has no periodic orbits, and no Reeb chords which leaves the region $U_1^1=\{\rho <1\}$ re-enters later.
\end{lemma}
\begin{proof}
Write $R_{\beta}=a\partial_{\theta}+b\partial_y+c\partial_z$, for functions $a,b,c:J^1S^1\rightarrow\mathbb{R}$. Since $R_{\beta}=\partial_z$ in $U^1$ it suffices to show that $c\not=0$ (so that $c>0$ everywhere). Therefore, towards a contradiction, assume that $c=0$ at some point in $J^1S^1$. We compute $d\beta$:
\begin{align*}
    d\beta&=df\wedge\alpha+f\wedge d\alpha\\
    &=f'(\rho)d\rho\wedge\alpha+f(\rho)d\alpha \\
    &=f'(\rho)(2ydy+2zdz)\wedge (dz-yd\theta)+f(\rho)d\theta\wedge dy \\
    &=2f'(\rho)\left(ydy\wedge dz +y^2d\theta\wedge dy+yzd\theta\wedge dz\right)+f(\rho)d\theta\wedge dy.
\end{align*}
Thus, at a point where $c=0$,
\begin{align*}
0=&d\beta(R_{\beta},-) \\
=&a\left(2y^2f'(\rho)dy+f(\rho)dy+2yzf'(\rho)dz\right) \\
+&b\left(2yf'(\rho)dz-2y^2f'(\rho) d\theta-f(\rho)d\theta\right).
\end{align*}
But this implies $a=0$ and $b=0$, as $y^2f'(\rho)+f(\rho)>0$. This shows that $c$ cannot vanish, as sought.
\end{proof}

Assume that $\Lambda\subset U^1$ is a Legendrian with finitely many \emph{Reeb chords}, i.e. integral curves of the Reeb vector field which start and ends  at $\Lambda$. We want to consider rigid $J$-holomorphic curves $u:D\smallsetminus\{x,y_1,\dots, y_n\}\rightarrow (\mathbb{R}_t\times U^1_2,d(e^t\beta))$, where $u$ has positive puncture at $x$  and negative punctures punctures at $y_1,\dots,y_n$ in the boundary $\partial D$, asymptotic to the Reeb chord $\alpha_0$ respectively $\beta_1,\dots,\beta_n$ (in the meaning of \cite[Definition 3.6]{MR3548486}), and mapping $\partial D\smallsetminus\{x,y_1,\dots,y_n\}$ to $\Lambda\times\mathbb{R}_t$. Here $J$ is a compatible generic almost complex structure, which is assumed to be cylindrical.
\begin{lemma}\label{curves}
Assuming $\Lambda\subset U_1^1\subset U^1_2$ to be small enough (by possibly rescaling in the $z$-direction). If $u$ is a $J$-holomorphic disc as above (with punctures converging to Reeb chords on 
$\Lambda$), then $\textup{im}( u)\subset \mathbb{R}_t\times U_1^1$.
\end{lemma}
\begin{proof}
Following e.g. the convention of \cite[Section 3]{MR3548486}, the \emph{total Hofer energy} of $u$ is defined to be
\[
E(u)=\int_ud\beta+\sup_{\varphi\in \mathcal{C}}\int_u\varphi(t)dt\wedge \beta,
\]
where $\mathcal{C}$ is the set of smooth non-negative functions $\varphi$ which have compact support, for which $\int_{\mathbb{R}}\varphi(t)dt=1$. It satisfies
\[
E(u)=2\ell(\alpha_0)-\sum_i\ell(\beta_i),
\]
where $\ell(\gamma)=\int_{\gamma}\beta$ is the length of a Reeb chord. By shrinking $\Lambda$ we can assume that $E(u)<C$, for any choice of constant $C>0$, since by assumption $u$ has a unique positive puncture and since there are only finitely many Reeb chords on $\Lambda$, which holds generically in contactisations and thus also for $\Lambda$ in $U^1_2$ by Lemma~\ref{chords}.

We want to find a lower bound $C_0>0$, depending only on the almost complex structure $J$, such that if $(s,p)\in\text{im}(u)$ for some $p\in U^1_2\smallsetminus\overline{U}_1^1$, then $E(u)>C_0$. If we can provide such a bound, the statement follows.

For any real number $s$, let
\[
A_s(u)=e^{-1}\int_{u\cap\{t\in[s-1,s+1]\}}d(e^{t-s}\beta).
\]
By definition of the total Hofer energy we have the inequalities $0\leq A_s(u)\leq E(u)$, where the second one follows from the identity $d(e^{t-s-1}\beta)=e^{t-s-1}d\beta+e^{t-s-1}dt\wedge \beta$, together with bounds $e^{t-s-1}\leq 1$, for $t\in[s-1,s+1]$, and $\int_{s-1}^{s+1}e^{t-s-1}dt=1-e^{-2}<1$.

Suppose that $p\in U^1_2\smallsetminus \overline{U}_1^1$ and assume that $(s,p)\in\text{im}(u)$, for some real number $s$. Suppose first that $s=0$. By monotonicity \cite[Proposition 4.3.1. (ii)]{MR1274929} applied to the compact manifold with boundary $[-1,1]\times U_2^1$, since for $\Lambda\subset U^1_1$ there is a fixed distance $r_0>0$ from $\Lambda$ to $\partial (U^1_2\smallsetminus\overline{U}^1_1)$ (which can only get larger if we make $\Lambda$ smaller in the $z$-direction), this implies that
\[
A_0(u)>C_0'r_0^2=C_0,
\]
where $C_0'$ only depends on $J$.

Suppose next that $s\not=0$. Then, by translation in the $\mathbb{R}$-direction we define a new $J$-holomorphic curve $\tilde{u}$ by $\tilde{u}(q)=(u_t(q)-s,u_{U^1_2}(q))$, where $u(q)=(u_t(q),u_{U^1_2}(q))\in\mathbb{R}_t\times U_2^1$. We note that then
\[
A_s(u)=A_0(\tilde{u})>C_0,
\]
using that $J$ is cylindrical.
\end{proof}
Thus, by Lemma~\ref{chords} and Lemma~\ref{curves}, together with Theorem~\ref{stopsect}:
\begin{theorem}\label{cesector}
For a suitable contact form on $\partial_{\infty}(\mathbb{C}\times\mathbb{C}^*)\smallsetminus \sigma_{\textup{stop}}=(S^1\times S^2\smallsetminus U^0,\xi_{\textup{std}})$ there exists an embedding of a convex neighbourhood of the zero-section in $J^1S^1$ into $U^1$ which preserves the contact form and identifies the zero-section with $\Lambda_1$ (see Figure~\ref{neighbourhoods}). It satisfies the following: For a suitable compatible cylindrical almost complex structure, the pseudoholomorphic disc count of any link $\Lambda \subset U^1$ coincides with the count in  $J^1S^1$ (for the relevant discs in the definition of the Chekanov-Eliashberg algebra). In particular, the two Chekanov--Eliashberg algebras are equal.
\end{theorem}
Thus, in practice, when computing the Chekanov--Eliashberg (see Section~\ref{invprel}) of the attaching spheres in $\partial_{\infty}(\mathbb{C}\times\mathbb{C}^*)\smallsetminus \sigma_{\text{stop}}=(S^1\times S^2\smallsetminus U^0,\xi_{\text{std}})$, we can do this in $(J^1S^1,\xi_{\text{std}})$.
\section{Symplectic realisation of type $\tilde{A}_n$ additive preprojective algebras}\label{wrappedsec}
We use Theorem~\ref{cesector} from the previous section to conclude that the Chekanov--Eliashberg sub-dg-algebra of a Legendrian $\Lambda \subset S^1\times S^2$ after putting a stop $\Lambda_{\text{stop}}$ (assuming $\Lambda$ is disjoint from a neighbourhood of this stop) can be computed in $J^1S^1$. We further show that this allows us to compute the partially wrapped Fukaya category by computing the Chekanov--Eliashberg algebra in $J^1S^1$. This also implies that we can compute the fully wrapped Fukaya category of the manifold obtained by handle attachment along $\Lambda\subset S^1\times S^2$ by dg-localisation of the module category of a finitely generated algebra computed in $J^1S^1$.

In particular the partially wrapped Fukaya category of cyclic plumbings of $T^*S^2$ with a stop can be computed by the Chekanov--Eliashberg algebra of an attaching link in $J^1S^1$. We show that the dg-algebra of this attaching link is the Ginzburg algebra of $\tilde{A}_n$.

\subsection{Preliminaries: symplectic invariants}\label{invprel}
We recall the definitions of the Chekanov--Eliashberg algebra of a Legendrian submanifold in a contact manifold and the wrapped Fukaya category of a Liouville sector. In particular we recall combinatorial models for computations in $(J^1\mathbb{R},\xi_{\text{std}})$, $(J^1S^1,\xi_{\text{std}})$ and $(\#^k(S^1\times S^2),\xi_{\text{std}})$. We also recall the relations between these two invariants via the surgery formula, which together with a generation result of the partially wrapped Fukaya category are the main tools that our results are based upon.

\subsubsection{The Chekanov--Eliashberg algebra}\label{cesection}

The Chekanov--Eliashberg algebra was proposed as a part of symplectic field theory in \cite{MR1826267}, and rigorously defined in $\mathbb{R}^3$ in \cite{MR1946550}. It was defined in the contactisation of an exact symplectic manifold in \cite{MR2299457}, and combinatorially in $J^1S^1$ in \cite{MR2131643} and in $\#^k(S^1\times S^2)$ in \cite{MR3356070} (see also \cite{MR4033516} for a finitely generated version). It is a semi-free dg-algebra, whose homology, also known as Legendrian contact homology, is a Legendrian isotopy invariant, more powerful than the previously known classical invariants, namely the smooth isotopy class, the Thurston-Bennequin number and the rotation number. 

Given a contact manifold $(M,\xi=\text{ker}\,\alpha)$, with chosen contact form $\alpha$. The Reeb vector field $R_{\alpha}$ is defined by
\begin{align*}
\iota_{R_{\alpha}}d\alpha=0,\\
\iota_{R_{\alpha}}\alpha=1.
\end{align*}
The non-trivial flow-lines of $R_{\alpha}$ from $\Lambda$ to itself are called \emph{Reeb chords}. Given a Legendrian submanifold $\Lambda\subset M$, decomposed into connected components $\Lambda=\Lambda_1\sqcup\cdots\sqcup\Lambda_n$. We say that $\Lambda$ is in \emph{chord generic} position by if the tangent space of the Legendrian at the starting point of the Reeb chord, is flown with the Reeb vector field to a subspace which is transverse inside $\xi$ to the tangent space of the Legendrian at the endpoint of the Reeb chord. Using a Legendrian isotopy we can always put $\Lambda$ in chord generic position, and we will therefore assume that this is the case, so that the Reeb chords form a discrete set. We define the \emph{Chekanov--Eliashberg algebra} $CE(\Lambda)$ of $\Lambda$ as follows.

We define a quiver $Q$ (with possibly countably infinitely many arrows) by picking one vertex $i$ for each connected component $\Lambda_i$, and one arrow $\alpha_c:i\rightarrow j$, for each Reeb chord $c$ which starts at $\Lambda_i$ and ends at $\Lambda_j$. As an algebra $CE(\Lambda)=\mathbb{C}Q$, that is, it is the path-algebra of the quiver (we write multiplication of arrows as composition). 

The grading is defined on generators using the Conley--Zehnder index and extended to words (multiplication/concatenations of arrows) as the sum of the gradings of the generators ($|\alpha\beta|=|\alpha|+|\beta|)$. In fact, for multiple component links there are many possible choices of gradings which depends on choices of certain paths. Furthermore, one must require that the Maslov classes of all the compenents of the Legendrian vanish in order to obtain a well-defined $\mathbb{Z}$-grading (otherwise one gets a grading in the integers modulo the minimal Maslov number. 

Assuming this now, we define the grading on generators as follows. Given a Reeb chord $c:\Lambda_i\rightarrow\Lambda_i$ starting and ending at the same component, we choose a capping path $\gamma_c$, i.e. $\gamma:I\rightarrow\Lambda$, such that $\gamma_c(0)=\text{starting point of }c$ and $\gamma_c(1)=\text{end point of }c$. This defines a path $\Gamma_c:I\rightarrow\xi_{|\gamma_c}$ of Lagrangian subspaces. We put
\[
|c|=\text{CZ}(\Gamma_c)-1,
\]
where $\text{CZ}$ denotes the Conley-Zehnder index as defined in \cite{MR2299457}.

To grade chords $c:\Lambda_i\rightarrow\Lambda_j$ between mixed components, we choose base-points $*_i\in\Lambda_i$ on all components, together with paths $\gamma_{ij}$ from $*_i$ to $*_j$ and paths $\Gamma_{ij}:I\rightarrow\xi_{|\gamma_{ij}}$. We assume that $\gamma_{ji}=\overline{\gamma}_{ij}$ and $\Gamma_{ji}=\overline{\Gamma}_{ij}$, where overline means the path traversed backwards, and $\gamma_{ij}*\gamma_{jk}\simeq\gamma_{ik}$ and $\Gamma_{ij}*\Gamma_{jk}\simeq\Gamma_{ik}$, so that in fact, we only need to pick paths $\gamma_i$ and $\Gamma_i$ from $*_i$ to $*_{n}$, $i=1,\dots,n-1$. Given $c:\Lambda_i\rightarrow\Lambda_j$ we also pick two paths, $\gamma_{c,j}$ from the end point of $c$ to $*_j$ and $\gamma_{c,i}$ from $*_i$ to the end point of $c$; as before they induce paths $\Gamma_{c,l}:I\rightarrow\xi_{\gamma_{c,l}}$, $l=i,j$. We put $\Gamma_c=\Gamma_{c,j}*\Gamma_{ji}*\Gamma_{c,i}$ (concatenation of pahts), and again the grading is defined by the formula $|c|=\text{CZ}(\Gamma_c)-1$. The choices of $\Gamma_{ij}$ above amounts to the choice of a Maslov potential, see \cite{MR1765826}.

The differential $\partial:CE(\Lambda)\rightarrow CE(\Lambda)$ (which by convention here is of degree $-1$) is defined on generators by counting holomorphic curves in the symplectisation, with one input (positive puncture) and many outputs (negative punctures). It is extended to the full algebra by requiring that it satisfies the graded Leibniz rule
\[
\partial(\alpha\beta)=\partial\alpha\cdot\beta+(-1)^{|\alpha|}\alpha\cdot\partial\alpha.
\]

Next we specialise to the cases of the 1-jet spaces $M=J^1N=T^*N\times\mathbb{R}_z$, for $N=\mathbb{R}$ or $N=S^1$, which are contact manifolds with contact structure $\xi=\text{ker}(\alpha)$, with $\alpha=dz-ydx$, where $x$ is a chart on $N$ and $(x,y)$ the induced chart on the cotangent bundle. In the front diagram $\pi_F(\Lambda)\subset N\times\mathbb{R}_z$, assuming that $\Lambda$ is in generic position, this is an immersed submanifold singularities are self-intersections or cusp edges. 

The Reeb vector field is $\partial_z$. The front projection $\pi_F(\Lambda)$ has no vertical ($z$-direction) tangencies and the link $\Lambda\subset J^1N$ can be recovered by the equations $y=dz/dx$. Thus the Reeb chords are found in the front projection going from points on lower strands to a points on upper strand such that their tangent spaces at the points are equal. They can be seen to be in one-to-one correspondence with right cusps and crossings (perhaps after perturbing) in the case of $N=\mathbb{R}$. In the case $N=S^1$ one must also add extra crossings coming from chords in the handle, seen in the resolution as in Figure~\ref{onejetresol}.

The choice of Maslov potential in the case of a $1$-jet space can be done in the following manner, which becomes combinatorial when the Legendrian is one-dimensional. We choose a locally constant function
$m:\Lambda\smallsetminus\pi_F^{-1}(\{\text{cusps}\})\rightarrow\mathbb{Z}$ such that $m$ decreases by $1$ when we pass downwards through a cusp, and increases by $1$ when we pass upwards through a cusp (when the Maslov class does not vanish one also needs to remove basepoints from the components). Then the gradings on the chords corresponding to right cusps are $1$, and for a chord $c$ identified with a crossing,
\[
|c|=m(\text{overstrand of } c)-m(\text{understrand of } c).
\]
Note that for Legendrian links with $n\geq 2$ connected component, the grading is not unique, but there are $\mathbb{Z}^{n-1}$ choices. In the case of $N=S^1$, there are two extra families of chords seen as crossings in the Lagrangian projection, as in Figure~\ref{onejetresol}. The left-hand side crossings are graded as usual, whereas for the \emph{special chords}, on the right-hand side, the grading is defined by $|c|=m(\text{overstrand of } c)-m(\text{understrand of } c)-1$.

For general $1$-jet spaces $M=J^1N$ of a manifold $N$, the discs for a particular choice of compatible almost complex structure can be counted in $N\times\mathbb{R}_z$ by counting rigid gradient flow trees, assuming that the Legendrian has no singularities except cusp edges or swallow tails in its front projection \cite{MR2326943}. For $N=\mathbb{R}$ or $N=S^1$, the differential can be computed combinatorially in the Lagrangian projection $\pi_L:J^1N\rightarrow T^*N$. Note that the Reeb chords are in one-to-one correspondence with double points in the Lagrangian projection. The differential $\partial \alpha$ at a double point $\alpha=\pi_L(c)$ counts discs with one positive puncture at $\alpha$ and several negative punctures at  $\beta_1=\pi_L(b_1),\dots,\beta_n=\pi_L(b_n)$, modulo holomorphic reparametrisation. More precisely,
\[
\partial\alpha=\sum_{\substack{n\geq 0,\,\beta_1,\dots,\beta_n\in R \\ u\in\mathcal{M}(\alpha;\beta_1,\dots,\beta_n)}}(-1)^{\sigma(u)}\beta_1\dots\beta_n,
\]
where $R$ is the set of Reeb chords or double points. The sets $\mathcal{M}(\alpha;\beta_1,\dots,\beta_n)$ are finite and non-empty for finitely many tuples, hence the sums above are finite (which is easily shown using an action argument). For $n=0$, if $\mathcal{M}(\alpha)\not=\emptyset$, then $\alpha$ must start and end at the same component and the empty word should be read as the corresponding idempotent.  The sets $\mathcal{M}(\alpha;\beta_1,\dots,\beta_n)$ are defined as the sets of discs $u:D\smallsetminus \{x,y_1,\dots,y_n\}\rightarrow T^*N$, where $D\subset\mathbb{C}$ is the unit disc and where $\{x,y_1,\dots,y_n\}\subset\partial D$ is cyclically ordered in the boundary, that are immersions such that $\lim_{z\rightarrow x}u(z)=\alpha$ and $\lim_{z\rightarrow y_i}=\beta_i$. Moreover, $u$ is required to map a neighbourhood of $x$ to one of the quadrants marked with a plus sign in Figure~\ref{chordsign}, and map a neighbourhood of $y_i$ to a quadrant marked with a minus sign.
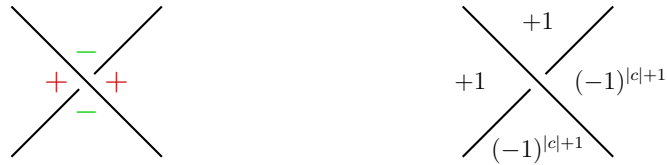
\begin{figure}[H]
\centering
\begin{tikzpicture}
\draw[thick] (0,2) to (2,0);
\draw[thick] (0,0) to (0.9,0.9);
\draw[thick] (1.1,1.1) to (2,2);
\node[red!80!black, very thick] at (1.4,1) {$+$};
\node[red!80!black] at (0.6,1) {$+$};
\node[green!80!black] at (1,0.6) {$-$};
\node[green!80!black] at (1,1.4) {$-$};

\draw[thick] (6,2) to (8,0);
\draw[thick] (6,0) to (6.9,0.9);
\draw[thick] (7.1,1.1) to (8,2);

\node[scale=0.8] at (7,0.1) {$(-1)^{|c|+1}$};
\node[scale=0.8] at (8.1,1) {$(-1)^{|c|+1}$};
\node[scale=0.8] at (7,1.8) {$+1$};
\node[scale=0.8] at (6.1,1) {$+1$};
\end{tikzpicture}
\caption{Reeb signs at the crossings on the left-hand side and orientation signs at crossings on the right-hand side, both shown in the Lagrangian projection. On the left-hand side: A positive puncture of a disc turns $90$ degrees on the left- or right-hand side quadrant, while the disc covers the same quadrant; the same for a disc with a negative puncture, but at the top or bottom quadrant. On the right-hand side: The orientation sign corresponding to chords $c$ are shown. For crossings with even degree this sign depends on a choice of an isotopy so that the strand with negative slope goes over the strand with positive slope. E.g. when the diagram comes from the resolution of a front (Figure~\ref{ngone}, and Figure~\ref{ngtwo}) this is case. For the special chords $s_{i,j}$ in Figure~\ref{onejetresol}, we choose the orientation signs obtained by rotating the picture $90$ degrees clockwise. Then they correspond to the short handle chords $c^0_{i,j}$ (without changing signs).}
\label{chordsign}
\end{figure}
Finally, a disc is either counted with a plus or a minus sign, which depends on the orientation signs at crossings, as shown in Figure~\ref{chordsign}, and choice of basepoint on each component. If the total number of crossings at which $u$ has a negative orientation sign together with the number of times a basepoint is traversed along the boundary is odd, then the disc is counted with a minus sign, i.e. $\sigma(u)=1$. Otherwise it is counted with plus sign, i.e. $\sigma(u)=0$.

A subcritical Weinstein domain of dimension $4$ has boundary $(\#^k(S^1\times S^2),\xi_{\text{std}})$, obtained by attaching $k$ $1$-handles to $B^4$. Up to quasi-isomorphism we can compute the Chekanov--Eliashberg algebra as follows, see \cite{MR3356070}.

Assume that the Weinstein handlebody diagram is in standard form, that is, the $1$-handles are represented by two balls each, which lie opposite to each other horizontally, and they are lined up vertically; the $2$-handles enter/exits the handles and go into a box between the $1$-handles where the linking happens. Legendrians can be isotoped so that their intersections with the $1$-handles consist of parallel strands. Therefore, for each $1$-handle with $k$ strands entering we get a countable family of handle chords (in \cite{MR4033516} this is simplified to finitely many by means of algebra, but it will not be used here). These are chords $c_{i,j}^p$, which go around in the handle from the $i$:th strand to $j$:th strand passing some basepoint $p$ times. We can think of them as in Figure~\ref{handlechords}, as chords going around a circle, with the strands on the right-hand side, and the chord $c_{i,j}^p$ going around the left-hand side of the circle $p$ times.

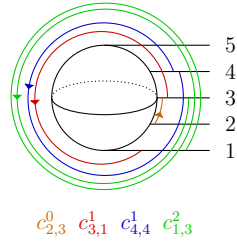
\begin{figure}[H]
\centering
\begin{tikzpicture}[scale=1.4]
\draw (0,0) circle (0.5cm);
\draw[densely dotted] (0.5,0) arc (0:180:0.5 and 0.16);
\draw (-0.5,0) arc (0:180:-0.5 and -0.16);

\draw (0,0.5) to (1,0.5);
\draw (0.433012,0.25) to (1,0.25);
\draw (0.5,0) to (1,0);
\draw (0.433012,-0.25) to (1,-0.25);
\draw (0,-0.5) to (1,-0.5);

\node[scale=0.7] at (1.2,0.5) {$5$};
\node[scale=0.7] at (1.2,0.25) {$4$};
\node[scale=0.7] at (1.2,0) {$3$};
\node[scale=0.7] at (1.2,-0.25) {$2$};
\node[scale=0.7] at (1.2,-0.5) {$1$};

\draw[red!80!black] (0.6,0) arc[start angle=0, end angle=307, radius=0.63];

\draw[blue!80!black,domain=21:379,samples=500] plot ({0.7*(\x*0.0002+1)*cos(\x)}, {0.7*(\x*0.0002+1)*sin(\x)});
\draw[green!80!black,domain=0:753,samples=500] plot ({0.8*(\x*0.0002+1)*cos(\x)}, {-0.8*(\x*0.0002+1)*sin(\x)});
\draw[orange!80!black,domain=-30:0,samples=500] plot ({0.55*cos(\x)}, {0.5*sin(\x)});

\node[scale=0.7,orange!80!black] at (-0.5,-1.2) {$c^0_{2,3}$};
\node[scale=0.7,red!80!black] at (-0.1,-1.2) {$c^1_{3,1}$};
\node[scale=0.7,blue!80!black] at (0.3,-1.2) {$c^1_{4,4}$};
\node[scale=0.7,green!80!black] at (0.7,-1.2) {$c^2_{1,3}$};

\node[green!80!black, scale=0.4, rotate=-90] at (-0.828,0) {\midarrow};
\node[blue!80!black, scale=0.4, rotate=-100] at (-0.718,0.1) {\midarrow};
\node[red!80!black, scale=0.4, rotate=-85] at (-0.657,-0.05) {\midarrow};
\node[orange!80!black, scale=0.4, rotate=71] at (0.526,-0.15) {\midarrow};
\end{tikzpicture}
\caption{Examples of handle chords in a handle where five strands enter (note that they may or may not lie on the same component of the link). The differential can be seen by observing how a chord can split into shorter chords in this picture (except for the idempotent terms appearing in $\partial c^1_{i,i}$).}
\label{handlechords}
\end{figure}

The description of the grading in front projection in terms of a choice of Maslov potential above extends Weinstein handlebody diagrams in standard position. One must now require that the value of $m$ remains constant when passing through a $1$-handle. Then the grading of the handle chords $c=c_{i,j}^p$ are
\[
|c_{i,j}^p|=2p-1+m(\text{endpoint of }c_{i,j}^p)-m(\text{starting point of } c_{i,j}^p)
\]
(where by endpoint/starting point mean any point at the corresponding strand near the $1$-handle). 

The differential at the handle chords is computed as follows:
\begin{align*}
\partial c^0_{i,j}&=\sum_{i<l<j}(-1)^{m(l)+m(j)}c^0_{l,j}c^0_{i,l}, \\
\partial c^1_{i,j}&=\delta_{i,j}e_{c(i)}+\sum_{i<l}(-1)^{m(l)+m(j)}c^1_{l,j}c^0_{i,l}+\sum_{l<j}(-1)^{m(l)+m(j)}c^0_{l,j}c^1_{i,l}, \\
\partial c^p_{i,j}&=\sum_{m=0}^p\sum_l(-1)^{m(l)+m(j)}c^{p-m}_{l,j}c^m_{i,l},\quad p\geq2.
\end{align*}
Here $\delta_{i,j}$ is the Kronecker delta, $e_{c(i)}$ is the idempotent corresponding to the $i$:th strand, and 
$m(i)$ is the value of the Maslov poitential at the $i$:th strand. In the last formula above, expressions of the form $c^0_{i,j}$ with $i\geq j$ should be interpreted as $0$. 

The differential at crossings is computed as usual by using the Lagrangian projection, but we now also allow for negative punctures converging at the handles, contributing $\beta=c_{i,j}^0$ to the words. See the left-hand side example in Figure~\ref{discpic} below. The orientation sign at negative punctures at handle chords is positive if the disc approaches the handle from the right-hand side (as in Figure~\ref{discpic}), and $(-1)^{m(j)-m(i)}$ if it is approaching $c^0_{i,j}$ from the left-hand side. 

\begin{figure}[H]
\centering
\begin{tikzpicture}[scale=1.4]
\draw (0,0) circle (0.5cm);
\draw[densely dotted] (0.5,0) arc (0:180:0.5 and 0.16);
\draw (-0.5,0) arc (0:180:-0.5 and -0.16);

\begin{knot}
\strand 
(0.25,0.433012) to[out=0,in=180] (1,0.433012) to[out=0,in=180] (2,-0.866024) to[out=0,in=180] (2.5,-0.866024);

\strand[name=A]
(0.5,0) to (2.5,0);
\strand[name=B]
(0.25,-0.433012) to (2.5,-0.433012);
\end{knot}

\draw[draw=none,fill, gray, opacity=0.5] (0.5,0) arc[start angle=0, end angle=-60, radius=0.5] (0.25,-0.433012) to (1.56,-0.433012) to[out=105,in=-70] (1.43,0) to (0.5,0);

\node[scale=0.5] at (1.5,0.13) {$\beta$};
\draw[red!80!black] (1.41,-0.35) to (1.47,-0.35);
\draw[red!80!black] (1.44,-0.32) to (1.44,-0.38);
\draw[green!80!black] (1.33,-0.07) to (1.39,-0.07);
\node[scale=0.5] at (1.7,-0.54) {$\alpha$};

\begin{knot}[flip crossing=4, flip crossing =3]
\strand (3,0) to[out=-45, in=135] (3.8,-0.8) to[out=-45,in=180] (4.5,-1) to[out=0,in=-90] (5,-0.5) to[out=90,in=0] (4.5,0.5) to[out=180,in=90] (4,0) to[out=-90,in=130] (4.2,-0.5);
\strand (3,-1) to[out=45,in=-135] (4,0) to [out=45,in=180] (4.5,0.2) to (5.5,0.2);
\strand (4.2,-0.5) to[out=-50,in=180] (5.5,-0.3);
\end{knot}

\draw[draw=none,fill, gray, opacity=0.5] (3.5,-0.5) to[out=-45,in=135] (3.8,-0.8) to[out=-45,in=180]  (4.5,-1) to[out=0,in=-90] (5,-0.5) to[out=90,in=0] (4.5,0.5) to[out=180,in=90] (4,0) to (3.5,-0.5);
\draw[draw=none,fill, gray, opacity=0.6] (4,0) to[out=-90,in=130] (4.2,-0.5) to[out=-50,in=-150] (5,-0.44) to[out=94,in=-72] (4.9,0.2) to (4.5,0.2) to[out=180,in=45] (4,0);

\draw[red!80!black] (3.64,-0.5) to (3.70,-0.5);
\draw[red!80!black] (3.67,-0.53) to (3.67,-0.47);
\draw[green!80!black] (4.89,-0.42) to (4.83,-0.42);
\draw[green!80!black] (4.85,0.14) to (4.79,0.14);

\node[scale=0.5] at (3.33,-0.47) {$\alpha$};
\node[scale=0.5] at (5.15,-0.53) {$\beta_1$};
\node[scale=0.5] at (5.03,0.32) {$\beta_2$};
\node[scale=0.5] at (3.85,0.04) {$\beta_3$};

\draw[fill] (4.5,-1) circle (0.02cm);
\node[scale=0.5] at (4.53,-1.13) {$*_1$};
\draw[fill] (5.3,0.2) circle (0.02cm);
\node[scale=0.5] at (5.33,0.35) {$*_2$};

\node[scale=0.4,rotate=-45] at (3.2,-0.2) {\midarrow};
\node[scale=0.4,rotate=45] at (3.2,-0.8) {\midarrow};

\draw[fill] (2.4,-0.866024) circle (0.02cm);
\draw[fill] (2.4,0) circle (0.02cm);
\draw[fill] (2.4,-0.433012) circle (0.02cm);

\node[scale=0.5] at (2.45,-0.986024) {$*_3$};
\node[scale=0.5] at (2.45,-0.12) {$*_2$};
\node[scale=0.5] at (2.45,-0.553012) {$*_1$};
\end{tikzpicture}
\caption{Examples of two discs in two different link diagrams in the Lagrangian projection. On the left-hand side showing that $\partial\alpha= (-1)^{|\beta|+1}\beta c^0_{1,2}+\cdots$. On the right hand side that $\partial\alpha = (-1)^{|\alpha|+|\beta_2|+1}t_1\beta_1\beta_2+\cdots$.}
\label{discpic}
\end{figure}
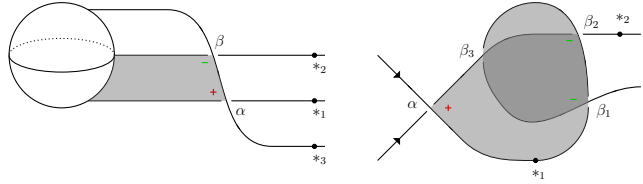

The following remark describes a combinatorial version of the Chekanov--Eliashberg algebra with loop space coefficients in the $3$-dimensional case. It based on the quasi-isomorphism between the dg-algebra $C_*(\Omega S^1;k)$, that is, chains on the based loop space of $S^1$ equipped with the Pontryagin product, and Laurent polynomials in one variable $k[t,t^{-1}]$. It is important because of its use in the surgery formula (Theorem~\ref{surgeryloop}, \cite{MR4634745}).
\begin{remark}[Loop space coefficients]\label{loopspace}
When putting stops rather than attaching handles at some components $\Lambda_{i_1},\dots\Lambda_{i_j}$ of a Legendrian link $\Lambda$, the corresponding Chekanov--Eliashberg algebra should be computed with loop space coefficients on these components. In the case of $3$-dimensional contact manifolds, i.e. $\Lambda_i\simeq S^1$, this can be modeled using the basepoints $*_{i_l}$ and choices of capping paths from the definition of the grading, together with a choice of orientation of each component $\Lambda_{i_l}$. In terms of generators of $CE(\Lambda)$, this adds two generators $t_{i_l},t_{i_l}^{-1}$ for each of the components $\Lambda_{i_1},\dots\Lambda_{i_j}$, such that $|t_{i_l}^{\pm 1}|=0$ and $\partial t_{i_l}^{\pm 1}=0$, and such that $e_jt_{i_l}^{\pm 1}=t_{i_l}^{\pm 1}e_j=0$, for $j\not=i_l$, $e_{i_l}t_{i_l}^{\pm 1}=t_{i_l}^{\pm 1}e_{i_l}=t_{i_l}^{\pm 1}$, and $t_{i_l}t_{i_l}^{-1}=t_{i_l}t_{i_l}^{-1}=e_{i_l}$. Moreover, discs $u$, say with positive puncture $\alpha$ and negative punctures $\beta_1,\dots,\beta_n$, in that order, whose boundary passes the basepoint $*_{j_i}$ $n_i$ times (counted with sign according to the orientation on $\Lambda_i$) between $\beta_{i}$ and $\beta_{i+1}$ (with $i=0,\dots,n$ and $\alpha=\beta_0=\beta_{n+1}$), now contributes the word 
\[
(-1)^{\sigma(u)}t_{j_0}^{n_0}\beta_1t_{j_1}^{n_1}\beta_2\cdots t_{j_{n-1}}^{n_{n-1}}\beta_nt_{j_n}^{n_n}
\]
to $\partial\alpha$ (which recovers $(-1)^{\sigma(u)}\beta_1\dots\beta_n$ by putting $t_i=e_i$). 
\end{remark}

Finally, some properties are more easily described in the front projection and some in the Lagrangian projection. One can always obtain the Lagrangian projection from the front projection using the \emph{Ng-resolution}, as in Figure~\ref{ngone}, Figure~\ref{ngtwo} and Figure~\ref{onejetresol}. This means that, after an isotopy of the front diagram (i.e. introducing no new crossings), we can assume that the Lagrangian projection is as shown up to isotopy. Note that this is sufficient for determining the rigid discs in the definition of the differential.

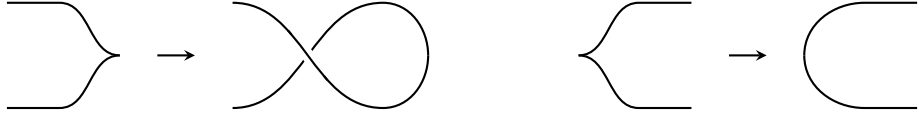
\begin{figure}[H]
\centering
\begin{tikzpicture}
\begin{knot}
\strand[thick]
(0,0.7) to[out=0,in=180] (0.7,0.7) to[out=0,in=180] (1.5,0);
\strand[thick]
(0,-0.7) to[out=0,in=180] (0.7,-0.7) to[out=0,in=180] (1.5,0);
\end{knot}

\draw[black,-stealth,thick] (2,0)  to (2.5,0) ;

\begin{knot}[clip width=6]
\strand[thick]
(3,0.7) to[out=0,in=180] (5,-0.7) to[out=0,in=-90] (5.6,0);
\strand[thick]
(5.6,0) to[out=90,in=0] (5,0.7) to[out=180,in=0] (3,-0.7);
\end{knot}

\begin{knot}
\strand[thick] 
(7.6,0) to[out=0,in=180] (8.4,0.7) to[out=0,in=180] (9.1,0.7);
\strand[thick]
(7.6,0) to[out=0,in=180] (8.4,-0.7) to[out=0,in=180] (9.1,-0.7);
\end{knot}

\draw[black,-stealth, thick] (9.6,0)  to (10.1,0) ;

\begin{knot}
\strand[thick]
(10.6,0) to[out=90,in=180] (11.4,0.7) to[out=0,in=180] (12.1,0.7);
\strand[thick]
(10.6,0) to[out=-90,in=180] (11.4,-0.7) to[out=0,in=180] (12.1,-0.7);
\end{knot}
\end{tikzpicture}
\caption{Showing how to resolve cusps when going from a front diagram to a Lagrangian diagram.}
\label{ngone}
\end{figure}

\begin{figure}[H]
\centering
\begin{tikzpicture}
\draw[thick] (-5,0) circle (0.5cm);
\draw[densely dotted] (-4.5,0) arc (0:180:0.5 and 0.16);
\draw[thick] (-5.5,0) arc (0:180:-0.5 and -0.16);
\draw[thick] (-8,0.433012) to (-5.25,0.433012);
\draw[thick] (-8,0.171) to (-5.47,0.171);
\draw[thick] (-8,-0.171) to (-5.47,-0.171);
\draw[thick] (-8,-0.433012) to (-5.25,-0.433012);

\draw[black,-stealth,thick] (-4,0)  to (-3.5,0) ;

\draw[thick] (0,0) circle (0.5cm);
\draw[densely dotted] (0.5,0) arc (0:180:0.5 and 0.16);
\draw[thick] (-0.5,0) arc (0:180:-0.5 and -0.16);
\begin{knot}[clip width =5.5]
\strand[thick] (-3,0.433012) to (-1.7,0.433012) to[out=0,in=180] (-0.5,-0.433012) to (-0.25,-0.433012);
\strand[thick] (-3,0.171) to (-1.9,0.171) to[out=0,in=180] (-1.1,-0.433012) to[out=0,in=180] (-0.7,-0.171) to (-0.47,-0.171);
\strand[thick] (-3,-0.171) to (-2.1,-0.171) to[out=0,in=180] (-1.7,-0.433012) to[out=0,in=180] (-0.9,0.171) to (-0.47,0.171);
\strand[thick] (-3,-0.433012) to (-2.2,-0.433012) to[out=0,in=180] (-1,0.433012) to (-0.25,0.433012);
\end{knot}
\end{tikzpicture}
\caption{Showing how the strands changes near a $1$-handle when passing to the Lagrangian projection, in the case of four strands going through the handle.}
\label{ngtwo}
\end{figure}
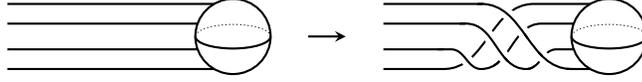
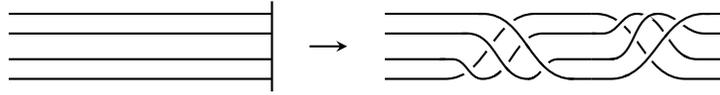
\begin{figure}[H]
\centering
\begin{tikzpicture}
\draw[thick] (-4.5,-0.6) -- (-4.5,0.6);
\draw[thick] (-8,0.433012) to (-4.5,0.433012);
\draw[thick] (-8,0.171) to (-4.5,0.171);
\draw[thick] (-8,-0.171) to (-4.5,-0.171);
\draw[thick] (-8,-0.433012) to (-4.5,-0.433012);

\draw[black,-stealth,thick] (-4,0)  to (-3.5,0) ;

\begin{knot}[clip width =5.5]
\strand[thick] (-3,0.433012) to (-1.7,0.433012) to[out=0,in=180] (-0.5,-0.433012) to (-0.25,-0.433012);
\strand[thick] (-3,0.171) to (-1.9,0.171) to[out=0,in=180] (-1.1,-0.433012) to[out=0,in=180] (-0.7,-0.171) to (-0.25,-0.171);
\strand[thick] (-3,-0.171) to (-2.1,-0.171) to[out=0,in=180] (-1.7,-0.433012) to[out=0,in=180] (-0.9,0.171) to (-0.25,0.171);
\strand[thick] (-3,-0.433012) to (-2.2,-0.433012) to[out=0,in=180] (-1,0.433012) to (-0.25,0.433012);
\end{knot}

\begin{knot}[clip width =4.5, flip crossing=5,flip crossing= 6]
\strand[thick] (-0.25,0.171) to (-0.1,0.171) to[out=0,in=180] (0.35,0.433012) to[out=0,in=180] (1.15,-0.171) to (1.5,-0.171);
\strand[thick] (-0.25,-0.171) to (0,-0.171) to[out=0,in=180] (0.63,0.433012) to[out=0,in=180] (1.25,0.171) to (1.5,0.171);
\strand[thick] (-0.25,0.433012) to(-0.2,0.433012) to[out=0,in=180]  (1,-0.433012) to (1.5,-0.433012);

\strand[thick] (-0.25,-0.433012) to (0,-0.433012) to[out=0,in=180]  (1.25,0.433012) to (1.5,0.433012);
\end{knot}
\draw[thick] (1.5,-0.6) -- (1.5,0.6);
\end{tikzpicture}
\caption{Showing the resolution when passing from the front projection to the Lagrangian projection in $J^1S^1$. The special chords, rightmost in the picture, where the strands with positive slopes cross over the strands with negative slopes corresponds to the short handle chords $c^0_{i,j}$ in the corresponding $1$-handle.}
\label{onejetresol}
\end{figure}

\subsubsection{The wrapped Fukaya category}

We sketch the definition and recall important results concerning the (partially) wrapped Fukaya category. The reader is referred to \cite{MR4106794}, \cite{MR4695507}, \cite{MR4634745}, \cite{MR4732675} and \cite{MR3427304} for in-depth treatments. Intuitively the Fukaya category is the $A_{\infty}$-category with objects being Lagrangian manifolds, where morphisms spaces are generated by intersection points and the $A_{\infty}$-maps are given by counting $J$-holomorphic polygons with punctures at the intersections, where the $A_{\infty}$-relations comes from the possible breakings of these curves. The wrapped Fukaya category is a version of this category adapted for a large class of open symplectic manifolds --- Liouville sectors --- where the Lagrangians are assumed to be cylinders over Legendrians and where the intersection points are counted after wrapping the Lagrangian submanifolds at infinity (in the cylindrical end) by using the Reeb flow. The partially wrapped Fukaya category allows for stops to stop the wrapping, e.g. in a neighbourhood of (a cylinder of) a Legendrian submanifold in the boundary at infinity. Equivalently, it can be seen as an enlargement of the class of objects for which the wrapped Fukaya category is defined, from Weinstein manifolds to Weinstein sectors.

Let $(W,\lambda)$ be a Liouville sector, we define \emph{wrapped Fukaya category} $\mathcal{W}(W,\lambda)$ as follows. The objects are exact Lagrangian submanifolds $L\subset W$ with cylindrical ends, i.e $\lambda$ pulls back to an exact form on $L$, and 
$L$ coincides with the cylinder $L=\mathbb{R}\times\Lambda$ under the identification $\tilde{\phi}:(\mathbb{R}_{t\geq 0}\times\partial_{\infty}W)\cong(W\smallsetminus \overline{W},\lambda)$, outside of a compact subset of $W$, for some Legendrian submanifold $\Lambda\subset\partial_{\infty}W$. We also include compact exact Lagrangian submanifolds, which can be thought of as having cylindrical ends with $\Lambda=\emptyset$. The Lagrangian submanifolds should also be equipped with extra data, which allows us to define $\mathbb{Z}$-gradings and $\mathbb{C}$-coefficients.

The homomorphism spaces $CW(L_0,L_1)$ are given by a limit of $\mathbb{C}$-vector spaces spanned by intersection points of $L_0$ and $\phi_H^1(L_1)$, for a family of Hamiltonians $H$.

As usual the definition depends on several choices, such as almost complex structures, but it yields an $A_{\infty}$-category $\mathcal{W}(W,\lambda)$ associated to $(W,\lambda)$, which is homologically unital, and independent of the choices up to quasi-equivalence. (Usually this is defined with cohomological grading, rather than the conventions in Section~\ref{algprelAinf}, but this does not cause any issues.)

Given a stop $\sigma\subset W$ (or union of stops) one can define $\mathcal{W}(W,\sigma)$, the \emph{partially wrapped Fukaya category} of $W$ \emph{stopped at} $\sigma$. This is another $A_{\infty}$-category in which the wrapping is not allowed to intersect the stop, and in which the holomorphic curves are also required to avoid the stop.  

The computations in this article do not use the definition of the (partially) wrapped Fukaya category explicitly, but rather they use the following two versions of the surgery formula. The first one concerns the attachment of standard Weinstein handles along Legendrian Links:
\begin{theorem}[\cite{MR2916289}]\label{surgerystd}
Suppose that $W$ is defined by attaching standard Weinstein handles along a Legendrian link $\Lambda\subset\partial_{\infty} W_0$, and let $\Lambda_1,\dots,\Lambda_m$ denote the connected components. Let $C_1,\dots, C_m$ denote the cocores of the handles. Then there is a surgery map defined by holomorphic curve counts, which induce and $A_{\infty}$-quasi-isomorphism
\[
\bigoplus_{i,j}CW(C_i,C_i)\simeq CE(\Lambda).
\]
\end{theorem}
The second formula (proved in \cite{MR4563001}) involves stopping along Legendrian submanifolds.
\begin{theorem}[\cite{MR4634745}]\label{surgeryloop}
Suppose that $W$ is obtained from $W_0$ by attaching handles along the link of Legendrian spheres $\Lambda^-=\Lambda^-_1\sqcup\dots\sqcup\Lambda_{m_1}^-$, and consider the partially wrapped Fukaya category of $W$ stopped at $\sigma$, where $\sigma$ is a stop put at the Legendrian link $\Lambda^{-}=\Lambda_1^+\sqcup\dots\sqcup\Lambda^+_{m_2}\subset \partial W_0 \smallsetminus \Lambda_-$. Then
\[
\bigoplus_{i,j}CW_{(W,\sigma)}(L_i^{\pm},L_j^{\pm})\simeq CE(\Lambda^-\cup\Lambda^+),
\]
where we on the right-hand side should put coeffients in chains on the based loop space on $\Lambda_1^+,\dots,\Lambda_{m_2}^+$. For $1$-dimensional Legendrians this is equivalent to adding generators $t_i^{\pm}$ on these components (see Remark~\ref{loopspace}). The sum on the left-hand side is over all possible combinations, where $L^-=C_i$ is the cocore of the handle, and where $L^+_i=D_i$ is the linking disc of the stop.
\end{theorem}
The following theorem relates the partially wrapped Fukaya category and the wrapped Fukaya category of the corresponding sector, obtained by removing a neighbourhood of the stop (see Section~\ref{geomprel}).
\begin{theorem}[\cite{MR4106794} ]\label{stopsecteq}
Let $W$ be a Liouville manifold with stop $\sigma$. Then the inclusion $W_{\sigma}\rightarrow W$ induces an equivalence
\[
D^{\pi}\mathcal{W}(W,\sigma)\simeq D^{\pi}\mathcal{W}(W_{\sigma}).
\]
\end{theorem}

Finally, to compute the (partially) wrapped Fukaya category we need not just the morphism spaces, but also a set of generators for the category, which for Weinstein manifolds are given by the cocores of the critical handles:
\begin{theorem}[\cite{MR4732675}, \cite{MR4106794}]\label{generation}
The wrapped Fukaya category obtained from attaching handles along $\Lambda_1,\dots,\Lambda_m$ to a subcritical Weinstein domain is generated by the cocores $C_1,\dots,C_m$ corresponding to the handles.

The partially wrapped Fukaya category obtained by attaching handles along the Legendrian link of spheres $\Lambda^-=\Lambda^-_1 \sqcup\cdots\sqcup\Lambda^-_{m_1}\subset \partial_\infty W_0$ to a subcritical Weinstein manifold $W_0$, and stopped along $\Lambda^+=\Lambda_1^+\sqcup\cdots\sqcup\Lambda_{m_2}^+\subset\partial_{\infty}W_0\smallsetminus\Lambda^-$, is generated by the corresponding cocores $C_1,\dots,C_{m_1}$ and the linking discs $D_1,\dots,D_{m_2}$.
\end{theorem} 

\subsection{Calculation of the Chekanov--Eliashberg algebra in $J^1S^1$ and $S^1\times S^2$}
We realise the Ginzburg algebra $\mathcal{G}(\tilde{A}_n)$ as the Chekanov--Eliashberg algebra of cyclically oriented $\tilde{A}_{n}$ of an attaching link in $J^1S^1$. We compare it to the algebra of the attaching link for cyclic plumbing of $T^*S^2$, computed in $S^1\times S^2$.

\subsubsection{The $\tilde{A}_{n}$-computation in $J^1S^1$}

We calculate $CE(\Lambda_{\text{cyc}})$, where $\Lambda_{\text{cyc}}\subset J^1\,S^1$ is the cyclic link of unknots $n+1$ as pictured in the front diagram shown in Figure~\ref{CEAn} below.
\begin{figure}[H]
\centering
\begin{tikzpicture}
\draw[very thick] (0,-1.1) to (0,1.1);
\begin{knot}[clip width=5, flip crossing=2,flip crossing =4]
\strand[thick] (0,0.7) to[out=0,in=180] (0.7,0.7) to[out=0,in=180]  (1.7,0);
\strand[thick] (0,-0.7) to[out=0,in=180] (0.7,-0.7) to[out=0,in=180]  (1.7,0);

\strand[thick] (0.8,0) to[out=0,in=180] (1.8,0.7) to[out=0,in=180] (2.0,0.7);
\strand[thick] (2.0,0.7) to[out=0,in=180] (2.7,0.7) to[out=0,in=180] (3.7,0);
\strand[thick] (2.0,-0.7) to[out=0,in=180] (2.7,-0.7) to[out=0,in=180] (3.7,0);
\strand[thick] (0.8,0) to[out=0,in=180] (1.8,-0.7) to[out=0,in=180] (2.0,-0.7);

\strand[thick] (2.8,0) to[out=0,in=180] (3.8,0.7) to[out=0,in=180] (4,0.7);
\strand[thick] (2.8,0) to[out=0,in=180] (3.8,-0.7) to[out=0,in=180] (4,-0.7);
\end{knot}
\draw[fill=black] (4.4,0) circle (0.01cm);
\draw[fill=black] (4.5,0) circle (0.01cm);
\draw[fill=black] (4.6,0) circle (0.01cm);

\begin{knot}[clip width=5, flip crossing =2]
\strand[thick] (5.0,0.7) to[out=0,in=180] (5.3,0.7) to[out=0,in=180] (6.2,0);
\strand[thick] (5.0,-0.7) to[out=0,in=180] (5.3,-0.7) to[out=0,in=180] (6.2,0);
\strand[thick] (5.3,0) to[out=0,in=180] (6.3,0.7) to[out=0,in=180] (7,0.7);
\strand[thick] (5.3,0) to[out=0,in=180] (6.3,-0.7) to[out=0,in=180] (7,-0.7);
\end{knot}

\draw[very thick] (7,-1.1) to (7,1.1);
\end{tikzpicture}
\caption{$\Lambda_{\text{cyc}}=\Lambda_1\sqcup\cdots\sqcup\Lambda_{n+1}$ in $J^1\,S^1$, shown in the front projection.}
\label{CEAn}
\end{figure}
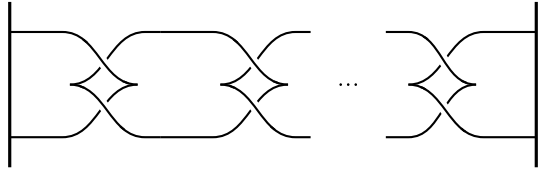
We have $n+1$ components $\Lambda_1,\dots,\Lambda_{n+1}$. We can apply the Ng-resolution (Figure~\ref{ngone} and Figure~\ref{onejetresol}) together with area-preserving Reidemeister moves of type III, to deduce that in a suitable representation of the Legendrian isotopy class the Reeb chords are in one-to-one correspondence with the right cusps and the intersection points (the Reeb chords corresponds to intersection points in Figure~\ref{lagAN}). Thus we have Reeb chords $\alpha_i:\Lambda_i\rightarrow\Lambda_{i+1}$, $\alpha_i^*:\Lambda_{i+1}\rightarrow\Lambda_i$ and $\zeta_i:\Lambda_i\rightarrow \Lambda_i$, for $i=1,\dots,n+1$. (This can also be seen easily by comparing the slopes in the front projection directly.) Thus the underlying path-algebra is the underlying algebra of the Ginzburg algebra with cyclically ordered $\tilde{A}_{n}$-quiver.

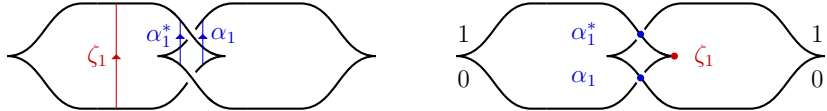
\begin{figure}[H]
\centering
\begin{tikzpicture}
\begin{knot}[clip width= 6,flip crossing=2]
\strand[thick] (0.8,0) to[out=0,in=180] (1.8,0.7) to[out=0,in=180] (2.0,0.7);
\strand[thick] (2.0,0.7) to[out=0,in=180] (2.7,0.7) to[out=0,in=180] (3.7,0);
\strand[thick] (2.0,-0.7) to[out=0,in=180] (2.7,-0.7) to[out=0,in=180] (3.7,0);
\strand[thick] (0.8,0) to[out=0,in=180] (1.8,-0.7) to[out=0,in=180] (2.0,-0.7);

\strand[thick] (2.8,0) to[out=0,in=180] (3.8,0.7) to[out=0,in=180] (4,0.7);
\strand[thick] (2.8,0) to[out=0,in=180] (3.8,-0.7) to[out=0,in=180] (4,-0.7);
\strand[thick] (4,0.7) to[out=0,in=180] (4.7,0.7) to[out=0,in=180] (5.7,0);
\strand[thick] (4,-0.7) to[out=0,in=180] (4.7,-0.7) to[out=0,in=180] (5.7,0);
\end{knot}

\draw[red!80!black] (2.25,-0.7) to (2.25,0.7);
\draw[blue!80!black] (3.4,-0.12) to (3.4,0.47);
\draw[blue!80!black] (3.1,-0.12) to (3.1,0.47);

\node[scale=0.4,rotate=90,red!80!black] at (2.25,0) {\midarrow};
\node[scale=0.4,rotate=90,blue!80!black] at (3.4,0.27) {\midarrow};
\node[scale=0.4,rotate=90,blue!80!black] at (3.1,0.27) {\midarrow};

\node[red!80!black, scale=0.8] at (2,0) {$\zeta_1$};
\node[blue!80!black, scale=0.8] at (2.82,0.275) {$\alpha_1^*$};
\node[blue!80!black, scale=0.8] at (3.68,0.27) {$\alpha_1$};

\end{tikzpicture}
\qquad
\begin{tikzpicture}
\begin{knot}[clip width= 6,flip crossing=2]
\strand[thick] (0.8,0) to[out=0,in=180] (1.8,0.7) to[out=0,in=180] (2.0,0.7);
\strand[thick] (2.0,0.7) to[out=0,in=180] (2.7,0.7) to[out=0,in=180] (3.7,0);
\strand[thick] (2.0,-0.7) to[out=0,in=180] (2.7,-0.7) to[out=0,in=180] (3.7,0);
\strand[thick] (0.8,0) to[out=0,in=180] (1.8,-0.7) to[out=0,in=180] (2.0,-0.7);

\strand[thick] (2.8,0) to[out=0,in=180] (3.8,0.7) to[out=0,in=180] (4,0.7);
\strand[thick] (2.8,0) to[out=0,in=180] (3.8,-0.7) to[out=0,in=180] (4,-0.7);
\strand[thick] (4,0.7) to[out=0,in=180] (4.7,0.7) to[out=0,in=180] (5.7,0);
\strand[thick] (4,-0.7) to[out=0,in=180] (4.7,-0.7) to[out=0,in=180] (5.7,0);
\end{knot}

\draw[red!80!black,fill=red!80!black] (3.7,0) circle (0.04cm);
\draw[blue!80!black,fill=blue!80!black] (3.25,-0.29) circle (0.04cm);
\draw[blue!80!black,fill=blue!80!black] (3.25,0.29) circle (0.04cm); 

\node[scale=0.8, blue!80!black] at (2.5,-0.3) {$\alpha_1$};
\node[scale=0.8, blue!80!black] at (2.5,0.3) {$\alpha_1^*$};
\node[scale=0.8, red!80!black] at (4.1,0) {$\zeta_1$};

\node[scale=0.8] at (5.6,0.3) {$1$};
\node[scale=0.8] at (5.6,-0.3) {$0$};

\node[scale=0.8] at (0.9,0.3) {$1$};
\node[scale=0.8] at (0.9,-0.3) {$0$};
\end{tikzpicture}
\caption{The Reeb chords $\zeta_1$, $\alpha_1$ and $\alpha_1^*$ shown in the front projection of $\Lambda_1\sqcup\Lambda_2$. The corresponding crossings and the Maslov potential is shown on the right-hand side.}
\label{grReeb}
\end{figure}
Next we calculate the grading. First note that $\text{rot}(\Lambda_i)=0$, and hence we will get a $\mathbb{Z}$-grading. We choose the Maslov potential as in Figure~\ref{grReeb}. With this choice we get:
\begin{align*}
|\zeta_i|&=1 \\
|\alpha_i|&=0-0=0 \\
|\alpha_i^*|&=1-1=0
\end{align*}

For the differential we note that there are no discs with positive puncture at either $\alpha_i$ or $\alpha_i^*$. There are four discs contributing to each differential at $\zeta_i$ (see Figure~\ref{lagAN} for one example), two of which contribute $e_i-e_i=0$, yielding
\[
\partial\zeta_i=\alpha_{i}^*\alpha_i-\alpha_{i-1}\alpha_{i-1}^*
\]
(counting modulo $n+1$).
\begin{figure}[H]
\centering
\begin{tikzpicture}
\draw[very thick] (0,-1.1) to (0,1.1);
\begin{knot}[clip width=6, flip crossing=3]
\strand[thick] (0,0.7) to[out=0,in=180] (0.5,0.7) to[out=0,in=180] (2.1,-0.7) to[out=0,in=-90] (2.8,0);
\strand[thick] (0,-0.7) to[out=0,in=180] (0.5,-0.7) to[out=0,in=180]  (2.1,0.7) to[out=0,in=90] (2.8,0);

\strand[thick] (2.4,0) to[out=90,in=180] (3.1,0.7) to[out=0,in=180] (4.5,-0.7);
\strand[thick] (2.4,0) to[out=-90,in=180] (3.1,-0.7) to[out=0,in=180] (4.5,0.7);
\end{knot}
\draw[fill=black] (4.7,0) circle (0.01cm);
\draw[fill=black] (4.8,0) circle (0.01cm);
\draw[fill=black] (4.9,0) circle (0.01cm);
\begin{knot}[clip width=6, flip crossing=1]
\strand[thick] (5.1,0.7) to[out=0,in=90] (5.8,0) to[out=-90,in=0] (5.1,-0.7);
\strand[thick] (6.3,0.7) to (6.1,0.7) to[out=180,in=90] (5.4,0) to[out=-90,in=180] (6.1,-0.7) to (6.3,-0.7);
\draw[very thick] (6.3,-1.1) to (6.3,1.1);
\end{knot}
\draw[draw=none,fill, gray, opacity=0.5] (3.8,0) to[out=110,in=0] (3.1,0.7) to[out=180,in=50] (2.63,0.5) to[out=-65,in=90] (2.8,0) to[out=-90,in=65] (2.63,-0.5) to[out=-50,in=180] (3.1,-0.7) to[out=0,in=-110] (3.8,0);
\node[scale=0.8] at (2.6,0.9) {$\alpha_1$};
\node[scale=0.8] at (2.6,-0.9) {$\alpha_1^*$};

\node[scale=0.8] at (5.6,0.9) {$\alpha_{n+1}$};
\node[scale=0.8] at (5.6,-0.9) {$\alpha_{n+1}^*$};

\node[scale=0.8] at (1.3,-0.5) {$\zeta_1$};
\node[scale=0.8] at (3.8,-0.5) {$\zeta_2$};

\draw[fill] (2.4,0) circle (0.04cm);
\draw[fill] (5.4,0) circle (0.04cm);

\node[scale=0.8] at (2.15,0) {$*_2$};
\node[scale=0.8] at (5.18,0.04) {$*_1$};

\draw[red!80!black] (3.45,0) to (3.65,0);
\draw[red!80!black] (3.55,0.1) to (3.55,-0.1);
\draw[green!80!black] (2.8,-0.5) to (3,-0.5);
\draw[green!80!black] (2.8,0.5) to (3,0.5);
\end{tikzpicture}
\caption{The Legendrian $\Lambda_{\text{cyc}}$ in the Lagrangian projection, obtained from the front projection by Ng-resolution and (area preserving) Reidemeister III moves. The shaded part shows the contribution of $-\alpha_1\alpha_1^*$ to $\partial \zeta_2$.}
\label{lagAN}
\end{figure}
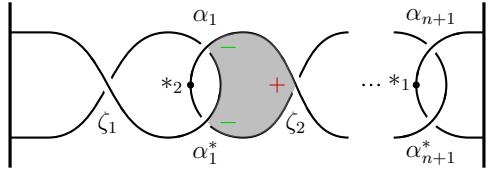
\begin{proposition}\label{ginzcalc}
The presentation of the Chekanov--Eliashberg dg-algebra of the cyclic link of $n$ standard unknots $CE(\Lambda_{\textup{cyc}})$ with choices made above, is equal as dg-algebras to the Ginzburg algebra $\mathcal{G}(\tilde{A}_n)$ of the cyclically oriented $\tilde{A}_{n}$-quiver.
\end{proposition}

\subsubsection{The $\tilde{A}_n$-attaching link in $S^1\times S^2$}

We have seen the attaching link of the Weinstein handles in Figure~\ref{plumbingpic}. Its Chekanov--Eliashberg algebra is shown in \cite{MR4033516} to be equivalent with the derived multiplicative preprojective algebra (this is true for plumbing according to any graph), which in turn is shown to be formal and quasi-isomorphic to the (non-deformed) multiplicative preprojective algebra in \cite{MR4457411}. However this, does not include the stop, whose attaching sphere $\Lambda_{\text{stop}}$ is included in red in Figure~\ref{attachfull} below.
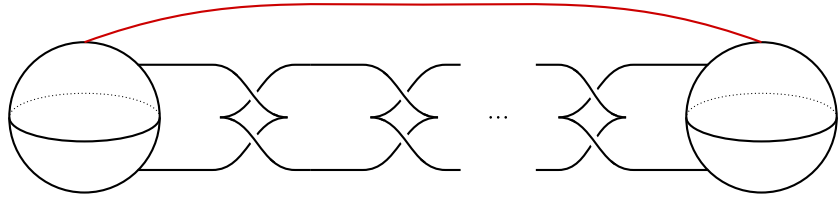
\begin{figure}[H]
\centering
\begin{tikzpicture}
\draw[thick] (-1,0) circle (1cm);
\draw[densely dotted] (0,0) arc (0:180:1 and 0.32);
\draw[thick] (-2,0) arc (0:180:-1 and -0.32);

\draw[thick] (8,0) circle (1cm);
\draw[densely dotted] (9,0) arc (0:180:1 and 0.32);
\draw[thick] (7,0) arc (0:180:-1 and -0.32);

\begin{knot}[clip width=5, flip crossing=2,flip crossing =4]
\strand[thick] (-0.293,0.7) to[out=0,in=180] (0.7,0.7) to[out=0,in=180]  (1.7,0);
\strand[thick] (-0.293,-0.7) to[out=0,in=180] (0.7,-0.7) to[out=0,in=180]  (1.7,0);

\strand[thick] (0.8,0) to[out=0,in=180] (1.8,0.7) to[out=0,in=180] (2.0,0.7);
\strand[thick] (2.0,0.7) to[out=0,in=180] (2.7,0.7) to[out=0,in=180] (3.7,0);
\strand[thick] (2.0,-0.7) to[out=0,in=180] (2.7,-0.7) to[out=0,in=180] (3.7,0);
\strand[thick] (0.8,0) to[out=0,in=180] (1.8,-0.7) to[out=0,in=180] (2.0,-0.7);

\strand[thick] (2.8,0) to[out=0,in=180] (3.8,0.7) to[out=0,in=180] (4,0.7);
\strand[thick] (2.8,0) to[out=0,in=180] (3.8,-0.7) to[out=0,in=180] (4,-0.7);
\end{knot}
\draw[fill=black] (4.4,0) circle (0.01cm);
\draw[fill=black] (4.5,0) circle (0.01cm);
\draw[fill=black] (4.6,0) circle (0.01cm);

\begin{knot}[clip width=5, flip crossing =2]
\strand[thick] (5.0,0.7) to[out=0,in=180] (5.3,0.7) to[out=0,in=180] (6.2,0);
\strand[thick] (5.0,-0.7) to[out=0,in=180] (5.3,-0.7) to[out=0,in=180] (6.2,0);
\strand[thick] (5.3,0) to[out=0,in=180] (6.3,0.7) to[out=0,in=180] (7.293,0.7);
\strand[thick] (5.3,0) to[out=0,in=180] (6.3,-0.7) to[out=0,in=180] (7.293,-0.7);
\end{knot}
\draw[red!80!black,thick] (-1,1) to[out=20,in=180] (3.5,1.5) to[out=0,in=160] (8,1);
\end{tikzpicture}
\caption{The attaching link from Figure~\ref{plumbingpic} with the stop (in red) added.}
\label{attachfull}
\end{figure}
Since we want to put a stop with attaching sphere the red component $\Lambda_{\text{stop}}$, and standard Weinstein handles at the other components, we will add generators $t_0^{\pm 1}$, modelling loop space coefficients, only on this component (and we use $e_0$ to denote the corresponding idempotent). 

From the Ng-resolution together with Reidemeister III moves to the link above (again area preserving), we obtain the following Lagrangian projection of a Legendrian link in $S^1\times S^2$, which we call $\Lambda_{\text{tot}}= \Lambda_1 \sqcup \ldots \sqcup \Lambda_{n+1}\sqcup\Lambda_{\text{stop}}$.
\begin{figure}[H]
\centering
\begin{tikzpicture}
\draw[thick] (0,0) circle (0.5cm);
\draw[densely dotted] (0.5,0) arc (0:180:0.5 and 0.16);
\draw[thick] (-0.5,0) arc (0:180:-0.5 and -0.16);

\draw[thick] (10,0) circle (0.5cm);
\draw[densely dotted] (10.5,0) arc (0:180:0.5 and 0.16);
\draw[thick] (9.5,0) arc (0:180:-0.5 and -0.16);

\begin{knot}[clip width=4.5, flip crossing=3, flip crossing=6, flip crossing=10,flip crossing=12,flip crossing=13]
\strand[thick]
(0.25,0.433012) to[out=0,in=180] (1,0.433012) to[out=0,in=180] (2,-0.433012) to[out=0,in=-90] (2.4,0);
\strand[thick]
(2.4,0) to[out=90,in=0] (2,0.433012) to[out=180,in=0] (1,-0.433012) to[out=180,in=0] (0.25,-0.433012);

\strand[thick]
(1.95,0) to[out=90,in=180] (2.35,0.433012) to[out=0,in=180] (3.35,-0.433012) to[out=0,in=-90] (3.75,0);
\strand[thick]
(3.75,0) to[out=90,in=0] (3.35,0.433012) to[out=180,in=0] (2.35,-0.433012) to[out=180,in=-90] (1.95,0);

\strand[thick]
(3.3,0) to[out=90,in=180] (3.7,0.433012) to[out=0,in=180] (4.7,-0.433012);
\strand[thick]
(4.7,0.433012) to[out=180,in=0] (3.7,-0.433012) to[out=180,in=-90] (3.3,0);

\strand[thick] 
(6,0.433012) to[out=0,in=180] (7,-0.433012) to[out=0,in=-90] (7.4,0);
\strand[thick]
(7.4,0) to[out=90,in=0] (7,0.433012) to[out=180,in=0] (6,-0.433012);

\strand[thick]
(6.95,0) to[out=90,in=180] (7.4,0.433012) to[out=0,in=180] (9.1,-0.433012) to (9.75,-0.433012);
\strand[thick] 
(9.75,0.433012) to[out=180,in=0] (9.1,0.433012) to[out=180,in=0] (7.4,-0.433012) to[out=180,in=-90] (6.95,0);

\strand[thick, red!80!black]
(0,0.5) to[out=20,in=180] (5,0.8) to[out=0,in=180] (7,0.8) to[out=0,in=120] (8.6,0.5) to[out=-60,in=180] (9.2,-0.7) to[out=0,in=-160] (10,-0.5);
\end{knot}

\draw[fill=black] (5.29,0) circle (0.01cm);
\draw[fill=black] (5.39,0) circle (0.01cm);
\draw[fill=black] (5.49,0) circle (0.01cm);

\draw[fill,red!80!black] (5,0.8) circle (0.025cm);
\node[scale=0.5, red!80!black] at (5,0.95) {$t_0$};
\node[scale=0.5] at (1.5,0.3) {$\zeta_1$};
\node[scale=0.5] at (2.85,0.3) {$\zeta_2$};
\node[scale=0.5] at (4.2,0.3) {$\zeta_3$};
\node[scale=0.5] at (6.5,0.42) {$\zeta_{n+1}$};
\node[scale=0.5] at (8.25,0.3) {$\zeta_1'$};

\node[scale=0.5] at (2.175,0.55) {$\alpha_1$};
\node[scale=0.5] at (2.175,-0.6) {$\alpha_1^*$};

\node[scale=0.5] at (3.525,0.55) {$\alpha_2$};
\node[scale=0.5] at (3.525,-0.6) {$\alpha_2^*$};

\node[scale=0.5] at (7.175,0.55) {$\alpha_{n+1}$};
\node[scale=0.5] at (7.175,-0.6) {$\alpha_{n+1}^*$};

\node[scale=0.5] at (8.7,0.55) {$\gamma_1$};
\node[scale=0.5] at (8.77,-0.57) {$\gamma_2$};
\end{tikzpicture}
\caption{The Lagrangian projection of the link $\Lambda_{\text{tot}}=\Lambda_1 \sqcup \ldots \sqcup \Lambda_{n+1}\sqcup\Lambda_{\text{stop}}$. The component drawn in red represents a stop rather then an attaching sphere.}
\end{figure}
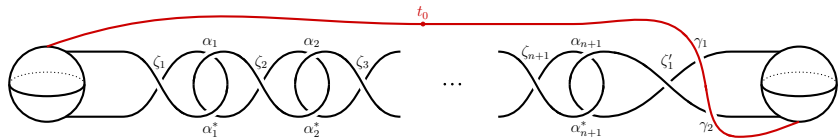

We thus get a DGA generated by the following Reeb chords: $\alpha_1,\alpha_1^*,\dots,\alpha_n,\alpha_n^*$, as before; now with one of the self-intersection generators split into two $\zeta_1,\zeta_1',\zeta_2,\dots,\zeta_n$; two new generators $\gamma_0$, $\gamma_1$, where the stop intersects $\Lambda_1$ in the projection; generators $t_0^{\pm 1}$, with relations $t_0t_0^{-1}=e_0=t_0^{-1}t_0$, corresponding to loop space coefficients on the stop; and finally the handle generators $c_{i,j}^p$, for $p\geq 0$ an integer, $1\leq i<j\leq 3$ if $p=0$, and when $p\geq 1$, $1\leq i,j\leq 3$. 

We have the following grading on the DGA (after choosing Maslov potential as in Theorem~\ref{grReeb}, together with $m(\text{stop/red component})=0$:
\begin{align*}
&|\alpha_i|=|\alpha_i^*|=0, &&|\zeta_1|=|\zeta_1'|=|\zeta_i|=1,\\
&|\gamma_1|=|\gamma_2|=0,&&|t_0^{\pm 1}|=0, \\
&|c^p_{1,2}|=2p,&&|c^p_{2,1}|=2p-2, \\
&|c^p_{1,3}|=2p,&&|c^p_{3,1}|=2p-2,\\
&|c^p_{2,3}|=2p-1,&&|c^p_{3,2}|=2p-1, \\
&|c^p_{1,1}|=|c^p_{2,2}|=|c^p_{3,3}|=2p-1.
\end{align*}
The differential can be computed as described above, but we won't need it here, except at the handle generators.

\begin{proposition}\label{quotone}
The Drinfeld quotient of $CE(\Lambda_{\textup{tot}})$ at the idempotents $e_1,\dots,e_{n+1}$, where $e_i$ is the idempotent corresponding to the component $\Lambda_i$, is trivial in homology.
\end{proposition}
\begin{proof}
The construction amounts to adding generators $\kappa_i$ of degree $1$ such that $\partial\kappa_i=e_i$, for $i=1,\dots,n+1$ (corresponding to loops in the underlying quiver of the dg-algebra, i.e. with $e_i\kappa_i=\kappa_i=\kappa_ie_i$ and $e_j\kappa_i=0=\kappa_ie_j$, if $j\not=i$). What remains is to show that $e_0$, the idempotent that corresponds to $\Lambda_{\text{stop}}$, is a boundary in this new dg-algebra (since then $1=e_0+e_1+\cdots+e_{n+1}$ is a boundary).

We have that
\[
\partial c^1_{33}=e_0-c^0_{13}c^1_{31}+c^0_{23}c^1_{32}.
\]
We let
\begin{align*}
a&=c^0_{1,3}\kappa_1c^1_{3,1}+c^0_{2,3}\kappa_1c^1_{3,2}, \\
b&=c^0_{2,3}\kappa_1c^0_{1,2}\kappa_1c^1_{3,1},
\end{align*}
and we note that
\begin{align*}
\partial a&=c^0_{2,3}c^0_{1,2}\kappa_1c^1_{3,1}+c^0_{1,3}c^1_{3,1}-c^0_{2,3}c^1_{3,2}-c^0_{2,3}\kappa_1 c_{1,2}^0c_{3,1}^1\\
\partial b&=-c^0_{2,3}c^0_{1,2}\kappa_1c^1_{3,1}+c^0_{2,3}\kappa_1c^0_{1,2}c^1_{3,1}.
\end{align*}
Hence
\[
\partial(c^1_{3,3}+a+b)=e_0.
\]
\end{proof}
\begin{corollary}\label{splitgen}
The perfect derived category of finitely generated (right) dg-modules over $CE(\Lambda_{\textup{tot}})$, is split-generated by the modules $e_iCE(\Lambda_{\textup{tot}})$, for $i=1,\dots,n+1$.
\end{corollary}
\begin{proof}
Apply Proposition~\ref{quotone} together with Proposition~\ref{quotgen}.
\end{proof}
In fact, the same strategy works more generally:
\begin{theorem}\label{alggen}
Let $\Lambda\subset S^1\times S^2$ be a link and let $\Lambda_{\textup{stop}}$ be a stop as in Figure~\ref{gengen}. Then the perfect derived category of finitely generated $CE(\Lambda\cup\Lambda_{\textup{stop}})$-modules is split-generated by the modules $e_iCE(\Lambda\cup\Lambda_{\textup{stop}})$, $i=1,\dots,m$, where $e_i$ are the idempotents corresponding to the connected components of $\Lambda$.
\end{theorem}
\begin{proof}
By Proposition~\ref{dgquotalg} it follows that the Drinfeld quotient of $CE(\Lambda\cup\Lambda_{\text{stop}})$ at the idempotents $e_1,\dots,e_m$ is quasi-isomorphic to $B=CE(\Lambda_{\text{stop}})$. Also, $H_{\bullet}(B)=0$, using $\partial_Bc^1_{1,1}=1$. Hence the theorem follows from Proposition~\ref{quotgen}.

Alternatively can can prove it directly using a spectral sequence as follows:

Let $\mathcal{A}=CE(\Lambda\cup\Lambda_{\text{stop}})[\kappa_1,\dots,\kappa_m]$, where each $\kappa_i$ is a generator of degree $1$ such that $\partial\kappa_i=e_i$ (and $e_i\kappa_i=e_i=\kappa_ie_i$ and $e_j\kappa_i=0=\kappa_ie_j$, if $j\not=i$); we want to show that $H_{\bullet}(\mathcal{A})=0$. It suffices to show that the homology of the sub-dg-algebra $\mathcal{B}$ generated by the handle chords and the $\kappa_i$ vanishes, because if $1$ is a boundary in $\mathcal{B}$ it needs to be a boundary also in $\mathcal{A}$; it even suffices to show that $H_0(\mathcal{B})=0$.

Let $\ell(c_{i,j}^p)=p(n+1)+(j-i)$ and $\ell(\kappa_i)=0$, which we think of as the length of the chords. Then $F^p\mathcal{B}=\langle\alpha\,|\,\ell(\alpha)\leq p\rangle$, the subalgebra of $\mathcal{B}$ generated by chords of length $p$ or less, defines a filtration $F^0\mathcal{B}\subset F^1\mathcal{B}\subset\cdots\subset F^p\mathcal{B}\subset\cdots$ on $\mathcal{B}$, which satisfies $\partial(F^p\mathcal{B})\subset F^p\mathcal{B}$.

This filtration defines a spectral sequence with $E^0_{d,p}=F^p\mathcal{B}^d/\mathcal{F}^{p-1}\mathcal{B}^{d-1}$. Then $E^1_{d,p}=0$ for $p$ not a multiple of $n+1$, using $\partial\kappa_i=e_i$, for $i=1,\dots,m$, to kill cycles either from the left or from the right (or both). For $r\geq n+1$, we have $E^r_{d,p}=0$, using 
\[
\partial c^1_{n+1,n+1}=e_0+\sum_{l=1}^{n}(-1)^{m(l)+m(n+1)}c^0_{l,n+1}c^1_{n+1,l},
\]
and hence that the induced differential in the spectral sequence kills $e_0$. Thus the spectral sequence converges to $E^{\infty}_{d,p}=0$, which implies that $H_{\bullet}(\mathcal{B})=0$.
\end{proof}
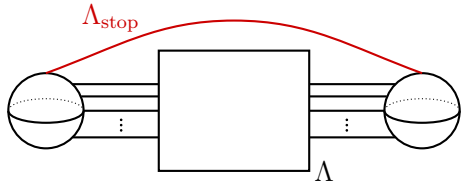
\begin{figure}[H]
\centering
\begin{tikzpicture}
\draw[thick] (0,0) circle (0.5cm);
\draw[densely dotted] (0.5,0) arc (0:180:0.5 and 0.16);
\draw[thick] (-0.5,0) arc (0:180:-0.5 and -0.16);

\draw[thick] (5,0) circle (0.5cm);
\draw[densely dotted] (5.5,0) arc (0:180:0.5 and 0.16);
\draw[thick] (4.5,0) arc (0:180:-0.5 and -0.16);

\draw[thick] (1.5,-0.8) rectangle (3.5,0.8);
\draw[red!80!black, thick] (0,0.5) to[out=20,in=180] (2.5,1.2) to[out=0,in=160] (5,0.5);
\draw[thick] (0.354,0.354) to (1.5,0.354);
\draw[thick] (0.469,0.191) to (1.5,0.191);
\draw[thick] (0.5,0) to (1.5,0);
\draw[fill=black] (1,-0.261) circle (0.01cm);
\draw[fill=black] (1,-0.191) circle (0.01cm);
\draw[fill=black] (1,-0.121) circle (0.01cm);
\draw[thick] (0.354,-0.354) to (1.5,-0.354);

\draw[thick] (4.646,0.354) to (3.5,0.354);
\draw[thick] (4.531,0.191) to (3.5,0.191);
\draw[thick] (4.5,0) to (3.5,0);
\draw[fill=black] (4,-0.261) circle (0.01cm);
\draw[fill=black] (4,-0.191) circle (0.01cm);
\draw[fill=black] (4,-0.121) circle (0.01cm);
\draw[thick] (4.646,-0.354) to (3.5,-0.354);
\node[scale=0.9] at (3.7,-0.8) {$\Lambda$};
\node[red!80!black,scale=0.9] at (0.86,1.23) {$\Lambda_{\text{stop}}$};
\end{tikzpicture}
\caption{The Legendrian link $\Lambda$ with $n$ strands going through the handle consisting of $m$ connected components, and Legendrian $\Lambda_{\text{stop}}$ at which we put a stop in red, shown in the front projection.}
\label{gengen}
\end{figure}

\subsection{Realisation of $\Pi(\tilde{A}_n)$ as a partially wrapped Fukaya category}\label{realisationsect}
We use the previous subsection, together with the surgery formula to argue that the cyclic plumbings of $T^*S^2$ with a stop attached, and hence the Milnor fibre with a fibre removed and a stop attached, has partially wrapped Fukaya category derived equivalent with finitely generated modules over the preprojective algebra over the type $\tilde{A}_n$ quiver. The argument is as follows.

\begin{theorem}\label{surgerycon}
Let $\Lambda\subset J^1S^1$ and let $\Lambda_{\textup{stop}}$ be as in Figure~\ref{gengen}, i.e. as in Theorem~\ref{stopsect}. Let $\iota:J^1S^1\rightarrow S^1\times S^2$ denote the inclusion. Then $\iota$ induces a quasi-isomorphism $CE(\Lambda)\xrightarrow{\textup{q.i.}}eCE(\iota(\Lambda) \cup\Lambda_{\textup{stop}})e$,  where $e$ is the sum of idempotents in $CE(\iota(\Lambda)\cup\Lambda_{\textup{stop}})$ corresponding to the connected components of $\iota(\Lambda)$. 
\end{theorem}
\begin{proof}
From Theorem~\ref{surgerystd} and Theorem~\ref{surgeryloop} together with Theorem~\ref{stopsecteq} we deduce: 
\[
CE(\Lambda)\simeq_{\text{q.i.}}CW_{X_{\sigma}}(C,C)\simeq_{\text{q.i.}}CW_{(X,\sigma_{\Lambda_{\text{stop}}})}(C,C)\simeq_{\text{q.i.}}eCE(\iota(\Lambda)\cup\Lambda_{\text{stop}})e,
\]
where the quasi-isomorphisms are as $A_{\infty}$-algebras. Here $X_{\sigma}$ denotes the Weinstein sector obtained by attaching Weinstein handles along $\Lambda\subset J^1S^1$ and $X$ the Weinstein manifold obtained by attaching Weinstein handles along $\iota(\Lambda)\subset S^1\times S^2$; $C$ denotes the direct sum of cocores of the Weinstein handles attached along $\Lambda$ and (by abuse of notation) along $\iota(\Lambda)$.
\end{proof}

From this we deduce: 
\begin{theorem}\label{cocoregen}
In the situation of Theorem~\ref{surgerycon}. Let $X$ be the Weinstein manifold obtained by attaching handles on $\iota(\Lambda)$. Then the partially wrapped Fukaya category stopped at $\sigma=\sigma_{{\Lambda}_{\textup{stop}}}$ is split-generated by the cocores. Thus,
\[
D^{\pi}\mathcal{W}(X,\sigma)\simeq D^{\pi}(CE(\Lambda)).
\]
\end{theorem}
\begin{proof}
This follows from Theorem~\ref{generation}, Theorem~\ref{surgeryloop}, Theorem~\ref{alggen} and Theorem~\ref{surgerycon}.
\end{proof}
Finishing the argument, the partially wrapped Fukaya category of cyclic plumbing $\#_{\tilde{A}_{n-1}}T^*S^2$ of cotangent bundles of $T^*S^2$, or equivalently, $\mathcal{W}(X_{\text{Mil}}\smallsetminus D_m,\sigma)$, of the (completed) Milnor fibre with a fibre removed $X_{\text{Mil}}\smallsetminus D_m$, stopped at the Legendrian  described in Section~\ref{stopsection} in the boundary at infinity, is then, by Theorem~\ref{generation}, quasi-equivalent to the $A_{\infty}$-category generated by the cocores $C_1,\dots,C_n$, corresponding to the connected components of the attaching link, and the linking disc $D_0$, corresponding to the stop $\sigma=\sigma_{\Lambda_{\text{stop}}}$. Using Theorem~\ref{surgeryloop} this is quasi-equivalent to the category of finitely generated dg-modules over $CE(\Lambda_{\text{tot}})$. But we saw that this category is split-generated by the Chekanov--Eliashberg algebra computed in the $1$-jet space, thus:
\[
D^{\pi}\mathcal{W}(X_{\text{Mil}}\smallsetminus D_m,\sigma)\simeq D^{\pi}(CE(\Lambda_{\text{cyc}})).
\]

Therefore:
\begin{theorem}\label{mainresult}
There is an equivalence of triangulated categories
\[
D^{\pi}\mathcal{W}(X_{\textup{Mil}}\smallsetminus D_m,\sigma_{\Lambda_{\textup{stop}}})\simeq D^{b}(\Pi(\tilde{A}_n)).
\]
\end{theorem}
\begin{proof}
By Theorem~\ref{cocoregen}, together with Theorem~\ref{ginzcalc} and Theorem~\ref{formal}, we have
\[
D^{\pi}\mathcal{W}(X_{\text{Mil}}\smallsetminus D_m,\sigma_{\Lambda_{\text{stop}}})\simeq D^{\pi}(CE(\Lambda_{\text{cyc}}))\simeq  D^{\pi}(\mathcal{G}(\tilde{A}_n))\simeq D^{\pi}(\Pi(\tilde{A}_n)).
\]
Furthermore, 
\[
 D^{\pi}(\Pi(\tilde{A}_{n}))\simeq D^b(\Pi(\tilde{A}_{n}))
\]
is well known and follows from \cite[Theorem 3.1.3]{MR1996800}. We conclude the statement.
\end{proof}

It is also possible to show that the stop is an iterated cone of the short handle chords starting at the strands below the stop and ending on it. From this one can conclude generation rather then split-generation.

\subsection{The Chekanov--Eliashberg algebra for the $\tilde{A}_n$-link in higher dimensions}
The Ginzburg dga of a quiver that was considered in Definition~\ref{defginzb} has a grading that makes it $2$-Calabi--Yau. There are other natural grading conventions for the Ginzburg algebra which makes it $m$-Calabi Yau, for $m\geq 3$ \cite[Section 6]{MR2795754} (here we do not consider the case with potential). Here we show how to realise these dg-algebras as Chekanov-Eliashberg algebras of higher dimensional versions of the links that we constructed above, and, consequently, their module categories as wrapped Fukaya categories. However, there is no obvious corresponding homological mirror symmetry statement, and in particular the bounded derived category of coherent sheaves of a crepant resolution of the $A_n$ singularity in dimension $3$ is not the derived module category of the preprojective algebra of an $\tilde{A}_{n}$-quiver. Moreover, in higher dimensions it is easy to realise the preprojective algebras of the graphs $\tilde{D}_n$ or $\tilde{E}_6$, $\tilde{E}_7$ and $\tilde{E}_8$ as Legendrian links (recall that we do not know any such 1-dimensional Legendrian representiative) since the discs which contribute higher order terms to the differential for $1$-dimensional Legendrians (see \cite{MR3692968}) become non-rigid in bhigher dimensions.

We first note that the Milnor fibre with a regular conic removed
\[
\{x_1^2+\cdots+x_{m+1}^2+x_{m+2}^{n+1}=1\}\smallsetminus\{x_{m+2}=0\}\xrightarrow{x_{m+2}}\mathbb{C}^*
\]
is obtained by handle attachment as in the construction in Section~\ref{geomeq}. That is we first attach a Weinstein $1$-handle, and then critical handles along an $\tilde{A}_{n}$-link going through the $1$-handle. As in Section~\ref{stopsection} we start by attaching a stop, or rather, equivalently we consider the sector obtained by removing a neighbourhood of the isotropic submanifold along which we attach the stop; then we attach the handles along the link in the boundary at infinity of this sector. 

More precisely, the Weinstein structure of the Milnor fibre with a regular conic removed can be constructed by starting with the trivial Lefschetz fibration the projection to the last coordinate
\[
\mathbb{C}^*\times\mathbb{C}^m\rightarrow\mathbb{C}, 
\]
with regular fibre $\mathbb{C}^*\times\mathbb{C}^{m-1}\cong T^*(S^1\times\mathbb{R}^{m-1})$. From it \cite[Construction 2.9]{MR4687585}, we get an open book decomposition with trivial monodromy of
$M=\partial_{\infty}(\mathbb{C}^*
\times\mathbb{C}^m)=(S^1\times S^{2m},\xi_{\text{std}})$. Here the fibration is $\pi:M\smallsetminus B\rightarrow S^1$, where the binding is the contact submanifold $B \cong S^1 \times S^{2m-2}$ and the pages $\Sigma_{\theta}=\overline{\pi^{-1}(\theta)}$ are symplectomorphic to $T^*(S^1\times \mathbb{R}^{m-1})$ (after smoothing and taking completions). 

Analogously to the $4$-dimensional case, i.e. $m=1$, we consider convex hypersurfaces around the isotropic skeleton of the page $\pi^{-1}( -1)$. It should be possible to generalize the $m=1$ case and show that removing a neighbourhood the skeleton is equivalent to removing a page, and thus obtaining a description of $(S^1\times S^{2m},\xi_{\text{std}})$ as two copies of $J^1(S^1\times\mathbb{R}^{m-1})$ glued together. The cutoff Reeb dynamics and the argument for boundedness of holomorphic curves carries over.

We now draw the linking of the unknots a bit differently. We construct $\Lambda\subset J^1(S^1\times\mathbb{R}^{m-1})$ by lifting the front diagram in Figure~\ref{ndim}.
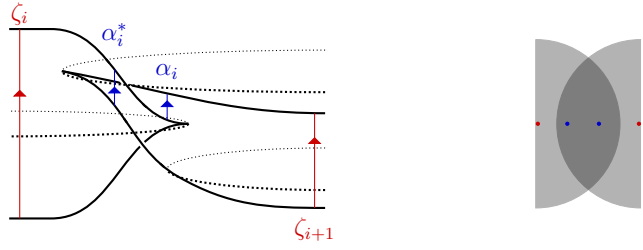
\begin{figure}[H]
\centering
\begin{tikzpicture}[scale=1.4]
\begin{knot}[flip crossing=2]
\strand[thick] (0,0.9) to[out=0,in=180] (0.4,0.9) to[out=0,in=180]  (1.7,0);
\strand[thick] (0,-0.9) to[out=0,in=180] (0.4,-0.9) to[out=0,in=180]  (1.7,0);

\strand[thick] (0.5,0.5) to[out=-10,in=180] (3,0.1);
\strand[thick] (0.5,0.5) to[out=-10,in=150] (1.6,-0.5) to[out=-30,in=180] (3.0,-0.8); 

\end{knot}
\draw[thick, densely dotted] (1.7,0) arc (0:90:1.7 and -0.12);
\draw[densely dotted] (1.7,0) arc (0:90:1.7 and 0.12);

\draw[thick, densely dotted] (3.0,0.3) arc (90:180:2.5 and -0.20);
\draw[densely dotted] (3.0,0.7) arc (90:180:2.5 and 0.20);

\draw[thick, densely dotted] (3.0,-0.63) arc (90:180:1.5 and -0.20);
\draw[densely dotted] (3.0,-0.23) arc (90:180:1.5 and 0.20);

\draw[red!80!black] (0.1,0.9) to (0.1,-0.9);
\draw[red!80!black] (2.9,-0.8) to (2.9,0.1);
\draw[blue!80!black] (1,0.17) to (1,0.515);
\draw[blue!80!black] (1.5,0.04) to (1.5,0.28);
\node[red!80!black,scale=0.8] at (0.1,1.05) {$\zeta_i$};
\node[blue!80!black,scale=0.8] at (1,0.85) {$\alpha_i^*$};
\node[blue!80!black,scale=0.8] at (1.5,0.5) {$\alpha_i$};
\node[red!80!black,scale=0.8] at (2.9,-1) {$\zeta_{i+1}$};

\draw[fill=black, opacity=0.3] (5,0.8) to[out=0,in=90] (5.8,0) to[out=-90,in=0] (5,-0.8) to (5,0.8);

\draw[fill=black,opacity=0.3] (6,0.8) to[out=180,in=90] (5.2,0) to[out=-90,in=180] (6,-0.8);
\draw[red!80!black,fill=red!80!black] (5.02,0) circle (0.015cm);
\draw[red!80!black,fill=red!80!black] (5.98,0) circle (0.015cm);
\draw[blue!80!black,fill=blue!80!black] (5.3,0) circle (0.015cm);
\draw[blue!80!black,fill=blue!80!black] (5.6,0) circle (0.015cm);

\node[red!80!black, scale=0.6,rotate=90] at (0.1,0.3) {\midarrow};
\node[red!80!black, scale=0.6,rotate=90] at (2.9,-0.15) {\midarrow};
\node[blue!80!black, scale=0.6,rotate=90] at (1,0.33) {\midarrow};
\node[blue!80!black, scale=0.6,rotate=90] at (1.5,0.2) {\midarrow};
\end{tikzpicture}
\caption{Shows in the front projection how to link Legendrian unknots in higher dimension. $S^1$-direction is horizontal in the pictures, and the right-hand side picture shows the linking from above.}
\label{ndim}
\end{figure}
The figure should be thought of as that the components are $S^{m-2}$-rotated symmetrically around the $\zeta_i$ Reeb chords.

To get the link one has to glue the left-hand side part of the knot to the right-hand side without creating new Reeb chords. Then one repeats the linking cyclically. 

The grading changes, and can be calculated by 
\[
|c|=m(\text{overstrand of }c)-m(\text{understrand of }c)+\text{index}_p(f_u-f_l)-1,
\]
where $f_u$ is a local height function of the upper strand and $f_l$ is a local height function of the lower strand. The height differences at $\zeta_i,\zeta_{i+1}$ are maxima, and thus the index is $m$. The height difference at $\alpha_i^*$ is a minima in the $S^1$-direction but a maxima in the other directions, and thus the index is $m-1$. The height difference at $\alpha_i$ is maximal in the $S^1$-direction but minimal in the other directions, and thus the index is $1$. The grading is thus:
\begin{align*}
|\zeta|&=1-0+m-1=m \\
|\alpha_i|&=1-1+1-1=0 \\
|\alpha_i^*|&=1-0+(m-1)-1=m-1.
\end{align*}
The differential in $J^1(S^1\times\mathbb{R}^{m-1})$ is obtained by counting rigid gradient flow trees (see \cite{MR2326943}), corresponding to the holomorphic curves. Combining grading (and action when m=2) we see that $\alpha_i$ and $alpha_i^*$ all are cycles. In the figure above one finds two rigid flows trees contributing $\partial\zeta_i=\alpha_i^*\alpha_i$ and $\partial\zeta_{i+1}=-\alpha\alpha_i^*$ to the differential, (the curve yielding cancelling idempotents not rigid in this dimension, i.e. they correspond to non-rigid trees, hence they are not counted by the differential).

Finally, since in higher dimensions (i.e. for $m\geq 2$) the stop is attached along an isotropic non-Legendrian submanifold, generation and surgery (Theorem~\ref{surgerystd} and Theorem~\ref{generation}) immediately implies that modules over the Chekanov--Eliashberg algebra computed in $J^1(S^1\times\mathbb{R}^{m-1})$ is equivalent to the partially wrapped Fukaya category of the Weinstein sector obtained by attaching critical handles along the Legendrian $\tilde{A}_n$-link (and no extra generation argument, see Section~\ref{realisationsect}, is needed).

\section{Homological mirror symmetry results}\label{mirrorsec}
Here we investigate homological mirror symmetry of the hyper-Kähler manifolds underlying the $A_n$-singularity, described in Example~\ref{HKquotient} and Example~\ref{loghkA}. Of primary interest is the homological mirror symmetry of the $S^1$-family of Kähler structures along the equators in Figure~\ref{hkfamilypic}, that is, when $u_I=0$, and the Kähler manifolds at the poles $u_I=1$, to which we now restrict our attention. The former structures correspond to the case when the symplectic manifolds are Weinstein, and the latter when the complex manifolds are given by the crepant resolution of the singularity. We will denote the structures at
the equators $u_I=0$ by $X_{\text{Mil}}$ and the structures at the poles $(1,0,0)$ by $X_{\text{res}}$; their corresponding
divisors will be denoted by $D_m$ and $D_r$, respectively.

We fix $\zeta$, and we begin by considering the log Calabi--Yau hyper-Kähler family, i.e. the family $X_{\zeta}\smallsetminus D_u$ from Example~\ref{loghkA}, obtained by removing anti-canonical normal crossing divisors $D_u'$ containing $D_u$ from a compactification of $X_{\zeta}$. It is anticipated that homological mirror symmetry for this family is obtained by hyper-Kähler rotation. We use a localisation argument on the resolution of the $A_n$-singularity to deduce one direction of homological mirror symmetry in this case, relating both sides to the non-deformed multiplicative preprojective algebra over the $\tilde{A}_n$-quiver. Moreover, we sketch the other direction, where $A_n$-Milnor fibres with a fibre removed are viewed as affine varieties, while the $A_n$-resolution of singularities is regarded as a non-exact symplectic manifold. We relate both these cases to deformed multiplicative preprojective algebras of the $\tilde{A}_n$-quiver. We find an equivalence of these algebras in an infinitesimally deformed sense. (See Figure~\ref{hmslogcyfigure}.)

Moreover, putting the divisor back in, either on the resolution or on the Milnor fibre, is expected to yield two new examples of homological mirror symmetry, where the mirror is the log Calabi--Yau mirror with a stop attached (viewed as an $A$-side) or equipped with a potential (viewed as a $B$-side). This is related to deformations of the additive preprojective algebras and the $A_n$-Ginzburg algebra. (See Figure~\ref{hmsresfigure} and Figure~\ref{hmsmilfigure}.) Again, the deformed preprojective algebras obtained on the $A$-side and the $B$-side, respectively, will be infinitesimally deformed equivalent.

This partially proves the homological mirror symmetry conjectured in Example~\ref{loghkA}, and the related statements without a divisor removed, analogous to the examples in Section~\ref{examplessec}. To be precise, in the non-exact setting our arguments assume certain foundational results on wrapped Fukaya categories which are not established in the literature. These assumptions are made in analogy with the exact setting, and they are used to relate the categories under consideration to (modules over) deformed preprojective algebras.

We have furthermore restricted our attention to the poles (the resolution) and the equator (an $S^1$-family of Milnor fibres) in the family of Kähler structures of the hyper-Kähler manifold (Figure~\ref{hkfamilypic}). However, this is not a severe restriction: the categories computed which involving deformations (see the north-west to south-east diagonal in Figure~\ref{hmslogcyfigure}, Figure~\ref{hmsresfigure} and Figure~\ref{hmsmilfigure}) could be interpreted to provide answers away from the pole/equator, for any choice of Kähler structure in the family, i.e. any choice of $u$ in the upper hemisphere of $S^2$ in Figure~\ref{hkfamilypic}. The important point is that the equivalence here is only shown in (and, indeed, only believed to be true in) a deformation sense, where we vary the areas on the exceptional spheres infinitesimally near $0$ and the complex structure is infinitesimally deformed near the structure at the pole. That is, we interpret the $B$-side along the equator as being obtained by an infinitesimal deformation of the complex structure at the pole, and vice versa for the $A$-side at the pole and the symplectic structure.

\subsection{Homological mirror symmetry in the log Calabi--Yau case}\label{otherdir}
The wrapped Fukaya category of plumbings of $T^*S^2$ has been computed in \cite{MR4033516}. In particular, for the cyclic plumbing, or equivalently for $X_{\text{Mil}}\smallsetminus D_m$,
the wrapped Fukaya category is derived equivalent to finitely generated modules of the non-deformed derived multiplicative preprojective algebra of $Q_{\tilde{A}_{n}}$ (the $\tilde{A}_n$-case can be easily computed via surgery and the Chekanov--Eliashberg algebra in $S^1\times S^2$ of the link in Figure~\ref{plumbingpic}). This dg-algebra is formal and quasi-isomorphic to its homology (concentrated in degree $0$), the non-deformed multiplicative preprojective algebra \cite{MR4582531}. Modules over this algebra is also a model for the derived category of coherent sheaves on the resolution with a divisor removed, which we show by localisation of the preprojective algebra of the $\tilde{A}_n$-quiver, in Proposition~\ref{bsidecycplumb}.

For the other direction we give a short sketch. However, the theory needed to make this a rigours computation of the wrapped Fukaya categories of the non-exact spaces appearing in Example~\ref{loghkA} is missing. We conjecture that these categories are derived equivalent with certain deformed multiplicative preprojective algebras.

We start by proving that for certain parameters the deformed multiplicative preprojective algebra of the $\tilde{A}_n$-quiver is derived Morita equivalent with derived category of coherent sheaves of some $A_n$-Milnor fibre with a fibre removed. 

\subsubsection{Type $A_n$-Milnor fibres and deformed multiplicative preprojective algebras}\label{multpreprojbside}

As a Lefschetz fibration, i.e. on the $A$-side, the Milnor fibre minus a fibre is unique and equivalent to the affine variety 
\[
\{Z^{n+1}-1=XY\}\cap\{Z\not=0\}\subset\mathbb{C}^3,
\] 
and the fibration is given by projecting to $z$-coordinate. However, when considered as algebraic varieties there are many Milnor fibres; they have the form $\{P(Z)=XY\}$, where $P(Z)$ is a degree $n+1$ polynomial without repeated roots. In \cite{MR520463}, the affine structures on the varieties arising in Example~\ref{HKquotient} away from the poles (Figure~\ref{hkfamilypic}) is described in terms of hyper-Kähler structure, where the roots are computed from areas of the exceptional spheres (see Remark~\ref{moreonfamily}). Note also that here $D_m=\{Z=0\}$, whereas in Example~\ref{loghkA} in $D_m=\{Z=-C\}$. The reason for this is a simple change of coordinates explained in Remark~\ref{substrmk}, where we also relate this to the parameters for the additive and multiplicative preprojective algebras (cf. Section~\ref{resstopsect}).

Since $X_{\text{Mil}}\smallsetminus D_m$ is affine, we have $D^{b}(\text{Coh}(X_{\text{Mil}}\smallsetminus D_m))=D^b(\mathbb{C}[X,Y,Z^{\pm1}]/(P(Z)-XY))$. We relate this to the deformed multiplicative preprojective algebra (see Example~\ref{multanex}) in the propositions below.

\begin{proposition}\label{parameters}
Let $B=\Lambda^q(\tilde{A}_n)$ be the deformed multiplicative preprojective algebra with parameters $q_1,\dots,q_{n+1}\in\mathbb{C}^{\times}$ such that $q_1\cdots q_{n+1}=1$. Then 
\[
A=e_{n+1}Be_{n+1}\cong\mathbb{C}[X,Y,Z^{\pm 1}]/(P(Z)-XY),
\]
where $P(Z)$ is of degree $n+1$. Furthermore, the roots of $P(Z)$ are $Q_i=\prod_{j=1}^iq_j$, for $i=1,\dots,n+1$.
\end{proposition}
\begin{proof}
We put $X=\alpha_n\cdots\alpha_1\alpha_{n+1}$, $Y=\alpha_{n+1}^*\alpha_1^*\cdots\alpha_n^*$ and $Z=z$. First of all $X$, $Y$ and $Z$ clearly generates $A$.

Next we show that $A$ is commutative. First of all
\[
ZX=z\alpha_n\cdots\alpha_1\alpha_{n+1}=(e_{n+1}+\alpha_{n+1}^*\alpha_{n+1})\alpha_n\cdots\alpha_1\alpha_{n+1}
\]
\[
=(e_{n+1}+(q_{n+1}^{-1}-e_{n+1}+q_{n+1}^{-1}\alpha_n\alpha_{n}^*)\alpha_n\dots\alpha_1\alpha_{n+1}
\]
\[
=q_{n+1}^{-1}\alpha_n\cdots\alpha_1\alpha_{n+1}+q_{n+1}^{-1}\alpha_n(q_n^{-1}-e_n+q_n^{-1})\alpha_{n-1}\cdots\alpha_1\alpha_{n+1}
\]
\[
=\cdots=q_{n+1}^{-1}\cdots q_2^{-1}\alpha_n\cdots\alpha_1\alpha_{n+1}+q_{n+1}^{-1}\cdots q_2^{-1}(q_1^{-1}-e_1+q_1^{-1}\alpha_{n+1}\alpha_{n+1}^*)\alpha_{n+1}
\]
\[
=X(1+\alpha_{n+1}^*\alpha_{n+1})=XZ.
\]
Simlarly $YZ=ZY$. By repeated application of the relations $e_i+\alpha_{i-1}\alpha_{i-1}^*=q_ie_i+q_i\alpha_i^*\alpha_i$ and $Z=z=\alpha_{n+1}^*\alpha_{n+1}+e_{n+1}$,
\begin{align*}
XY&=P(Z), \\
YX&=Q(Z)
\end{align*}
where $P$ and $Q$ are of degree $n+1$. Since $X$ (or $Y$) commute with $Z$ and hence with $XY$ (and with $YX$) we conclude that 
\[
X(XY-YX)=XYX-XYX=0.
\]
Similarly $Y(XY-YX)=0$. Then, in fact $XA(XY-YX)=0$, and hence $XB(XY-YX)=0$. Using that $B$ is prime, i.e. that $rBs=0$ implies either $r=0$ or $s=0$, \cite[Proposition 6.11]{MR4582531} we conclude that $XY-YX=0$.

Let $A_{i}=\alpha_n\cdots\alpha_i\alpha_i^*\dots\alpha_n^*$, for $i=1,\dots, n$. Then
\[
XY=\alpha_n\cdots\alpha_1\alpha_{n+1}\alpha_{n+1}^*\alpha_1^*\dots\alpha_n^*=\alpha_n\cdots\alpha_1(q_1-e_1+q_1\alpha_1^*\alpha_1)\alpha_1^*\cdots\alpha_n^*
\]
\[
=(q_1-1)A_1+q_1\alpha_n\cdots\alpha_1\alpha_1^*\alpha_1\alpha_1^*\cdots\alpha_n^*=(q_1-1)A_{1}+q_1\alpha_n\cdots\alpha_2(q_2-e_2+q_2\alpha_2^*\alpha_2)\alpha_1\alpha_1^*\cdots\alpha_n^*
\]
\[
=(q_1q_2-1)A_1+q_1q_2\alpha_n\cdots\alpha_3(q_3-e_3+q_3\alpha_3^*\alpha_3)\alpha_2\alpha_1\alpha_1^*\cdots\alpha_n^*
\]
\[
=\dots=(q_1\cdots q_{n+1}-1)A_{1}+q_1\cdots q_{n+1}a_{n+1}^*a_{n+1}A_1=(Z-1)A_1.
\]
In the same way,
\[
A_{i}=(q_i\cdots q_{n+1}-1)A_{i+1}+q_{i}\cdots q_{n+1}(Z-1)A_{i+1}=((q_i\cdots q_{n+1})Z-1)A_{i+1}.
\]
We conclude that
\[
P(Z)=XY=(Q_1^{-1}Z-1)\cdots(Q_{n+1}^{-1}Z-1).
\]
\end{proof}

\begin{proposition}\label{multloc}
In the setting in Theorem~\ref{parameters} above, if furthermore, for all $1\leq i\leq j\leq n$, we have $\prod_{i\leq l\leq j}q_l\not=1$, i.e. the roots of $P(Z)$ are distinct, then $D^b(B\textup{-Mod})\simeq D^b(A\textup{-Mod})$.
\end{proposition}
\begin{proof}
We replace $B$ by the quasi-isomorphic semi-free dg-algebra $\mathcal{G}_M^q(\tilde{A}_{n+1})$ \cite[Theorem 3.7, Proposition 4.4]{MR4582531}. The Drinfeld quotient at $e_{n+1}$ is then quasi-isomorphic to $\mathcal{G}_M^{q'}(A_n)$, by Proposition~\ref{dgquotalg}, where $q'=(q_1,\dots,q_n)$, i.e. the first $n$ entries in $q$. 

Using that $H_0(\mathcal{G}_m^{q'}(A_n))=\Lambda^{q'}(A_n)\not=0$ if and only if there are $i<j$ such that $\prod_{i\leq l\leq j}q_l=1$ \cite[Example 2.3]{MR4582531}, we now apply Proposition~\ref{quotgen} to deduce that $D^{\pi}(A)\simeq D^{\pi}(B)$ (since $\Lambda^{q'}(A_n)=0$, then in particular $1=0$ in the algebra $H_{\bullet}(\mathcal{G}_M^{q'}(A_n))$). Finally $D^b(A)=D^{\pi}(A)$, and hence $D^b(A)\simeq D^b(B)$, because $A$ is a regular affine algebra \cite[Theorem 3.1.3]{MR1996800}.
\end{proof}

\subsubsection{Deformed multiplicative preprojective algebras of type $\tilde{A}_n$ and the non-exact case}\label{multpreprojnonexact}
We will now consider a slightly different setting, using a (log-)Novikov parameter $s$. That is, we consider the deformed multiplicative preprojective algebra with coefficients in the field of formal Laurent series $k=\mathbb{C}((s))$ (that is, power series in $s^{\pm 1}$ with only finitely many terms with negative exponent). Associated to $X_{\zeta}$ are the parameters $\lambda_1,\dots,\lambda_{n+1}$, that are related to the symplectic areas of the exceptional divisors of the resolution (see Example~\ref{loghkA} and Remark~\ref{moreonfamily}), satisfying $\lambda_1+\dots+\lambda_{n+1}=0$ and $\sum_{l=i}^j\lambda_l\not=0$, for all choices $1\leq i\leq j\leq n$. We put 
\[
q_i(s)=e^{s\lambda_i}=\sum_{m\geq 0}\frac{\lambda_i^m}{m!}s^m\in\mathbb{C}((s)),
\]
and we denote the deformed multiplicative preprojective algebra over $\mathbb{C}((s))$ with parameters $q_i(s)$, by $\Lambda^{q(s)}(\tilde{A}_n)$. Note that $\lambda_1+\cdots+\lambda_{n+1}=0$ implies that $q_1(s)\cdots q_{n+1}(s)=1$, and that $\sum_{l=i}^j\lambda_l\not=0$ implies that $\prod_{l=i}^jq_l(s)\not=1$.

Likewise, we denote the deformed derived  multiplicative preprojective algebra defined over $\mathbb{C}((s))$ with parameters $q_i(s)=e^{s\lambda_i}$ by $\mathcal{G}_M^{q(s)}(\tilde{A}_n)$.

\begin{conjecture}\label{cycconj}
The perfect derived non-exact wrapped Fukaya category $D^{\pi}\mathcal{W}(X_{\textup{res}}\smallsetminus D_r)$ is triangulated equivalent to $D^b(\Lambda^{q(s)}(\tilde{A}_n))$, with $q_i(s)=e^{s\lambda_i}$.
\end{conjecture}
\begin{proof}[Sketch of proof]
As a symplectic manifold $X_{\text{res}}$ is $(X_{\zeta},\omega_I)$ from Example~\ref{HKquotient}. After also removing $D_r$, this is topologically the cyclic plumbing of $T^*S^2$, but with respect to $\omega_I$ the $(-2)$-spheres are symplectic (see \cite[Lemma 40]{MR2720232}). Assuming results analogous to Theorem~\ref{generation} and Theorem~\ref{surgerystd} for the non-exact wrapped Fukaya category (i.e. the surgery formula and generation by cocores), then we can conclude that $\mathcal{W}(X_{\text{res}}\smallsetminus D_r)$ is perfect derived equivalent to modules over the endomorphism algebra of the Lagrangian cocores (which survive when adding areas to the spheres), which in turn can be computed by the surgery formula. This computation should then show that $D^{\pi}\mathcal{W}(X_{\text{res}}\smallsetminus D_r)$ is equivalent to the perfect derived category of the deformed derived multiplicative preprojective algebra $\mathcal{G}_M^{q(s)}(\tilde{A}_n)$ analogously to \cite{MR4033516}, or by direct computation via Figure~\ref{plumbingpic} and Remark~\ref{loopspace}. More precisely, this is achieved by computing the Chekanov--Eliashberg algebra with loop space coefficients $t_i^{\pm 1}$, and then specialising $t_i\mapsto q_i(s)=e^{s\lambda_i}\in\mathbb{C}((s))$. Note that the non-exactness forces the introduction of a (log-)Novikov parameter $s$. 

The dg-algebra $\mathcal{G}_M^{q_i(s)}(\tilde{A}_n)$ is formal and quasi-isomorphic to the deformed multiplicative preprojective algebra $\Lambda^{q_i(s)}(\tilde{A}_n)$ \cite{MR4582531} (or by Remark~\ref{formalityremark}), which implies that $D^{\pi}(\mathcal{G}_M^{q(s)}(\tilde{A}_n))$ is equivalent to $D^b(\Lambda^{q(s)}(\tilde{A}_n))$ (here we also use that $D^{\pi}(\Lambda^{q(s)}(\tilde{A}_n))=D^{b}(\Lambda^{q(s)}(\tilde{A}_n))$; see the proof of Proposition~\ref{multloc}).  
\end{proof}
Note that the cyclic plumbing contains an $A_n$-plumbing where the spheres are the exceptional divisors $E_1,\dots,E_n$. The extending vertex corresponds to a sphere that should be homologous to $-[E_1]-\cdots-[E_n]$, which explains the relation $\lambda_1+\dots+\lambda_{n+1}=0$.
\subsubsection{Equivalences as infinitesimal variations of $\Lambda^q(\tilde{A}_n)$ near $q=(1,\dots,1)$}\label{multinfvar}
It is described in Example~\ref{loghkA} how $(X_{\zeta}\smallsetminus D_u,I_u)$, for $u\in O\smallsetminus\{(1,0,0)\}$, i.e., away from the north pole in the upper hemisphere, can be represented as the affine variety
\[
\{XY=P(Z)\}\smallsetminus\{Z=-C\},
\]
for $P(Z)=(Z-\Lambda_1)\cdots(Z-\Lambda_{n+1})$. In particular this holds for $X_{\zeta}\smallsetminus D_u$, along the equator, here denoted by $X_{\text{Mil}}\smallsetminus D_m$. The roots are given by $\Lambda_i=\lambda_1+\cdots+\lambda_{i}$. By the coordinate change $Z\mapsto CZ-C$, $X\mapsto C^{n+1}X$ this can be rewritten as
\[
\{XY=(Z-(C^{-1}\Lambda_1+1))\cdots(Z-(C^{-1}\Lambda_{n+1}+1))\}\smallsetminus \{Z=0\}.
\]
Thus, by Section~\ref{multpreprojbside}, we conclude that
\[
D^b(\text{Coh}(X_{\text{Mil}}\smallsetminus D_m))=D^b(\Lambda^q(\tilde{A_n})),
\]
where the parameters are given by
\[
q_i=\frac{1+C^{-1}\Lambda_i}{1+C^{-1}\Lambda_{i-1}}=\frac{1+C^{-1}(\lambda_1+\dots+\lambda_i)}{1+C^{-1}(\lambda_1+\dots+\lambda_{i-1})},
\]
or equivalently, by the relation $C^{-1}\lambda_i=q_1\cdots q_{i-1}(q_i-1)$. (Cf. Remark~\ref{substrmk}.)

Note that this means that we can also view $C^{-1}$ as a parameter deforming the roots. By the description in Example~\ref{loghkA} and Remark~\ref{moreonfamily}, $C^{-1}$ is equivalently seen as a parameter describing a deformation of the complex structure near the pole in the underlying hyper-Kähler structure. 

If we regard $s=C^{-1}$ as an infinitesimal variation parameter, that is, we consider the deformed multiplicative preprojective algebra above over the ring $\mathbb{C}[[s]]/(s^2)$, then this algebra can be identified with $\Lambda^{q(s)}(\tilde{A}_n)$ obtained in Section~\ref{multpreprojnonexact}, provided that this algebra is also interpreted with coefficients in $k=\mathbb{C}[[s]]/(s^2)$. Indeed, the deformation parameters are then 
\[
q_i(s)=e^{s\lambda_i}=1+s\lambda_i,
\]
and hence 
\[
q_1(s)\cdots q_{i-1}(s)(q_i(s)-1)=(1+s\lambda_1)\cdots(1+s\lambda_{i-1})(s\lambda_i)=s\lambda_i.
\]
In other words, the partially wrapped Fukaya category of the resolution determines the family of affine varieties along the equator as a limit case, if we choose the parameter $C\gg0$ sufficiently large.
\subsubsection{Coherent sheaves on $X_{\text{res}}\smallsetminus D_r$}

Recall that $X_{\text{res}}$ is a smooth variety obtained by repeated blowing up singularities, starting from the singular variety $Z^{n+1}=XY$. Complexes of coherent sheaves on $X_{\text{res}}$ correspond to complexes of $\Pi(\tilde{A}_n)$-modules (after considering derived categories), and the global sections of the structure sheaf are in correspondence with the central elements of this algebra.

We compute $D^b(\text{Coh}(X_{\text{res}}\smallsetminus D_r))$. Since we remove a divisor, this means that the category should be a localisation of $D^b(\Pi(\tilde{A}_n))$. Indeed:
\begin{proposition}\label{bsidecycplumb}
Let $D_r=\{Z+1=0\}\subset X_{\textup{res}}$. There is an equivalence $D^b(\textup{Coh}(X_{\textup{res}}\smallsetminus D_r))=D^b(\textup{Coh}(X_{\textup{res}}))/\mathcal{M}\cong D^b(\Lambda^1(\tilde{A}_n))$, where $\mathcal{M}$ is the triangulated subcategory generated by coherent sheaves with support on $D_r$. The equivalence can be induced from the localisation $\Pi(\tilde{A}_n)\rightarrow \Lambda^1(\tilde{A}_n)$ which makes $1+\sum_i\alpha_i^*\alpha_i$ invertible.
\end{proposition}
\begin{proof}
Removing $D_r$ corresponds to killing all complexes of coherent sheaves with support at $D_r$. That is, $D^b(\text{Coh}(X_{\text{res}}\smallsetminus D_r))=D^b(\text{Coh}(X_{\text{res}})/\mathcal{M}$, where, equivalently, $\mathcal{M}=\text{thick}\langle im(\iota_*)\rangle$, for $\iota:\{Z+1=0\}\hookrightarrow X_\text{res}$, i.e., $\mathcal{M}$ is the triangulated subcategory generated by complexes of coherent sheaves with homology annihilated by the global section of the structure sheaf of the subvariety defined by $Z+1$.

Let $S=\mathbb{C}[U,V]^{\mathbb{Z}_{n+1}}=\Gamma(\mathcal{O}(X_{\text{res}}))$, the ring of $\mathbb{Z}_{n+1}$-invariant polynomials (the action is by $\xi\cdot(z_1,z_2)=(\xi z_1,\xi^{-1} z_2)$, where we act by the multiplicative group generated by $\xi=e^{\frac{2\pi i}{n+1}}$). Let $M=\mathbb{C}[U,V]$ be the canonical $S$-module. Then $\Pi(\tilde{A}_n)=\mathbb{C}[U,V]*\mathbb{Z}_{n+1}=\text{End}_S(M)$ (where the middle term is the skew-symmetric algebra, see the end of Section~\ref{secmckay}) can be seen as the non-commutative crepant resolution (see e.g. \cite{MR3698338}), via the fully-faithful functor
\begin{align*}
D^b(S)&\rightarrow D^b(\Pi(\tilde{A}_n)), \\
N&\mapsto N\otimes_S^{\mathbb{L}}M.
\end{align*}
Here $M$ is an $\text{End}_S(M)$-module via evaluation, which is compatible with $S$-module structure, viewing $S\subset\text{End}_S(M)$ as canonically embedded as the centre.

Therefore, using that $S\subset \text{End}_S(M)$ acts as multiplication by global sections of the sheaf $\mathcal{O}(X_{\text{res}})$ of regular functions, $D^b(\text{Coh}(X_{\text{res}}))/\mathcal{M}$ can equivalently be computed as $D^b(\Pi(\tilde{A}_n))/\mathcal{M}'$, where $\mathcal{M}'$ is the triangulated subcategory generated by modules $N$ which are annihilated by $UV+1$, which is the element corresponding to $Z+1$ under the isomorphism $\mathbb{C}[U,V]^{\mathbb{Z}_{n+1}}\cong\mathbb{C}[X,Y,Z]/(XY-Z^{n+1})$. Equivalently again, $\mathbb{C}[X,Y,Z]/(XY-Z^{n+1})\cong S\hookrightarrow \Pi(\tilde{A}_n)$,
where $S=Z(\Pi(\tilde{A}_n))$, is identified with the centre via
\[
X=\sum_{i=1}^{n+1}\alpha_{i-1}\cdots\alpha_{i+1}\alpha_i,\quad Y=\sum_{i=1}^{n+1}\alpha_{i}^*\cdots\alpha_{i-2}^*\alpha_{i-1}^*,\quad Z=\sum_{i=1}^{n+1}\alpha_i^*\alpha_i.
\]
Thus, as an element in $\Pi(\tilde{A}_n)$,
\[
p=Z+1=\sum_{i=1}^{n+1}(\alpha_i^*\alpha_i+e_i).
\]

The algebra obtained by localisation at this central element, $\Pi(\tilde{A}_n)[p^{-1}]$, is by definition $\Lambda^1(\tilde{A}_n)$ according to the convention used in \cite{MR4582531}, which is  equivalent to Definition~\ref{defmultpreproj} by \cite[Remark 4.7]{MR4582531}. Let $i:\Pi(\tilde{A}_n)\rightarrow\Lambda^1(\tilde{A}_n)$ denote the localisation map. By definition of localisation, since $p$ is not a zero-divisor, for a $\Pi(\tilde{A}_n)$-module $M$, its pushforward satisfies $i_*M=0$ if and only if $M\cdot p=0$. Furthermore, this pushforwards is exact. Thus the induced quotient $D^b(\Pi(\tilde{A}_n))\rightarrow D^b(\Lambda^1(\tilde{A}_n))$ has the the same kernel as the quotient $D^b(\Pi(\tilde{A}_n))\rightarrow D^b(\Pi(\tilde{A}_n))/\mathcal{M}'$. Hence, by the universal property of Verdier localisation the triangulated categories $D^b(\Lambda^1(\tilde{A}_n))$ and $D^b(\Pi(\tilde{A}_n))/\mathcal{M}'$ are equivalent. 
\end{proof}
\begin{remark}\label{bsidecycplumbremark}
Note that the removal of different choices of divisor of the form $D_r=\{Z+C=0\}$, for $C\not=0$, defines isomorphic varieties. Isomorphisms of the different localisations of the preprojective algebra $\Pi(\tilde{A}_n)$ can be achieved by e.g. rescaling all arrows $\alpha_i$, $i=1,\dots,n+1$ by the factor $C$. Cf. Example~\ref{loghkA}.
\end{remark}
\begin{figure}[H]
\centering
\begin{tikzpicture}
    \draw[thick] (0,0) rectangle (8.9,4.1);
    \node at (1,2.4) {$A$-side};
    \node at (1,0.8) {$B$-side};
    \draw[very thick] (2.5,4.1) to (2.5,0);
    \draw[very thick] (0,3.2) to (8.9,3.2);

    \node at (4,3.6) {$X_{\text{res}}\smallsetminus D_r$};
    \node at (7.5,3.6) {$X_{\text{Mil}}\smallsetminus D_m$};

    \draw[thick] (0,1.6) -- (8.9,1.6);
    \draw[thick] (5.7,0) -- (5.7,4.1);

    \node at (4,1.1) {$D^b(\Lambda^1(\tilde{A}_n))$};
    \node at (4.1,0.5) {Prop. \ref{bsidecycplumb}};'

    \node at (7.2,2.7) {$D^b(\Lambda^1(\tilde{A}_n))$};
    \node at (7.2,2.1) {\cite{MR4033516}, \cite{MR4582531}};

    \node at (4,2.7) {$D^b(\Lambda^{q(s)}(\tilde{A}_n))$?};
    \node at (4,2.1) {Conj. \ref{cycconj}};

    \node at (7.2,1.1) {$D^b(\Lambda^q(\tilde{A}_n))$};
    \node at (7.2,0.5) {Prop. \ref{multloc}};
\end{tikzpicture}
\caption{Summary of the results in Section~\ref{otherdir}. When considered in the same underlying hyper-Kähler manifold $X_{\zeta}$, the parameters in the main diagonal define the same algebras when considered as infinitesimal variations of $\Lambda^1(\tilde{A}_n)$ (see Section~\ref{multinfvar}).}
\label{hmslogcyfigure}
\end{figure}
\subsection{Homological mirror symmetry for resolutions of singularities}\label{hmsres}
Motivated by the computation of the derived category of coherent sheaves on the resolution of singularities without a divisor removed (Theorem~\ref{derivedmckay}), we realised the derived category of finitely generated $\Pi(\tilde{A}_n)$-modules as a wrapped Fukaya category of the cyclic plumbing of $T^*S^2$, or equivalently the $A_n$-Milnor fibre with a fibre removed, with a stop attached (Theorem~\ref{mainresult}). In this way we extended one direction of the derived equivalence in Section~\ref{otherdir} to a derived equivalence between complexes of coherent sheaves on the resolution of $A_n$-singularities with the wrapped Fukaya category of the $A_n$-Milnor fibre with a divisor removed and a stop attached. To get the full statement of homological mirror symmetry on the resolution of $A_n$-singularities we need to consider it as an $A$-side as well. Then, instead of attaching a stop to it, we should equip the cyclic plumbing (i.e. the Milnor fibre with a fibre removed) with a potential.

\subsubsection{The Milnor fibre $X_{\text{Mil}}\smallsetminus D_m$ as a B-side via a Landau--Ginzburg model}\label{sectrivialpot}

The potential is a holomorphic function $w_m:X_{\text{Mil}}\smallsetminus D_m\rightarrow\mathbb{C}$. The mirror category, called a \emph{Landau--Ginzburg model} should be given by a product of singularity categories
\[
DB(w_m)=\prod_{\lambda\in\text{Sg}(w_m)}D_{\text{sg}}(w_m^{-1}(\lambda))
\]
(see e.g. \cite{MR2101296}). The potential is computed on the $A$-side, i.e. in $X_{\text{res}}\smallsetminus D_r$. In the compact Fukaya category it is defined by counting holomorphic discs with Maslov index $2$, that has boundary on the Lagrangian skeleton of $X_{\text{res}}\smallsetminus D_r$. In the wrapped Fukaya category, the same disc counts give rise to Borman--Sheridan classes (see \cite[Section 1.6.4]{MR4696565}), that give rise to a cyclic word $w_m\in CE(\Lambda)$ (more precisely an element in $HH_0(CE(\Lambda))$, the zeroth degree Hochschild homology), which can be obtained by counting Maslov index $2$ discs with boundary on the Legendrian attaching link $\Lambda$ and punctures converging to Reeb chords (see \cite{MR4489822} for a precise definition).

We start by considering the $\tilde{A}_0$-case. Here $X_{\text{res}}\smallsetminus D_r=X_{\text{Mil}}\smallsetminus D_m=\{XY=Z-1\}\subset\mathbb{C}^3\smallsetminus\{Z=0\}$ and the potential is computed in \cite[Section 5.3]{MR4489822} as $\tilde{w}_0=X+Y^2Z^{-1}$.
However, this computation is done in the compactification (without $\mathbb{C}P_{\infty}^1$ removed), and in the open case considered here the disc contributing $Y^2Z^{-1}$ does no longer exists. Thus
\[
w_0=X,
\]
where $X$ is identified with the Reeb chord $\alpha$ (see the proof of Proposition~\ref{parameters}).

In general, the $A_n$-Milnor fibre with a fibre removed is given by the $(n+1)$-fold covering space of the $\tilde{A}_0$-case, and the potential is therefore given by $\alpha_n\dots\alpha_1\alpha_{n+1}\in CE(\Lambda)$ (or more preciely the corresponding cyclic word in $HH_0(CE(\Lambda))$) which we identify with the regular function $w_m:X_{\text{Mil}}\smallsetminus D_m\rightarrow\mathbb{C}$ given by $w(X,Y,Z)=X$.

This regular function is smooth, which implies that $DB(w_m)=0$.

\subsubsection{Resolution of $A_n$-singularities $X_{\text{res}}$ as an $A$-side via non-exact wrapped Fukaya categories}
\begin{remark}\label{nonexactlin}
Under similar assumptions to those of Conjecture~\ref{cycconj}, on generators and morphisms on the non-exact Fukaya category of $X_{\text{res}}$, i.e. that the non-exact Fukaya category of $X_{\text{res}}$ is $\Lambda^{q(s)}(A_n)$, with $q_1(s),\dots,q_n(s)$, satisfying $\prod_{l=i}^jq_l(s)\not=1$, for all choices $1\leq i\leq j\leq n$. We have, as in the proof of Proposition~\ref{multloc}, that this category vanishes.

This ought also be a corollary of the main result in \cite{MR2720232}, that $SH^{\bullet}(X_{\text{res}})=0$. One of the caveats with this result (and with Conjecture~\ref{cycconj}) is that the theoretical foundations (e.g. generation and surgery formula) for the non-exact category is not, to the knowledge of the author, worked out.
\end{remark}

\begin{figure}[H]
\centering
\begin{tikzpicture}
    \draw[thick] (0,0) rectangle (12,4.1);
    \node at (1,2.4) {$A$-side};
    \node at (1,0.8) {$B$-side};
    \draw[very thick] (2.5,4.1) to (2.5,0);
    \draw[very thick] (0,3.2) to (12,3.2);

    \node at (4.5,3.6) {$X_{\text{res}}$};
    \node at (9.3,3.6) {$X_{\text{Mil}}\smallsetminus D_m+\text{stop/potential}$};

    \draw[thick] (0,1.6) -- (12,1.6);
    \draw[thick] (6.6,0) -- (6.6,4.1);

    \node at (4.5,1.1) {$D^b(\Pi(\tilde{A}_n))$};
    \node at (4.5,0.5) {\cite{MR1752785} (Thm. \ref{derivedmckay})};'

    \node at (9.3,2.7) {$D^b(\Pi(\tilde{A}_n))$};
    \node at (9.3,2.1) {Thm. \ref{mainresult}};

    \node at (4.5,2.7) {$0$};
    \node at (4.5,2.1) {\cite{MR2720232} (Rmk. \ref{nonexactlin})};

    \node at (9.3,1.1) {$0$};
    \node at (9.3,0.5) {Section \ref{sectrivialpot}};
\end{tikzpicture}
\caption{Summary of the results in Section~\ref{hmsres}.}
\label{hmsresfigure}
\end{figure}
\subsection{Homological mirror symmetry for Milnor fibres: deformations}\label{hmsmil}
Finally, we consider the case where we put back the divisor on the Milnor fibre, and where we put a stop or a potential instead on the resolution of singularities; we divide this case into two sections. The Milnor fibre is a linear plumbing, i.e. plumbing according to the $A_n$-diagram, of $T^*S^2$ (see Figure~\ref{linear}). Its wrapped Fukaya category has been computed in \cite{MR3692968}, and is equivalent to $D^{\pi}(\mathcal{G}(A_n))$ (this is immediately seen from Figure~\ref{linear}, using the generation result together with the surgery formula). We relate this to a category computed by a potential on the resolution of singularities with a divisor removed.

Before that, analogously to the exact case we conjecture that the wrapped Fukaya category of the resolution with a divisor removed and a stop attached can be computed by a Chekanov--Eliashberg algebra in $J^1S^1$, and is given by modules over a deformation of the preprojective algebra of type $\tilde{A}_n$ (introduced in \cite{MR1620538}). Such deformations of the preprojective algebras also model coherent sheaves on the $A_n$-Milnor fibres.

\subsubsection{Sketch of attaching a stop on $X_{\text{res}}\smallsetminus D_r$}\label{resstopsect}

Now we want to compute $CE(\Lambda)$ in $J^1S^1$, but with certain area parameters coming from the non-exactness of the symplectic manifold. This yields conjecturally:
\begin{conjecture}\label{resconj}
The non-exact partially wrapped Fukaya category of $X_{\textup{res}}\smallsetminus D_r$ with a stop $\sigma$ satisfies
\[
D^{\pi}\mathcal{W}(X_{\textup{res}}\smallsetminus D_r,\sigma)\cong D^b(\Pi^{\lambda(s)}(\tilde{A}_n)),
\]
that is, it is equivalent to the bounded derived category of finitely generated modules over the deformed preprojective algebra, where the parameters are in the field of Laurent series $\mathbb{C}((s))$ and given by 
\[
\lambda_i(s)=q_1(s)\cdots q_{i-1}(s)(q_i(s)-1),
\]
where $q_i(s)=e^{s\lambda_i}\in\mathbb{C}((s))$, and where the $\lambda_i\not=0$ are determined by the symplectic area of the non-exact Lagrangian spheres. 
\end{conjecture}
We sketch a proof below. See also Conjecture~\ref{parameters}.

Let $CE^{q(s)}(\Lambda)$ denote the dg-algebra obtained by computing the Chekanov--Eliashberg algebra of the $\tilde{A}_n$-link (Figure~\ref{lagAN}) with loop space coefficients $t_i^{\pm 1}$ (see Remark~\ref{loopspace}), and then specialising $t_i\mapsto q_i(s)\in\mathbb{C}((s))$ (and consider the algebra as defined over this field).

We also make the following auxiliary definition: For $q=(q_1,\dots,q_{n+1})$, $q_i\in k^{\times}$, let $\mathcal{G}^{q}_{M,\,\text{non-loc}}(\tilde{A}_n)$ denote the dg-algebra, defined in almost the same way as the deformed derived multiplicative preprojective algebra, the only difference being here we don't have the generators $z^{\pm 1}$, $\tau$ (see Definition~\ref{defdermultpreproj} and Example~\ref{multanex}). This is just, up to quasi-isomorphism, the dg-algebra obtained from $\mathcal{G}_M^q(\tilde{A}_n)$ with a non-localised variable $z$, as follows by cancellation of generators. (In the following we will mainly be interested in the cases $k=\mathbb{C}$ and $k=\mathbb{C}((s))$.)

\begin{lemma}\label{nonloccechoice}
By choosing orientations and placement of basepoint in the right way, we have $CE^{q(s)}(\Lambda)=\mathcal{G}^{q(s)}_{M,\,\textup{non-loc}}(\tilde{A}_n)$ on the nose.
\end{lemma}
\begin{proof}
The orientations are chosen so that the unknots in the front projection (see e.g. Figure~\ref{grReeb}) are oriented clockwise (i.e. the strand with greater $z$-coordinate is oriented from left to right). We change the place of the basepoints in the Lagrangian projection in Figure~\ref{lagAN} so that, on each component $\Lambda_i$, the basepoint $*_i$ lies after the crossing corresponding to $\zeta_i$ and before the crossing corresponding to $\alpha_i^*$ following the orientation of the component.
\end{proof}
Let $\Lambda^q_{\text{non-loc}}(\tilde{A}_n)=H_0(\mathcal{G}^{q}_{M,\text{ non-loc}}(\tilde{A}_n))$, i.e. the only difference to $\Lambda^q(\tilde{A}_n)$ is that we don't have the inverse $z^{-1}$.
\begin{lemma}\label{substlemma}
The dg-algebra $\mathcal{G}^{q}_{M,\,\textup{non-loc}}(\tilde{A}_n)$, with parameters $q_i\in k^{\times}$ such that $q_1\cdots q_{n+1}=1$, is isomorphic to the deformed Ginzburg algebra $\mathcal{G}^{\lambda}(\tilde{A}_n)$, with parameters $\lambda_i=q_1\cdots q_{i-1}(q_i-1)$, satisfying $\lambda_1+\cdots+\lambda_{n+1}=0$. Furthermore, with the same parameters, the algebra $\Lambda^q_{\textup{non-loc}}(\tilde{A}_n))$ is isomorphic to the deformed multiplicative preprojective algebra $\Lambda^q(\tilde{A}_n)$.
\end{lemma}
\begin{proof}
The substitutions $\alpha_i\leftrightarrow q_1\cdots q_i\alpha_i$ while setting $\lambda_i=q_1\cdots q_{i-1}(q_i-1)$ gives rise to both isomorphisms.
\end{proof}

\begin{proposition}\label{deformedformal}
The deformed Ginzburg algebra of the (cyclically oriented) $\tilde{A}_n$-quiver $\mathcal{G}^{\lambda}(\tilde{A}_n)$ is formal and quasi-isomorphic to the deformed preprojective algebra $\Pi^{\lambda}(\tilde{A}_n)$.
\end{proposition}
\begin{proof}
This is contained in \cite[Lemma 9.1]{kalck2018relativesingularitycategoriesii}. It follows from Theorem~\ref{formal}, together with a comparison of spectral sequences associated to $G(\tilde{A}_n)$ and $G^{\lambda}(\tilde{A}_n)$ respectively, showing that their first two pages coincide.
\end{proof}
\begin{remark}\label{formalityremark}
By Lemma~\ref{deformedformal} and Lemma~\ref{substlemma}, $\mathcal{G}^{q}_{M,\,\text{non-loc}}(\tilde{A}_n)$ is formal and quasi-isomorphic to $\Lambda^q_{\text{non-loc}}(\tilde{A}_n)$, when $q_1\cdots q_{n+1}=1$. One can show that the multiplicative set $S\subset\Lambda^q(\tilde{A}_n)$ generated by $e_i+\alpha_i^*\alpha_i$, $i=1\dots,n+1$, is a right Ore set (this follows from the defining relations $e_i+\alpha_{i-1}\alpha_{i-1}^*=q_ie_i+q_i\alpha_i^*\alpha_i$ in Definition~\ref{defmultpreproj}, Example~\ref{multanex}). By \cite[Remark 4.7]{MR4582531}, localising at $S$ provides an alternative model for the (deformed) derived multiplicative preprojective algebra. Therefore, by \cite[Proposition 5.14]{MR3771137}, $G_M^q(\tilde{A}_n)$ is formal and quasi-isomorphic to its degree $0$ homology $\Lambda^q(\tilde{A}_n)$. This provides an alternative proof of this formality, in the case $q_1\cdots q_{n+1}=1$, which is the case used in Section~\ref{otherdir}.
\end{remark}
\begin{proof}[Sketch of proof of Conjecture~\ref{resconj}]
This is analogous to Conjecture~\ref{cycconj}. Computing via the surgery formula, the expected category is $D^{\pi}(CE^{q(s)}(\Lambda))$. The parameters $q_i=e^{s\lambda_i}\in\mathbb{C}((s))\smallsetminus\{0\}$ are defined using a (log-)Novikov parameter $s$. The area parameters $\lambda_i$ satisfy $\sum_{l=i}^j\lambda_l\not=0$, for all choices $1\leq i\leq j\leq n$ and $\lambda_1+\cdots+\lambda_{n+1}=0$, which implies that $\prod _{i\leq l\leq j}q_l(s)\not=1$ and $q_1(s)\cdots q_{n+1}(s)=1$.

By Lemma~\ref{nonloccechoice}, Lemma~\ref{substlemma}, Proposition~\ref{deformedformal}, we have $D^{\pi}(CE^{q(s)}(\Lambda))=D^{\pi}(\Pi^{\lambda(s)}(\tilde{A}_n))$. One concludes the statement by applying \cite[Theorem 1.5]{MR1620538}, to replace $D^{\pi}$ with $D^b$. 
\end{proof}

\subsubsection{Deformed preprojective algebras and Milnor fibres}\label{addpreprojbside}

For the other side, we consider the $A_n$-Milnor fibre $X_{\text{Mil}}=\{XY=P(Z)\}$, where $P(Z)$ is a polynomial of degree $n+1$ without repeated roots. Then $D^b(\text{Coh}(X_{\text{Mil}}))=D^b(\mathbb{C}[X,Y,Z]/(XY-P(Z)))$, and we relate this situation to the deformed preprojective algebras as follows.
\begin{proposition}\label{nonlocmil}
For all choices of $\lambda_i\in\mathbb{C}$ such that $\sum_{i=1}^{n+1}\lambda_i=0$, 
we have
\[
e_{n+1}\Pi^{\lambda}(\tilde{A}_{n})e_{n+1}\cong\mathbb{C}[X,Y,Z]/(XY-P(X)),
\]
where $P(X)$ is a polynomial of degree $n+1$, whose roots are $\Lambda_i=\sum_{j=1}^{i}\lambda_j$,  for $i=1,\dots,n+1$.

If $\lambda_i$ are such that for all $l< k$ we have $\sum_{l\leq i\leq k}\lambda_i\not=0$, i.e. the roots are distinct, then also
\[
D^b(\Pi^{\lambda}(\tilde{A}_{n}))\cong D^b(\mathbb{C}[X,Y,Z]/(XY-P(Z))).
\]
\end{proposition}
\begin{proof}
This is proven in the same way as Proposition~\ref{parameters} and Proposition~\ref{multloc} (see also Remark~\ref{substrmk}), where we put $X=\alpha_{n}\cdots\alpha_1\alpha_{n+1}$, $Y=\alpha_{n+1}^*\alpha_1^*\cdots\alpha_n^*$ and $Z=\alpha_{n+1}^*\alpha_{n+1}$. The condition $\sum_{i=1}^{n+1}\lambda_i=0$ ensures commutativity of the variables $X$, $Y$ and $Z$. The polynomial obtained is
\[
P(Z)=(Z-\Lambda_1)\cdots(Z-\Lambda_{n+1}).
\]

For the second part, one notes that $\Pi^{\lambda}(A_n)\not=0$ if and only if there exists $l<k$ such that $\sum_{l\leq i\leq k}\lambda_l=0$ (see also \cite[Example 2.3]{MR4582531}).
\end{proof}
\begin{remark}\label{substrmk}
In terms of the relations between the parameters $\lambda_i$ and $q_i$ in the proof of Lemma~\ref{substlemma}, note that the condition $\prod_{l\leq i\leq k}q_i\not=1$ for all $l< k$ is equivalent to $\sum_{l\leq i\leq k}\lambda_i\not=0$ for all $l<k$, and the condition $\sum_{i=1}^{n+1}\lambda_i=0$ is equivalent to $\prod_{i=1}^{n+1}q_i=1$.

Moreover, comparing Proposition~\ref{nonlocmil} with Proposition~\ref{parameters}, the roots are related by $\Lambda_i=\sum_{j=1}^{i}\lambda_{j}=\prod_{j=1}^iq_j-1=Q_i-1$. The equation $XY-P(Z)=0$ in Proposition~\ref{parameters} is obtained from the same equation in Proposition~\ref{nonlocmil} by the substitutions $X\leftrightarrow Q_1\cdots Q_{n+1}X$ and $Z\leftrightarrow Z+1$, which is in line with the substitutions in the proof of Lemma~\ref{nonlocmil}. These changes of coordinates corresponds to a change of coordinates on the Milnor fibre which puts the divisor $D_m$ at preimage of $0$ instead of $-1$, as in Example~\ref{loghkA}. In Proposition~\ref{parameters}, we thus localise the variable $Z$ rather than $Z+1$.

\end{remark}
\subsubsection{Equivalences as infinitesimal variations of $\Pi^{\lambda}(\tilde{A}_n)$ near $\lambda=(0,\dots,0)$}\label{addinfvar}
This section follows Section~\ref{multinfvar} closely. Along the equator $u_I=0$, it is described in Example~\ref{loghkA} how $(X_{\zeta},I_u)$, here denoted by $X_{\text{Mil}}$, can be represented as the affine variety
\[
\{XY=(Z-\Lambda_1)\cdots(Z-\Lambda_{n+1})\},
\]
where the roots of the polynomial on the right-hand side, denoted $P(Z)$, are $\Lambda_i=\lambda_1+\cdots+\lambda_{i}$.

Deforming the complex structure (see Remark~\ref{moreonfamily}) corresponds to adding a deformation parameter $s$ deforming the roots, namely
\[
\{XY=(Z-s\Lambda_1)\cdots(Z-s\Lambda_{n+1})\}.
\]
Note that (in contrast to the situation in Section~\ref{multinfvar}) these affine varieties are all isomorphic. Using $\lambda_i=\Lambda_{i}-\Lambda_{i-1}$, for $i=1,\dots,n+1$ (with $\Lambda_0=\Lambda_{n+1}=0$), this is equivalent to deforming the parameters $\lambda_i$. Thus, by Section~\ref{addpreprojbside}, we conclude that
\[
D^b(\text{Coh}(X_{\text{Mil}}))=D^b(\Pi^{s\lambda}(\tilde{A_n})).
\]

If we regard $s$ as an infinitesimal variation parameter, that is, we consider the deformed preprojective algebra above over the ring $k=\mathbb{C}[[s]]/(s^2)$, then this algebra can be identified with $\Pi^{\lambda(s)}(\tilde{A}_n)$ obtained in Section~\ref{resstopsect}, provided that this algebra is also interpreted with coefficients in $k=\mathbb{C}[[s]]/(s^2)$. Indeed, the deformation parameters are then
\[
\lambda_i(s)=q_1(s)\cdots q_{i-1}(s)(q_i(s)-1)=(1+s\lambda_1)\cdots(1+s\lambda_{i-1})(s\lambda_i)=s\lambda_i.
\]
\subsection{Homological mirror symmetry for Milnor fibres: LG-potential}\label{hmsmiltwo}
We find a potential $w_r$ and compute the triangulated category of singularities at the singular points (see Section~\ref{prelsingcat}). Similarly to above (Section~\ref{sectrivialpot}) the potential counts Maslov $2$ index discs with boundary on the divisor and can be found by considering the resolution of singularities as a covering space of the self-plumbing (the $\tilde{A}_0$-case), using \cite[Section 5.3]{MR4489822}. The potential is given by the composition $w_r:X_{\text{res}}\smallsetminus D_r\rightarrow \{XY=Z^{n+1}\}\smallsetminus\{Z=-1\}\xrightarrow{\pi_X}\mathbb{C}$, where $\pi_X$ means projection onto the $X$ coordinate. For this potential we prove the following result. 

The regular function $w_r$ has an isolated singular value equal to $0$. Thus, $DB(w_r)=D_{\text{sg}}(w_r^{-1}(0))$. Geometrically, $w_r^{-1}(0)$ is the non-affine variety given by the exceptional divisor (which is an $A_n$ configuration of copies of $\mathbb{C}P^1$ with self-intersection $-2$) together with the plane $\{X=0\}\subset \{XY=Z^{n+1}\}$. However, the structure as a variety is also determined by the order of vanishing of $w$, which is $n+1,n,\dots,2$, respectively, on the complex projective lines, and $1$ on $\{X=0\}$.
\begin{theorem}\label{respot}
For the potential $w_r:X_{\textup{res}}\smallsetminus D_r\rightarrow\mathbb{C}$ given by the composition with projection onto the $X$ coordinate we have an equivalence $D_{\textup{sg}}(w_r^{-1}(0))\simeq D^{\pi}(\mathcal{G}(A_{n})\otimes_{\mathbb{C}}\mathbb{C}[T^{\pm1}])$, where $|T|=-2$.
\end{theorem}
\begin{remark}\label{parameterremark}
There is an asymmetry here in that this $B$-side model, that is, the singularity category, has $\mathbb{C}[T^{\pm}]$-coefficients, with $|T|=2$. The author believes that this is due to the fact that we are not considering the linear plumbing and the cyclic plumbing by themselves, but we are considering a non-exact embedding of the cyclic plumbing of $n+1$ copies of $T^*S^2$ into the linear plumbing of $n$ copies of $T^*S^2$, that moreover is not compatible with the trivialisations of the complex determinant bundles (which means that the $\mathbb{Z}$-gradings are not compatible). In general, such a non-exact embedding requires the use of some kind of Novikov coefficients. Using this, we can obtain the same category on the $A$-side. Here, on the linear plumbing this parameter is specialised by putting $T=1$, which need not yield a convergent theory on the $A$-side in general, and which also destroys the grading.
\end{remark}
\begin{proof}[Proof of Theorem~\ref{respot}]
First of all, we can consider $w$ as a function from all of $X_{\text{res}}$, and ignore the removal of $D_r$, since it is a closet set disjoint from the singular locus. Note that then $\pi_X$ is identified with the function from the invariant ring $\mathbb{C}[X,Y]^{\mathbb{Z}_{n+1}}\rightarrow\mathbb{C}$ given by $X^{n+1}$. By \cite[Theorem 8.6]{MR2504756} we can consider the non-commutative model given by $G$-equivariant modules over $\mathbb{C}[X,Y]/(X^{n+1})$, that is,
\[
D_{\text{sg}}(w_r^{-1}(0))\cong D_{\text{sg}}^{G}(\pi_X^{-1}(0))=D_{\text{sg}}(A),
\]
where $A=(\mathbb{C}[X,Y]/(X^{n+1}))*\mathbb{Z}_{n+1}$ (the skew-symmetric algebra, see Section~\ref{secmckay} for the definition of the skew-symmetric product). Moreover, this algebra is isomorphic to the preprojective algebra modulo all full turn paths in the original $\tilde{A}_{n}$ arrows: 
\[
A= \Pi(\tilde{A}_{n})/(\alpha_i\alpha_{i-1}\cdots\alpha_{n+1}\alpha_1\cdots\alpha_{i-1}\;|\;i=1,\dots,n+1).
\]
Thus we want to show that $D_{\text{sg}}(A)\simeq D^{\pi}(\mathcal{G}(A_{n})\otimes_{\mathbb{C}}\mathbb{C}[T^{\pm 1}])$.

This follows from Lemma~\ref{dbsggeom} and Theorem~\ref{dbsgcatalg} below, together with the computation of the minimal model of parallel strands through a $1$-handle in the author's work \cite{ryd2025}.
\end{proof}

\begin{lemma}\label{dbsggeom}
The Weinstein manifold obtained by attaching $n+1$ Weinstein $2$-handles along $n+1$ parallel Legendrian strands through a Weinstein $1$-handle is equivalent to the linear plumbing of $n$ copies of $T^*S^2$. In other words, there is an equivalence of perfect derived categories of the Chekanov--Eliashberg algebras of the different Legendrian links depicted in Figure~\ref{sgeq}.
\end{lemma}
\begin{proof}
The first step in Figure~\ref{sgeq} is obtained by consecutive handle sliding all other $2$-handles over the top $2$-handle and pulling them through the $1$-handle, followed by cancelling the $1$-handle with the top $2$-handle (see e.g. \cite[Figure 9]{MR4417717}). The second step is obtained by a repeated series of $2$-handle slides and Reidemeister moves, as follows: First we slide the top component over the next one. Then we can pull this one out on the right-hand side. Then we repeat, sliding the second unknot from the top (the one which is now on top) over the next one. Then we can pull this one out using Reidemeister moves. Continuing this eventually yields the Legendrian link on the right-hand side (after some extra modifications using Reidemeister moves).
\end{proof}
\begin{figure}[H]
\centering
\begin{tikzpicture}
\draw[thick] (-5,0) circle (0.5cm);
\draw[densely dotted] (-4.5,0) arc (0:180:0.5 and 0.16);
\draw[thick] (-5.5,0) arc (0:180:-0.5 and -0.16);

\draw[thick] (-4.75,0.43) to (-2.75,0.43);
\draw[thick] (-4.53,0.17) to (-2.97,0.17);
\draw[thick] (-4.53,-0.17) to (-2.97,-0.17);
\draw[thick] (-4.75,-0.43) to (-2.75,-0.43);

\draw[thick] (-2.5,0) circle (0.5cm);
\draw[densely dotted] (-2,0) arc (0:180:0.5 and 0.16);
\draw[thick] (-3,0) arc (0:180:-0.5 and -0.16);
\draw[-stealth] (-1.7,0) to (-1.3,0);

\begin{knot}[flip crossing=3,flip crossing=5]
\strand[thick] (-1,0) to[out=0,in=180] (-0,0.6) to[out=0,in=180] (1,0)
to[out=180,in=0] (0,-0.6) to[out=180,in=0] (-1,0);

\strand[thick] (-1,0.3) to[out=0,in=180] (0,0.9) to[out=0,in=180] (1,0.3) to[out=180,in=0] (0,-0.3) to[out=180,in=0] (-1,0.3);

\strand[thick] (-1,-0.3) to[out=0,in=180] (0,0.3) to[out=0,in=180] (1,-0.3)
to[out=180,in=0] (0,-0.9) to[out=180,in=0] (-1,-0.3);
\end{knot}

\draw[-stealth] (1.3,0) to (1.7,0);

\begin{knot}[flip crossing=2,flip crossing=4]
\strand[thick] (2,0) to[out=0,in=180] (3,0.6) to[out=0,in=180] (4,0)
to[out=180,in=0] (3,-0.6) to[out=180,in=0] (2,0);

\strand[thick] (3.15,0) to[out=0,in=180] (4.15,0.6) to[out=0,in=180] (5.15,0)
to[out=180,in=0] (4.15,-0.6) to[out=180,in=0] (3.15,0);

\strand[thick] (4.3,0) to[out=0,in=180] (5.3,0.6) to[out=0,in=180] (6.3,0)
to[out=180,in=0] (5.3,-0.6) to[out=180,in=0] (4.3,0);
\end{knot}
\end{tikzpicture}
\caption{On the left-hand side is the Legendrian link in $S^1\times S^2$ consisting of $n+1$ parallel strands. It is equivalent to the Legendrian link in $\mathbb{R}^3$ that defines the linear plumbing $n$ copies of $T^*S^2$, shown on the right-hand side. Here, shown in the front projection, for $n=3$.}
\label{sgeq}
\end{figure}
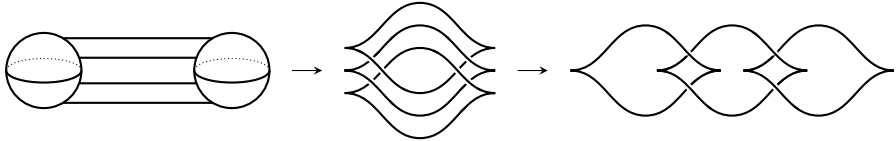
Therefore, as noted in the proof of Theorem~\ref{respot}, by \cite{ryd2025}, $D^{\pi}(\mathcal{G}(A_n))$ is quasi-equivalent to $D^{\pi}(B)$, where $B$ is the following $A_{\infty}$-algebra. (Note however, that the grading differs from \cite{ryd2025}, which is due to the choice of Maslov class. The gradings can be made to coincide, either by changing that choice, or by shifting the degree of the generators in Lemma~\ref{sggeneration}.)
\begin{definition}\label{sgquiveralgdef1}
Let $Q_B=(Q_0,Q_1)$ be the cyclic quiver with $n+1$ vertices and two arrows $a_i,b_i:i\rightarrow i-1$ (indices are considered modulo $n+1$). We grade this quiver by putting $|a_i|=0$ and $|b_i|=-1$, for $i=1,\dots,n$, and $|a_{n+1}|=2$ and $|b_{n+1}|=1$. Let $B=\mathbb{C}Q/I$, where $I$ is the ideal generated by the relations $b_{i-1}b_i=0$ (except for $n=1$, where $b_{i-1}b_i=\mu_2(b_{i-1},b_i)=e_i$) and $a_{i-1}b_i=b_{i-1}a_i$ (modulo $n+1$). We turn $B$ into an $A_{\infty}$-algebra by equipping it with the higher product $\mu_{n+1}(a^{m_{n+1}}b_{i_{m+1}},\dots,a^{m_1}b_{i_1})=a^{m_1+\dots+m_{n+1}}$. The notation should be read as follows: on the left-hand side, $a^{m_i}b_j=a_{i+m_i}\cdots a_{i+1}b_j$, and $\mu_{n+1}$ is non-zero if and only if the entries are composable in $\mathbb{C}Q$, in which case the right-hand side is equal to $a^{m_1+\cdots+m_{n+1}}=a_{i_1+m_1+\cdots+m_{n+1}-1}\cdots a_{i_1}$.
\end{definition}
\begin{figure}[H]
\centering
\begin{tikzpicture}
\draw[fill=black] (-0.1,0) circle (0.01cm);
\draw[fill=black] (0,0) circle (0.01cm);
\draw[fill=black] (0.1,0) circle (0.01cm);

\draw[fill=black] (2,0) circle (0.06cm);
\draw[fill=black] (3,1.73) circle (0.06cm);
\draw[fill=black] (2,3.46) circle (0.06cm);
\draw[fill=black] (0,3.46) circle (0.06cm);
\draw[fill=black] (-1,1.73) circle (0.06cm);

\draw[black,-stealth] (-0.93,2) [out=80, in=220] to (-0.27,3.39);
\draw[black,-stealth] (-0.72,1.8) [out=40, in=-100] to (-0.07,3.19) ;

\draw[black,-stealth] (0.2,3.66) [out=20, in=160] to (1.8,3.66) ;
\draw[black,-stealth] (0.2,3.26) [out=-20, in=-160] to (1.8,3.26) ;

\draw[black,stealth-] (2.93,2) [out=100, in=-40] to (2.27,3.39);
\draw[black,stealth-] (2.72,1.8) [out=140, in=-80] to (2.07,3.19) ;

\draw[black,stealth-] (-0.93,1.46) [out=-80, in=140] to (-0.27,0.07);
\draw[black,stealth-] (-0.72,1.66) [out=-40, in=100] to (-0.07,0.27) ;

\draw[black,-stealth] (2.97,1.46) [out=-100, in=40] to (2.27,0.07);
\draw[black,-stealth] (2.72,1.66) [out=-140, in=80] to (2.07,0.27) ;

\draw[black,-stealth] (1.8,0.2) [out=160, in=20] to (0.2,0.2);
\draw[black,-stealth] (1.8,-0.2) [out=200, in=-20] to (0.2,-0.2) ;

\node[scale=0.6] at (0,3.71) {$2$};
\node[scale=0.6] at (2,3.71) {$1$};
\node[scale=0.6] at (-1.25,1.73) {$3$};
\node[scale=0.6] at (3.43,1.73) {$n+1$};
\node[scale=0.6] at (2.03,-0.25) {$n$};

\node[scale=0.6] at (0,2.4) {$a_3$};
\node[scale=0.6] at (1,2.9) {$a_2$};
\node[scale=0.6] at (2,2.4) {$a_1$};
\node[scale=0.6] at (1.94,1.06) {$a_{n+1}$};
\node[scale=0.6] at (1,0.6) {$a_n$};
\node[scale=0.6] at (0.05,1.06) {$a_{4}$};

\node[scale=0.6] at (-1,2.8) {$b_3$};
\node[scale=0.6] at (1,4) {$b_2$};
\node[scale=0.6] at (3,2.8) {$b_1$};
\node[scale=0.6] at (3.08,0.6) {$b_{n+1}$};
\node[scale=0.6] at (1,-0.6) {$b_n$};
\node[scale=0.6] at (-0.9,0.6) {$b_{4}$};

\end{tikzpicture}
\caption{The quiver $Q_{B}$. The relations can be summarised by $b^2=0$ and $ba=ab$; the higher operations can be summarised by $\mu_{n+1}(a^{n_1}b,\dots,a^{n_{n+1}}b)=\pm a^{n_1+\cdots+n_{n+1}}$ and otherwise $0$.}
\label{sgquiver}
\end{figure}
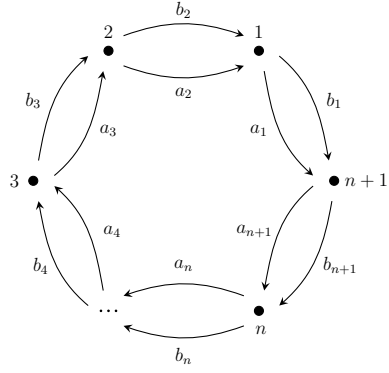
We will from now on use the notation $\alpha=\alpha_1+\cdots+\alpha_ {n+1}$ (hence $\alpha^{m}=\sum_i\alpha_{i-1}\alpha_{i-1}\cdots\alpha_{i-m}$, where indices are read modulo $n+1$). In particular $(e_i\Pi(\tilde{A}_n))/(\alpha^m)=(e_i\Pi(\tilde{A}_n))/(e_i\alpha^m)$, and we have morphisms given by left multiplication $\alpha^m:e_i\Pi(\tilde{A}_n)\rightarrow e_{i+m}\Pi(\tilde{A}_n)$ (and similarly replacing $\Pi(\tilde{A}_n)$ by A).
\begin{theorem}\label{dbsgcatalg}
The singularity category $D_{\textup{sg}}(A)=D_{\textup{sg}}(\Pi(\tilde{A}_n)/(\alpha^{n+1}))$ is equivalent to the triangulated category $D^{\pi}(B\otimes_{\mathbb{C}}\mathbb{C}[T^{\pm1}])$, where $|T|=-2$ and the $A_{\infty}$-structure is extended to be linear in $T^{\pm 1}$.
\end{theorem}
We divide the theorem into four lemmas.
\begin{lemma}\label{sggeneration}
The derived category of $A$-modules $D^b(A)$ is  split-generated by the modules $C_i=e_i(A/(\alpha))$, for $i=1,\dots,n+1$. Hence, the singularity category $D_{\textup{sg}}(A)$ is also split-generated by modules $C_i$, $i=1,\dots,n+1$.
\end{lemma}
\begin{proof}
Let $M_{\bullet}\in\text{Ch}(A)$ be a chain complex, bounded below and with bounded homology, defining an arbitrary object in $D^b(A)$. Without loss of generality, by shifting $M_{\bullet}$, we can assume that $M_i=0$, for $i<0$. This object has a representative given by a complex of projective modules, i.e., a complex $Q_{\bullet}'$, such that each $Q_i'$ is projective, together with a quasi-isomorphism $Q_{\bullet}'\rightarrow M_{\bullet}$ (where we note that $Q'_{\bullet}$ may be unbounded but $Q_i'=0$, for $i<0$).

Let $N$ be the smallest integer such that $H_i(M_{\bullet})=0$, for all $i\geq N$. Consider $K=\text{ker}(M_N\rightarrow M_{N-1})$ as a $\Pi(\tilde{A}_n)$-module. As such, it has a projective resolution 
\[
\tilde{Q}_{2}\xrightarrow{\tilde{f}_2} \tilde{Q}_{1}\xrightarrow{\tilde{f}_1} \tilde{Q}_0\simeq_{\text{q.i.}}K,
\]
using that $\Pi(\tilde{A}_n)$ has global dimension $2$ \cite[Example 3.4]{MR1752785}. We define a new complex $Q_{\bullet}$ of projective $A$-modules, which is built from $C_i$ in finitely many steps taking cones, sums and summands, and which is quasi-isomorphic to $M_{\bullet}$. It will consist of a finite projective part, followed by an infinite alternating part, both of which can be built from the $C_i$ in finitely many steps. We do this in two steps, as follows: First, let $Q_{\bullet}''$ be defined by $Q_i''=Q_i'$, for $i\leq N-1$, by $Q_{N+i}''=\tilde{Q}_i\otimes_{\Pi(\tilde{A}_n)}A=\tilde{Q}_i/(\alpha^{n+1})$, for $i=0,1,2$, and $Q_i''=0$, otherwise. Here $\tilde{Q}_i$ is a direct sum of indecomposable projective $\Pi(\tilde{A}_n)$-modules, i.e., modules of the form $e_j\Pi(\tilde{A}_n)$, and $Q_i''=\tilde{Q}_i/(\alpha^
{n+1})$ denotes the corresponding direct sum of indecomposable projective $A$-modules $e_jA=(e_j\Pi(\tilde{A}_n))/(\alpha^{n+1})$. Note that $Q''$ has a chain map to $M_{\bullet}$ that induces isomorphism in homology in all degrees except possibly $N+1$ and $N+2$. In these degrees the map in homology is simply vanishing. This constitute the first finite projective part.

If the complex $Q''$ is exact in degrees $N+1$ and $N+2$ we are done, otherwise we next show how to amend this non-exactness in $Q''$ to produce the sought complex $Q$ that is quasi-isomorphic to $M_{\bullet}$. This will be done by adding the alternating part. Given any indecomposable direct summand $e_jA$ in $Q_{N+i}''$, for $i=1$ or $i=2$. The map $e_jA\hookrightarrow Q''_{N+i}\xrightarrow{f_1}Q_{N+i-1}''$ is given by left multiplication with $\sum c_{r,s}\alpha^r(\alpha^*)^s$, for some finite sum, where $\alpha^r$ denotes the product of $r$ consecutive $\alpha_i$ that ends at the vertex $j$, and, similarly, $(\alpha^*)$ denotes the product of consecutive $\alpha^*_i$ that ends at the vertex where $\alpha^r$ starts, and $c_{r,s}\in\mathbb{C}$ (this follows from the commutator relations defining the preprojective algebra); here we use $f_i=\tilde{f}_i\otimes_{\Pi(\tilde{A}_n)}A$ to denote the induced maps. Let $J^i$ be a set indexing the indecomposable summands $Q^i_j$ of $Q_{N+i}''$, for $i=1,2$ (here we use that finitely generated projective modules have unique decomposition, up to re-ordering the summands, into indecomposable projective modules, using that $R$ is semi-simple). For each $j\in J^i$, let $f_{i,j}$ denote the map given by the composition $Q^i_j\hookrightarrow Q_{N+i}''\rightarrow Q''_{N+i-1}$. In the sum corresponding to $f_{i,j}$, let $r_{i,j}$ denote the minimal $r$-exponent. 

We obtain an exact complex by adjoining $(e_{(j-(n+1-r_{i,j}))}A)/(\alpha^{r_{i,j}})$, (in degree $N+i+1$) together with the map into $Q^i_j$ given by multiplication by $\alpha^{n+1-r_{i,j}}$, for each indecomposable summand in $Q_{N+1}''$ and $Q_{N+2}''$, if $r_{i,j}\leq n$. In this case, we make this into a complex of projective modules $Q_{\bullet}$ by replacing each $(e_{(j-(n+1-r_{i,j}))}A)/(\alpha^{r_{i,j}})$ by its resolution below given by an iterated cone of the $C_i$. And, when $r_{i,j}\geq n+1$, then the map above is equal to multiplication by $0$, and we remove this summand from the complex $Q''_{N+i}$.

Since the resulting complex $Q_{\bullet}$ is given by a part which is obtained from the projective modules in finitely many steps (taking direct sums, summands and cones and applying shifts), followed by the resolutions of finitely many quotients $(e_jA)/(\alpha^m)$, to conclude that $M_{\bullet}$ lies in $\text{thick}\langle C_1,\dots,C_{n+1}\rangle$, it only remains to show that $(e_jA)/(\alpha^{m})\in \text{thick}\langle C_1,\dots,C_{n+1},\rangle$, for $m=1,\dots,n+1$. (Note that this statement includes the projective modules $P_i=e_iA=(e_i\Pi(\tilde{A}_n))/(\alpha^{n+1})$.)

We start by considering the following projective resolutions:
\[
\begin{tikzcd}
(e_iA)/(\alpha^m)\quad\simeq\quad P_i&\arrow{l}[swap]{\alpha^m}P_{i-m}& \arrow{l}[swap]{\alpha^{n+1-m}}P_i&\arrow{l}[swap]{\alpha^m}P_{i-m}&\arrow{l}[swap]{\alpha^{n+1-m}}\cdots
\end{tikzcd}
\]
In particular, when $m=1$ this gives a projective resolution of $C_i$, and we denote the corresponding complex by $D^1_i$. We consider the iterated cones
\[
D_i^{m+1}=\text{cone}(b_{m,i}:\text{cone}(b_{m-1,i}(\cdots \text{cone}(b_{1,i}:D^1_{i}\rightarrow D^1_{i-1})\cdots)\rightarrow D_{i-m}^1)
\]
obtained using the morphisms $b_{m,i}$, given by
\[
\begin{tikzcd}
P_{i-m+1}\oplus\cdots \oplus P_{i-1}\oplus P_{i}\arrow{d} &\arrow{l}[swap]{d_1} P_{i-m}\oplus\cdots \oplus P_{i-2}\oplus P_{i-1}\arrow{d}{f_1}&\arrow{l}[swap]{d_2} P_{i-m+1}\oplus\cdots \oplus P_{i-1}\oplus P_{i}\arrow{d}{f_2}&\arrow{l}[swap]{d_1}\cdots \\
    0&\arrow{l}P_{i-m}&\arrow{l}[swap]{\alpha}P_{i-m-1}&\arrow{l}[swap]{\alpha^n}\cdots
\end{tikzcd}
\]
alternating between the two maps $f_{1;m,i}=f_1$ and $f_{2;m,i}=f_2$ defined by
\begin{align*}
f_1(x_{i-m},\dots,x_{i-1})&=x_{i-m}, \\
f_2(x_{i-m+1},\dots,x_{i})&=-\alpha^{n-1}x_{i-m+1}-\cdots-\alpha^{n-m+1}x_{i-1}+\alpha^{n-m}x_i.
\end{align*}
Here the differentials $d_1=d_{1;m,i}$ and $d_{2;m,i}=d_2$ are defined recursively from the maps $f_1,f_2$, in the definition of the iterated cones $D_i^{m+1}$ above. On $D_i^{m}$, they are given by
\[
d_1=\begin{bmatrix}
    -\alpha&1&0&\cdots& 0&0 \\
    0&-\alpha&1&\dots&0&0\\
    0&0&-\alpha&\cdots&0&0 \\
    &\ddots&&\ddots&&\vdots\\
    0&0&0&\cdots&-\alpha&1 \\
    0&0&0&\cdots&0&\alpha
\end{bmatrix},
\qquad
d_2=\begin{bmatrix}
    -\alpha^{n}&-\alpha^{n-1}&\cdots&-\alpha^{n-m+2}&\alpha^{n-m+1}\\
    0&-\alpha^{n}&\cdots&-\alpha^{n-m+3}&\alpha^{n-m+2} \\
    0&0&\cdots&-\alpha^{n-m+4}&\alpha^{n-m+3} \\
    &\ddots&&\ddots&\vdots\\
    0&0&\cdots&-\alpha^{n}&\alpha^{n-1} \\
    0&0&\cdots&0&\alpha^{n}
\end{bmatrix}.
\]

To see that $e_iA/(\alpha^m)$ is a summand of $D^m_i$ (which is all we need for split-generation) we consider the composition of $g_{m,i}=g$ and $h_{m,i}=h$ below:
\[
\begin{tikzcd} 
\arrow{d}{g_1}P_i&\arrow{d}{g_2}\arrow{l}[swap]{\alpha^m} P_{i-m}&\arrow{l}[swap]{\alpha^{n+1-m}} P_i\arrow{d}{g_1}&\arrow{l}[swap]{\alpha^{m}}\cdots \\
\arrow{d}{h_1}P_{i-m+1}\oplus\cdots\oplus P_i&\arrow{d}{h_2}\arrow{l}[swap]{d_1}P_{i-m}\oplus\cdots\oplus P_{i-1}&\arrow{d}{h_1}\arrow{l}[swap]{d_2}P_{i-m+1}\oplus\dots\oplus P_{i}&\arrow{l}[swap]{d_1}\cdots \\
P_i&\arrow{l}[swap]{\alpha^m} P_{i-m}&\arrow{l}[swap]{\alpha^{n+1-m}} P_i&\arrow{l}[swap]{\alpha^{m}}\cdots
\end{tikzcd}
\]
A simple verification shows that $g$ and $h$ define chains maps and that $h\circ g=\text{Id}$, where $h$ and $g$, respectively, are given by
\begin{align*}
g_1(x_{i})&=(0,\dots,0,x_i) \\
g_2(x_{i-m})&=(x_{i-m},\alpha x_{i-m},\dots,\alpha^{m-1}x_{i-m}) \\
h_1(x_{i-m+1},\dots,x_i)&=-\alpha^{m-1}x_{i-m+1}-\cdots-\alpha x_{i-1}+x_i \\
h_2(x_{i-m},\dots,x_{i-1})&=x_{i-m}.
\end{align*}

\end{proof}
Next, let $B_T$ be the $A_{\infty}$-algebra defined as follows.
\begin{definition}
We define a graded quiver $Q_B^T$, by adding loops $T_i$ of degree $|T_i|=-2$ to each vertex $i\in Q_0$ of the quiver $Q_B$ in Figure~\ref{sgquiver}, and changing the grading so that $|a_{n+1}|=0$ and $|b_{b+1}|=-1$. As an algebra, we let $B_T$ be given as the quotient of the path algebra, where we impose the same relations as in $B$ (see Definition~\ref{sgquiveralgdef1}), and furthermore the commutator relations for $T_i$: $T_{i-1}b_{i}=b_{i-1}T_i$ and $T_{i-1}a_{i}=a_{i-1}T_i$. The $A_{\infty}$-structure is defined by $\mu_2$ being the product (concatenation), $\mu_{n+1}(T^{r_{n+1}}a^{s_{n+1}}b_{j_{n+1}},\dots,T^{r_{1}}a^{s_{1}}b_{j_{1}})=a^{\sum_is_i}T^{1+\sum_ir_i}$ and $\mu_i=0$, if $i\not=2,n+1$. The $\mu_{n+1}$ should be understood as only being non-zero when the input is concatenable, i.e. the $T^{r_{i}}a^{s_{i}}b_{j_{i}}$ ends in the same vertex as $T^{r_{i+1}}a^{s_{i+1}}b_{j_{i+1}}$ starts, and when each input contains a $b_i$; then the output starts in the same vertex as $T^{r_{1}}a^{s_{1}}b_{j_{1}}$ and ends in the same vertex as $T^{r_{n+1}}a^{s_{n+1}}b_{j_{n+1}}$. Another way to describe this is by putting $\mu_{n+1}(b_{i-1},\dots,b_{i+1},b_{i})=T_{i}$, and extending linearly over words in $a_i$ and $T_i$.
\end{definition}
Let $C=\bigoplus_{i=1}^{n+1}C_i$, with $C_i$ as in Lemma~\ref{sggeneration}. For each $C_i$, recall the minimal projective resolution above:
\[
\begin{tikzcd}
D_i=D^1_i\quad=\quad P_i&\arrow{l}[swap]{\alpha} P_{i-1}&\arrow{l}[swap]{\alpha^n} P_i&\arrow{l}[swap]{\alpha}\cdots
\end{tikzcd}
\] 
Let $D=\bigoplus_{i=1}^{n+1}D_i$. Then $\text{End}_A(D)$, the algebra of all $A$-morphisms $f:D\rightarrow D$ of chain complexes of modules (not just chain maps), is a dg-algebra with differential given on homogeneous maps $f$, by $\partial=\partial_{\mathcal{B}}(f)=\partial_{D}\circ f+(-1)^{|f|}f\circ \partial_{D}$, where $\partial_{D}$ is the alternating differential on $D$, as shown on $D_i$ above. By definition, $\text{Ext}^n_A(C,C)=H_n(\text{End}_A(D))$.
\begin{lemma}\label{rhomhomology}
We have
\[
H_{\bullet}(\textup{End}_A(D))\cong B_T,
\]
as algebras.
\end{lemma}
\begin{proof} 
By definition, the $n$:th homology of $\text{End}_A(D)$ can be computed as the degree $n$ chain maps modulo the null-homotopic ones. We look at the homogeneous parts separately, where we can describe the chain maps as follows, dividing into the four cases non-negative even degree, positive odd degree, negative even degree and negative odd degree.

Suppose first that $f:D_i\rightarrow D_j$ is of non-positive even degree. Then, using only that it is a morphism of projective $A$-modules in each degree, it is of the form:
\[
\begin{tikzcd}
P_i\arrow{d}&\arrow{l}[swap]{\alpha}P_{i-1}\arrow{d}&\arrow{l}[swap]{\alpha^n}\cdots&\arrow{l}[swap]{\alpha}P_{i-1}\arrow{d}&\arrow{l}[swap]{\alpha^n}P_i\arrow{d}{f_1}&\arrow{l}[swap]{\alpha}P_{i-1}\arrow{d}{f_2}&\arrow{l}[swap]{\alpha^n}P_{i}\arrow{d}{f_3}&\arrow{l}[swap]{\alpha}\cdots \\
0&\arrow{l}0&\arrow{l}\cdots&\arrow{l}0&\arrow{l}P_{j}&\arrow{l}[swap]{\alpha}P_{j-1}&\arrow{l}[swap]{\alpha^n}P_{j}&\arrow{l}[swap]{\alpha}\cdots
\end{tikzcd}
\]
We can write $f_k=\sum_{a_k,b_k}\lambda_{a_k,b_k}\alpha^{a_k}(\alpha^*)^{b_k}$ (where the morphism is given by left multiplication by this element), where each sum is over all pairs $a_k,b_k$ such that $a_k-b_k$ is equivalent to $j-i$ modulo $n+1$, and $\lambda_{a_k,b_k}\not=0$ for finitely many pairs $a_k,b_k$. The diagram commutes if and only if the following conditions hold: For $k$ odd, we have $\lambda_{a_k,b_k}=\lambda_{a_{k+1},b_{k+1}}$ for each pair $a_k=a_{k+1}$, $b_k=b_{k+1}$ and $a_k\not=n$; for $k$ even, for all pairs $a_k,b_k$ with $a_k=0$, we have $\lambda_{a_k,b_k}=\lambda_{a_{k+1},b_{k+1}}$, for each pair $a_k=a_{k+1}$ and $b_k=b_{k+1}$. From this it follows that we can write $f=f^0+f'$, where $f^0_k=\sum_{b_k}\lambda_{0,b_k}(\alpha^*)^{b_k}$, with $\lambda_{0,b_k}=\lambda_{0,b_{l}}$ when $b_k=b_l$, and where $f'=\sum_{a_k,b_k}\lambda_{a_k,b_k}\alpha^{a_k}(\alpha^*)^{b_k}$, with $a_k\geq 1$.

We next define a homotopy $H$ showing that $f'\simeq 0$, i.e. is null-homotopic (see diagram below). We write $f'_k=(f'_k)^n+(f'_k)^{<n}$, where $(f'_k)$ is the part with $a_k=n$. Recall that, if $k$ is odd, by the previous paragraph $(f'_k)^{<n}=(f'_{k+1})^{<n}$. We define $H$ inductively as follows: We put $H_1=f_1'/\alpha$, $H_2=-(f_1')^n/\alpha^n$, for $k\geq 3$ odd $H_k=f_k'/\alpha-\alpha^{n-1}H_{k-1}$ and, for $k\geq 4$ even, $H_k=(f'_k)^n-(f'_{k-1})^n+H_{k-2}$. (This is well-defined because the sum is finite in each degree.)
\[
\begin{tikzcd}
P_i\arrow{d}&\arrow{l}[swap]{\alpha}P_{i-1}\arrow{d}&\arrow{l}[swap]{\alpha^n}\cdots&\arrow{l}[swap]{\alpha}P_{i-1}\arrow{d}\arrow{dr}{0}&\arrow{l}[swap]{\alpha^n}P_i\arrow{d}{f_1'}\arrow{dr}{H_1}&\arrow{l}[swap]{\alpha}P_{i-1}\arrow{d}{f_2'}\arrow{dr}{H_2}&\arrow{l}[swap]{\alpha^n}P_{i}\arrow{d}{f_3'}\arrow{dr}{H_3}&\arrow{l}[swap]{\alpha}\cdots \\
0&\arrow{l}0&\arrow{l}\cdots&\arrow{l}0&\arrow{l}P_{j}&\arrow{l}[swap]{\alpha}P_{j-1}&\arrow{l}[swap]{\alpha^n}P_{j}&\arrow{l}[swap]{\alpha}\cdots
\end{tikzcd}
\]

Secondly, we consider a map of negative odd degree:
\[
\begin{tikzcd}
P_i\arrow{d}&\arrow{l}[swap]{\alpha}P_{i-1}\arrow{d}&\arrow{l}[swap]{\alpha^n}\cdots&\arrow{l}[swap]{\alpha^n}P_{i}\arrow{d}&\arrow{l}[swap]{\alpha}P_{i-1}\arrow{d}{f_1}&\arrow{l}[swap]{\alpha^n}P_{i}\arrow{d}{f_2}&\arrow{l}[swap]{\alpha}P_{i-1}\arrow{d}{f_3}&\arrow{l}[swap]{\alpha^n}\cdots \\
0&\arrow{l}0&\arrow{l}\cdots&\arrow{l}0&\arrow{l}P_{j}&\arrow{l}[swap]{\alpha}P_{j-1}&\arrow{l}[swap]{\alpha^n}P_{j}&\arrow{l}[swap]{\alpha}\cdots
\end{tikzcd}
\]
Similarly to above, we write $f_k=\sum_{a_k,b_k}\lambda_{a_k,b_k}\alpha^{a_k}(\alpha^*)^{b_k}$. For this to be chain map we note that the following must hold. For odd $k$, and a pair $a_k,b_k$, if $a_k=0$, then $\lambda_{a_k,b_k}=\lambda_{a_{k+1},b_{k+1}}$, when $a_{k+1}=n-1$ and $b_k=b_{k+1}$. For even $k$, we must have $\lambda_{a_k,b_k}=0$, when $a_k<n-1$. 

Therefore, we can write $f=f^0+f'$ where $f_k^0=\sum_{b_k}\lambda_{0,b_k}(\alpha^*)^{b_k}$, when $k$ is odd, and \\$f^0_k=\sum_{b_k}\lambda_{b_k}\alpha^{n-1}(\alpha^*)^{b_k}$, when $k$ is even; $f'_k=\sum_{b_k}\lambda_{n,b_k}\alpha^n(\alpha^*)^{b_k}$, when $k$ is even, and $f'_k=\sum_{a_k,b_k}\lambda_{a_k,b_k}\alpha^{a_k}(\alpha^*)^{b_k}$, such that $a_k\geq 1$, when $k$ is odd. We can now construct a homotopy $H$ showing that $f'\simeq 0$ (see the diagram below): Let $H_1=f_1'/\alpha$ and $H_2=f_2'/\alpha^n-f'_1/\alpha$. For $k\geq 3$ odd, we put $H_k=f_3/\alpha-H_{k-1}$, and for $k\geq 4$ even, we put $H_k=f_k/\alpha^n-H_{k-1}$
\[
\begin{tikzcd}
P_i\arrow{d}&\arrow{l}[swap]{\alpha}P_{i-1}\arrow{d}&\arrow{l}[swap]{\alpha^n}\cdots&\arrow{l}[swap]{\alpha^n}P_{i}\arrow{d}\arrow{dr}{0}&\arrow{l}[swap]{\alpha}P_{i-1}\arrow{d}{f_1'}\arrow{dr}{H_1}&\arrow{l}[swap]{\alpha^n}P_{i}\arrow{d}{f_2'}\arrow{dr}{H_2}&\arrow{l}[swap]{\alpha}P_{i-1}\arrow{d}{f_3'}\arrow{dr}{H_3}&\arrow{l}[swap]{\alpha^n}\cdots \\
0&\arrow{l}0&\arrow{l}\cdots&\arrow{l}0&\arrow{l}P_{j}&\arrow{l}[swap]{\alpha}P_{j-1}&\arrow{l}[swap]{\alpha^n}P_{j}&\arrow{l}[swap]{\alpha}\cdots
\end{tikzcd}
\]

The third and fourth cases, positive degree maps of even respectively odd degrees are likewise shown to be null-homotopic. This is because in the decomposition $f=f^0+f'$, the map $f^0$ must be equal to $0$, otherwise $f$ will not define a chain map (the first square with $0\leftarrow P_i$ in the top row will not commute); that $f'$ is null-homotopic is shown in the exact same way as above. 

We can conclude from this that the chain maps that are not null-homotopic can all be described as linear combinations of compositions of the following: the degree $0$ map $D_i\rightarrow D_{i-1}$ given by multiplication by $\alpha_{i-1}^*$, which we identify with $a_i$; the degree $-1$ map $b_i=b_{1,i}$, as described in the previous proof; and the degree $-2$ map $T_i:D_i\rightarrow D_i$, given by 
\[
\begin{tikzcd}
P_i\arrow{d}&\arrow{l}[swap]{\alpha}P_{i-1}\arrow{d}& \arrow{l}[swap]{\alpha^{n}}P_i\arrow{d}{\text{id}}&\arrow{l}[swap]{\alpha}P_{i-1}\arrow{d}{\text{id}}&\arrow{l}[swap]{\alpha^n}\cdots \\
0&\arrow{l}0&\arrow{l}P_{i}&\arrow{l}[swap]{\alpha}P_{i-1}&\arrow{l}[swap]{\alpha^n}\cdots
\end{tikzcd}
\]
All these maps commute. 

Indeed, suppose first that $f^0$, as above is of even degree, say $2l$. Then we can assume that $f^0_k$ is of the form $x\mapsto \lambda (\alpha^*)^bx$, for some number $b\geq 0$ (by looking at the possible summands separately). Then $f^0=T^k_{i-b}a_{i-b+1}\cdots a_{i-1}a_i$. Suppose next that $f^0$ is of odd degree, say $2l+1$. Then we can assume that $f^0_k$ is of the form 
\[
x\mapsto
\begin{cases}
\lambda(\alpha^*)^b &\mbox{ if }k\text{ is odd}, \\
\lambda\alpha^n(\alpha^*)^b &\mbox{ if }k\text{ is even}.
\end{cases}
\]
In this case $f^0=T^l_{i-b-1}b_{i-b}a_{i-b+1}\cdots a_i$.

There can furthermore be no more relations among them, except in the case $n=1$, where $b_{i+1}b_i=T_i$ (multiplication is induced by composition). Therefore there is an algebra-isomorphism $H_{\bullet}(\text{End}_A(D))\cong B_T$. 
\end{proof}
\begin{lemma}\label{sgend}
We can upgrade the isomorphism in Lemma~\ref{rhomhomology} to an isomorphism
\[
\textup{Ext}^{\bullet}_A(C,C)\cong B_T,
\]
as $A_{\infty}$-algebras.
\end{lemma}
\begin{proof}
The $A_{\infty}$-algebra $\text{Ext}_A^{\bullet}(C,C)=H_{\bullet}(\text{RHom}_A(C,C))$ is by definition the minimal model of the dg-algebra $\mathcal{B}=\text{End}_{A}(D)$. By \cite{MR2533303} and \cite{MR1672242} (here following the exposition of these results in \cite[Section 6, Theorem 17]{MR4432239}), we can find the $A_{\infty}$-structure in the following way. We choose a graded decomposition $\mathcal{B}=H\oplus\text{im}(\partial)\oplus L$ as graded vector spaces, where $H_k$ is a complement of $\text{im}(\partial_{k+1})$ in the cycles $\text{ker}(\partial_k)\subset\mathcal{B}$ and $L_k$ is a complement of the cycles in $\mathcal{B}_k$. We also pick an isomorphism $\iota:\text{Ext}_A^{\bullet}(C,C)\rightarrow H$, and we let $\pi:\mathcal{B}\rightarrow \text{Ext}_A^{\bullet}(C,C)$ denote the composition $\iota^{-1}\circ\pi'$, where $\pi':\mathcal{B}\rightarrow H$ is the projection. Given the decomposition, note that $\partial_k$ maps $H_k\oplus\text{im}(\partial_{k+1})\oplus L_k$ onto $\text{im}(\partial_k)$ with kernel $H_k\oplus\text{im}(\partial_{k+1})$. Thus, we can define an "inverse" $h$ to $\partial$ by $h_{|H_k\oplus L_k}=0$ and $h_{|\text{im}(\partial_{k})}=(\partial_k)^{-1}$. Using this, inductively, for $m\geq 1$, we define $\lambda_n:\mathcal{B}^{\otimes m}\rightarrow\mathcal{B}$ by
\[
h\lambda_1=-\text{id}_{\mathcal{B}},\qquad \lambda_2=\text{composition},\qquad\lambda_n=\sum_{\substack{i+j=n  \\ i,j\geq 1}}(-1)^{j+1}\lambda_2(h\lambda_i\otimes h\lambda_j).
\]
Then, for $m\geq 2$, $\mu_m=\pi\circ\lambda_n\circ \iota$ defines the $A_{\infty}$-structure.

To make it more readable we skip the indices, and note that they are implicitly determined. In the above situation, we choose a decomposition with $H_k=\{T^ra^sb\}$, for $k=2r+1$ odd, and $H_k=\{T^ra^s\}$, for $k=2r$ even. Note that $b^2$ is given by multiplication with $\alpha^{n-1}$ composed with $T$. We define $h(Tb^2)$ to be the morphism below:
\[
\begin{tikzcd}
P_i\arrow{d}&\arrow{l}[swap]{\alpha^n}P_{i-1}\arrow{d}{0}&\arrow{l}[swap]{\alpha^{n}}P_i\arrow{d}{\alpha^{n-2}}&\arrow{l}[swap]{\alpha}P_{i-1}\arrow{d}{0}&\arrow{l}[swap]{\alpha^n}P_i\arrow{d}{\alpha^{n-2}}&\arrow{l}[swap]{\alpha}\cdots \\
0&\arrow{l}P_{i-2}&\arrow{l}[swap]{\alpha}P_{i-3}&\arrow{l}[swap]{\alpha^n}P_{i-2}&\arrow{l}[swap]{\alpha}P_{i-3}&\arrow{l}[swap]{\alpha^n}\cdots
\end{tikzcd}
\]
Then $\lambda_2(h\lambda_1(b),h\lambda_2(b,b))$ is given by
\[
\begin{tikzcd}
P_i\arrow{d}&\arrow{l}[swap]{\alpha}P_{i-1}\arrow{d}& \arrow{l}[swap]{\alpha^{n}}P_i\arrow{d}{\alpha^{n-2}}&\arrow{l}[swap]{\alpha}P_{i-1}\arrow{d}{0}&\arrow{l}[swap]{\alpha^n}P_i\arrow{d}{\alpha^{n-2}}&\arrow{l}[swap]{\alpha}\cdots \\
0&\arrow{l}0&\arrow{l}P_{i-3}&\arrow{l}[swap]{\alpha}P_{i-4}&\arrow{l}[swap]{\alpha^n}P_{i-3}&\arrow{l}[swap]{\alpha}\cdots
\end{tikzcd}
\]
and $\lambda_2(h\lambda_2(b,b),h\lambda_1(b))$ by
\[
\begin{tikzcd}
P_i\arrow{d}&\arrow{l}[swap]{\alpha}P_{i-1}\arrow{d}& \arrow{l}[swap]{\alpha^{n}}P_i\arrow{d}{0}&\arrow{l}[swap]{\alpha}P_{i-1}\arrow{d}{\alpha^{n-2}}&\arrow{l}[swap]{\alpha^n}P_i\arrow{d}{0}&\arrow{l}[swap]{\alpha}\cdots \\
0&\arrow{l}0&\arrow{l}P_{i-3}&\arrow{l}[swap]{\alpha}P_{i-4}&\arrow{l}[swap]{\alpha^n}P_{i-3}&\arrow{l}[swap]{\alpha}\cdots
\end{tikzcd}
\]

Thus, $\lambda_3(b,b,b)$ is the sum of two alternating maps, and is given by multiplication by $\alpha^{n-2}$ composed with $T$. 

Next, we let $h(\lambda_3(b,b,b))=h(Tb^3)$ be given by 
\[
\begin{tikzcd}
P_i\arrow{d}&\arrow{l}[swap]{\alpha}P_{i-1}\arrow{d}{0}&\arrow{l}[swap]{\alpha^{n}}P_i\arrow{d}{\alpha^{n-3}}&\arrow{l}[swap]{\alpha}P_{i-1}\arrow{d}{0}&\arrow{l}[swap]{\alpha^n}P_i\arrow{d}{\alpha^{n-3}}&\arrow{l}[swap]{\alpha}\cdots \\
0&\arrow{l}P_{i-3}&\arrow{l}[swap]{\alpha}P_{i-4}&\arrow{l}[swap]{\alpha^n}P_{i-3}&\arrow{l}[swap]{\alpha}P_{i-4}&\arrow{l}[swap]{\alpha^n}\cdots
\end{tikzcd}
\]
If we compute $\lambda_4$, it will be the sum of two alternating maps $\lambda_2(h\lambda_1(b)\otimes h\lambda_3(b,b,b))$ and \\$\lambda_2(h\lambda_3(b,b,b)\otimes h\lambda_1(b))$ (and one map $\lambda_2(h\lambda_2(b,b)\otimes h\lambda_2(b,b))=0$). It will be equal to multiplication by $\alpha^{n-3}$ composed with $T$.

We can continue and define $h$ in the analogous way (starting with $0$ and alternating with the appropriate power of $\alpha$). Then $h\lambda_i(b,\dots,b)\otimes h\lambda_j(b\dots,b)=0$, for $i,j\geq 2$, and, inductively, 
\[
\lambda_j(b,\dots,b)=h\lambda_1(b)\otimes h\lambda_{j-1}(b,\dots,b)+h\lambda_{j-1}(b,\dots,b)\otimes h\lambda_1(b),
\]
which is then equal to multiplication by $\alpha^{n+1-j}$ composed with $T$.

From this we deduce that $\mu_j(b,\dots,b)=0$, for $2<j<n+1$. Moreover, we deduce that $\mu_{n+1}(b,\dots,b)=T$.

As $h\lambda_{n+1}(b,\dots,b)=0$, this process will stop, and we will have $\lambda_j(b,\dots,b)=0$, if $j>n+1$.

Likewise, if we replace $b$ with any word $w$ in $a$ and $T$, then $h(w)=0$. Hence, $\lambda_j(\dots,w,\dots)=0$.

Finally, if we replace $b$ by any word of the form $T^ra^sb$, then $T^ra^s$ can always be factored out of the equations defining $\lambda_j$. We conclude the formula described right above the lemma.
\end{proof}
In particular, $D^b(A)\cong D^{\pi}(B_T)$, but we are interested in the singularity category. Note especially, in the proof of Lemma~\ref{sggeneration}, that the projective $A$-modules are quasi-isomorphic to the twisted complexes $D_i^{n+1}$. Under this isomorphism these twisted complexes corresponds to objects in $D^{\pi}(B_T)$ are quasi-isomorphic to the twisted complexes over $B_T$ given by
\[
P_i\simeq S_i=(e_{i+1}B_T\oplus\cdots\oplus e_{i-1}B_T\oplus e_{i}B_T,f_i),\qquad f_i=
\begin{bmatrix}
0&b_{i+2}&0&\cdots&0\\
0&0&b_{i+3}&\cdots&0 \\
&\ddots&&\ddots \\
0&0&0&\cdots &b_{i}\\
0&0&0&\cdots&0
\end{bmatrix}.
\] 
Since $R$ is semi-simple, the indecomposable projective modules $P_1,\dots,P_{n+1}$ corresponding to the nodes generate $\text{Perf}(A)$. We deduce that \[
D_{\text{sg}}(A)\simeq D^{\pi}(B_T)/D^{\pi}(S_1,\dots,S_{n+1})\simeq D^{\pi}(B_T/\langle S_1,\dots,S_{n+1}\rangle),
\]
where the middle quotient is the Verdier quotient, and the last quotient is the $A_{\infty}$-quotient (see Section~\ref{algprelAinf}) due to \cite{MR2259271}, where the second quasi-isomorphism is also established.

Let $B_{T^{\pm}}$ denote the $A_{\infty}$-algebra obtained by adding $T_{i}^{-1}$ at all nodes, i.e., adding loops to the underlying quiver and imposing the extra relations $T_i^{-1}T_i=e_i=T_iT_i^{-1}$, extending the commutating relations with $b_i$ and $a_i$, as well as the linearity under $\mu_{n+1}$. 
\begin{lemma}
The $A_{\infty}$-quotient of $B_T$ at the twisted complexes $S_1,\dots,S_{n+1}$ is quasi-isomorphic to the $A_{\infty}$-algebra $B_{T^{\pm}}$.
\end{lemma}
\begin{proof}
We have an inclusion $\iota:B_T\rightarrow B_{T^{\pm}}$ (which is an $A_{\infty}$-morphism with $f_1=\iota$, $f_k=0$, $k\geq 2$). This defines a functor on the $A_{\infty}$-categories of twisted complexes, $\text{Tw}\,B_T\rightarrow\text{Tw}\,B_{T^{\pm}}$, which we here also denote by $\iota$ (see Section~\ref{algprelAinf}). Thus, there is a twisted complex $\iota(S_i)$ over $B_{T^{\pm}}$ for each $S_i$. The following calculation shows that
\[
\mu_1^{\iota(S_i)}(g_i)=\mu_1^{\text{Tw}\,B_{T^{\pm}}}(g_i)=\text{Id}_{\iota(S_i)},\quad\text{where}\quad g_i=\begin{bmatrix}
0&0&\cdots&0
\\
\vdots&\vdots&\ddots&\vdots \\
0&0&\cdots& 0 \\
b_{i+1}T_{i+1}^{-1}&0&\cdots&0
\end{bmatrix}.
\]
Indeed,
\[
\mu_1^{\text{Tw}\,B_{T^{\pm}}}(g_i)=\sum\mu_{n+1}^{\text{Add}_{B_{T^{\pm}}}}(\iota(f_i),\dots,\iota(f_i),g_i,\iota(f_i),\dots,\iota(f_i)),
\]
where the sum is over all different positions for $g_i$. On components, evaluating where $g_i$ is in the $l$:th position, for $l=1,\dots,n+1$,
\[
\left(\mu_{n+1}^{\text{Add}_{B_{T^{\pm}}}}(\iota(f_i),\dots,\iota(f_i),g_i,\iota(f_i),\dots,\iota(f_i))\right)_{s,t}
\]
\[
=\begin{cases}
\mu_{n+1}^{B_{T^{\pm}}}(b_{i+l-1},\dots,b_iT_i^{-1},\dots,b_{i+l})=e_{i+l}, &\mbox{if }s=t=n+2-l, \\
0,&\mbox{otherwise}
\end{cases}
\]
(which is most easily seen from multiplying matrices).

Therefore, by \cite[Section 3]{MR2259271} or, equivalently, by the universal property in \cite[Theorem 4.13]{oh2024infinitycategoricaluniversalpropertiesquotients}, there is a morphism $h=(h_n)_{n\geq 1}$ of $A_{\infty}$-algebras making the diagram commute:
\[
\begin{tikzcd}
B_T\arrow{r}{\iota}\arrow{d}&B_{T^{\pm}}\\
B_T/\langle S_1,\dots,S_{n+1}\rangle\arrow[dashed]{ur}[swap]{\exists h}
\end{tikzcd}
\]
We need to show that $h$ defines a quasi-isomorphism. 

We can consider the derived localisation $L_T(B_T)$ of $B_T$ at $T=T_1+\cdots+T_{n+1}$, in the sense of \cite{MR3771137}, replacing it with a quasi-isomorphic dg-algebra. We choose to consider the semi-free dg-algebra replacement given by the cobar-bar construction $\Omega B(B_T)$ of $B_T$; it comes with a canonical quasi-isomorphism $\Omega B(B_T)\rightarrow B_T$. We have an inclusion of dg-algebras $\Omega B(B_T)\hookrightarrow L_T(\Omega B(B_T))$ and $H_{\bullet}(L_T(B_T))=H_{\bullet}(L_T(\Omega B(B_T)))=\tilde{B}_{T^{\pm}}$, where $\tilde{B}_{T^{\pm}}$ by definition is the minimal model of $L_T(B_T)=L_T(\Omega B(B_T))$. Note that $\tilde{B}_{T^{\pm}}$ is equal to $B_{T^\pm}$ as algebras by \cite[Theorem 5.3]{MR3771137}, since $T$ is a central element in homology, but not necessarily as  $A_{\infty}$-algebras. However, there is a map of $A_{\infty}$-algebras $i=(i_k)_{k\geq 1}$ induced by localisation $B_T\rightarrow\tilde{B}_{T^{\pm}}$ such that $i_1$ is the inclusion: this is obtained by taking the composition $B_T\rightarrow \Omega B(B_T)\rightarrow L_T(\Omega B(B_T))\rightarrow\tilde{B}_{T^{\pm}}$, where the first and the last maps are $A_{\infty}$-quasi-isomorphisms whose arity-$1$ parts are cycle choosing respectively projection to homology class. 

Consider next the cobar-bar construction also of the $A_{\infty}$-algebra $B_{T^{\pm}}$, $\Omega B(B_{T^{\pm}})$, we then have the commutative diagram, where the vertical arrows are quasi-isomorphisms:
\[
\begin{tikzcd}
\Omega B(B_T) \arrow[hook]{r}{\Omega B(\iota)} \arrow{d}  &\Omega B(B_{T^{\pm}}) \arrow{d}\\
B_{T}\arrow[hook]{r}[swap]{\iota}&  B_{T^{\pm}}
\end{tikzcd}
\] 
By the universal property \cite[Definition 3.3]{MR3771137} there is a morphism of dg-algebras, making the diagram below commute, because $\Omega B(\iota)(T')$ is invertible in homology, for any representative $T'$ of the homology class $T$:
\[
\begin{tikzcd}
\Omega B(B_T)\arrow{r}{\Omega B(\iota)}\arrow{d} & \Omega B(B_{T^{\pm}}) \\
L_T(\Omega B(B_T))\arrow[dashed]{ur}
\end{tikzcd}
\]

Going back to the level of minimal modeles, this induces a morphism of $A_{\infty}$-algebras $j=(j_k)_{k\geq 1}$ such that the diagram below commutes:
\[
\begin{tikzcd}
B_T\arrow{r}{\iota}\arrow{d}{i}&B_{T^{\pm}} \\
\tilde{B}_{T^{\pm}}\arrow{ur}[swap]{j}
\end{tikzcd}
\]
By commutativity, $j_1$ is the identity restricting to $B_T\subset\tilde{B}_{T^{\pm}}$. Using that $\mu_2^{B_{T^{\pm}}}(i_1,i_1)=i_1\circ\mu_2^{\tilde{B}_{T^{\pm}}}$, which holds as $\mu_1^{B_{T^{\pm}}}=0$ and $\mu_1^{\tilde{B}_{T^{\pm}}}=0$, we deduce that $j_1(T_i^{-1})=T_i^{-1}$. Thus, $j_1$ must be the identity, and in particular, $j$ is an isomorphism. Hence, we will identify $\tilde{B}_{{T^{\pm}}}$ (which is defined up to $A_{\infty}$-quasi-isomorphism) with $B_{T^{\pm}}$, and $L_T(\Omega B(B_T))$ wtih $\Omega B(B_{T^{\pm}})$.

By above, derived localisation of $T$ implies killing $S_i$, $i=1,\dots,n+1$, yielding the map $h$. On the other hand, killing all $S_i$ implies localisation of all $T_i$, which follows from looking at the Verdier quotient, using that $\text{cone}(T_i)$ is isomorphic to the perfect complex 
\[
S_i'\simeq P_i\xleftarrow{\alpha}P_{i-1}
\]
(see the description of T in the the proof of Lemma~\ref{rhomhomology}). To more precise, there is a quasi-isomorphism between the dg-algebra $\text{End}_A(D)$ and the $A_{\infty}$-algebra $ B_T$. T twisted complexes $S_i$ corresponds to twisted complexes $S_i'$ under this quasi-isomorphism, and, thus, the dg-quotient \\$\text{End}_A(D)/\langle S_1',\dots,S_{n+1}'\rangle$ (which is the endomorphism dg-algebra, obtained by taking the quotient in the dg-category of twisted complexes over $\text{End}_A(D)$) is quasi-isomorphic to $A_{\infty}$-quotient $B_T/\langle S_1,\dots,S_{n+1}\rangle$. The element $T\in\text{End}_A(D)/\langle S_1',\dots,S_{n+1}'\rangle$ is identified with the corresponding endomorphism in $\text{Tw}\,\text{End}_A(D)/\langle S_1',\dots,S_{n+1}'\rangle$, and its homology class is identified with its class in $H_0\text{Tw}\,\text{End}_A(D)/\langle S_1',\dots,S_{n+1}'\rangle$, which is identified with the morphism also denoted by $T$ in $D^{\pi}(\text{End}_A(D)/\langle S_1',\dots,S_{n+1}'\rangle)\simeq D_{\text{sg}}(A)$.

Therefore, by \cite[Definition 3.3]{MR3771137}, there is dg-morphism making the diagram commute:
\[
\begin{tikzcd}
\Omega B(B_T) \arrow{r}\arrow{d}&L_T(\Omega B(B_T))\arrow[dashed]{dl} \\
\Omega B(\text{End}_A(D)/\langle S_1',\dots,S'_{n+1}\rangle)
\end{tikzcd}
\]
Here, the vertical map is defined as the composition 
\[
\Omega B(B_T)\rightarrow \Omega B(B_T/\langle S_1,\dots,S_{n+1}\rangle)\xrightarrow{\sim_{\text{q.i.}}}\Omega B( \text{End}_A(D)/\langle S_1',\dots,S_{n+1}'\rangle).
\]

Thus, there is an $A_{\infty}$-morphism $\tilde{h}=(\tilde{h_n})_{n\geq 1}:B_{T^{\pm}}\rightarrow B_T/\langle S_1,\dots,S_{n+1}\rangle$, and a commutative diagram:
\[
\begin{tikzcd}
B_T\arrow{r}\arrow{d}&B_{T^{\pm}}\arrow[bend left=20]{dl}{\tilde{h}} \\
B_T/\langle S_1,\dots,S_{n+1}\rangle\arrow{ur}{h}
\end{tikzcd}
\]
The universal property in \cite{MR3771137}, together with the universal property in \cite[Theorem 4.13]{oh2024infinitycategoricaluniversalpropertiesquotients}, furthermore implies that the dg-morphism must be a quasi-inverse. Indeed, 
$\Omega B(h \circ\tilde{h})$ must be equal to $\text{Id}_{L_T(\Omega B(B_T))}$ in the homotopy category of the under category $\Omega B(B_T)$ (in dg-algebras). This implies that $h\circ\tilde{h}\simeq\text{Id}_{B_{T^{\pm}}}$. Moreover, similarly, $h\circ\tilde{h}$ is equal to the identity $\text{Id}_{B_T/\langle S_1,\dots, S_{n+1\rangle}}$ in the $\infty$-category of $A_{\infty}$-categories, which implies that they are homotopic $A_{\infty}$-morphisms. Hence the map induced by $h$ on the minimal model of $B_T/\langle S_1,\dots,S_{n+1}\rangle$ is an $A_{\infty}$-isomorphism.

\end{proof}
\begin{proof}[Proof of Theorem~\ref{dbsgcatalg}]
We can make the substitutions $b_{n+1}\leftrightarrow b_{n+1}T_{n+1}^{-1}$ and $a_{n+1}\leftrightarrow a_{n+1}T_{n+1}^{-1}$. Then $\mu_{n+1}(a_i,a_{i-1},\dots,a_{i+1})=e_{i+1}$, and $\mu_{n+1}$ is still linear in the $b_i$ and $T_i^{\pm 1}$. This shows that $B_{T^{\pm}}\cong B\otimes_{\mathbb{C}}\mathbb{C}[T^{\pm 1}]$, with $|T|=-2$, isomorphic as $A_{\infty}$-algebras, where we tensor with a graded algebra on the right-hand side such that the $A_{\infty}$-relations are extended to be linear in $T^{\pm 1}$.
\end{proof}

\begin{figure}[H]
\centering
\begin{tikzpicture}
    \draw[thick] (0,0) rectangle (12,4.1);
    \node at (1,2.4) {$A$-side};
    \node at (1,0.8) {$B$-side};
    \draw[very thick] (2.5,4.1) to (2.5,0);
    \draw[very thick] (0,3.2) to (12,3.2);

    \node at (5.2,3.6) {$X_{\text{res}}\smallsetminus D_r+\text{stop/potential}$};
    \node at (10,3.6) {$X_{\text{Mil}}$};

    \draw[thick] (0,1.6) -- (12,1.6);
    \draw[thick] (7.9,0) -- (7.9,4.1);

    \node at (5.2,1.1) {$D^{\pi}(\mathcal{G}(A_n)\otimes_{\mathbb{C}}\mathbb{C}[T^{\pm1}])$};
    \node at (5.2,0.5) {Thm. \ref{respot} };'

    \node at (10,2.7) {$D^{\pi}(\mathcal{G}(A_n))$};
    \node at (10,2.1) {\cite{MR3692968}};

    \node at (5.2,2.7) {$D^b(\Pi^{\lambda(s)}(\tilde{A}_n))$?};
    \node at (5.2,2.1) {Conj. \ref{resconj}};

    \node at (10,1.1) {$D^b(\Pi^{\lambda}(\tilde{A}_n))$};
    \node at (10,0.5) {Prop. \ref{nonlocmil}};
\end{tikzpicture}
\caption{Summary of the results in Section~\ref{hmsmil} and Section~\ref{hmsmiltwo}. See Remark~\ref{parameterremark} for the asymmetry in the off-diagonal. When considered in the same underlying hyper-Kähler manifold $X_{\zeta}$, the parameters in the main diagonal define the same algebras when considered as infinitesimal variations of $\Pi(\tilde{A}_n)$ (see Section~\ref{addinfvar}).}
\label{hmsmilfigure}
\end{figure}

\addcontentsline{toc}{section}{References}
\printbibliography

\end{document}